\documentclass{article}
\usepackage[square,comma,numbers,sort&compress]{natbib}

\usepackage{graphicx}
\graphicspath{{figures/}}

\usepackage{amsmath}
\usepackage{amssymb}
\DeclareMathAlphabet{\mathbbb}{U}{bbold}{m}{n} 

\usepackage{xr-hyper}
\usepackage{hyperref}
\newcommand*\sref[1]{S\ref{#1}}

\usepackage{color}
\usepackage[dvipsnames]{xcolor}
\usepackage{cancel}
\usepackage{stmaryrd} 

\usepackage{tikz}
\usepackage[outline]{contour}
\contourlength{1.5pt}

\usepackage{mathtools}
\newcommand{\eqcolon}{\mathrel{\resizebox{\widthof{$\mathord{=}$}}{\height}{ $\!\!=\!\!\resizebox{1.2\width}{0.8\height}{\raisebox{0.23ex}{$\mathop{:}$}}\!\!$ }}}
\newcommand{\coloneq}{\mathrel{\resizebox{\widthof{$\mathord{=}$}}{\height}{ $\!\!\resizebox{1.2\width}{0.8\height}{\raisebox{0.23ex}{$\mathop{:}$}}\!\!=\!\!$ }}}

\usepackage{amsthm}
\newtheorem{theorem}{Theorem}
\newtheorem{definition}{Definition}
\newtheorem{lemma}[theorem]{Lemma}
\newtheorem{proposition}[theorem]{Proposition}
\newtheorem*{remark}{Remark}

\usepackage[left=2.cm,right=2.cm,top=2.cm,bottom=2cm]{geometry}

\title{Towards a Geodesic Flow Bundle Formalism of General Relativity: Riemannian case}

\author{Adrian Bo\^itier and Shubhanshu Tiwari}

\date{\today}

\begin{document}

\maketitle

\begin{abstract}
Gravity is a phenomenon which arises due to the space-time geometry. The main equations that describe gravity are the Einstein equations. To understand the consequences of these field equations we need to calculate the free particle worldlines to the geometries, which solve these field equations e.g. the Schwarzschild metric solves the Einstein equations and we would need to solve the geodesic equations to this metric. If we were to describe the space-time geometry in terms of geodesics instead of the metric, we could skip the step of solving for the metric and solve for the geodesics directly. In this work we have developed a formalism doing that: we use the bundle of the arclength parametrized geodesics (geodesic flow bundle GFB) from all points in the manifold to describe a Riemannian geometry.\\
Our formalism uses infinitesimal spherical triangles as generating elements, to solve geometric problems. We relate the geodesic flow bundle to Gaussian curvature and develop a method to calculate the geodesics geometrically starting from a Gaussian curvature field. The result amounts to a generalization of the cosine- and sine-laws for constant curvature to varying curvature fields.\\
In this work we restrict ourselves to the Riemannian case and positive curvature fields. We expand the triangulation problem in power series and calculate the main result up to second order. The method itself could be extended to treat more generic cases in particular pseudo-Riemannian geometries which relate directly to Einsteins equations.\\

We test our results against the sphere and perform consistency checks for some examples of varying curvature fields.\\
As a by-product we generalize the notion of integration to products and derive a relation analogue to the main theorem of calculus, for product integrals.
\end{abstract}

\tableofcontents

\section{Introduction}
The question, whether general relativity predicts gravitational waves, is an old one and is rigorously answered for the case of the cosmological constant being 0. However, the current observations suggest that the universe is in an accelerated expansion which is best described by the standard model of cosmology ($\Lambda$-CDM) where the cosmological constant ($\Lambda$) is positive \cite{Planck:2018vyg}. Moreover, we now have multiple observations of gravitational wave (GW) events \cite{LIGOScientific:2021djp} experimentally proving their existence. These observational results motivate strongly the theoretical exploration of gravitational waves with a positive $\Lambda$. In this work, we lay the basic framework which might be helpful in proving the existence of GW with a positive $\Lambda$.\\

Before such a proof could be attempted one should find an invariant definition of a gravitational wave, since coordinate artifacts can be (and were) confused with GW as the discussion in~\cite{GWMathematics} shows. An invariant definition was achieved by~\cite{InvarGWdef}. The Bondi-Sachs formalism~\cite{Bondi_1962,Sachs_1962} uses outgoing null-rays to create a coordinate system on an asymptotically flat space-time. Using this formalism Robinson and Trautman~\cite{Trautman} managed to prove the existence of gravitational waves in GR for $\Lambda=0$. Unfortunately, this proof is not straight forwardly generalizable to arbitrary backgrounds, since there the asymptotic symmetry group is $\text{Diff}(\mathcal{I})$ which makes it difficult to find an invariant notion of energy-momentum carried by gravitational waves (GW). Ashtekar and collaborators in a series of papers pointed out the difficulties of extending the Bondi-Sachs formalism to $\Lambda\neq0$ and achieved a description of GW on a de-Sitter background in the weak field limit~\cite{AshtekarPosLam,Ashtekar_2014,Ashtekar_2015_Aug,Ashtekar_2015_Nov,Ashtekar_2019}. A description of GW on arbitrary backgrounds in full GR, however is not yet achieved.\\

The Bondi-Sachs formalism works with a $\frac{1}{r}$-expansion. This had to be replaced with a late-time expansion in the case of the background de-Sitter metric, and it is reasonable to assume that it has to be specifically adapted to whatever background metric one is working on, which for a generic background cannot be done.\\
We took inspiration from the Bondi-Sachs formalism in another way: There, a special chart is created using outgoing null-rays which achieves that we have some understanding of what these coordinates represent and are thus not confused about whether the metric describes a wave or not.\\
We tried to go a step further and use the geodesics themselves to describe a geometry instead of using some of them to set up a coordinate system and then use the metric in that chart for this purpose. Consequently, we relate the geodesics directly to curvature i.e., we construct the analogue relation to:\\
\begin{equation}\label{eq: R_curvature}
    R^\mu_{\nu\rho\sigma} = -\Gamma^\mu_{\nu\rho,\sigma} + \Gamma^\mu_{\nu\sigma,\rho}
    - \Gamma^\eta_{\nu\rho}\Gamma^\mu_{\eta\sigma} + \Gamma^\eta_{\nu\sigma}\Gamma^\mu_{\eta\rho}, \quad 
    \Gamma^\mu_{\rho\sigma} = \frac{1}{2}g^{\mu\nu}\left( g_{\nu\rho,\sigma} + g_{\nu\sigma,\rho} - g_{\rho\sigma,\nu} \right),
\end{equation}
where the metric is related to curvature via second order coupled partial differential equations.\\
If one manages to find a metric to a given curvature tensor one in general faces the problem we pointed out above. Namely that we have to understand the coordinates in which the metric is expressed to understand if it describes a GW or not. One way to achieve this is to solve the geodesic equations
\begin{equation}\label{eq: geod.eq.}
    \frac{D\gamma^\mu}{dt} = \ddot{\gamma}^\mu + \Gamma^\mu_{\rho\sigma}(\gamma)\dot{\gamma}^\rho\dot{\gamma}^\sigma = 0, \qquad
    \gamma: \ I\subset\mathbb{R} \ \to \ \mathcal{M}
\end{equation} 
and then use these geodesics to construct a new coordinate system as it is done in the Bondi-Sachs formalism or for the Edington-Finkelstein coordinates.\\
Proper time distances between two spacetime events can then be calculated by integrating the tangent vector field along such a geodesic, since distances on a differential manifold are defined as the infimum of the arclengths of connecting curves.
\begin{equation}\label{eq: distance}
    d(x,y) = \inf\,\{l[\gamma]|\gamma: [a,b]\to\mathcal{M}, \gamma(a)=x, \gamma(b)=y\}, \qquad
    l[\gamma] = \int_a^b \sqrt{g_{\gamma(t)}(\dot{\gamma}(t),\dot{\gamma}(t))}\,dt
\end{equation}

In our formalism the quantity describing geometry is related to the one describing curvature by a transcendental equation.\\

In this work we restrict ourselves to the Riemannian case which describes the spatial sub-sheets of a space-time. Thus there is no time evolution in the problem, but we are currently working on an extension of the formalism to the pseudo-Riemannian case, for which the Riemannian case serves as a foundation.\\

\subsection{Outline and resulting Procedure}
Since we want to describe a geometry using only geodesics, we start by defining and rederiving many fundamental concepts.\\
A geometry is essentially a collection of distances between all points in a space and angles between all direction at any point. Distances in a manifold are inherently related to geodesics. If the infimum in Eq.~\eqref{eq: distance} exists, there is a connecting curve with minimal arclength, which is by definition called a geodesic. We thus use arclength parametrized geodesics $\gamma$ together with a map to the $(n-1)$-sphere $\mathcal{S}^{n-1}$, representing the directions at any given point, to get as close to (fundamental) geometrical properties as possible. Interestingly, this ended up excluding the null-geodesics and we thus work on the complement of geodesics at the origin to the ones used in the Bondi-Sachs formalism. We bundle the geodesics of flows from all points together, analogue to the tangent bundle and relate this geodesic flow bundle to curvature via the sine- and cosine-laws on any geodesic triangle.\\
This way we manage to skip the step from the Riemann-tensor~\eqref{eq: R_curvature} and the Einstein equations to the geodesic equations~\ref{eq: geod.eq.} (red arrows in Fig.~\ref{fig: Overview}) and reduce the coupled non-linear differential equations to an analytic operation (green-gray arrow in Fig.~\ref{fig: Overview}). We achieve this by design: if the metric is the main quantity. One relates the metric to curvature and then first solves for the metric before one can calculate any other quantity one might want from there. So, if we make geodesics to the main quantity, we calculate them directly from curvature. For the same reason we restrict the parametrization of the geodesics to arclength to get the distances directly ingrained in our main quantity. We develop a spherical triangulation method in the follow-up paper, which allows us to calculate the geodesic flow bundle from a curvature field. The distances and angles can then be then read out from the arguments of the geodesic flow bundle and thus the integral in Eq.~\eqref{eq: distance} reduces to a transcendental equation.\\

The following diagram Fig.~\ref{fig: Overview} summarizes, how the geodesic flow bundle is related to the metric one. The arrows show, from which quantities we can calculate others. The analogue to the Riemann curvature tensor in our formalism is the source curvature $K$ (Sec.~\ref{sec: Curvature}). We are not certain yet, whether an analogue of the energy-momentum tensor is needed.
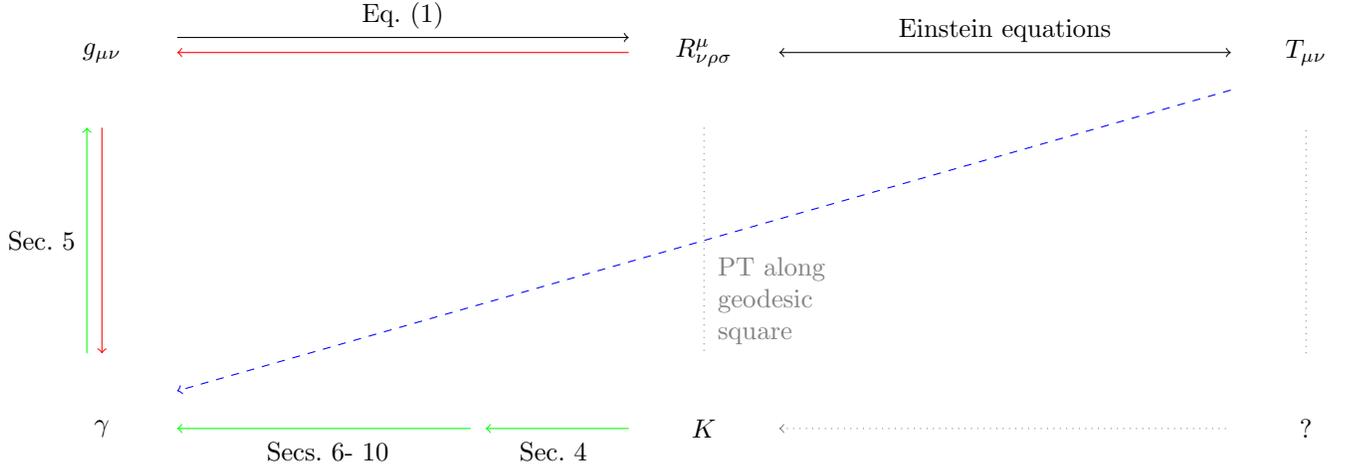
\begin{figure}[h!]
    \begin{tikzpicture}
        \node at (0,0) {$g_{\mu\nu}$};
        \draw[<-,red] (1,0) -- (7,0);
        \draw[->] (1,0.2) -- (7,0.2);
        \node at (4,0.5) {Eq.~\eqref{eq: R_curvature}};
        \node at (8,0) {$R^\mu_{\nu\rho\sigma}$};
        \draw[<->] (9,0) -- (15,0);
        \node at (12,0.3) {Einstein equations};
        \node at (16,0) {$T_{\mu\nu}$};
        
        \draw[->,red] (0,-1) -- (0,-4);
        \draw[->,green] (-0.2,-4) -- (-0.2,-1);
        \node at (-0.8,-2.5) {Sec.~\ref{sec: metric}};
        
        \draw[gray,dotted] (8,-1) -- (8,-4);
        \node[gray,align=left] at (8.9,-3.3) {PT along\\geodesic\\square};
        
        \node at (0,-5) {$\gamma$};
        \draw[<-,green] (1,-5) -- (4.9,-5);
        \node at (3,-5.3) {Secs.~\ref{sec: Triangulation}-~\ref{sec: Limit of the second order triangulation}};
        \draw[<-,green] (5.1,-5) -- (7,-5);
        \node at (6,-5.3) {Sec.~\ref{sec: Curvature}};
        \node at (8,-5) {$K$};
        \draw[<-,gray,dotted] (9,-5) -- (15,-5);
        \node at (16,-5) {?};
        \draw[gray,dotted] (16,-4) -- (16,-1);
        
        \draw[->,blue,dashed] (15,-0.5) -- (1,-4.5);
    \end{tikzpicture}
    \caption{Our ultimate goal is represented by the blue dashed arrow. The black arrows stand for calculations which are straight forward, whereas the red ones mark in general unsolvable equations. Green stands for the steps which are derived in this work and the gray dotted relations are yet unknown or not properly worked out (Parallel Transport).}
    \label{fig: Overview}
\end{figure}

\newpage

The presentation is not strictly structured according to the resulting methodology, since we would otherwise have to use concepts which were not introduced jet. We define the geodesic flow bundle (GFB) in Sec.~\ref{sec: GFB} and investigate its properties. Then we construct a special case of normal charts, which we call faithful normal chart in Sec.~\ref{sec: Faithful normal chart}. Ultimately, we want to use this formalism on the Einstein equations. The first step in that direction, is to connect the geodesic flow bundle to curvature, which we do in Sec.~\ref{sec: Curvature}. We find that it is more appropriate in this formalism to treat curvature as a source of geometry rather then an effect of it and thus introduce the notion of a source curvature field. This field can be related to the Riemann curvature tensor using parallel transport. We discuss how a natural notion of parallel transport arises from the GFB in Appendix~\ref{sec: Parallel transport} but discussing its relation to Riemann curvature in full detail is  postponed to a later publication. \\
We relate our formalism to the metric in Sec.~\ref{sec: metric}, to lay a ground for comparison and check, whether our results so far are consistent, before we start calculating the geodesic flow bundle from a source curvature field. This calculation is done via the triangulation scheme described in Sec.~\ref{sec: Triangulation}, which consists of cutting a geodesic triangle into increasing numbers of smaller triangles and applying the sine- and cosine-laws for constant curvature on the approximately spherical small triangles. We start the triangulation, by cutting the geodesic triangle into slices (thin triangles) and then cut the slices into smaller triangles. In Sec.~\ref{sec: A Slice of the Solution} we summarize the triangulation of a slice and how we calculate the desired quantities from the arguments. To complete the triangulation, we need to link all slices together which leads to a recursion across the slices. We present the solution to this recursion in Sec.~\ref{sec: Recursion across slices}. To arrive at a precise analytic result, we take the limit of cutting the triangle into infinitely many infinitesimal spherical triangles. To calculate this limit we need to find the limits of infinite products. We introduce a notion of product integrals in Sec.~\ref{sec: Product integrals} and present an analogue to the main theorem of calculus for product integrals. The limits of the triangulation are presented in the following Sec.~\ref{sec: Limit of the second order triangulation}. We put our results into context in Sec.~\ref{sec: sine- & cosine-laws} and conclude that we derived a generalization of the sine- and cosine-laws to varying curvature fields.\\

We provide some examples in the Supplemental Material to visualize the geodesic flow bundle and the faithful normal charts. We also demonstrate, how one can calculate the metric from the GFB on the example of the $2$-sphere, $3$-sphere, flat $n$-dimensional space and give an example of a geometry, which cannot be described using a metric. The supplement also contains the detailed calculations to the expansion of the sine- and cosine-laws, the triangulation, mathematical methods to calculate the limits and then finally the limit calculations. We check or results on some examples and provide these tests in its last section.\\

Finally, a method emerges from this discussion, through which we can calculate the geodesic flow bundle up to second order in the arclength parameter from a source curvature field. The resulting procedure works as follows:
\begin{enumerate}
\setcounter{enumi}{-1}
    \item We treat curvature as a source field for geometry and consider it our input from which we calculate the geodesics.
    \item Then we systematically embed two-dimensional hypersurfaces in an n-dimensional manifold which splits the problem into a one of tilt angles and a two-dimensional one. Since we restrict ourselves to Riemannian plane GFBs this angular problem is relatively simple and consists of solving transcendental equations obtained from rotations.
    \item The remaining two-dimensional problem contains the largest part of the complexity and amounts to an extension of the sine- and cosine-laws to varying curvature fields. We solve it up to second order via the spherical triangulation of a geodesic triangle.
\end{enumerate}
$ $\\
The same procedure can be used in the case of a pseudo-Riemannian space-time.
\color{black}

\subsection{Comparison to Regge Calculus}
Regge calculus (RC) was introduced by T. Regge~\cite{Regge} to establish a coordinate free treatment of general relativity. He proposed a discretisation of space-time by triangulating it with flat triangles. With increasing finesse of the triangulation the piece-wise linear space converges to the curved manifold. RC found applications in discrete quantum gravity~\cite{Williams-DiscreteQG} and numerical relativity~\cite{Khavari-PhD}, where its discrete nature is a good match for numerical methods.\\

Since the main part of this work involves a triangulation, we compare our formalism to Regge Calculus to avoid confusion of the two.\\
Regge Calculus is nowadays used as a finite element method, to solve the Einstein equations numerically.
The idea behind the geodesic flow bundle formalism is to describe a geometry via arclength parametrized geodesics instead of using a scalar product field.\\

In Regge Calculus the space time is approximated by a piecewise flat construction of triangles. The entire information about curvature is concentrated in (n-2)-dimensional subsimplices called “hinges” or “bones” in the form of the deficit angle.\\
In our formalism we relate the geodesics to curvature via sine- and cosine-laws, which we generalize to varying curvature fields. Whilst the main result are the second order terms in the power series of these generalized sine- and cosine-laws, the formalism itself is not an approximation and we strive to capture all degrees of freedom of smooth curvature fields on differential manifolds.\\

In Regge Calculus one evolves a tessellated $3$-dimensional hypersurface in time like direction by evolving vertex after vertex and connecting the new vertices with their original ones and their neighbouring vertices which achieves a triangulation of the 4-dimensional spacetime. It is thus called a (3+1) evolutionary scheme.\\
Given a curvature field we calculate a geodesic from an arbitrary point in an arbitrary direction in terms of the flow from an origin point (which i.g. can be any other point), by first embedding the resulting geodesic triangle in the $n$-dimensional manifold (n-dim. problem) and then we use the generalized sine- and cosine-laws to solve the $2$-dimensional triangle problem. The triangulation is used to derive the derive the generalization of the sine- and cosine-laws to varying curvature by approximate the curvature field with a step function and then taking the limit to an infinitely fine spherical triangulation of the curved geodesic triangle.\\

The outcome of the flat triangulation in Regge Calculus is the Connection- or Incidence Matrix containing all edge lengths and the information on how they are connected.\\
The result of applying the generalized (curved) sine- and cosine-laws to an arbitrary curvature field is the geodesic flow bundle in terms of the geodesic flow from another point w.l.o.g. the origin point. This contains all geodesics in the manifold in arclength parametrized form.\\

Both methods can be applied to general curvature distributions and do not require any symmetry in the problem.\\
Regge Calculus will run into trouble, if the space- and the time-like coordinates would change their role, since one could not setup a sensible evolution of the initial tessellated hypersurface.\\
The GFB could in general be applied to pretty much all circumstances, with the caveat that recursions may not be solvable and some quantities may not be possible to be written in explicit form and thus the limits could not be calculated. We also may have to restrict ourselves to a finite cutoff of the power series.
In this work we restrict ourselves to the simplest cases, but the methods could in general be used on any other cases with the aforementioned potential problems.\\

A finite volume (subset or equal) of the manifold is triangulated with flat triangles in RC. The triangles fit together and form a consistent piece-wise flat manifold, where the curvature is captured in the discontinuities. The angles around a vertex do not add up to $2\pi$ and this deficiency is used to represent the curvature.
It is a bottom-up method where one approximates a curved space-time increasingly well with increasing finess.\\
Only a geodesic triangle, which is automatically $2$-dimensional, is triangulated by spherical triangles when we derive the fundamental solution for the GFB. While the shared sides of neighbouring triangles have the same length they do not fit together, if they were embedded in a $3$-dimensional ambient space, since they are spherical triangles of different curvature. The angles around each point where geodesics cross add up to $2\pi$.
It is a top-down method, where we start from the actual geodesics on the full non-linear manifold and then construct an expansion to solve the problem.\\

In Regge Calculus one can find the analogue action and one gets a discretised version of the Bianchi identities. It looks like a discretised version of the metric formalism.\\
The geodesic flow bundle shows a general incompatibility with the metric formalism. As we can see in Appendix~\ref{sec: Parallel transport} and Sec.~\ref{sec: Curvature}, it is not straight forward to compare the parallel transport arising from the GFB with the connection of the metric formalism or the $n$-dimensional Gaussian curvature field with the Riemann tensor. The easiest way is to calculate the metric in faithful normal coordinates (FNC), Secs.~\ref{sec: Faithful normal chart},~\ref{sec: metric}, from the GFB and then work with that metric. The other direction is in general not possible since it would require solving the geodesic equations for all cases.\\

\section{Definition of the Geodesic Flow Bundle}\label{sec: GFB}
Usually, one defines a metric which is a scalar product field, to impose a geometry on a manifold. The geometric quantities i.e. distances and angles can then be calculated from the metric.\\
But not all notions of distances can be expressed as scalar products, like for example the Manhattan metric \footnote{In the context of differential manifold we interpret this as imposing the Manhattan metric consistently on an $\varepsilon$-Ball around each point in the manifold. Constructing such a metric using a scalar product field would require us to represent the metric via a scalar product on every tangent space. This notion of distance on a vector space is however not representable by a matrix.}. So, by using the scalar product we are restricting our studies to geometries which can be expressed in form of a scalar product.\\

Despite the fact, that we can define a scalar product coordinate independently we still need to choose coordinates when we actually want to use it. And we are trying to use a linear structure on an otherwise in general nonlinear entity. It is no surprise that it works since we can approximate everything linearly. We might however loose or obscure some of the non-linear structure when we do this.\\

If we are given a metric expressed in an arbitrary chart, it in general is very difficult do discern, what geometry we are dealing with, since the components are all coordinate dependent and thus cannot be interpreted as long as we do not understand what coordinates we are working with. To get a grasp on the problem we need the geodesics in a parametrization affine to their arclenghts, since we then understand the parametrization and we know that curves are 1-dimensional submanifolds and thus coordinate independent. Additionally, we know that geodesics trace out the shortest paths between two sufficiently close points. To find these geodesics however, we would have to solve the geodesic equations and then integrate to compute a distance.
We propose to instead describe the geometry on a manifold in a more fundamental and less encrypted form by using the geodesics themselves instead of a scalar product.

\subsection{Geodesic flow from a point}
Instead of choosing a scalar product field or imposing a metric, a function which maps two points to a positive real number, we choose a family of curves $\{\gamma_{\hat{\Omega}}\}_{\hat{\Omega}\in\mathbb{S}^{n-1}}$ starting from a origin point $o\in\mathcal{M}$ on the manifold $\mathcal{M}^n$ and heading out in all directions $\Omega\in\mathcal{D}_o$. We consider this family of curves as geodesics by definition and define the distance from $o$ to any other point $p$ to be the parameter value of the curve at that point. We thereby consider the parameter $\lambda$ to be the arclength by definition as well.\\
Let $\mathcal{M}^n$ be an $n$ dimensional differential manifold and $x\in\mathcal{M}$.
\begin{align}
    &\begin{matrix}
    \gamma_{o,\Omega}: & (-\varepsilon,\varepsilon) & \to & \mathcal{M} \\
    & \lambda & \mapsto & \gamma_{o,\Omega}(\lambda)
    \end{matrix}, \quad 
    \varepsilon>0 \quad \wedge \quad \gamma_{\hat{\Omega}}(0) = o; \\
    &d(o,x) \coloneq \lambda, \quad \text{with } \lambda\in(0,\varepsilon) \text{ and } \Omega\in\mathcal{D}_o \text{ s.t.} \quad \gamma_{o,\Omega}(\lambda) = x \label{eq: d(o,x)}
\end{align}
As a definition for the set of directions at a point $p\in\mathcal{M}$ we can use the following equivalence class of curves through $p$.
\begin{equation}
    \mathcal{D}_p \coloneq \{[\gamma_{p}]_\sim\,|\,\gamma_p:(-\varepsilon,\varepsilon)\to\mathcal{M},\, \text{smooth},\ \gamma_p(0)=p\}, \qquad \gamma_p\sim\gamma_p'\ :\Leftrightarrow\ \dot{\gamma}_p(0) \propto \dot{\gamma}_p'(0).
\end{equation}
That way we do not need to impose any additional structure, other than the manifold $\mathcal{M}$ itself. Directions should already be a property of a manifold and thus it should not be necessary to introduce additional structure to describe them.\\
The family of curves $\{\gamma_{o,\Omega}\}_{\Omega\in\mathcal{D}_o}$ defines a flow in the neighbourhood around $o$ without the origin point $o$ itself.
\begin{definition}[Geodesic Flow from a Point]
The geodesic flow from $x\in\mathcal{M}$ is the map:\\
\begin{minipage}{.5\linewidth}
    \begin{align}
        &\qquad \begin{matrix}
        \phi: & \mathcal{D}_x & \to & C^\infty(I;\mathcal{M}) \\
        & \Omega & \mapsto & \left.\gamma_{\Omega}\right|_I
        \end{matrix}, \quad I = (0,\varepsilon), \quad \text{s.t.}
    \end{align}
\end{minipage}
\begin{minipage}{.5\linewidth}
    \begin{align}
        &\bullet\quad \gamma_{\Omega} \text{ are injective} \\
        &\bullet\quad \mathcal{D}_x\times I \ \to \ \mathcal{M} \ \text{ is smooth} \label{eq: Vfield} \\
        &\qquad \Rightarrow \quad X(\gamma_{\Omega}(\lambda)) = \dot{\gamma}_{\Omega}(\lambda)
        \in \Gamma(T\mathcal{M}\backslash\{x\}) \notag \\
        &\bullet\quad \gamma_{\Omega}(\lambda) = \gamma_{\Omega'}(\lambda') \quad
        \Rightarrow \quad \lambda = \lambda'
    \end{align}
\end{minipage}
\end{definition}
As concluded from condition~\eqref{eq: Vfield} the tangent vecors of the flow from $x$ form a smooth vecor field on $\mathcal{M}\backslash\{x\}$, which by definition $X$ is a central field of unit vectors. We can write the flow in a more conventional way, to recap:\\

A local flow of $X\in\Gamma(T\mathcal{M})$ is a family $\{\phi^t\}_{t\in I}$, with $\phi^t:\ U\subset\mathcal{M}\ \to\ \phi^t(U)\subset\mathcal{M}$, where $0\in I\subset\mathcal{R}$ is an interval. The function $\phi^t(x) \coloneq \phi(x,t) = \gamma_x(t)$ is given by the integral curves i.e. they satisfy:
\begin{align}
    \frac{\partial\phi(x,t)}{\partial t} = X(\phi(x,t)) \ \wedge \ \phi(x,0) = x, \quad
    \forall x\in U\subset\mathcal{M},\ \forall t\in I
\end{align}
With a bit of tweaking on the arguments the conditions are satisfied trivially, since we start from the integral curve.
\begin{align}
    \phi^t(x) = \phi(x,t) = \gamma_{\Omega}(\lambda+t) = \phi(\Omega,\lambda+t), \quad \text{with } \Omega \text{ and } \lambda \text{ s.t. } \gamma_{\Omega}(\lambda) = x
\end{align}

\subsubsection{Angular structure at a point}\label{sec: Angular structure}
With the geodesic flow from a point $p\in\mathcal{M}$ we have so far defined the distances between $p$ and any other point in $\mathcal{M}$ using the map $d(p,\cdot)$ from~\ref{eq: d(o,x)}. We now study the choice of angles at a point.\\

An angle is measured by the circle arc $L$, which two directions enclose, divided by the radius of the circle $r$ or more simply by the unit circle arc, that two directions enclose. Thereby a special case of a metric is imposed / defined on the set of directions $\mathcal{D}_p$ at a point $p\in\mathcal{M}$.
\begin{equation}
    \measuredangle:\ \mathcal{D}_p\times\mathcal{D}_p \ \to \ [0,2\pi)
\end{equation}

The unit sphere $\mathcal{S}^{(n-1)}$ is the union of all unit circles around a point $p\in\mathcal{M}^n$. It thus provides a convenient way of defining the metric on $\mathcal{D}_p$.\\
\begin{proposition}
The set of directions $\mathcal{D}_p$ at a point $p\in\mathcal{M}^n$ forms a manifold, which is diffeomorphic to the sphere $\mathcal{S}^{n-1}$.
\begin{equation}
    \Phi_\mathcal{D}: \ \mathcal{D}_p \ \to \ \mathcal{S}^{n-1}, \quad
    \Omega = [\gamma_p]_\sim \ \mapsto \ \hat{\Omega}
\end{equation}
\end{proposition}

\begin{proof}
The maps
\begin{align}
    \Phi_\mathcal{D}:\ \gamma\in\Omega, \quad \gamma \ \to \ \hat{\Omega} = \frac{\dot{\gamma}}{\Vert\dot{\gamma}\Vert}\in\mathcal{S}^{n-1} \qquad
    \Phi_\mathcal{D}^{-1}:\ \hat{\Omega} \ \mapsto \ \{\gamma\in C^\infty((-\varepsilon,\varepsilon);\mathcal{M})|\gamma(0)=p,\ \dot{\gamma}\propto\hat{\Omega}\} = \Omega
\end{align}
are independent of the representant and smooth inverse of each other.
\end{proof}

The unit sphere is defined as a subset of $\mathbb{R}^n$ endowed with the standard scalar product, which induces the embedding $\imath$.
\begin{equation}\label{eq: Def. Sphere}
    \mathcal{S}^{n-1} \coloneq \{x\in\mathbb{R}^n|\Vert x\Vert = 1\}, \quad \Rightarrow \quad
    \imath: \ \mathcal{S}^{n-1} \hookrightarrow \mathbb{R}^n
\end{equation}
This embedding together with the standard scalar product $\cdot$ on the ambient space $\mathbb{R}^n$ defines an angular metric on $\mathcal{S}^{(n-1)}$. Combined with a choice of diffeomorphism $\Phi_\mathcal{D}$ this induces a metric on $\mathcal{D}_p$.
\begin{equation}
    \begin{matrix}
    \measuredangle: & \mathcal{D}_p\times\mathcal{D}_p & \overset{\Phi_\mathcal{D}}{\to} & \mathcal{S}^{n-1}\times\mathcal{S}^{n-1} & \to & [0,2\pi) \\
    & (\Omega,\Omega') & \mapsto & (\hat{\Omega},\hat{\Omega}') & \mapsto & d(\hat{\Omega},\hat{\Omega})
    \end{matrix}
\end{equation}
\begin{equation}
    d(\hat{\Omega},\hat{\Omega}') \coloneq \inf_\gamma\{ l(\gamma) | \gamma: [0,1]\to\mathcal{S}^{n-1}, \gamma(0) = \hat{\Omega}, \gamma(1) = \hat{\Omega}' \}, \qquad
    l[\gamma] = \int_0^1 \Vert\dot{\gamma}(t)\Vert\,dt
\end{equation}
In other words, angles are the distances on the unit sphere between two directions.\\

In the case of a space-time this concept generalizes to generalized angles being distances on the unit hyperboloid under the Minkowski metric. Upon further investigation we found that some direction pairs are related to angles and others to arclengths of the unit hyperbola which are known as rapidities.\\
In the case of a space-time this concept generalizes to some direction pairs being related by arclengths of unit hyperbolas rather than great-circle arcs.

\begin{definition}[Direction Sphere at a Point]
The direction sphere is the sphere $\mathcal{S}^{n-1}$ endowed with the metric $\measuredangle$, representing the directions $\mathcal{D}_p$ at a point $p\in\mathcal{M}$, using the map $\Phi_\mathcal{D}$.
\end{definition}

\begin{remark}
The degree of freedom (d.o.f.) of the arclength is thereby replaced by the auxiliary d.o.f. and thus the ambient space is well suited to model the direction sphere at a point.
\end{remark}

\begin{proposition}
The construction of the direction sphere allows us to calculate the angle between two directions $\Omega,\Omega'\in\mathcal{D}_p$ via the embedding $\imath$:
\begin{equation}\label{eq: angle from skp}
    \measuredangle(\Omega,\Omega') = \cos^{-1} \imath\circ\Phi_\mathcal{D}(\Omega) \cdot \imath\circ\Phi_\mathcal{D}(\Omega').
\end{equation}
\end{proposition}

\begin{proof}
Due to the rotation symmetry of the sphere, we can without loss of generality assume that the two directions are mapped to the $1,2$-plane and specifically $\imath\circ\Phi_\mathcal{D}(\Omega) = \hat{\Omega} = \hat{e}_1$. The direction vectors on $\mathcal{S}^{n-1}$ and the connecting geodesic are then given by:
\begin{equation}
    \hat{\Omega} = \imath\circ\Phi_\mathcal{D}(\Omega) = \begin{pmatrix} 1 \\ 0 \\ \vec{0} \end{pmatrix}, \quad
    \hat{\Omega}' = \imath\circ\Phi_\mathcal{D}(\Omega') = \begin{pmatrix} \cos\theta \\ \sin\theta \\ \vec{0} \end{pmatrix}, \qquad
    \gamma(t) = \begin{pmatrix} \cos(\theta t) \\ \sin(\theta t) \\ \vec{0} \end{pmatrix}\in\mathcal{S}^{n-1}\subset\mathbb{R}^n
\end{equation}
We see, that the arclength of the geodesic agrees with the scalar product of the direction vectors:
\begin{equation}
    l[\gamma] = \int_0^1\sqrt{(-\theta\sin(\theta t))^2 + (\theta\cos(\theta t))^2}\,dt = \theta
    = \cos^{-1}\cos\theta = \cos^{-1}\hat{\Omega}\cdot\hat{\Omega}'
\end{equation}
\end{proof}

To use the geodesic flow in concrete examples we need to systematically parametrize its curves. Since the curves are labelled by the direction in which they head out from a point $o$ we can use a parametrization of the unit sphere. We still want to have as much geometrical information i.e. maintain the angular structure in our geodesic flow and we want to note what the arbitrary choices of such a parametrization are.\\

In two dimensions we can pick a direction and label it with $0$ i.e. define it as our origin or initial direction. We will refer to points lying on the geodesic heading out in this direction as, with $q(c) = \gamma_{o,0}(c)$.
If we now demand that the angular parameter represents actual angles i.e.
\begin{equation}\label{eq: 2d angle cond}
    \Delta\varphi \coloneq \measuredangle(\gamma_{o,\varphi},\gamma_{o,\varphi'}) = |\varphi - \varphi'|
\end{equation}
then we only have two choices left: in which direction do we want to increase the parameter and what domain do we want to use. For example $(-\pi,\pi]$ or $[0,2\pi)$ as domains and then we could also invert the angular parameter $\varphi\mapsto-\varphi$.\\
The system is constrained because the angles are inherently $2$-dimensional and thus to a large degree dictate what coordinates we need to pick if they are supposed to reflect the angular structure.\\

In this work we will use the following parametrization map:
\begin{equation}
    \hat{\Omega}^{(1)}: \ [0,2\pi) \ \to \ \mathcal{S}^1, \qquad \varphi \ \mapsto \ \hat{\Omega}^{(1)}(\varphi) = \binom{\cos\varphi}{\sin\varphi}
\end{equation}
It should be noted that we used the arclength parameter of the geodesic flow from a point we name $0$ on $\mathcal{S}^1$.\\

Because angles are $2$-dimensional concepts, the angular structure becomes much more complex and we have much more choices for our angular parameters, when we introduce another dimension:\\
We begin in the same way as before by picking a random direction and calling it original direction $\Omega_0$ and mapping it to $\Phi_\mathcal{D}(\Omega_0) = (1,0,0) \eqcolon \hat{0}$ on $\mathcal{S}^2$. We then consider a geodesic $\gamma_{q,\hat{\beta}_1}$ heading out form a point $q(c)$ on $\gamma_{o,\hat{0}}$ in a direction $\hat{\beta}_1$. We now have two angular variables for every geodesic and we start labelling the ones heading out from $o$, which cross $\gamma_{q,\hat{\beta}_1}$ with $\varphi$. We choose to increase the first angular parameter $\varphi$ in direction of $\gamma_{q,\hat{\beta}_1}$. The directions of these geodesics from $o$ may in general be mapped to a generic curve on $\mathcal{S}^2$ from $\hat{0}$ i.e. the second angular parameter of these geodesics could vary: $\vartheta = \vartheta(\varphi)$. We can however choose the second angular variable to be $\vartheta = 0$ for all these geodesics, which would map the directions to the great arc $\gamma_{\hat{0},0}^{\mathcal{S}^2}$ on the $2$-sphere.
\begin{equation}
    \hat{\Omega}^{(2)}(\varphi,\vartheta) \quad \text{s.t.} \quad \gamma_{o,\hat{\Omega}^{(2)}(\varphi,0)}(l) = \gamma_{q,\hat{\beta}_1}(\lambda) \quad \Rightarrow \quad \bigcup_\varphi [\gamma_{o,\hat{\Omega}^{(2)}(\varphi,0)}]_\sim \
    \overset{\Phi_\mathcal{D}}{\mapsto} \ \gamma_{\hat{0},0}^{\mathcal{S}^2}.
\end{equation}

Now we see that we can again take the arclength of the $\mathcal{S}^1$-flow as the $\varphi$ parameter, embedding the $2$-dimensional case in the $3$-dimensional one:
\begin{equation}
    \hat{\Omega}^{(1)}(\varphi) \ \hookrightarrow \ \hat{\Omega}^{(2)}(\varphi,0) = \begin{pmatrix}
    \cos\varphi \\ \sin\varphi \\ 0
    \end{pmatrix}
\end{equation}

We continue this procedure, by considering a geodesic $\gamma_{q,\hat{\Omega}_\beta}$ from $q$, which does not lie in the initial plane. As before, we can choose to label the geodesics $\gamma_{o,\hat{\Omega}^{(2)}(\varphi,\vartheta)}$ which cross $\gamma_{q,\hat{\Omega}_\beta}$ in a way such that their directions are mapped to a great arc $\gamma_{\hat{0},\vartheta}^{\mathcal{S}^2}$ on $\mathcal{S}^2$.
\begin{equation}
    \bigcup_\varphi [\gamma_{o,\hat{\Omega}^{(2)}(\varphi,\theta)}]_\sim \
    \overset{\Phi_\mathcal{D}}{\mapsto} \ \gamma_{\hat{0},\vartheta}^{\mathcal{S}^2}
\end{equation}

\begin{figure}
\centering\includegraphics[width=.5\linewidth]{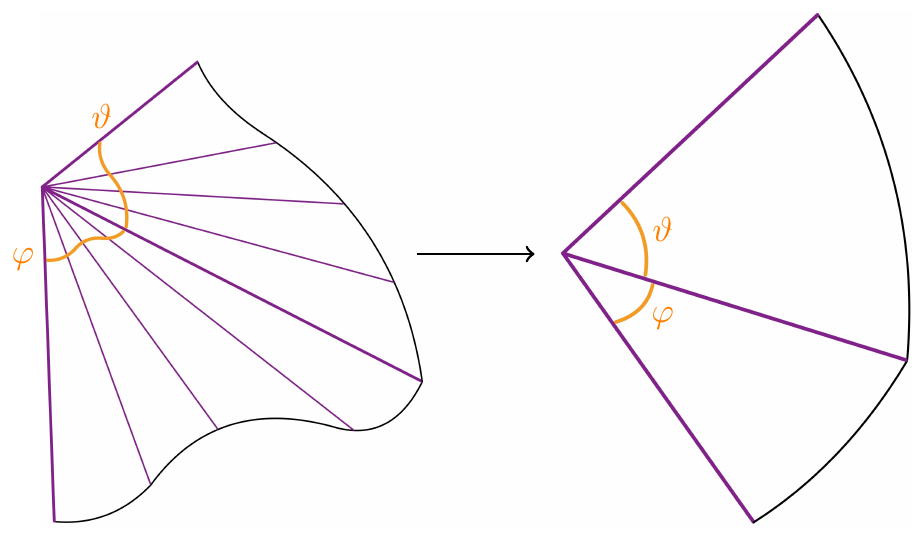}
    \caption{The sketch shows, how we adjust the parametrization of the directions, to flatten the geodesic surfaces. This way the angular parameters are more representative of the angular structure at that point.}
    \label{fig: Flatten geodesic surfaces}
\end{figure}

We end up with the parametrization:
\begin{equation}\label{eq: Omega2}
    \hat{\Omega}^{(2)}: [0,2\pi)\times\left(-\frac{\pi}{2},\frac{\pi}{2}\right) \ \to \ \mathcal{S}^2, \qquad
    (\varphi,\vartheta) \ \mapsto \ \hat{\Omega}^{(2)}(\varphi,\vartheta) = \begin{pmatrix}
    \cos\varphi \\ \cos\vartheta\sin\varphi \\ \sin\vartheta\sin\varphi
    \end{pmatrix}
\end{equation}
By comparing \eqref{eq: Omega2} with the 2-sphere in the Supplemental Material we see, that the angular structure at $o$ is just the flow from $o_{\mathcal{S}^2} = \hat{0}$ on the $2$-sphere. The first angular parameter $\varphi$ coincides with the arclength $\lambda_{\mathcal{S}^2}$ and the second angular parameter $\vartheta$ is the same as the first and only angular parameter $\varphi_{\mathcal{S}^2}$ of the geodesic flow on $\mathcal{S}^2$: $\varphi = \lambda_{\mathcal{S}^2}, \ \vartheta = \varphi_{\mathcal{S}^2}$.
So, we connected the geodesic flow on a $3$-dimensional manifold, which has $2$ angular degrees of freedom, with the one on the $2$-dimensional sphere, which has $1$ angular degree of freedom.\\
Furthermore, we can parametrize the flow from $o$ on the $3$-sphere, using the flow from $o$ on the $2$-sphere and then precede to parametrizing the flow on the $4$-sphere with the one on the $3$-sphere and so on, see Fig.~\ref{fig: sphere flow chain}.
\begin{equation}\label{eq: angular structure sequence}
    \gamma_{o,\hat{\Omega}^{(0)}}^{\mathcal{S}^1}(\lambda) \ \overset{\begin{matrix}
    \lambda\mapsto\varphi
    \end{matrix}}{\longrightarrow} \ \gamma_{o,\hat{\Omega}^{(1)}(\varphi)}^{\mathcal{S}^2}(\lambda) \ \overset{\begin{matrix}
    \lambda\mapsto\varphi \\ \varphi\mapsto\vartheta
    \end{matrix}}{\longrightarrow} \gamma_{o,\hat{\Omega}^{(2)}(\varphi,\vartheta)}^{\mathcal{S}^3}(\lambda) \
    \overset{\begin{matrix}
    \lambda\mapsto\varphi \\ \varphi\mapsto\vartheta_1 \\ \vartheta\mapsto\vartheta_2
    \end{matrix}}{\longrightarrow} \gamma_{o,\hat{\Omega}^{(3)}(\varphi,\vartheta_1,\vartheta_2)}(\lambda) \ 
    \overset{\begin{matrix}
    \lambda\mapsto\varphi \\ \varphi\mapsto\vartheta_1 \\ \vartheta_1\mapsto\vartheta_2 \\ \vartheta_2\mapsto\vartheta_3
    \end{matrix}}{\longrightarrow} \ \ldots
\end{equation}

That way we can construct a meaningful parametrization in the sense, that all parameters correspond to angles i.e. arc-lengths on the unit sphere and not some unknown function of them.
\begin{equation}\label{eq: direction vector}
    \hat{\Omega}^{(n)}: \ [0,2\pi)\times\left(-\frac{\pi}{2},\frac{\pi}{2}\right)^{n-2} \ \to \ \mathcal{S}^{n-1}, \qquad \hat{\Omega}^{(n-1)}(\varphi,\vartheta_1,\ldots\vartheta_{n-2})
    = \begin{pmatrix} \cos\varphi \\
    \cos\vartheta_{n-2}\cos\vartheta_{n-3}\ldots\cos\vartheta_1\sin\varphi \\
    \sin\vartheta_1\sin\varphi \\ \sin\vartheta_2\cos\vartheta_1\sin\varphi \\
    \sin\vartheta_3\cos\vartheta_2\cos\vartheta_1\sin\varphi \\ \vdots \\
    \sin\vartheta_{n-2}\cos\vartheta_{n-3}\ldots\cos\vartheta_1\sin\varphi
    \end{pmatrix}
\end{equation}

\begin{figure}[h!]
    \centering\includegraphics[width=.8\linewidth]{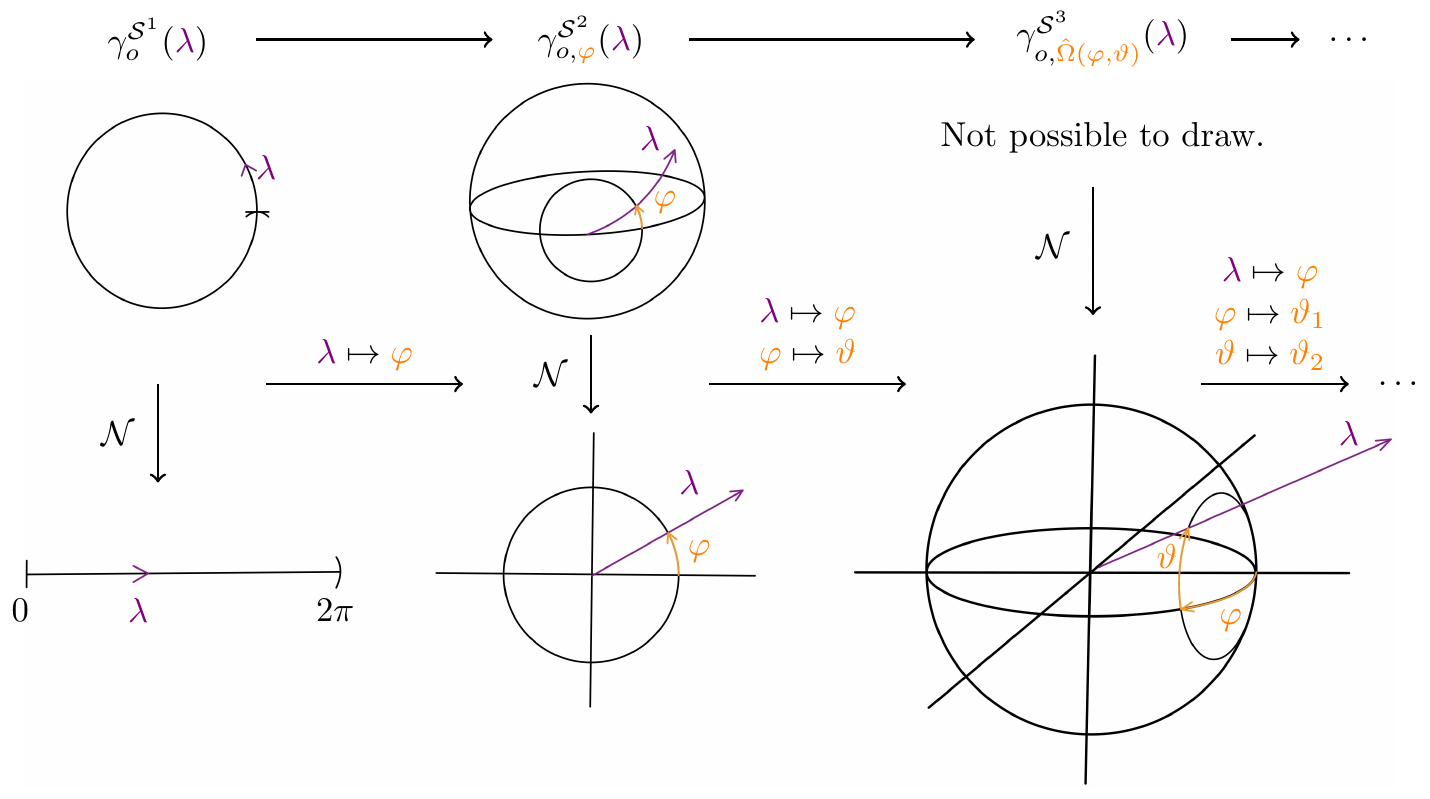}
    \caption{The sketch shows, how we use the flow from a point $o$ on the lower dimensional sphere to parametrize the directions on the next higher dimensional sphere. We draw the spheres which can be drawn, that is $\mathcal{S}^1$ and $\mathcal{S}^2$. We also draw the flow from $o$ in faithful normal coordinates which we introduce in Sec.~\ref{sec: Faithful normal chart}, since we can draw this up to $\mathcal{S}^3$.}
    \label{fig: sphere flow chain}
\end{figure}

Using the map $\Phi_\mathcal{D}$ we can parametrize the directions:
\begin{equation}
    [0,2\pi)\times\left(-\frac{\pi}{2},\frac{\pi}{2}\right)^{n-2} \ \overset{\hat{\Omega}^{(n-1)}}{\to} \
    \mathcal{S}^{n-1} \ \overset{\Phi_\mathcal{D}^{-1}}{\to} \ \mathcal{D}_p, \qquad
    (\varphi,\vartheta_1,\ldots,\vartheta_{n-1}) \ \mapsto \ \hat{\Omega} \ \mapsto \ \Omega
\end{equation}
and the angle between two parametrized directions is calculated via:
\begin{equation}
    \measuredangle(\Omega,\Omega') = \cos^{-1}\hat{\Omega}\cdot\hat{\Omega}', \quad
    \hat{\Omega} = \imath\circ\Phi_\mathcal{D}(\Omega)
    = \imath\circ\Phi_\mathcal{D}\circ\Phi_\mathcal{D}^{-1}\circ\hat{\Omega}^{(n-1)}(\varphi,\vartheta_1,\ldots,\vartheta_{n-1})
    = \imath\circ\hat{\Omega}^{(n-1)}(\varphi,\vartheta_1,\ldots,\vartheta_{n-1})
\end{equation}

We will from now on identify the directions with the positions on the sphere and the images of the map $\imath\circ\hat{\Omega}^{n-1}$ and abbreviate it with $\hat{\Omega}$.\\

The angular labelling of the flow from a point $o$ may seem very similar to the exponential map. We explain the difference between the two concepts in Appendix~\ref{sec: Exp. map}.\\

\subsection{Geodesic Flow Bundle}
So far we have defined the distances from a point $o$ to any other point $p$ in $U\subset\mathcal{M}$ and the angles at $o$. To complete the geometric information we need a geodesic flow at every point. Just as we organize the tangent spaces at all points into a bundle, we now organize all geodesic flows into the geodesic flow bundle.

\begin{definition}[Geodesic Flow Bundle]
We define the geodesic flow bundle by:
\begin{equation}
    \Phi\mathcal{M} \coloneq \bigcup_{x\in\mathcal{M}} \{x\}\times\phi_x = \bigcup_{x\in\mathcal{M}} \{x\}\times\{ \gamma_{x,\hat{\Omega}} \}_{\hat{\Omega}\in\mathbb{S}^{n-1}}
\end{equation}
It is often handy, to collect the geodesic flows in a function. Thus, we introduce the flow bundle function:
\begin{equation}
    \Phi: \ \mathcal{M}\times\mathbb{S}^{n-1}\times I \ \to \ \mathcal{M}, \qquad
    (x,\hat{\Omega},\lambda) \ \mapsto \ \Phi(x;\hat{\Omega},\lambda)
\end{equation}
which we will for most cases consider to be smooth and has to satisfy the following consistency conditions:
\begin{align}
    &\bullet\quad \gamma_{x,\Omega}(\lambda) = y \quad \Rightarrow \quad
    \gamma_{y,\Omega}(\tau) = \gamma_{x,\Omega}(\lambda+\tau),\ \tau\in(0,\varepsilon) \ \wedge \
    \gamma_{y,-\Omega}(\tau) = \gamma_{x,\Omega}(\lambda-\tau),\ \tau\in(0,\lambda) \label{eq: consistency} \\
    &\bullet\quad \forall\,x,y,z\in\mathcal{M}, \text{ with } \gamma_{x,\Omega}(\lambda) = z, \
    \gamma_{x,\Omega'}(\tau) = y \text{ and } \gamma_{y,\omega}(t) = z : \quad
    \lambda \leqslant \tau + t \quad \text{(triangle inequality)} \label{eq: triangle ineq.}\\
    &\bullet\quad \lambda = \tau + t \quad \Rightarrow \quad \Omega = \Omega', \quad
    \wedge \quad \gamma_{y,\omega}(t) = \gamma_{x,\Omega}(\tau+t), \label{eq: aligned case}
\end{align}
\end{definition}
\begin{remark}
Here $\Omega$ denotes the respective directions at the point where the geodesic heads out. More explicitly: when the first index is the point $x$, then the second index $\Omega$ is a direction at $x$.
\begin{equation}
    \gamma_{x,\Omega} \quad \Rightarrow \quad \Omega = \Omega(x) = [\gamma_{x,\Omega}]_\sim \in \mathcal{D}_x
\end{equation}
\end{remark}

The directions are supposed to be interpreted in the following way:
\begin{align}
    \Omega(y) = [\tau\mapsto\gamma_{x,\Omega}(\lambda+\tau)]_\sim \in \mathcal{D}_y, \quad
    -\Omega(y) = [\tau\mapsto\gamma_{x,\Omega}(\lambda-\tau)]_\sim \in \mathcal{D}_y
\end{align}
Note however, that we will see in the next subsection, that in a curved manifold it is not possible to label all geodesics in a way that the direction labels remain constant along the geodesics. We can do this for the geodesics from one point, we chose the ones from the origin $o$, but will have changing direction labels for geodesics which do not cross $o$.\\

Intuitively we think of a geodesic as a path which locally always goes straight on i.e. never changes it's direction. The consistency condition~\eqref{eq: consistency} incorporates that, by demanding, that if we head out from a point $y$ on a geodesic in the direction or opposite direction of that geodesic, then that geodesic from $y$ overlaps with the previous one.\\
The third condition~\eqref{eq: aligned case} asserts, that in the case where all three points are aligned on the same geodesic the triangle inequality equality~\eqref{eq: triangle ineq.} reduces to the consistency condition~\eqref{eq: consistency}.
\begin{equation}
    \text{The distances are now defined by:} \qquad
    d(x,y) = \lambda, \quad \gamma_{x,\hat{\Omega}}(\lambda) = y \ \wedge \ \gamma_{y,-\hat{\Omega}}(\lambda) = x
\end{equation}

The definition of the geodesic flow bundle is valid for arbitrary cases including the case of a space-time. What requires adapting is the angular structure, where the Riemannian case forms a part of the pseudo-Riemannian angular structure. The other parts work analogously with hyperbolas replacing the circles.

\subsubsection{Angular structure on the Geodesic Flow Bundle}\label{sec: global angular structure}
Since we are dealing with the directions at every point $p\in\mathcal{M}$ we need to promote our parametrization function to a field on $\mathcal{M}$:
\begin{equation}
    \hat{\Omega}^{(n-1)}_p: \ \mathcal{M} \ \to \ \text{Diff}(P,\mathcal{S}^{n-1}), \quad
    p \ \mapsto \ \hat{\Omega}^{(n-1)}_p, \qquad
    P = [0,2\pi)\times\left(-\frac{\pi}{2},\frac{\pi}{2}\right)^{n-2}
\end{equation}
The vector space $P$ (or subset of a vector space) in general denotes any viable angular parameter space but we stick to our choice above to assure consistency throughout this work.\\

We continue our angular labelling scheme from $o$ to any other point $p\in\mathcal{M}$, while satisfying the consistency condition~\eqref{eq: consistency} and restricting ourselves to angles as parameters. We can incorporate~\eqref{eq: consistency} more directly into the angular labelling of the flow from $o$ by keeping the angular label constant i.e. assigning the same label to $\Omega(p)$ and $\Omega(o)$.\\

We start with the $2$-dimensional case, where we had the parametrization function $\hat{\Omega}^{(1)}_o(\varphi)$ at $o$ and thus the geodesics from $o$ where labelled with $\gamma_{o,\varphi}$. Let $p\in\mathcal{M}$ be the point, which can be reached by heading out in direction $\alpha$ from $o$ and covering the distance $l$, then the label of two directions at $p$ are already determined by~\eqref{eq: consistency}:\\
\begin{align}
    &p = \gamma_{o,\alpha}(l); \quad \gamma_{o,\alpha}(l+\lambda) \overset{!}{=} \gamma_{p,\alpha}(\lambda) \
    \wedge \ \gamma_{o,\alpha}(l-\lambda) \overset{!}{=} \gamma_{p,\alpha+\pi}(\lambda) \\
    &\Rightarrow \quad
    \alpha \overset{\hat{\Omega}^{(1)}_p}{\mapsto} \Omega_\alpha \coloneq [\gamma_{o,\alpha}(l+\lambda)]_\sim \ \wedge \ \alpha+\pi \overset{\hat{\Omega}^{(1)}_p}{\mapsto} \Omega_{\alpha+\pi} \coloneq [\gamma_{o,\alpha}(l-\lambda)]_\sim, \quad \Omega_\alpha,\,\Omega_{\alpha+\pi}\in\mathcal{D}_p,
\end{align}
If we now demand, that the angular parametrization is smooth, especially at $o$ and that $P$ is the same and corresponds to angles everywhere, then the parametrization field $\hat{\Omega}^{1}_p$ is uniquely determined on $\mathcal{M}$.
\begin{equation}
     \lim_{l\to0}\hat{\Omega}^{(1)}_p = \hat{\Omega}^{(1)}_o \quad \Rightarrow \quad
     \measuredangle_p([\gamma_{o,\alpha}(l+\lambda)]_\sim,[\gamma_{p,0}]_\sim) \overset{!}{=} \alpha \ \wedge \
    \gamma_{p,0} = \gamma_{p,\alpha-\alpha}
\end{equation}
Consequently the geodesic with angular label $0$ at $p$ is the one which encloses the angle $\alpha$ with $\gamma_{o,\alpha}$ at $p$ and lies in direction of decreasing $\varphi$. It is as expected uniquely determined but we may be a bit surprised when we look at the example of $\mathcal{S}^2$ plotted in Fig.~\ref{fig: geodesic flow bundle on S^2}, since $\gamma_{p,0}$ intuitively seems to point in a different direction than $\gamma_{o,0}$. If we move $p$ closer to the original geodesic $\gamma_{o,0}$ the angle $\alpha$ tends to $0$ and in that limit the two geodesics coincide.
\begin{equation}
    \lim_{\alpha\to0}\gamma_{p,\alpha} = \gamma_{p,0} \ \wedge \ \lim_{\alpha\to0}\gamma_{o,\alpha} = \gamma_{o,0} \quad \Rightarrow \quad \lim_{\alpha\to0} \gamma_{p,0}(\lambda) = \gamma_{o,0}(l+\lambda)
\end{equation}
We also note, that geodesics from $q$ do not enclose the same angle with the $0$-geodesic from another point $p$ and thus their angular label changes although we continue in the same direction.\\
\begin{figure}[h!]
    \centering
    \includegraphics[width=\linewidth]{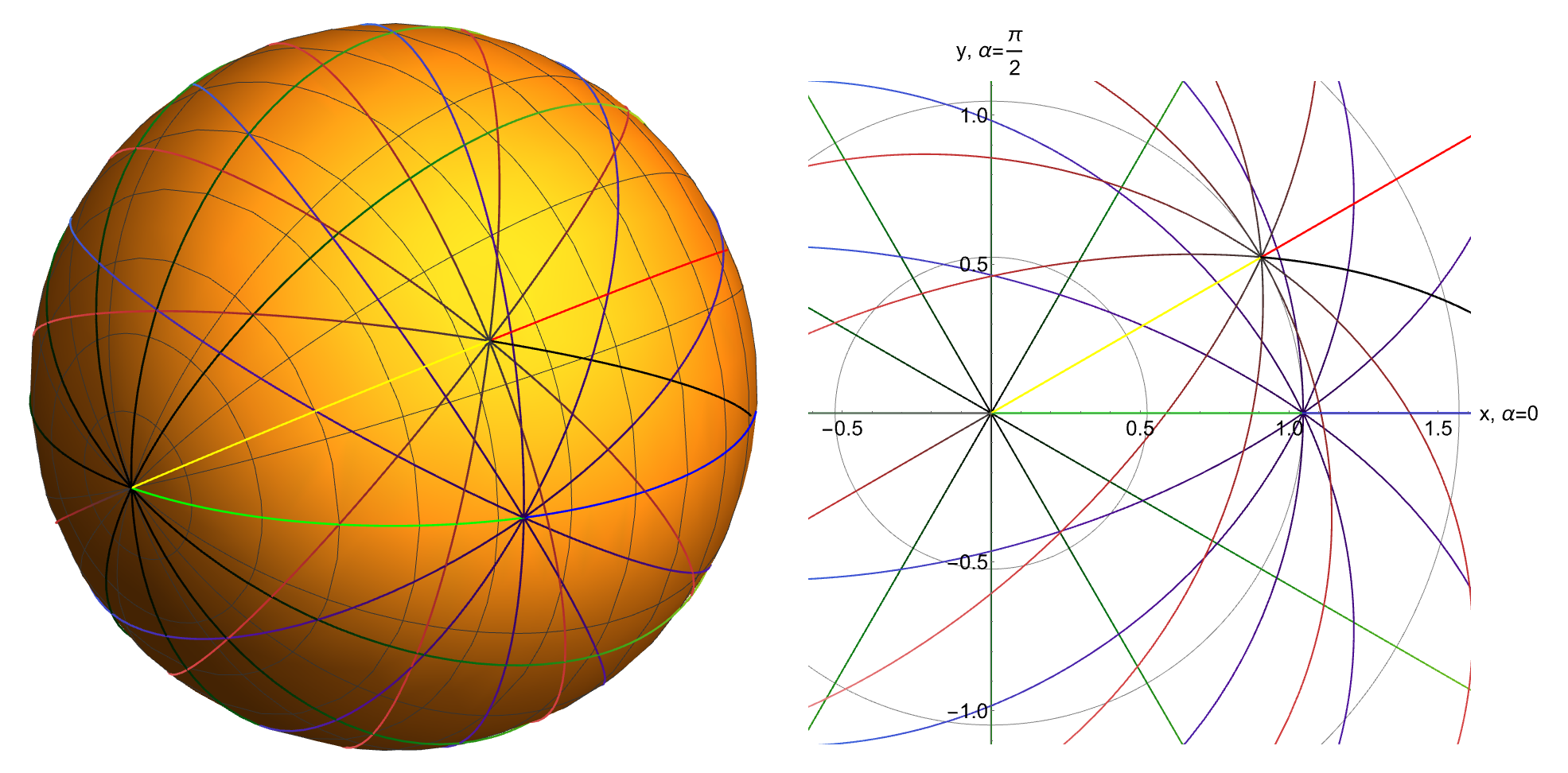}
    \caption{We plot a collection of curves from the geodesic flow bundle on $\mathcal{S}^2$ embedded in its ambient space $\mathbb{R}^3$ on the left and in the faithful normal chart at $o$, introduced in Sec.~\ref{sec: Faithful normal chart}, on the right. The green thick line is the original geodesic $\gamma_{o,0}$ in this example and is continued by the blue thick one $\gamma_{q,0}$ which is thus also labelled with $0$. The angle between the original geodesic and the yellow thick line at $o$ is $\alpha = \frac{\pi}{6}$ and thus the yellow thick line $\gamma_{o,\frac{\pi}{6}}$ and its continuation, the red thick line $\gamma_{o,\frac{\pi}{6}}$ have that angular label. The black thick line $\gamma_{p,0}$ is the one with the label $0$ and thus the angle between it and the red thick one is $\frac{\pi}{6}$ as well.}
    \label{fig: geodesic flow bundle on S^2}
\end{figure}

We are now ready to investigate the $3$-dimensional case. We again consider a point $p\in\mathcal{M}$ in an arbitrary direction $\hat{\Omega}_\alpha \coloneq \hat{\Omega}^{(2)}(\alpha_1,\alpha_2)$ at a distance $l$ from $o$. Then again the labels of the forward and backwards direction from $o$ are already determined.
\begin{align}
    &p = \gamma_{o,\hat{\Omega}_\alpha}(l); \quad \gamma_{o,\hat{\Omega}_\alpha}(l+\lambda) \overset{!}{=} \gamma_{p,\hat{\Omega}_\alpha}(\lambda) \
    \wedge \ \gamma_{o,\hat{\Omega}_\alpha}(l-\lambda) \overset{!}{=} \gamma_{p,-\hat{\Omega}_\alpha}(\lambda) \\
    &\Rightarrow \quad
    (\alpha_1,\alpha_2) \overset{\hat{\Omega}^{(2)}_p}{\mapsto} \Omega_\alpha \coloneq [\gamma_{o,\hat{\Omega}_\alpha}(l+\lambda)]_\sim \ \wedge \ (\alpha_1+\pi,\alpha_2) \overset{\hat{\Omega}^{(1)}_p}{\mapsto} -\Omega_\alpha \coloneq [\gamma_{o,-\hat{\Omega}_\alpha}(l-\lambda)]_\sim, \quad \Omega_\alpha,\,-\Omega_\alpha\in\mathcal{D}_p,
\end{align}

\begin{definition}[Geodesic Surfaces] Under the condition that geodesics form consistent surfaces a geodesic surface is defined as the image set of a one-parameter group of geodesics:\\
\begin{align}
    U_{\vec{\vartheta}} \coloneq \phi_{\vec{\vartheta}}^\varphi([0,2\pi)), \qquad
    \begin{matrix}
    \phi_{\vec{\vartheta}}^\varphi: & [0,2\pi) & \to & \Phi\mathcal{M} \\
    & \varphi & \mapsto & \gamma_{x,\hat{\Omega}^{(n-1)}(\varphi.\vec{\vartheta})}
    \end{matrix},
\end{align}
for $x$ and $\vec{\vartheta} = (\vartheta_1,\ldots,\vartheta_{n-2})$ fixed.
\end{definition}
\begin{remark}
These surfaces are naturally parametrized by $\lambda$ and $\varphi$ and are sub-flows of the flow from $x$.
\end{remark}

In $3$-dimensions the geodesic surfaces can be labelled with the $\alpha_2$-parameter, which represents the angle between the two geodesic surfaces $U_{\alpha_2}$ and $U_0$ at every point $p\in\mathcal{M}$, and can be written as sets:

\begin{equation}
    U_{\alpha_2} \coloneq \bigcup_\varphi\gamma_{o,\hat{\Omega}^{(2)}(\varphi,\alpha_2)}
\end{equation}

\begin{definition}[Plane Geodesic Flow Bundles]
When all geodesics in a geodesic surface $U_{\vec{\vartheta}}$ can be labelled consistently with $\vec{\vartheta}$, we call it a plane geometry and we say that it has a plane geodesic flow bundle.
\end{definition}
\begin{remark}
In a faithful normal chart (see Sec.~\ref{sec: Faithful normal chart}) geodesic surfaces appear as planes.
\end{remark}

We see on the example of the $3$-sphere that such geometries exists and we thus restrict ourselves to plane geometries in this work. We use this example to visualize the labeling scheme we describe in the following in Fig.~\ref{fig: Angular Structure}.\\
The first angular parameter of a geodesic heading out from a point $q = \gamma_{o,\hat{0}}(c)$ on the original geodesic will change if we look at the coinciding geodesic from a point $p = \gamma_{q,\hat{\Omega}^{(2)}(\beta_1,\beta_2)}(a)$ further up i.e.:
\begin{equation}
    \gamma_{q,\hat{\Omega}^{(2)}(\beta_1,\beta_2)}(a+\lambda) = \gamma_{p,\hat{\Omega}^{(2)}(\beta_1',\beta_2)} \quad \text{where in general:} \quad \beta_1 \neq \beta_1',
\end{equation}
as we can see from the discussion above of the $2$-dimensional case. The second angular parameter does not change however by our assumption (plane GFB). We will see in Sec.~\ref{sec: metric} that this assumption assures that the metric can be diagonalized.
\begin{figure}[h!]
    \centering\includegraphics[width=\linewidth]{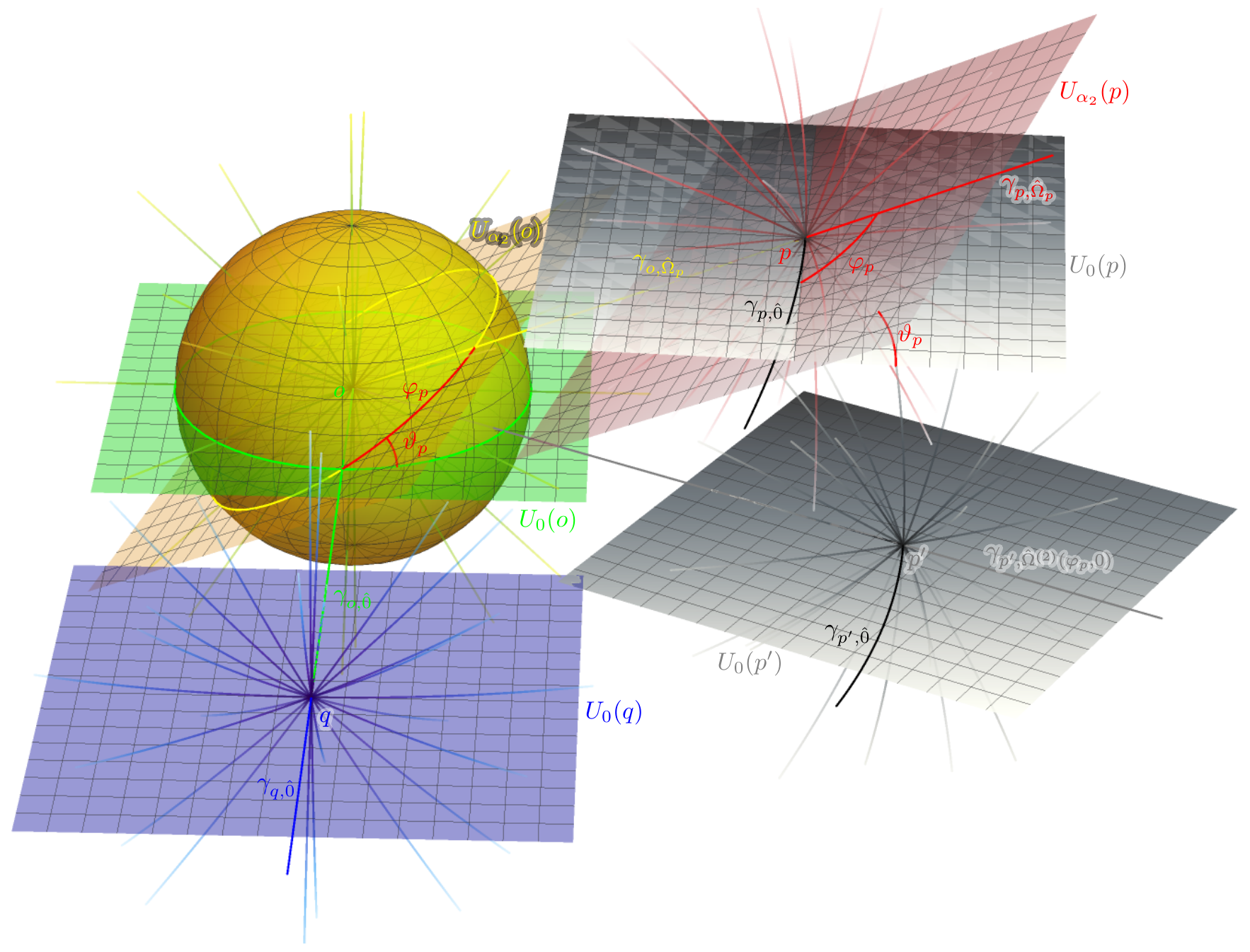}
    \caption{We sketch, how we rotate by $-\varphi_p$ in the plane $U_{\alpha_2}$, to find the geodesic $\gamma_{p,\hat{0}}$ in direction $\hat{0}$ from $p$ and then tilt the plane $U_{\alpha_2}$ by $-\alpha_2$, using $\hat{0}(p)$ as axis.}
    \label{fig: Angular Structure}
\end{figure}

To complete the labelling scheme let us investigate the $3$-dimensional analogue of the geodesic in original direction at $p$ and the consistency with the limiting cases:\\
Finding the geodesic with label $\hat{0}$ at $p$ is less straight forward in this case. In continuation of $\hat{\Omega}^{(2)}$ from $o$ to other points $p$ inherit the angular structure of the two dimensional case in the surfaces $U_{\alpha_2}$. So, we know from the previous case, that we just have to rotate $\hat{\Omega}_\alpha$ back by $\alpha_1$ in that surface $U_{\alpha_2}$ at $p$, to arrive at $\hat{0}_p$. In other words the direction labelled with $\hat{\Omega}^{(2)}_p(\alpha_1,\alpha_2)$ means, that it encloses the angle $\alpha_1$ with $\hat{0}_p = \hat{\Omega}^{(2)}(0,0)$ and is tangential to the surface $U_{\alpha_2}$. We interpret the angular labels at $p$ in the same way as we did at $o$, which guaranties consistency in the limiting case $l\to0$.\\
The original surface $U_0$ at $p$ is the one, whose intersection with $U_{\alpha_2}$ is $\gamma_{p,\hat{0}}$ and encloses the angle $\alpha_2$ with $U_{\alpha_2}$. In the plot in Fig.~\ref{fig: Angular Structure}, which is the $3$-dimensional analogue to Fig.~\ref{fig: geodesic flow bundle on S^2}, the tangent plane of $U_0$ at $p$ is not parallel to the ones at $o$ and $q$ but will agree with those in the limits of $a\to0$ or $l\to0$.\\

The demand, that we incorporate the consistency condition~\eqref{eq: consistency} into the angular labelling of the flow from $o$ and that we only use angles i.e. arclengths of the unit sphere to parametrize directions determines the angular parametrization of the entire geodesic flow bundle uniquely, given the angular parametrization at $o$.\\

The arguments given for the $3$-dimensional case can be generalized to $n$-dimensions since it is the lowest dimensional generic case. Surfaces then have multiple parameters $\alpha_2,\ldots,\alpha_{n-1}$, which are specific choices of the parameters $\vartheta_1,\ldots,\vartheta_{n-1}$. And we know now, that the angular parametrization at $o$ is carried along the geodesics from $o$ in a unique fashion.

\section{Faithful Normal Chart}\label{sec: Faithful normal chart}
The geodesic flow bundle induces a special case of normal coordinates at every point $o\in\mathcal{M}^n$. We can construct a chart on a neighborhood $U$ around $o$ by labelling every point $p\in U$ with the arclength parameter $\lambda$ and the angular parameters $\varphi,\vartheta_2,\ldots,\vartheta_{n-2}$ of the geodesic flow from $o$ to $p$.
\begin{definition}[Faithful Normal Chart]
The faithful normal chart (FNC) at a point $o\in\mathcal{M}$ is the tuple $(\mathcal{U},\mathcal{N}_o)$, where the map is defined as:
\begin{equation}
    \mathcal{N}_o: \ \mathcal{U}\subset \mathcal{M} \ \to \ \mathbb{R}^n, \quad
    p \ \mapsto \ (l,\varphi,\vartheta_1,\ldots,\vartheta_{n-2}), \qquad
    \hat{\Omega}\in\mathcal{S}^{n-1}, \ l\in\mathbb{R}_+ \quad \text{s.t.} \quad 
    p = \gamma_{o,\hat{\Omega}(\varphi,\vartheta_1,\ldots,\vartheta_{n-2})}(l).
\end{equation}
\end{definition}
We get a special normal chart, which correctly represents the angles at $o$ and the distances from $o$ to any other point $p$, which can be reached via geodesic from $o$.\\
The chart tells us, that we are supposed to head out in the direction $\hat{\Omega}(\varphi,\vartheta_1,\ldots,\vartheta_{n-2})$ and walk straight on for a distance $l$ to arrive at the point $p$. So, we are essentially completing the chart on $\mathcal{S}^{n-1}$, given by the inverse of the parametrization function $\hat{\Omega}^{-1}\circ\Phi_\mathcal{D}: \ \mathcal{D}_o \ \to \mathcal{S}^{n-1} \ \to \ P$, with the geodesic parametrization by arclength.\\

From here we can get a Cartesian version of the faithful normal coordinates by using the embedding $\imath$ of $\mathbb{S}^{n-1}$ at $o$ into $\mathbb{R}^n$, which we specified in~\eqref{eq: direction vector}:
\begin{equation}
    \mathcal{C}_o: \ \mathcal{U} \ \to \ \mathbb{R}^n, \quad p \ \mapsto \ \Vec{x} = (x_1,\ldots,x_n), \qquad
    \Vec{x} = l\,(\imath\circ\hat{\Omega}^{(n-1)})(\varphi,\vartheta_1,\ldots,\vartheta_{n-2})
\end{equation}
We can construct this embedding via group actions on $o$. An arbitrary point $p$ can be reached, by a translation along $\gamma_{o,\hat{0}}$ by c, a rotation in the $x_1,x_2$-plane with angle $\varphi$ at $o$ and subsequent rotations in the $x_2,x_i$-planes at $o$ with angles $\vartheta_i$, with $i\in\{3,\ldots,n\}$.
\begin{equation}
    \hat{\Omega}^{(n-1)}(\varphi,\vartheta_1,\ldots,\vartheta_{n-2}) = R_{x_2,x_n}(\vartheta_{n-2})\circ\ldots\circ R_{x_2,x_3}(\vartheta_1)\circ R_{x_1,x_2}(\varphi)\cdot\hat{0}
\end{equation}
where $\hat{0} = (1,0,\ldots,0)$.\\

It is more natural to work with the polar form, since this does not require an additional construction, but we use this mapping to generate plots.\\

In this chart the geodesic surfaces $U_{\vartheta_1,\ldots,\vartheta_{n-1}}$ are planes and thus all geodesics lie in these planes through $o$. The parameters $\vartheta_1,\ldots,\vartheta_{n-1}$ describe the angles by which these planes are tilted from the original plane $U_{0,\ldots,0}$. So, we basically accomplished to embed the $2$-dimensional case in the $n$-dimensional one and we dealt with the relations between these planes. This will allows us to reduce an $n$-dimensional problem into a $2$-dimensional one.\\
We will see in Sec.~\ref{sec: Curvature} that this plane structure is exactly what we need, to describe curvature in our formalism and relate it tho the geodesic flow bundle.\\

The faithful normal coordinates would precisely represent how an observer at $o$ would think about a Riemannian manifold. The straight lines are paths which the observer would follow, if it would decide to walk straight on. If the observer would after a distance $c$ decide to change direction and continue walking straight on, then that path would also appear as a straight line from $o$, since it lies in the same geodesic plane.\\

The geodesic flow bundle $\Phi\mathcal{M}$ does not require charts but provides an atlas $\mathcal{A}_{FNC} \coloneq \bigcup_{p\in\mathcal{M}}\mathcal{N}_p$ which is tightly related to it and carries it's entire geometric information. In other words we can choose charts, but we don't have to, if we don't want to.\\

We demonstrate the the geodesic flow bundle in faithful normal charts on the Manhattan metric and explain the calculations to the plot in Fig.~\ref{fig: geodesic flow bundle on S^2} in detail in the Supplemental Material.\\
From the view point of the metric formalism we would expect, that we would now introduce a notion of connection and parallel transport, to define the Riemann curvature tensor. We find in Appendix~\ref{sec: Parallel transport} that a natural notion of parallel transport is already ingrained in the GFB formalism. We do however not need it to relate the GFB to curvature as we will see in the following section. The parallel transport will become useful, when one wants to relate the GFB formalism to the Riemann tensor and can provide a consistency check of our results, which we outline in Appendix~\ref{sec: geod. eq. test recipe}.

\section{Curvature}\label{sec: Curvature}
In this section we define a notion of curvature, which is suited to our approach. We work from the assumption, that curvature is a property of any point in the manifold, which may be described by some set of numbers. A property which can in a sense be seen as the local source or generator of the manifolds global geometry.
\begin{equation}
    K: \ \mathcal{M} \ \to \ A, \quad p \ \mapsto \ K(p), \qquad A \text{ a set.}
\end{equation}
Since we view the curvature as the source of the geometry we are looking for a way to calculate the geodesic flow bundle from the values of the field.\\
In the metric formalism we work from the geometry and define a quantity which "measures" an effect of the non-flatness of that geometry. Thus the defined curvature quantities describe an effect of some source and not the source itself.

\subsection{Relating geodesics to curvature}\label{sec: relating geod. to curvature}
We saw, that in faithful normal coordinates all geodesics in plane GFB's lie in planes through $o$ and all geodesics from $o$ are straight lines. Thus the entire geometric information lies in the relation between two flows. More specifically in the flow from a point $p$ expressed in terms of of the flow from another point $o$. This is exactly describing the flow from $p = \gamma_{o,\hat{\Omega}_p}(c)$, with $\hat{\Omega}_p = \hat{\Omega}^{(n-1)}(\varphi_p,\vartheta_{1,p},\ldots,\vartheta_{n-2,p})$ in faithful normal coordinates at $o$: $\mathcal{N}_o(\gamma_{p,\hat{\Omega}}(\lambda))$. But we already know, that $\mathcal{N}_o(\gamma_{p,\hat{\Omega}}(\lambda))$ lies in a plane through $o$ which turns this into a $2$-dimensional problem as sketched in Fig.~\ref{fig: dimensional reduction}. We can without loss of generality consider the geodesic triangle spanned by $o$ and $\gamma_{q,\hat{\beta}}(\lambda)$, with $\hat{\beta} = \hat{\Omega}^{(n-1)}(\beta,0,\ldots,0)$, lying on the original geodesic $q = \gamma_{o,\hat{0}}(c)$, since a generic triangle spanned by a geodesic $\gamma_{p,\hat{\Omega}_\beta}$, with $\hat{\Omega}_\beta = \hat{\Omega}^{(n-1)}(\beta_1,\ldots,\beta_{n-1})$ from a point $p = \gamma_{o,\hat{\Omega}_p}(l)$ is tilted by the angles $\beta_2,\ldots,\beta_{n-2}$, then rotated in the $x_1,x_2$-plane with angle $\varphi_p$ and then again everything is tilted with the angles $\vartheta_{1,p},\ldots,\vartheta_{2,p}$.
\begin{figure}[h!]
    \centering\includegraphics[width=\linewidth]{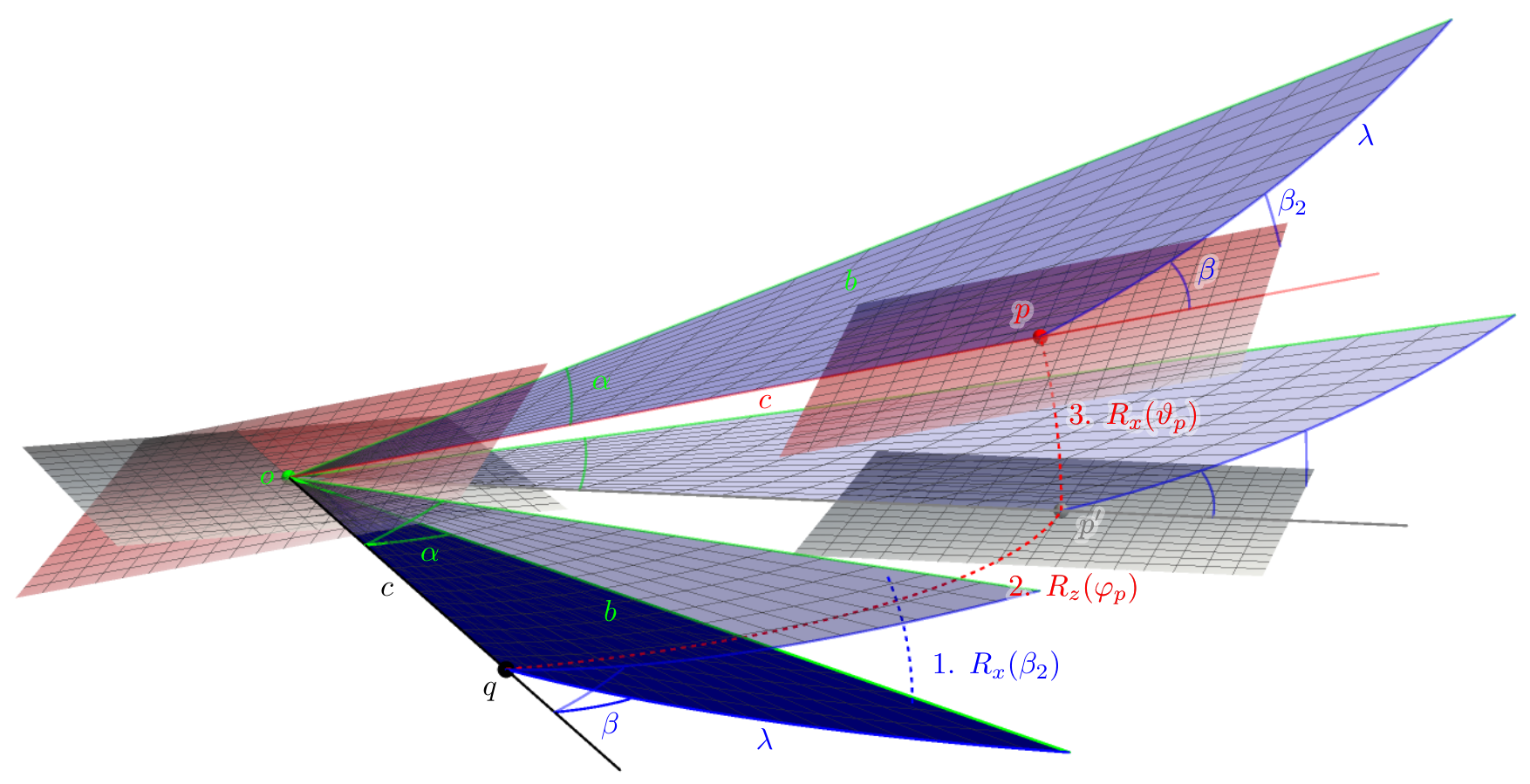}
    \caption{We sketch, how we tilt and rotate the geodesic triangle from the original plane to the point $p$ with tilt $\beta_2$ with respect to the rotated original plane (red).}
    \label{fig: dimensional reduction}
\end{figure}

In other words our generic problem is expressing the geodesic $\gamma_{q,\beta}(\lambda)$ from $q = \gamma_{o,0}(c)$ in faithful normal coordinates in $2$ dimensions at $o$. This automatically leads us to a triangle, since we need to find the geodesic $\gamma_{o,\alpha}(l)$ connecting $o$ with $p = \gamma_{q,\beta}(\lambda)$. More precisely we need to find the arclength of the top-line $b$ and the opening angle $\alpha$, given the distance to $q$ i.e. arc-length of the base line $c$, direction angle $\beta$, arclength $\lambda$ and a curvature field $K$.
\begin{equation}\label{eq: l and varphi}
    \left.\mathcal{N}_o(\gamma_{q,\hat{\beta}}(\lambda))\right|_{U_0} = \binom{b(\lambda)}{\alpha(\lambda)}, \qquad
    b(K;\beta,c,\lambda), \quad \alpha(K;\beta,c,\lambda), \qquad \hat{\beta} = \hat{\Omega}^{(n-1)}(\beta,0,\ldots,0)
\end{equation}
We call $b$ and $\alpha$ the fundamental solution to the curved triangle problem and we describe in Sec.~\ref{sec: Triangulation} and the following ones, how we calculate these two quantities from a curvature field on the geodesic triangle.\\

\subsection{Infinitesimal triangles as geometry generators}\label{sec: geometry generators}
In this section we investigate the curvature at a single point $o$ and since the entire geometric information lies in the relation between two flows (the flow from the origin point $o$ in FNC looks the same for every geometry), we consider the flow from a neighbouring point $q = \gamma_{o,0}(\varepsilon)$ at a fixed but arbitrarily small distance $\varepsilon$ from $o$. We then trace out the first infinitesimal part of the flow from $q$ i.e. only until an arbitrarily small parameter value $\delta\lambda>0$.\\
The important question is now, whether a single value determines the infinitesimal start of all the geodesics of the flow from $q$ i.e. the distribution $\delta l(\beta)$ only has one degree of freedom or whether one can do more to it. We aim to answer the following question \textit{Is it possible to choose $\delta l(\beta)$ differently for every $\beta$ or are there restrictions to that?}\\

\begin{lemma}
We find that the flat space limit restricts $\delta l(\beta)$ to one degree of freedom. Thus, there is a single curvature value per infinitesimal triangle.
\end{lemma}

\begin{proof}
If we assume, that an infinitesimal triangle in an arbitrary manifold can be sufficiently described by a spherical, flat, or pseudo-spherical one, then the first infinitesimal part of the geodesic is uniquely defined by the two infinitesimals $\delta\lambda$, $\delta l$ and the angle $\beta$, since with these a point can uniquely be determined. The distance coordinate can be calculated from a Gaussian curvature value $K$, using the cosine law for constant curvature
\begin{equation}
    \delta l(K,\beta) = \cos_K^{-1}\left( \cos_K\varepsilon\cos_K\delta\lambda + K\sin_K\varepsilon\sin_K\delta\lambda\cos(\pi-\beta)
    \vphantom{\sqrt{2}}\right), \ \cos_Kx \coloneq \cos(\sqrt{K}x), \ \sin_Kx \coloneq \frac{\sin(\sqrt{K}x)}{\sqrt{K}}
\end{equation}
and from there we can calculate the angle coordinate, using the sine law for constant curvature and thus obtain the first infinitesimal segment of a geodesic from $q$ in faithful normal coordinates:
\begin{figure}[h!]
    \begin{minipage}{0.5\linewidth}
        \begin{align}
            &\varphi(K,\beta) = \sin^{-1}\left( \frac{\sin_K\delta\lambda}{\sin_K\delta l(K)}\sin(\pi-\beta) \right), \\
            &\mathcal{N}_o(\gamma_{q,\beta}(\delta\lambda)) = \binom{\delta l}{\varphi}(K,\beta)
        \end{align}
        \begin{align*}
            &\forall\varepsilon, \delta\lambda > 0, \quad \forall\beta\in[0,2\pi) \quad \exists K\in\mathbb{R} \\
            &\text{to generate any consistent distribution} \quad \delta l(\beta)
        \end{align*}
    \end{minipage}
    \begin{minipage}{.5\linewidth}
        \centering\includegraphics[width=\linewidth]{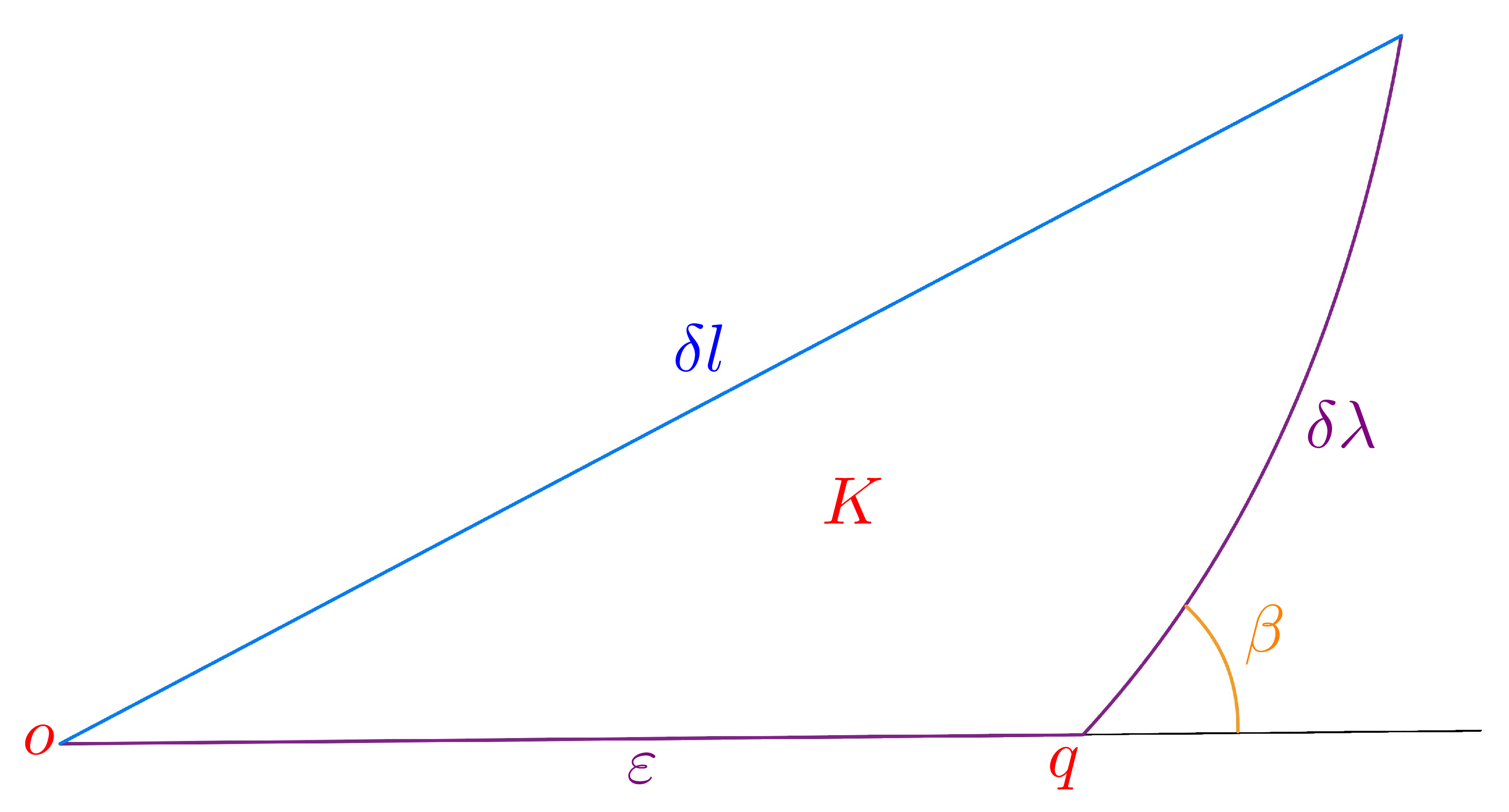}
    \end{minipage}
    \caption{A sketch of the infinitesimal triangle.}
    \label{fig:my_label}
\end{figure}

The question is though, whether this allows us to generate all consistent distributions (geodesic stomps) $\delta l(\beta)$. Every Riemannian manifold is locally flat, which means that in the limit of $\varepsilon,\delta\lambda\to0$ the relation between the $4$ quantities has to converge to the cosine law in a plane:
\begin{equation}\label{eq: flat limit}
    \delta l^2(\beta) \overset{\varepsilon,\delta\lambda\to0}{\longrightarrow}
    \varepsilon^2 + \delta\lambda^2 - \varepsilon \delta\lambda\cos\beta
\end{equation}
We cannot modify the relation for a flat space, by adding a term of second order or lower, since we would violate the limit above~\eqref{eq: flat limit}. On the other hand, only the leading order matters, since we are considering an infinitesimal triangle. Thus, the only change we can make whilst still satisfying~\eqref{eq: flat limit} is multiplying both sides with a constant factor $K$ which is exactly what the constant curvature cosine law does in this limit.
\begin{align}
    &\delta l^2 = \varepsilon^2 + \delta\lambda^2 - \varepsilon \delta\lambda\cos(\pi-\beta) \quad \overset{K\in\mathbb{R}}{\mapsto} \quad \cos_K\delta l = \cos_K\varepsilon\cos_K\delta\lambda + K\sin_K\varepsilon\sin_K\delta\lambda\cos(\pi-\beta) \\
    &\overset{\varepsilon,\delta\lambda\to0}{\longrightarrow} \quad
    K\delta l^2 \sim K\varepsilon^2 + K\delta\lambda^2 - K\varepsilon \delta\lambda\cos(\pi-\beta)
\end{align}
\end{proof}

So, in general the geodesics from $q$ are generated by a sequence of triangles spanned by the side lengts $\delta\lambda$ with curvature values $K_1, K_2, ..$ as sketched in Fig.~\ref{fig: generating a geodesic}. In the limit of $\delta\lambda\to0$ we take infinite curve segments and a direction dependent curvature field $K(\varphi)$ at $o$ and we get a flow with finite arclength.
\begin{figure}[h!]
    \centering\includegraphics[width=.5\linewidth]{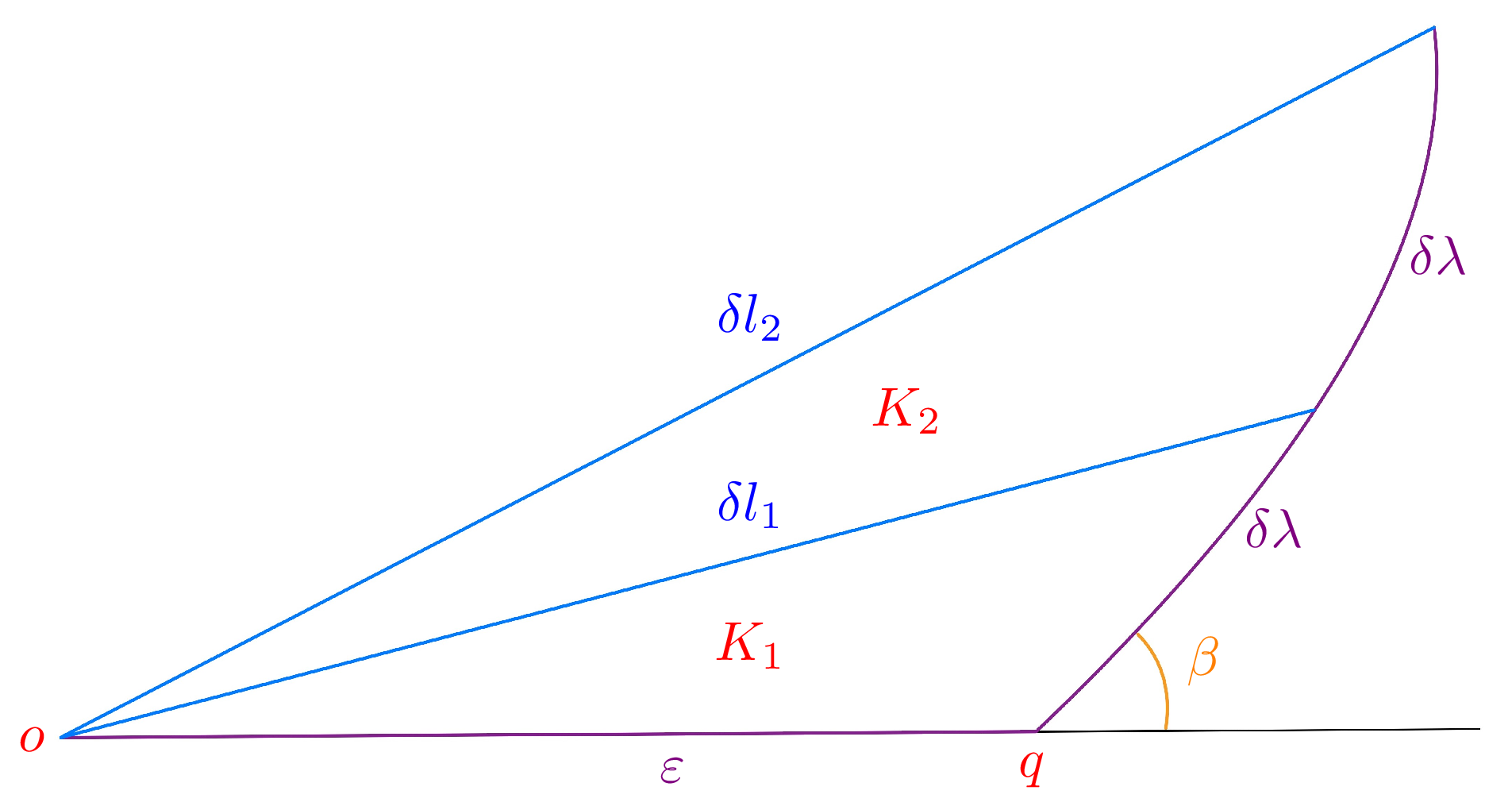}
    \caption{A sketch that visualizes, how infinitesimal triangles generate a geodesic.}
    \label{fig: generating a geodesic}
\end{figure}

We conclude that any geodesic flow bundle and thus any geometry can be generated by infinitesimal triangles with constant curvature. More precisely, if we divide a manifold into an increasingly fine sampling of constant curvature triangles, we can describe any geometry.\\
This means that the curvature degrees of freedom of a $2$-dimensional manifold are one curvature value per infinitesimal triangle which in terms of this construction here means, that every point can at most have a direction dependent curvature field $K(\varphi)$, because we can make the triangles as slim as we like. So, measuring the curvature at a point $o$ could be achieved by pointing an arbitrarily slim triangle attached to $o$ in different directions and thus scanning the curvature at that point like a radar.\\
The same argument can be made for a space-time using the cosine- and sine-laws for triangles on a pseudo-Riemannian hyperboloid.\\

\subsection{Curvature in 2 dimensions}
\begin{lemma}
The curvature at any point in a $2$-dimensional real Riemannian manifold can be sufficiently described by a single real number $K\in\mathbb{R}$.
\end{lemma}
\begin{proof}
To show that the curvature at a point of a $2$-dimensional real Riemannian manifold can be sufficiently described by a single real number we consider two isosceles triangles with legs of length $\varepsilon>0$ and vertex angles $\delta$. Motivated by the findings in the previous section we consider these generator triangles to have a priori different curvatures $K_1$ and $K_2$ and thus side lines with different lengths $\delta\lambda_1$ and $\delta\lambda_2$. This means that the geodesics which form the sides of the two triangles are completely determined by the curvature values in these triangles: $\delta\lambda_1(K_1)$ and $\delta\lambda_2(K_2)$.\\
But the isosceles triangle with vertex angle $2\delta$ is also a generator triangle and thus its side length is completely determined by it's curvature value $\delta\lambda(K)$. And we know from the previous section, that it only has one. In the limit of $\varepsilon,\delta\to0$ the union of the two smaller triangles coincide with the larger one and $\delta\lambda_1(K_1) + \delta\lambda_2(K_2) = d\lambda(K)$. Since we can only have one curvature value in a generator triangle, we conclude that the three otherwise conflicting values must agree $K_1 = K_2 = K$.
\begin{figure}[h!]
    \begin{minipage}{.5\linewidth}
        \begin{align*}
            &\exists K_1,K_2\in\mathbb{R}: &\
            \cos_{K_1}\delta\lambda_1 &= \cos_{K_1}^2\varepsilon + \sin_{K_1}^2\varepsilon\cos\delta \\
            & & \cos_{K_2}\delta\lambda_2 &= \cos_{K_2}^2\varepsilon + \sin_{K_2}^2\varepsilon\cos\delta \\
            &\text{but we also have:} &\ \delta\lambda_1 + \delta\lambda_2 &\overset{\delta\to0}{\longrightarrow} \delta\lambda \quad \text{and} \\
            &\exists K\in\mathbb{R}: &\
            \cos_K\delta\lambda &= \cos_K^2\varepsilon + \sin_K^2\varepsilon\cos2\delta \\
            \vphantom{a}\\
            &\Rightarrow \ K_1 = K_2 = K
        \end{align*}
    \end{minipage}
    \begin{minipage}{.5\linewidth}
        \centering\includegraphics[width=\linewidth]{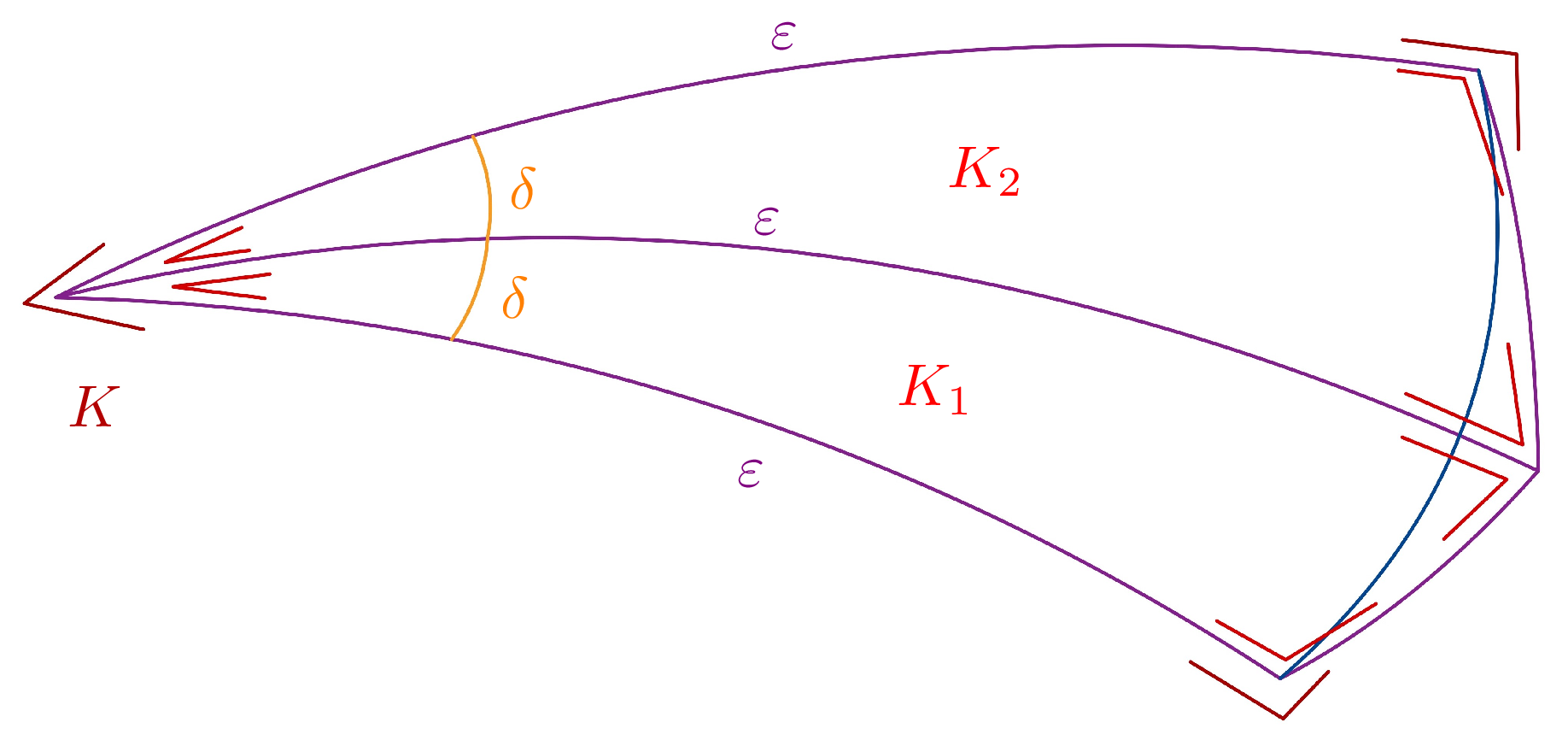}
    \end{minipage}
    \caption{Two joined infinitesimal triangles form another larger, but still infinitesimal and thus generating triangle, when the side line is adjusted to be a geodesic.}
    \label{fig:my_label}
\end{figure}

So, we find that the curvature in a $2$-dimensional manifold can be described by a scalar field. This means, that sectional curvature describes all degrees of freedom to create a geometry on a manifold.
\end{proof}

\subsection{The $n$-dimensional case}\label{sec: n-dim}
\begin{definition}[Great circle function]
We call a function on the great circles of the $(n-1)$-sphere a great circle function:
\begin{equation}
    o: \ \mathrm{O}^n \ \to \ \mathbb{R}, \quad \mathrm{O}^n \coloneq \{\mathcal{S}^1\subset\mathcal{S}^{n-1}\ |\ \mathcal{S}^1\ \text{a geodesic on}\ \mathcal{S}^{n-1}\}, \qquad n\geqslant 2.
\end{equation}
\end{definition}
\begin{remark}
A smooth great arc function on the directions $\mathcal{D}_p$ around $p\in\mathcal{M}$ is an element in $C^\infty(\Phi_\mathcal{D}^{-1}(\mathrm{O}^n);\mathbb{R})$.
\end{remark}
\begin{theorem}
Using the mapping $\Phi_\mathcal{D}$ of directions $\mathcal{D}_p$ at $p$ to the sphere $\mathcal{S}^{n-1}$ we can describe all degrees of freedom of curvature at this point by a great circle function.
\end{theorem}

\begin{proof}
If the isosceles triangles from the previous argument are tilted to one-another in a higher dimensional manifold i.e., do not lie in the same geodesic surface. Then that argument does not work anymore since we cannot form a covering triangle by taking the union. Thus, we retain the degrees of freedom of curvature in different planes at a point.
\begin{figure}[h!]
    \centering\includegraphics[width=.5\linewidth]{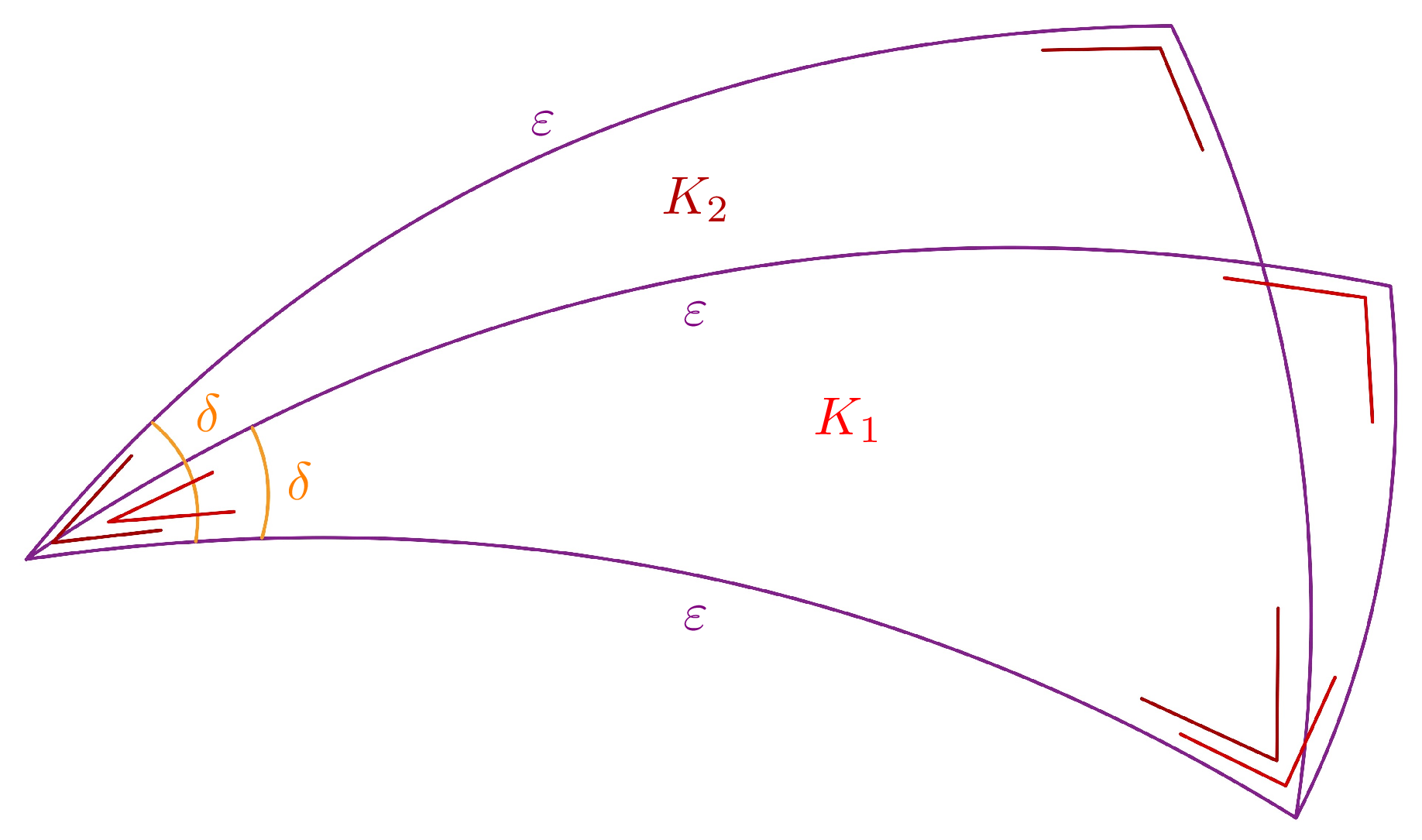}
    \caption{If the two infinitesimal geodesic triangles do not lie in the same geodesic surface, then they cannot be joined to a larger triangle.}
    \label{fig:my_label}
\end{figure}

If we imagine an arbitrarily small sphere around the origin point $o$ in FNC (faithful normal coordinates), then the geodesic planes on which we can have one curvature value intersect with that sphere on a great circle. So, we do have a sort of direction dependent curvature at $o$, where all directions which are mapped to the same great circle on the sphere must have the same curvature value. Thus, the curvature at $o$ can be described by a smooth function on the great circles of a sphere.\\
\end{proof}

\begin{definition}[Source Curvature]
We define a smooth source curvature field as a great arc field that determines the geometry over application of the cosine law to infinitesimal triangles:
\begin{equation}
    K: \ \mathcal{M} \ \to C^\infty(\Phi_\mathcal{D}^{-1}(\mathrm{O}^n);\mathbb{R})
\end{equation}
\begin{align}
    &K_0 = K(o)(\mathcal{C}); \qquad \lim_{\varepsilon\to0}\cos_{K_0}(\varepsilon c) = \lim_{\varepsilon\to0}\cos_{K_0}(\varepsilon a)\cos_{K_0}(\varepsilon b) - K_0\sin_{K_0}(\varepsilon a)\sin_{K_0}(\varepsilon b)\cos\theta, \\
    &o,q,p\in\mathcal{M},\ \text{with}\ q = \gamma_{o,\Omega}(\varepsilon c), \
    p = \gamma_{q,\tilde{\Omega}}(\varepsilon a) = \gamma_{o,\Omega'}(\varepsilon b), \qquad
    \Omega,\Omega'\in \mathcal{C} = \Phi_\mathcal{D}^{-1}(C)\subset\mathcal{D}_o, \quad C\in\mathrm{O}^n, \qquad \tilde{\Omega}\in\mathcal{D}_q, \notag \\
    &\theta = \measuredangle_p([\gamma_{q,\tilde{\Omega}}(\varepsilon a+\lambda)]_\sim,[\gamma_{o,\Omega'}(\varepsilon b+\lambda)]_\sim) \notag
\end{align}
\end{definition}

We can use the geodesic flow bundle on the sphere, to parameterize the great circles on the sphere and thus a great circle function at a point. Looking back at Eq.~\eqref{eq: angular structure sequence} and Fig.~\ref{fig: sphere flow chain} in Sec.~\ref{sec: Angular structure} we remind ourselves, that the flow bundle parameters $\lambda_{\mathcal{S}^2},\varphi_{\mathcal{S}^2}$ on the sphere $\mathcal{S}^2$ correspond to the angular parameters $\varphi,\vartheta$ of the flow from $o$ in the manifold $\mathcal{M}^3$. So, the curvature at $o$ can be parametrized with the direction parametrization function $\hat{\Omega}$ restricted to the parameter domain $P' = [0,\pi)\times\left(-\frac{\pi}{2},\frac{\pi}{2}\right]$. We will mark the angular parameters of curvature with a subscript $K$ to differentiate them from the direction parameters of a geodesic:
\begin{equation}
    \lambda_{\mathcal{S}^2} = \varphi_K \ \wedge \ \varphi_{\mathcal{S}^2} = \vartheta_K \quad \Rightarrow \quad
    K_o: \ [0,\pi)\times\left(-\frac{\pi}{2},\frac{\pi}{2}\right] \ \to \ \mathbb{R}, \quad
    (\varphi_K,\vartheta_K) \ \mapsto \ K_o(\varphi_K,\vartheta_K)
\end{equation}

Using the sequence from Eq.~\eqref{eq: angular structure sequence} again we can straight forwardly extend the curvature parametrization to $n$-dimensions:
\begin{align}
    &\lambda^{\mathcal{S}^{n-1}} = \varphi, \quad \varphi^{\mathcal{S}^{n-1}} = \vartheta_1, \quad
    \vartheta^{\mathcal{S}^{n-1}}_1 = \vartheta_2 \quad \ldots \quad \vartheta^{\mathcal{S}^{n-1}}_{n-3} = \vartheta_{n-2} \\ &\Rightarrow \quad
    \begin{matrix}
    K_o: & P'=[0,\pi)\times\left(-\frac{\pi}{2},\frac{\pi}{2}\right]^{n-2} & \to & \mathbb{R} \\
    & (\varphi,\vartheta_1,\ldots,\vartheta_{n-2}) & \mapsto & K_{\hat{\Omega}^{(n-1)}|_{P'}}(o)
    \end{matrix}
\end{align}

And to parametrize the curvature on the entire manifold $\mathcal{M}^n$ we promote it to a field
\begin{equation}
    K: \ \mathcal{M} \ \to \ C^\infty(P';\mathbb{R}), \quad p \ \mapsto \ K_{\hat{\Omega}^{(n-1)}|_{P'}}(p)
\end{equation}

\subsection{Parametrization of a $3$-dimensional Gaussian curvature field}\label{subsec: Curvature parametrization}
We saw, that the flow from the initial geodesic $\gamma^{\mathcal{S}^2}_{q(\varphi),\vartheta}$ parametrizes all great-circles on the $2$-sphere. Thus we can provide a cleaner definition of the set of great-circles and great-circle functions on the set of directions $\mathcal{D}_p\simeq\mathcal{S}^2$ at a point $p\in\mathcal{M}^3$:
\begin{equation}\label{eq: Great-arcs}
    o: \ \mathrm{O}^3 = \left\{\gamma^{\mathcal{S}^2}_{q(\varphi_K),\vartheta_K}|\,\varphi_K\in[0,\pi),\vartheta_K\in\left(-\frac{\pi}{2},\frac{\pi}{2}\right]\right\} \ \to \ \mathbb{R}
\end{equation}

We generated the geodesics of a flow from an arbitrary point $p\in\mathcal{M}$ from the curvature value of an infinitesimal generator triangle. We want to know now, how to read out the correct value from the curvature function $K_p(\varphi_K,\vartheta_K)$ at $p\in\mathcal{M}^3$ to such a geodesic heading out from $p$ in direction $\hat{\Omega}_\beta = \hat{\Omega}(\beta_1,\beta_2)$.\\
Since geodesic surfaces become planes in faithful normal coordinates and these planes are directly related to the great-circles, we need the relation between a geodesic $\gamma_{p,\hat{\Omega}_\beta}$ from a point $p$ and the geodesic $\gamma^{\mathcal{S}^2}_{q(\varphi_K),\vartheta_K}$ on the $2$-sphere $\mathcal{S}^2\simeq\mathcal{D}_o$ at $o$, which corresponds to the plane $U$, in which $\gamma_{p,\hat{\Omega}_\beta}$ lies (see Fig.~\ref{fig: Ring-Function evaluation}). The plane $U$ is obtained by rotating the original plane $U_0$ in the following way:
\begin{equation}
    U = R_x(\vartheta_p)\circ R_z(\varphi_p)\circ R_x(\alpha_2-\vartheta_p)\cdot U_0
\end{equation}
We want to have the parameters $\varphi_K$ and $\vartheta_K$ we used in~\eqref{eq: Great-arcs}. This amounts to translating to the simpler rotation:
\begin{equation}
    U = R_z(\varphi_K)\circ R_x(\vartheta_K)\cdot U_0
\end{equation}
\begin{figure}[h!]
    \centering\includegraphics[width=.5\linewidth]{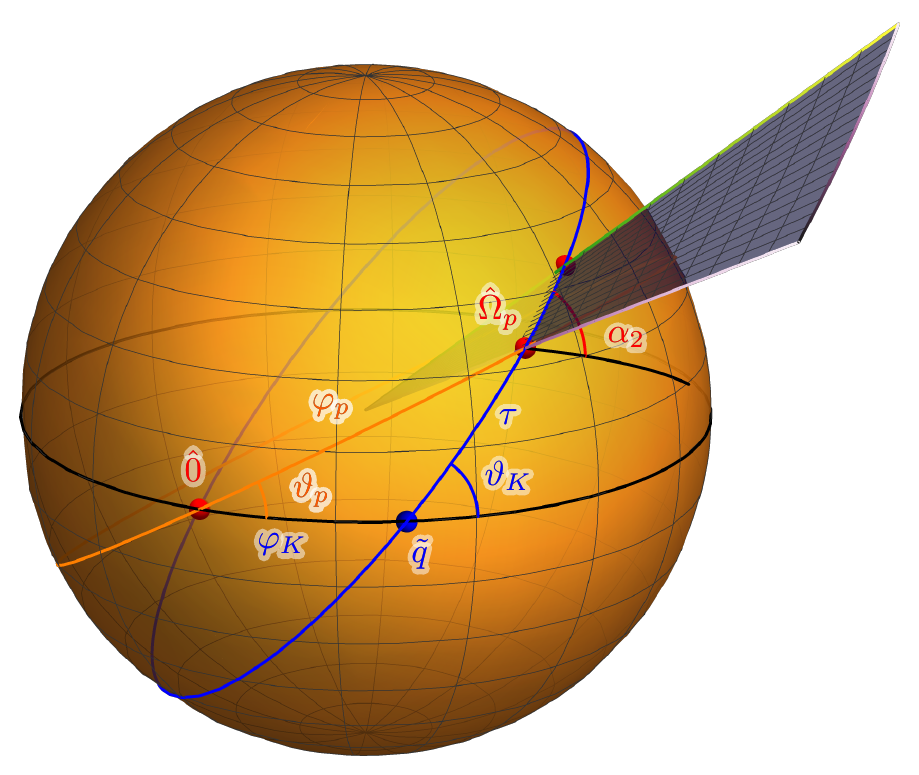}
    \caption{The black great-arc is the intersection of the original plane with the unit sphere in yellow, which represents the directions of the point $o$ in the center of the sphere. The great-arc in blue is associated to the plane in which the geodesic in question $\gamma_{p,\hat{\Omega}_\beta}$ lies. The tilt parameter $\alpha_2=\beta_2$ in red is the angle between that great-arc and the geodesic in black with angular label $0$ on the sphere at that point $\hat{\Omega}_p\in\mathcal{S}^2\simeq\mathcal{D}_o$.}
    \label{fig: Ring-Function evaluation}
\end{figure}

We use the sine and cosine laws on $\mathcal{S}^2$ (i.e. $K=1$) to relate the great-circle function parameters $\varphi_K,\vartheta_K\in P'$ to the direction parameters $\varphi_p,\vartheta_p$ of the point $p = \gamma_{o,\hat{\Omega}_p}(l)$ and the tilt $\alpha_2 = \beta_2$ of the plane, in which the geodesic in question $\gamma_{p,\hat{\Omega}_\beta}$ lies. In other words we solve the equation:  $\gamma^{\mathcal{S}^2}_{\hat{\Omega}_p,\alpha_2}(\tau+\lambda) \overset{!}{=} \gamma^{\mathcal{S}^2}_{\hat{\Omega}(\varphi_K,0),\vartheta_K}(\lambda)$.\\

We are describing a geodesic from an arbitrary point $\hat{\Omega}_p$ on $\mathcal{S}^2$ in terms of one originating from a point $\hat{\Omega}(\varphi_K,0) = \gamma^{\mathcal{S}^2}_{\hat{0},0}(\varphi_K)$ on the initial geodesic $\gamma^{\mathcal{S}^2}_{\hat{0},0}$. This is essentially applying a special case of the consistency condition~\eqref{eq: consistency}.\\

The problem splits into different cases which we discuss more in Appendix~\ref{sec: Branches}. For simplicity we restrict ourselves to one case.
\begin{align}
    \varphi_K = \tan^{-1}\left( \frac{\sin\varphi_p}{\left| \cot(\alpha_2 - \vartheta_p)\sin\vartheta_p\right|
    + \cos\varphi_p\cos\vartheta_p}\right), \quad
   \vartheta_K = \sin^{-1}\left(\frac{\sin\varphi_p\sin(\alpha_2-\vartheta_p)}{\sin\varphi_K(\varphi_p,\vartheta_p,\alpha_2)}\right)
\end{align}
Extending it to more cases is not difficult, but a tedious task to adjust the branches of the inverse trigonometric functions.\\

Ultimately, we parametrize the curvature field $K$ with the location $p\in\mathcal{M}$ in faithful normal coordinates $\mathcal{N}(p) = (l,\varphi_p,\vartheta_p)$ and the orientation parameters $\varphi_K,\vartheta_K$ of the geodesic plane, in which the geodesic in question lies: $K_p(\varphi_K,\vartheta_K)$.

\subsection{Restrictions to the curvature field}
We parametrized the curvature field and found, that the degrees of freedom at every point correspond to the great-circles on the $(n-1)$-sphere at this point. But there might be dependencies which reduce the degrees of freedom. To investigate, whether this is the case we start with choosing a scalar curvature field on every plane through $o$ in a smooth manner: $K_{\hat{\Omega}^{(n-1)}|_{P'}}(l,\varphi)$. Now all geodesics on these planes in FNC are determined, including the ones, which form the curved surface depicted in Fig.~\ref{fig: d.o.f of curvature}.\\
\begin{figure}[h!]
    \centering\includegraphics[width=0.6\linewidth]{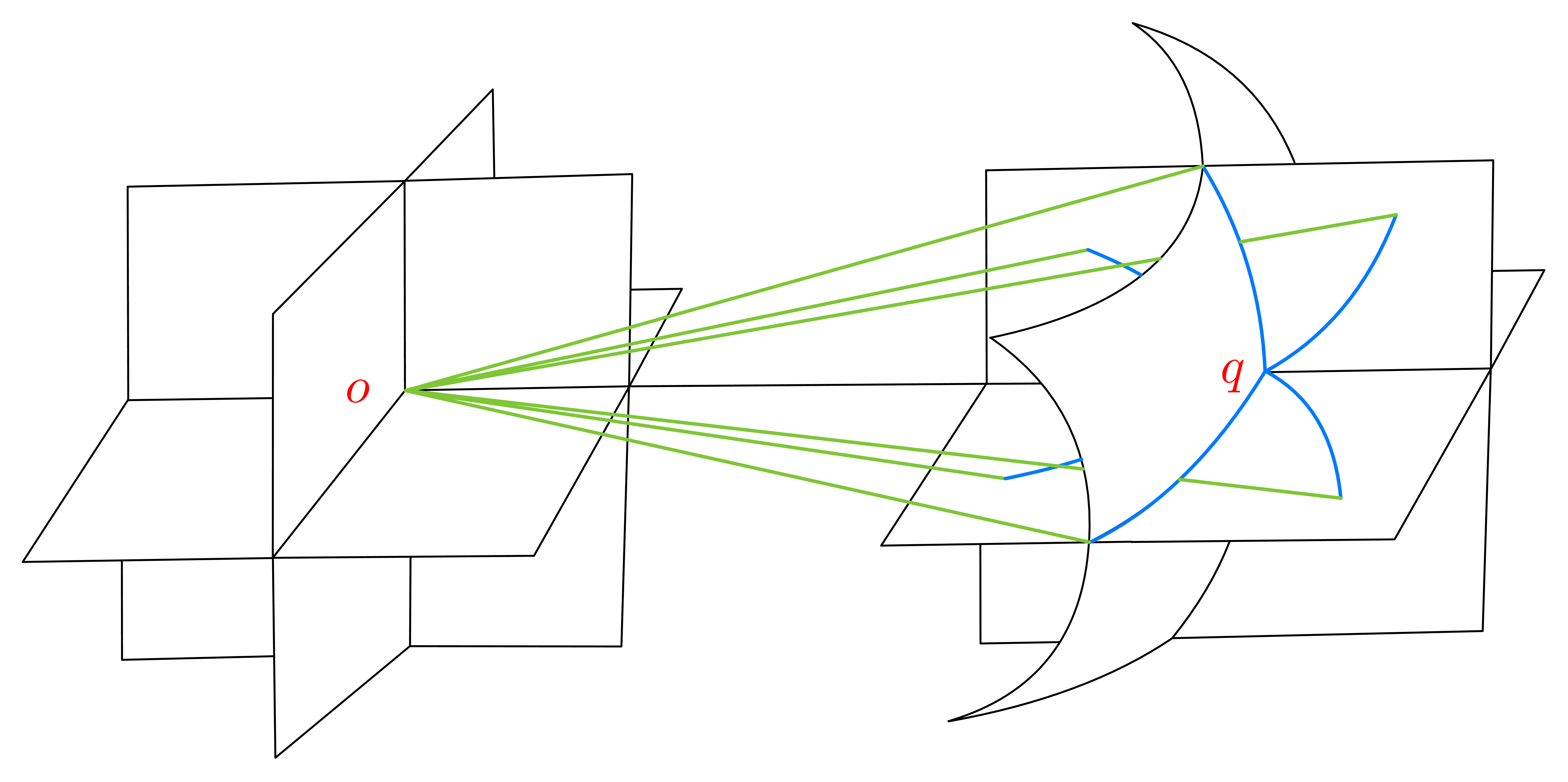}
    \caption{A sketch of the geodesic surfaces, geodesics of the flow from $o$ in green and from another point $q$ in blue.}
    \label{fig: d.o.f of curvature}
\end{figure}
But we know, that all geodesics lie on geodesic planes through $o$. So, the entire geometry of a manifold is determined by the curvature field on the geodesic surfaces through a single point w.l.o.g. $o$.\\

It should be noted that the remaining degrees of freedom are much more than the Riemann tensor can incorporate. We describe an example of a geometry in the Supplemental Material~\sref{subsec: NotASphere}, where the metric formalism fails to properly describe it.

\section{Relation to the metric}\label{sec: metric}
We constructed a collection of charts in Sec.~\ref{sec: Faithful normal chart}, which are tightly related to the geodesic flow bundle. These charts allow us, to relate the geodesic flow bundle to the metric. In this section we derive the procedure to calculate the components of the metric in a faithful normal chart.\\

For this we need the tangent vector fields of the coordinate lines of $l$ and $\hat{\Omega}$. With coordinate lines we mean the curves traced out by increasing only one coordinate of the faithful normal chart $\mathcal{N}$. This creates the flows:
\begin{equation}
    \phi_\mathcal{N}^l(p) = (l_p+l,\vec{\vartheta}_p), \quad \mathcal{N}(p) = (l_p,\vec{\vartheta}_p), \qquad \phi_\mathcal{N}^\varphi, \qquad \phi_\mathcal{N}^{\vartheta_1}, \quad \ldots \quad \phi_\mathcal{N}^{\vartheta_{n-2}},
\end{equation}
we collect the angular parameters in the parameter vectors $\vec{\vartheta} = (\varphi,\vartheta_1,\ldots,\vartheta_{n-2}),\ \vec{\vartheta}_p = \vartheta^i_p\in P$.\\

In the case of the coordinate $l$ the coordinate lines coincide with the geodesic flow from $o$ and thus the tangent vector field is just it's parameter derivative:
\begin{equation}
    \phi_\mathcal{N}^l = \gamma_{o,\hat{\Omega}_p} \quad \Rightarrow \quad
    \partial_l(p(\hat{\Omega}_p,l_p)) = \dot{\gamma}_{o,\hat{\Omega}_p}(l_p) = \left.\frac{d}{d\lambda}\gamma_{o,\hat{\Omega}_p}(\lambda)\right|_{\lambda=l_p}
\end{equation}

The geodesic $\gamma_{o,\hat{\Omega}_p}$ is by definition parametrized by arclength $\lambda$ and thus: $\Vert\gamma_{o,\hat{\Omega}_p}(\lambda)\Vert = 1$ $\forall\lambda\in I$.
\begin{equation}
    g_{ll}(p) = g_p(\partial_l(p),\partial_l(p)) = g_p(\dot{\gamma}_{o,\hat{\Omega}_p}(l_p),\dot{\gamma}_{o,\hat{\Omega}_p}(l_p))
    = \Vert\gamma_{o,\hat{\Omega}_p}(\lambda)\Vert^2 = 1 \quad \forall p\in U
\end{equation}

The coordinate lines of the angular parameters $\vartheta^i$ lie on $(n-1)$-spheres and their tangent vector fields are thus orthogonal to the $l$-line rays.\\
\begin{align}
    &\partial_{\vartheta^i}(p) = \dot{\phi}_\mathcal{N}^{\vartheta^i}(p) = \left.\frac{d}{d\vartheta^i}\phi_\mathcal{N}^{\vartheta^i}(p)\right|_{\vartheta^i=\vartheta^i_p}; \qquad
    \dot{\phi}_\mathcal{N}^{\vartheta^i}(p) \perp \dot{\phi}_\mathcal{N}^l(p) = \dot{\gamma}_{o,\hat{\Omega}_p}(l_p) \quad
    \Leftrightarrow \quad \beta \coloneq \measuredangle_p([\phi_\mathcal{N}^{\vartheta^i}(p)]_\sim,[\gamma_{p,\hat{\Omega}_p}]_\sim) = \frac{\pi}{2} \\
    &\Rightarrow \quad g_{\vartheta^il}(p) = g_p(\dot{\phi}_\mathcal{N}^{\vartheta^i}(p),\dot{\gamma}_{o,\hat{\Omega}_p}(l_p))
    = \Vert\dot{\phi}_\mathcal{N}^{\vartheta^i}(p)\Vert \Vert\dot{\gamma}_{o,\hat{\Omega}_p}(l_p)\Vert \cos\beta = 0 \quad
    \forall p\in U
\end{align}

The vector fields of our angular parameters form an orthogonal basis on the $(n-1)$-sphere, wherever it is not degenerate.
We can calculate the vector fields of the angular coordinate lines, by taking derivatives of the angular parametrization function $\hat{\Omega}$, since the $l$-coordinate held constant anyways. Since the angular parametrization we use is generated by subsequently acting rotations in orthogonal planes\footnote{The concept of rotations around an axis only works in the special case of $3$-dimenstions. In higher dimensions the ases would generate to $n-2$-dimensional submanifolds. In the $2$-dimensional case we can see however, that rotations are an action in a surface. There is no additional direction in which we could have an axis.\\
In this work we consider rotations as the action of $O(n-1)$ on the $(n-1)$-spere, representing directions at a point $p$, embedded in $\mathbb{R}^n$: $\imath(\mathcal{S}^{n-1})$\\
Thus it is a local operation which happens in a plane in the ambient space which translates to an action which happens on a ring on the manifold of directions $\mathcal{D}_p$ at $p$.} (considering the embedding $\imath$ of $\mathcal{S}^{n-1}$ in $\mathbb{R}^n$) on the initial direction $\hat{0}$, the orbits of each group action are automatically orthogonal to the subspace in which the orbits of the previous actions lie. And thus all angular coordinate lines are orthogonal to each other.\\
\begin{align}
    \dot{\phi}_\mathcal{N}^{\vartheta^i}(p) = l_p\hat{\Omega}_{,\vartheta^i}(\vec{\vartheta}_p), \qquad
    \hat{\Omega}(\vec{\vartheta}) = R_{x_n,x_2}(\vartheta_{n-2})\circ\ldots\circ R_{x_3,x_2}(\vartheta_1)\circ R_{x_1,x_2}(\varphi)\cdot\hat{0}, \quad \Rightarrow \quad \dot{\phi}_\mathcal{N}^{\vartheta^i} \perp \dot{\phi}_\mathcal{N}^{\vartheta^j}, \ i\neq j
\end{align}
Thus, the metric in faithful normal coordinates is always diagonalized: $g_{\vartheta^i\vartheta^j}(p) = 0, \ i\neq j, \ \forall p\in\mathcal{U}$.
\footnote{We are not talking about the angular structure at a point $p$ here, since this is in general not correctly represented in the faithful normal chart. We are using the angular structure at $o$, which is projected outwards, using geodesics.}\\

The coordinate lines $\phi_\mathcal{N}^{\vartheta^i}(p)$ are in general no geodesics (consider for example the constant $\theta$ lines $\phi_\mathcal{N}^\varphi(p)$ on the $2$-sphere) and thus their tangent vectors are in general no unit vectors. To find the diagonal entry of $\vartheta^i$ in the metric we use the tangent vector of a geodesic heading out in direction of increasing $\vartheta^i$ from a point $p$ and the already determined components of the metric.\\
\begin{align}
    \hat{\vartheta}^i \coloneq N\frac{d\hat{\Omega}}{d\vartheta^i}(p)\in\mathcal{S}^{n-1}, \quad
    \dot{\gamma}_{p,\hat{\vartheta}^i}(0) =& \left.\frac{d}{d\lambda}\mathcal{N}_o(\gamma_{p,\hat{\vartheta}^i})^j\right|_{\lambda=0} \partial_j(p) = \dot{l}(p,\hat{\vartheta}^i;0)\partial_l(p) +
    \dot{\vartheta}^j(p,\hat{\vartheta}^i;0)\partial_{\vartheta^j}(p)  \notag \\
    =& \dot{l}(p,\hat{\vartheta}^i;0)\partial_l(p) +
    \dot{\vartheta}^i(p,\hat{\vartheta}^i;0)\partial_{\vartheta^i}(p), \quad \text{using} \quad \dot{\vartheta}^j(p,\hat{\vartheta}^i;0) = 0, \ j\neq i
\end{align}
where $N$ is a normalization constant to make it a unit vector.\\
Since this is a geodesic (and in general not a coordinate line) it has length $1$. Writing this in components gives us an equation with which we can determine the missing metric component:\\
\begin{minipage}{.6\linewidth}
    \begin{align}\label{eq: metric}
        \Vert\dot{\gamma}_{p,\hat{\vartheta}^i}(0)\Vert^2
        &= \mathcal{N}(\dot{\gamma}_{p,\hat{\vartheta}_i}(0))^i\ \mathcal{N}(\dot{\gamma}_{p,\hat{\vartheta}_i}(0))^j\
         g_p(\partial_i(p),\partial_j(p)) \notag \\
        &= \dot{l}(p,\hat{\vartheta}^i;0)^2 \underbrace{g_{ll}(p)}_{=1} + \dot{\vartheta}^i(p,\hat{\vartheta}^i;0)^2 g_{\vartheta^i\vartheta^i}(p) = 1 \notag
    \end{align}
\end{minipage}
\begin{minipage}{.4\linewidth}
    \begin{equation}
        \Rightarrow \qquad g_{\vartheta^i\vartheta^i}(p)
        = \frac{1 - \dot{l}(p,\hat{\vartheta}^i;0)^2}{\dot{\vartheta}^i(p,\hat{\vartheta}^i;0)^2}
    \end{equation}
\end{minipage}

We saw in Sec.~\ref{sec: relating geod. to curvature} that calculating the geodesic flow bundle from the curvature field boils down to calculating the length of the top-line $b$ and opening angle $\alpha$ of a triangle from a curvature field $K$ and is thus essentially a $2$-dimensional problem. So, we want to express this tangent vector in terms of the solution to this problem~\eqref{eq: l and varphi} (fundamental solution), which we will derive in the following sections.\\
The geodesics $\gamma_{q,\hat{\Omega}_\beta}$ from a point $q = \gamma_{o,\hat{0}}(c)$ on the original geodesic can be directly written in terms of the fundamental solution $b(K)$, $\alpha(K)$ and the tilts $\beta_2,\ldots,\beta_{n-1}$ at $q$ of the plane in which the triangle lies. In a plane GFB the tilt parameters $\alpha_2,\ldots,\alpha_{n-1}$ at $o$ are the same as the ones at $q$ and we have:\\
\begin{equation}
    l = b(K;c,\beta_1;\lambda), \quad \alpha_1 = \alpha(K;c,\beta_1;\lambda), \quad \alpha_2 = \beta_2, \quad
    \ldots, \quad \alpha_{n-1} = \beta_{n-1};
\end{equation}
\begin{equation}\label{eq: FNC equation}
    \gamma_{q,\hat{\Omega}_\beta}(\lambda) \overset{!}{=} l\,\hat{\Omega}(\vec{\alpha}) = \gamma_{o,\hat{\Omega}_\alpha}(l) \quad \Rightarrow \quad
    \mathcal{N}_o(\gamma_{q,\hat{\Omega}_\beta}) = \left( b(K;c,\beta_1;\lambda), \alpha(K;c,\beta_1;\lambda),
    \beta_2, \ldots, \beta_{n-1} \vphantom{\sqrt{2}}\right)^T
\end{equation}
To calculate the flow from a different point $p$ using the fundamental solution $(b,\alpha)$ we just rotate the tilted triangle to $p$ according to our parametrization function $\hat{\Omega}$, see Fig.~\ref{fig: Triangle embedding}\footnote{Note that we are using $\beta_2$ as a coordinate in Fig.~\ref{fig: Triangle embedding} rather than the angle between the red plane and the blue triangle as it is used in Eq.~\eqref{eq: FNC equation}. The angle between the red plane $U_{\vartheta_p}$ and the blue triangle $U_{\beta_2}$ in the plot is given by $\beta_2-\vartheta_p$. So, the two conventions are related by: $\beta_2\leftrightarrow\beta_2-\vartheta_p$.\\
The $\beta_i$ are the angular coordinates at $q$, but not anymore, when we rotate them to $p$. So, they are the angular labels the geodesic $\gamma_{p,\hat{\Omega}_\beta}$ would have, if it were heading out from $q$.}. Then we have to solve Eq.~\eqref{eq: FNC equation} again to obtain the components $(l,\varphi,\vartheta_2,\ldots,\vartheta_{n-2})$ in FNC.
\begin{align}
    \gamma_{p,\hat{\Omega}_\beta}(\lambda) &= R_{x_n,x_2}(\vartheta_{n-2,p})\circ\ldots\circ R_{x_3,x_2}(\vartheta_{1,p})\circ R_{x_1,x_2}(\varphi_p)\cdot\gamma_{q,\hat{\Omega}_\beta}(l) \notag \\
    &= l(K)\,R_{x_n,x_2}(\vartheta_{n-2,p})\circ\ldots\circ R_{x_3,x_2}(\vartheta_{1,p})\circ R_{x_1,x_2}(\varphi_p)\cdot\hat{\Omega}(\varphi(K),\beta_2,\ldots,\beta_{n-1})
    \overset{!}{=} l\,\hat{\Omega}(\vec{\vartheta}) = \gamma_{o,\hat{\Omega}}(l)
\end{align}
Note, that $l$ drops out and this equation is only about the angular part, which is the same for every geometry. One gets a solution of the type: $\vartheta^i(\hat{\Omega}_p,\hat{\Omega}_\alpha)$.\\
The $2$-dimensional case is trivial and we present the solutions for the $3$-dimensional case in the Supplemental Material, but in general this is a nontrivial transcendental equation and attempting to derive the general $n$-dimensional solution is not within the scope of this work.\\
\begin{figure}[h!]
    \centering\includegraphics[width=1\linewidth]{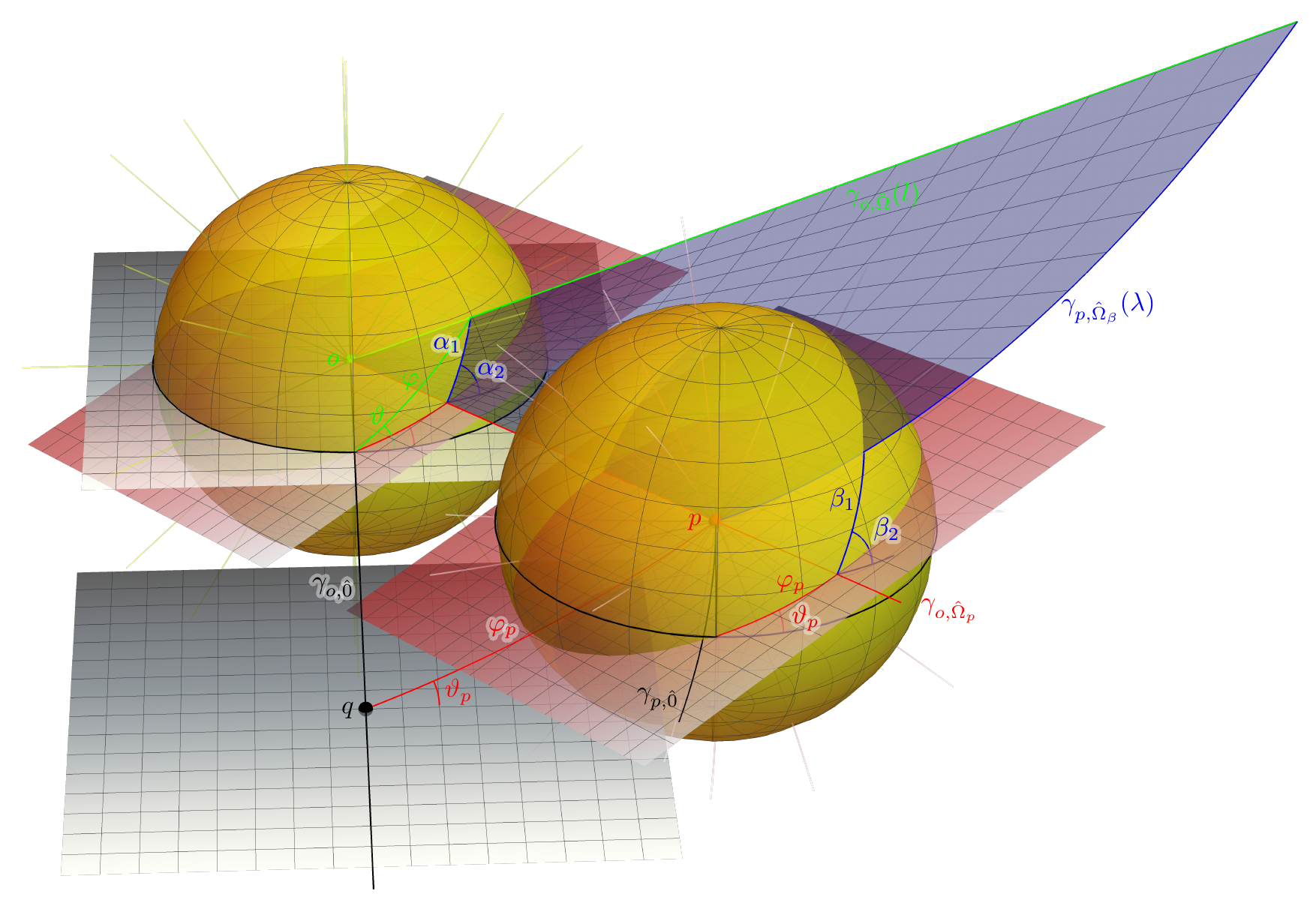}
    \caption{The figure shows, how a geodesic triangle is embedded in an $n$-dimensional manifold and how we rotate and tilt the flow from $q$ to get the flow from $p$.}
    \label{fig: Triangle embedding}
\end{figure}

Now we are ready to do this with the geodesics $\gamma_{p,\hat{\vartheta}^i}$, which we need to calculate the metric components. We calculated the direction vectors $\hat{\vartheta}^i$ at an arbitrary point $p$, since we are solving the problem at a point $q$ on the original geodesic we have to rotate that direction vectors back to $q$, to find the parameters $\vec{\beta}$:
\begin{equation}
    \hat{\Omega}_\beta = \hat{\Omega}(\vec{\beta}) \overset{!}{=} R_{x_1,x_2}^{-1}(\varphi_p)\circ R_{x_3,x_2}^{-1}(\vartheta_{1,p})\circ\ldots\circ R_{x_n,x_2}^{-1}(\vartheta_{n-2,p})\cdot\hat{\vartheta}^i(p)
    \quad \Rightarrow \quad \vec{\beta}^i
\end{equation}
Finally, the FNC component functions of the geodesics $\gamma_{p,\hat{\vartheta}^i}$ to the coordinate direction vectors $\hat{\vartheta}^i$ are determined by the direction vectors $\hat{\Omega}_p$  of $p$ (containing the information to rotate the triangle from $q$ to $p$) and $\hat{\Omega}_\alpha$\footnote{The triangle vector $\hat{\Omega}_\alpha$ can be interpreted as a $n$-dimensional generalization of the opening angle, additionally describing the orientation of the opening angle via the tilt parameters.}, describing the triangle as it would be at q. The triangle vector $\hat{\Omega}_\alpha$ is composed of the opening angle $\alpha_1$, which is the angular component of the fundamental solution $\alpha(K)$, and the tilts $\vec{\beta}^i$.
\begin{equation}
    \vartheta^i(p,\hat{\vartheta}^i;\lambda) = \vartheta^i\left( \hat{\Omega}_p , \hat{\Omega}\left( \alpha(K;c,\beta_1^i;\lambda),\beta_2^i,\ldots,\beta_{n-1}^i \vphantom{\sqrt{2}}\right) \right)
    = \mathcal{N}_o(\gamma_{p,\hat{\vartheta}^i}(\lambda))^{i+1}
\end{equation}
The $+1$ in the index is because $l$ is the first coordinate which is just the distance component of the fundamental solution evaluated for the correct angular coordinate:
\begin{equation}
    l(p,\hat{\vartheta}^i;\lambda) = b(K;c,\beta_1^i;\lambda)
\end{equation}
We demonstrate this procedure on the examples in the Supplemental Material~\sref{sec: Examples} for which we can calculate the functions $l(K)$ and $\varphi(K)$ with the constant curvature cosine- and sine-laws. The counter example in~\sref{subsec: NotASphere} demonstrates, that the metric cannot pick up all the curvature degrees of freedom, that a GFB can describe.

\newpage

\section{Spherical triangulation}\label{sec: Triangulation}
We want to write all geodesics in faithful normal coordinates, which amounts to describing geodesics from an arbitrary point $p$ in terms of the flow from $o$. So, this is essentially describing the relations between the flows from different points. This is, where the entire information about the geometry lies. The flow from $o$ always appears as straight lines with varying lengths. So, there is not much one can learn about a geometry from the flow of a single point.\\
We have seen in the Secs.~\ref{sec: n-dim} and~\ref{sec: metric} how one can extend the solution for the $2$-dimensional problem at a point $q$ on the initial line to $n$-dimensions and arbitrary points $p\in\mathcal{M}$. The bulk of the complexity of determining geodesics from curvature lies in this $2$-dimensional problem. We are going to dedicate the main part of the paper to calculating the leading order terms of the fundamental solution i.e. distance $b$ (top line) from $o$ to $p$ and the angle $\alpha$ (opening angle) between the original geodesic $\gamma_{o,\hat{0}}$ and the geodesic leading to $p$, given the distance $c$ (base line) to a point $q$ on $\gamma_{o,\hat{0}}$, the direction $\beta$ (direction angle) in which the geodesic in question $\gamma_{q,\beta}$ heads out from $q$ and the distance $a$ (side line) which we follow along this geodesic, to arrive at $p$.\\
\begin{equation}
    b(K;\beta,a,c), \qquad \alpha(K;\beta,a,c), \qquad K\in C^\infty(\mathcal{M};\mathbb{R})
\end{equation}
This spans the triangle $o$, $q$, $p$ (Fig.~\ref{fig: scheme-main}) where we know two side-lengths $c$ and $a$ and the an angle $\pi-\beta$.\\
\begin{figure}[h!]
    \centering\includegraphics[width=\linewidth]{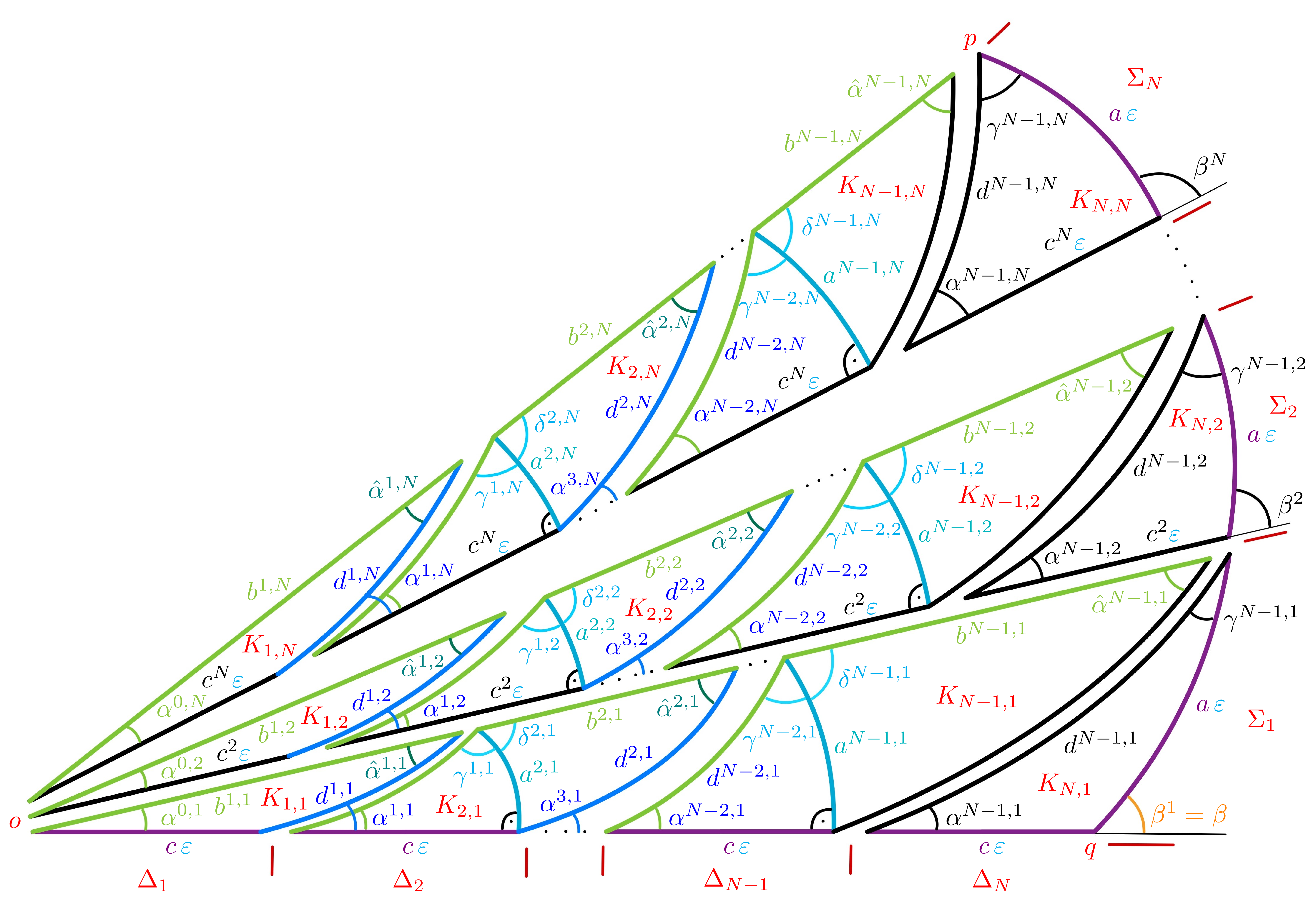}
    \caption{This is a depiction of the scheme we use to approximate the curved triangle. The curvature takes the constant values $K_{ij}$ on each segment $\Delta_{ij}$ in each slice $\Sigma_j$, which are depicted detached from the other segments.}
    \label{fig: scheme-main}
\end{figure}

If this was a triangle on a constant curvature surface we could immediately solve it by applying cosine- and sine-laws. Since this is not a constant curvature surface however, we are going to cut it into so many pieces, such that each piece has approximately constant curvature and then we apply cosine- and sine-laws to these pieces.\\
We first cut the triangle into $N$ slices $\Sigma_j$ by dividing the geodesic $\gamma_{q,\beta}$ into equally long sections. The points
\begin{equation}
    p_j \coloneq \gamma_{q,\beta}(ja\varepsilon), \quad j\in\{1,\ldots,N\}, \quad \varepsilon=\frac{1}{N}
\end{equation}
are connected with geodesics $\gamma_{o,\bar{\alpha}^{0,j}}$ from $o$ to complete the triangles.
\begin{equation}
    \bar{\alpha}^j \coloneq \sum_{l=1}^j\alpha^{0,l}, \qquad \gamma_{o,\bar{\alpha}^j}(c^j) = p_j, \qquad
    j\in\{1,\ldots,N\}, \qquad \alpha = \bar{\alpha}^N
\end{equation}
where $\alpha^{0,j}$ are the opening angles of the slices $\Sigma_j$ and $c^j$ the lengths of the slices baselines.\\
We write $\alpha^{ij}$, when both indices are letters and separate the two by a coma, when one of them is a number, to avoid confusions like $\alpha^{2l} = \alpha^{2,l}$ or $\alpha^{2l,1}$, since we drop the second index in the next section, where it is always $j=1$. We do not use comas for partial derivatives in this work.\\

Then we cut each slice into further pieces, by dividing the baselines $\gamma_{o,\bar{\alpha}^j}$ of each slice into equally long parts. We connect the points
\begin{equation}
    q_{ij} \coloneq \gamma_{o,\bar{\alpha}^{j-1}}(ic^j\varepsilon), \quad i,j\in\{1,\ldots,N\}, \qquad \bar{\alpha}^0 = 0
\end{equation}
on the baselines with the top-line of each slice via geodesics, which leave the points $q_i$ in directions $\beta^{i,j} = \frac{\pi}{2}$. This choice of angles simplifies the problem greatly but restricts the validity of the solution to $a < c$. Since this construction might break otherwise. The angles $\beta^{i,j}$ may have to be tilted sharper to still create triangles in this way. The notation is chosen to be consistent with the directions, in which the sections $\gamma_{p_j,\beta^{j+1}}$ of the side geodesic $\gamma_{q,\beta}$ head out from the points $p_j$:
\begin{equation}
    \beta^j = \beta^{N,j}, \quad \text{and} \quad \beta^1 = \beta, \quad \text{s.t.} \quad \gamma_{p_{j-1},\beta^j}(a\varepsilon) = p_j, \quad j\in\{1,\ldots,N\} \quad \text{with} \quad p_0 = q \ \text{and} \ p_N = p
\end{equation}
With exception of the pieces at $o$ these pieces are curved rectangles instead of triangles. So, we divide them into two triangles via diagonal geodesics $\gamma_{q_{i,j},\alpha^{i,j}}$, whose lengths we denote with $d^{i,j}$. We label the angles between these diagonal geodesics and the baseline geodesics with $\alpha^{i,j}$. This choice is consistent with the opening angles of the slices for $i = 0$.\\
Instead of choosing constant curvature on the square peaces we had initially we pair the triangles between two diagonal geodesics to segments $\Delta_{ij}$ and approximate the curvature field with a step function, which has the constant values $K_{ij}$ on these segments. We choose to pick the values at the lower left edges of the segments $q_{i-1,j}$ and write the points in FNC.\\
\begin{equation}
    K_{ij} = K(q_{i-1,j}) = K((i-1)c^j\varepsilon,\bar{\alpha}^{j-1}), \qquad q_{0,j} = o
\end{equation}

We will then start with calculating the opening angle $\alpha^{0,1}$ and the arclength of the top-line $\bar{b}^1$ of the first slice.
\begin{equation}
    \bar{b}^j \coloneq \sum_{i=1}^{N-1}b^{ij}, \quad j\in\{1,\ldots,N\} \qquad \bar{b}^j = c^{j+1}, \quad j\in\{1,\ldots,N-1\} \quad \text{and} \quad b = \bar{b}^N
\end{equation}

The first step, which we will refer to as "Problem 1" is straight forward. The segment $\Delta_N$ is a triangle, of which we have the side-lengths of the baseline $c^1\varepsilon = c\,\varepsilon$ and the side-line $a\varepsilon$ and the angle between them $\pi-\beta^1 = \beta$. We can apply the cosine-law to get length of the diagonal $d^{N-1,1}$ and then the sine-law to obtain the two missing angles $\alpha^{N-1,1}$ and $\gamma^{N-1,1}$.\\

But now there are no other triangles in the first slice $\Sigma_1$ of which we have enough known quantities to calculate anything. What we can do instead, is assuming the side-length $d^{1,1}$ and angle $\alpha^{1,1}$ in $\Delta_1$ as given for the moment and determine $b^{1,1}$ and $\alpha^{0,1}$ as functions of $d^{1,1}$ and $\alpha^{1,1}$. We call this step "Problem 2" of "Type 1".\\

The idea is now to repeat this process and determine $b^{2,1}$, $d^{1,1}$ and $\alpha^{1,1}$ in terms of $d^{2,1}$ and $\alpha^{2,1}$ and so on and thus move the problem to the right until we reach $d^{N-1,1}$ and $\alpha^{N-1,1}$, which we already know from "Problem 1". It seems however, that we are cheating. By backwards inserting we would get all quantities of all triangles, although we did not have enough information to determine these in the first place. And indeed this procedure fails, since the segments $\Delta_i$ with $i\in\{2,\ldots,N-1\}$ are no triangles.\\
The segment $\Delta_2$ does consists of $2$ triangles with the same curvature, but we have the same amount of information for $2$ triangles as we had in "Problem 2, Type 1". That is $c$, $d^{2,1}$ and $\alpha^2,1$. The missing piece of information comes from the fact that the top-lines $b^{i,1}$ for $i\in\{1,\ldots,N-1\}$ form a geodesic i.e. there are no kinks, and thus the angles $\hat{\alpha}^{i,1}$, $\gamma^{i,1}$ and $\delta^{2,1}$ add up to $\pi$. This relates the segments to each other and allows us to push the problem all the way to the right.\\
We would now first do the same as in "Problem 2, Type 1" for the lower left triangle of $\Delta_2$ and calculate $d^{1,1}$ and $\alpha^{1,1}$ in terms of $a^{2,1}$. The angle $\hat{\alpha}^{1,1}$ is a function of $d^{1,1}$ and $\alpha^{1,1}$ which in turn are now functions of $a^{2,1}$. We can now obtain an equation by calculating $\delta^{2,1}$ from both the upper right and lower left triangle, using that $\delta^{2,1} = \pi-\hat{\alpha}^{1,1}-\gamma^{1,1}$. We can solve this equation ($\delta$-equation) for $a^{2,1}$ to express all quantities in $\Delta_2$ in terms of $d^{2,1}$ and $\alpha^{2,1}$. We call this procedure "Problem 2" of "Type 2" and this is the one which we repeat until we get to $\Delta_{N-1}$.\\

What changes in $\Delta_{N-1}$ is that we already have the expressions for $d^{N-1,1}$ and $\alpha^{N-1,1}$. This leads to some boundary effects in formulas for this problem, which we name "Problem 2" of "Type 3".\\

Once we calculated $\bar{b}^1$, $\alpha^{0,1}$, $\hat{\alpha}^{N-1,1}$ and $\gamma^{N-1,1}$ from the first slice $\Sigma_1$ we can reinterpret $\bar{b}^1 = c^2$ and $\hat{\alpha}^{N-1,1} + \gamma^{N-1,1} = \beta^2$ and apply the solution of $\Sigma_1$ to $\Sigma_2$ with the new set of input variables $a$, $c^2$ and $\beta^2$. We can repeat this procedure successively until we have $\bar{b}^N = b$. At this point we can calculate $\alpha$ by summing over all $\alpha^{0,i}$ i.e. $\bar{\alpha}^{0,N} = \alpha$.

\subsection{Curvature values}\label{subsec: Curvature values}
To use the results of the two dimensional triangle on the original geodesic from above for a generic case, we can just embed it in an $n$-dimensional manifold, where $q$ is replaced with a generic point $p'$, by adjusting the curvature values.\\
The first thing we need to do, is to move and tilt the triangle and the points $q_{ij}$ from a triangle attached to $q$ to $p$. In Sec.~\ref{sec: metric} we derived, how to write a geodesic heading out from $p$ in direction $\hat{\Omega}_\beta$ in faithful normal coordinates. The points $q_{ij}$ are such a case:
\begin{equation}
    \mathcal{N}_o(\gamma_{p,\hat{\Omega}_\beta}(\lambda))
    = \binom{l(K;\lambda)}{\vec{\vartheta}(\hat{\Omega}_p,\hat{\Omega}_\alpha)}, \quad \hat{\Omega}_\alpha(\varphi(K;\lambda),\beta_2,\ldots,\beta_{n-1})
\end{equation}
where $\hat{\Omega}_{p'}$ takes care of moving the edge point $q$ to $p'$ and $\hat{\Omega}_\alpha$ are the opening angle of the triangle and the tilts, as they would be, if it were at $q$.\\

To calculate $l(K)$ and $\varphi(K)$ we need to pick the curvature field on the plane in which the geodesic lies. We discussed in Sec.~\ref{subsec: Curvature parametrization} how we do this in three dimensions:
\begin{equation}
    K(l',\varphi') = K(p'(l',\varphi'))(\vec{\vartheta}_K), \quad
    \vec{\vartheta}_K(\hat{\Omega}_p;\beta_2,\ldots,\beta_{n-2}), \qquad \varphi_K(\varphi_p,\vartheta_p,\alpha_2), \quad \vartheta_K(\varphi_p,\vartheta_p,\alpha_2),\ \text{for} \ n=3
\end{equation}

From the parametrization of the $3$ dimensional case in Sec.~\ref{subsec: Curvature parametrization} we see, that we need to rotate the triangle by $\tau$ to move $q$ to $p'$. Thus the curvature values in faithful normal coordinates on that geodesic plane, with label $\vec{\vartheta}_K$, are given by:
\begin{equation}
    K_{ij} = K(R(\tau)\cdot q_{ij})(\vec{\vartheta}_K) = K_{\vec{\vartheta}_K}(ic^j\varepsilon,\tau + \bar{\alpha}^{j-1})
\end{equation}
We use Eqs.~\eqref{eq: alpha>=theta},~\eqref{eq: alpha<theta & <0},~\eqref{eq: alpha<theta & >0} and~\eqref{eq: varphi_K} to calculate the angle and side line of the triangle on the direction sphere at the origin $\tau$:
\begin{align}
    \tau = \text{sign}(\vartheta_p\alpha_2)\sin^{-1}\left( \frac{\sin\vartheta_p\sin\varphi_K(\varphi_p,\vartheta_p,\alpha_2)}{\sin(\alpha_2-\vartheta_p)} \right)
\end{align}
The sign of $\alpha_2$ follows from the three sine-laws and we need to change the sign when $\vartheta_p$ becomes negative, since we are use the side length $\tau$ as a coordinate, which is negative in this case.\\

Finally, the input arguments from which we calculate the coordinate functions $l(\lambda)$ and $\varphi(\lambda)$ are related to this geodesic by:
\begin{equation}
    c = l, \quad \beta = \beta_1, \quad a = \lambda
\end{equation}

From this we know, how to apply the triangulation in the special case presented above to a more generic case. We will now for simplicity restrict ourselves to the special case and drop the curvature parameters, when we carry out the triangulation in the following sections.
\begin{figure}[h!]
    \centering
    \includegraphics[width=0.5\linewidth]{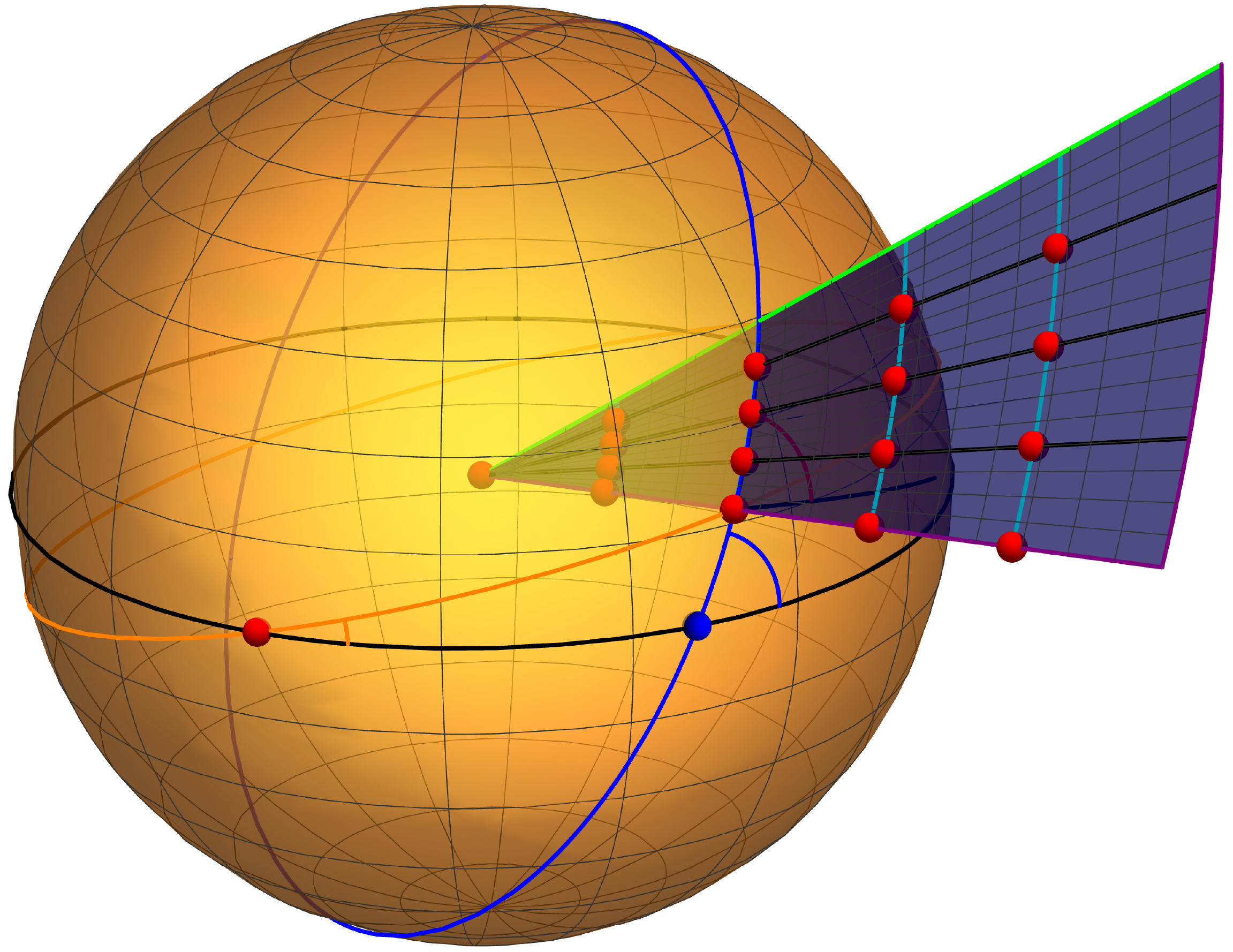}
    \caption{We display the triangulation of a geodesic in a similar sketch as in Fig.~\ref{fig: Ring-Function evaluation} (the labels are still the same and we thus omit them here). We mark the points on the blue triangle from which the curvature values $K_{ij}$ are taken with red spheres. We leave the diagonal lines away and chose the trivial special case of $3$-dimensional flat space to simplify the plotting.
    The difference to a generic case is, that the rib-lines in turquoise would not be connected from slice to slice.}
    \label{fig:my_label}
\end{figure}

\section{A Slice of the Solution}\label{sec: A Slice of the Solution}
Most of the calculations discussed in the previous section (Sec.~\ref{sec: Triangulation}) are repeated applications of the cosine- and sine-law for a surface of constant curvature $K$.
\begin{equation}
    \cos_Kc = \cos_Ka\cos_Kb + K\sin_Ka\sin_Kb\cos\gamma, \qquad \frac{\sin_Ka}{\sin\alpha} = \frac{\sin_Kb}{\sin\beta} = \frac{\sin_Kc}{\sin\gamma}, \quad
    \alpha,\beta,\gamma\notin\pi\mathbb{Z}
\end{equation}
The vast majority of the functions can be expressed in terms of the $C$ "inverse cosine-law" and $SC$ "inverse sine- on the inverse cosine-law" functions, which we define as:
\begin{align}
    C(K;\gamma,a,b) \coloneq& \cos_K^{-1}\left( \cos_Ka\cos_Kb + K\sin_Ka\sin_Kb\,\cos\gamma \vphantom{\sqrt{K}}\right) \notag \\
    =& \frac{1}{\sqrt{K}}\cos^{-1}\left( \cos\sqrt{K}a\cos\sqrt{K}b + \sin\sqrt{K}a\sin\sqrt{K}b\,\cos\gamma \right)\\
    SC(K;\gamma,a,b) \coloneq& \sin^{-1}\left( \frac{\sin_Ka\,\sin\gamma}{\sin_KC(K;\gamma,a,b)} \right)
    = \sin^{-1}\left( \frac{\sin\sqrt{K}a\,\sin\gamma}{\sqrt{1-\left( \cos\sqrt{K}a\cos\sqrt{K}b + \sin\sqrt{K}a\sin\sqrt{K}b\,\cos\gamma  \right)^2}} \right). \notag
\end{align}
We restrict ourselves to the case of positive curvature $K>0$ and discuss the negative curvature case as well as the cosine laws on the sphere and their power series in more detail in the Supplemental Material.\\

We use the two functions to carry out the triangulation on a slice $\Sigma_j$. The $\delta$-equation is a transcendental equation, and we could not solve it exactly. So, we instead expand the $C$ and $SC$ functions in $\varepsilon=\frac{1}{N}$ and then solve the problem up to second order. The systematic expansion is a non-trivial and cumbersome calculation, thus we provide the solution in the Supplemental Material.\\
Since the second order problem is linearized after the expansion we can use a linear ansatz for the top angles $\hat{\alpha}_2^{ij}$ in the upper triangles $\hat{\Delta}_i$ of the segments, which allows us to extract a recursion from the repeated application of "Problem 2, Type 2". The recursion can be solved by a continued fraction, which cancels to a simple one.\\
After solving the recursion we get expressions for the rib lines $a^{ij}$, which depend only on the diagonal line $d^{ij}$ and opening angle $\alpha^{ij}$ of their segment $\Delta_i$, and they in turn depend only on the rib line $a^{i+1,j}$ of the next segment $\Delta_{i+1}$. Thus, we can start from the last segment $\Delta_N$, which we can solve directly and subsequently backwards insert the calculated quantities. Since we want to do this for an arbitrary number of segments $N$ we need a generic expression for the backwards insertion in the $i$-th segment. By solving the resulting recursion for the rib lines we obtain explicit expressions for the top lines, top angles and opening angles of any slice for all $N>3$.\\
This calculation is carried out in detail for the first slice $\Sigma_1$ and then generalized to an arbitrary slice $\Sigma_j$ in the Supplemental Material. In the following two subsections we summarize and discuss the results of the triangulation up to second order of a generic slice $\Sigma_j$.\\

\subsection{$0$-th order solution}
The slices are congruent to each other and are all calculated in the same way with different base lines and direction angles as arguments. The side lines are $a\varepsilon$ for every slice $\Sigma_j$. The base line of the slice $\Sigma_j$ is the top line of the previous slice $\Sigma_{j-1}$ and the direction angle of $\Sigma_j$ is the opposite angle to the top angle of $\Sigma_{j-1}$. Thus, the two arguments have higher order correction terms for $j\geqslant2$, whereas the ones of the first slice $c^1$ and $\beta^1$ do not, since these are arguments of the entire triangulation. We expand these relations up to second order in $\varepsilon$:
\begin{equation}\label{eq: recursive relation}
    c_0^j + c_2^j\varepsilon^2 = \bar{b}_0^{j-1} + \bar{b}_2^{j-1}\varepsilon^2, \qquad \beta_0^j + \beta_2^j\varepsilon^2 = \bar{\gamma}_0^{j-1} + \bar{\gamma}_2^{j-1}\varepsilon^2, \qquad c^1 = c, \qquad \beta^1 = \beta.
\end{equation}
In our calculation in the Supplemental Material~\sref{subsec: Solving Problems of different Types} we show, that all first order terms vanish.\\

The zeroth order triangulation can be done using the cosine- and sine-laws for the Euclidean plane ($K=0$):
\begin{align}\label{eq: c_0^j -> K_ij}
    &c_0^j = \bar{b}_0^{j-1} = \sqrt{ (j-1)^2a^2 + N^2c^2 + 2N(j-1)\,a\,c\cos\beta}\,\varepsilon, \qquad
    \bar{\alpha}_0^j = \sin^{-1}\frac{j\,a\sin\beta}{\sqrt{ j^2a^2 + N^2c^2 + 2Nj\,a\,c\cos\beta}}, \\
    &\beta_0^j = \bar{\gamma}_0^{j-1} = \sin^{-1}\left( \frac{Nc\sin\beta}{\sqrt{ (j-1)^2a^2 + N^2c^2 + 2N(j-1)\,a\,c\cos\beta}} \right).
\end{align}

Since all the slices are congruent it seems to be reasonable at first to triangulate the first one and then Taylor expand the two functions $\bar{b}^1$ and $\bar{\gamma}^1$ in $\varepsilon$ up to second order. We drop the $j$-index for functions in the first slice $j=1$:
\begin{align}
    c^j \approx&\, \bar{b}_0(a,c_0^{j-1} + c_2^{j-1}\varepsilon^2,\beta _0^{j-1}
    + \beta _2^{j-1} \varepsilon^2)
    + \bar{b}_2(K_{i j};a,c_0^{j-1}+c_2^{j-1}\varepsilon^2,\beta_0^{j-1}
    + \beta_2^{j-1}\varepsilon^2)\varepsilon^2 + \mathcal{O}(\varepsilon^3) \label{eq: b expansion} \\
    \approx&\, \bar{b}_0(a,c_0^{j-1},\beta_0^{j-1})
    + \left\{ \left[ \partial_c\,\bar{b}_0(a,c_0^{j-1},\beta_0^{j-1}) \right]c_2^{j-1}
    + \left[ \partial_\beta\,\bar{b}_0(a,c_0^{j-1},\beta_0^{j-1}) \right]\beta_2^{j-1} + \bar{b}_2(K_{ij};a,c_0^{j-1},\beta_0^{j-1}) \right\}\varepsilon^2 + \mathcal{O}(\varepsilon^3) \notag \\
    \vphantom{a}\notag\\
    \beta^j \approx&\, \bar{\gamma}_0(a,c_0^{j-1}+c_2^{j-1}\varepsilon^2,\beta_0^{j-1}+\beta_2^{j-1}\varepsilon^2)
    + \bar{\gamma}_2(K_{ij};a,c_0^{j-1}+c_2^{j-1}\varepsilon^2,\beta_0^{j-1}+\beta_2^{j-1}\varepsilon^2)
    + \mathcal{O}(\varepsilon^3) \label{eq: beta expansion} \\
    \approx&\, \bar{\gamma}_0(a,c_0^{j-1},\beta_0^{j-1})
    + \left\{ \left[ \partial_c\,\bar{\gamma}_0(a,c_0^{j-1},\beta_0^{j-1}) \right]c_2^{j-1}
    + \left[ \partial_\beta\,\bar{\gamma}_0(a,c_0^{j-1},\beta_0^{j-1}) \right]\beta_2^{j-1}
    + \bar{\gamma}_2(K_{ij};a,c_0^{j-1},\beta_0^{j-1}) \right\}\varepsilon^2 + \mathcal{O}(\varepsilon^3) \notag
\end{align}
Due to the non-linearity of this triangulation the Taylor expansion on the entire slice does not provide the correct result and we have to expand each segment on its own and generalize each problem of each type to adapt them to second order correction terms in the arguments. We collect all the correction terms to the functions of the first slice in the argument correction coefficients $f^{j,c}$ and $f^{j,\beta}$, so that we can write the functions in the $j$-th slice in a similar form:
\begin{align}\label{eq: separated form}
    &f^1(c,\beta) \quad \mapsto \quad f^j(c^j,\beta^j) \approx f^1(c_0^j + c_2^j\varepsilon^2, \beta_0^j + \beta_2^j\varepsilon^2) \approx \overset{\circ}{f}{}^j + \overset{\circ}{f}{}^{j,c}c_2^j\varepsilon^2 + \overset{\circ}{f}{}^{j,\beta}\beta_2^j\varepsilon^2 \\
    &\text{interior values:} \quad \overset{\circ}{f}{}^j \coloneq f^1(c_0^j,\beta_0^j), \qquad \text{argument correction coefficients:} \quad \overset{\circ}{f}{}^{j,c} \neq \partial_c f^1(c,\beta)|_{c\mapsto c_0^j, \beta\mapsto \beta_0^j}
\end{align}

\subsection{$2$-nd order solution to $\Sigma_j$}\label{subsubsec: 2-nd order of Sigma_j}
Since there are often reoccurring terms we can define the following substitutes, to simplify our results.
\begin{align}\label{eq: substitutes}
    &X_N \coloneq Nc + a\cos\beta, \quad e^i \coloneq i\,a\sin\beta, \quad Z_N = a + Nc\cos\beta\\
    &Y_N \coloneq \sqrt{a^2 + N^2c^2 + 2Nac\cos\beta}, \quad k_N^i = \sqrt{X_N^2 + (e^i)^2}
    \quad \Rightarrow \quad Y_N = k_N^1 \label{eq: substitute relation}
\end{align}

When treat the $j$-th slice we replace all $c\to c_0^j$ and $\beta\to\beta_0^j$. We can simplify the generalized substitutes, when we express them in terms of the original arguments $c$ and $\beta$:
\begin{equation}
    Y_N^j \coloneq \left.Y_N\right|_{c\mapsto c_0^j, \beta\mapsto\beta_0^j} = \sqrt{j^2a^2 + N^2c^2 + 2jNac\cos\beta}, \qquad Z_N^j \coloneq \left.Z_N\right|_{c\mapsto c_0^j, \beta\mapsto\beta_0^j}
    = ja + Nc\cos\beta.
\end{equation}
We can describe the generalized substitutes as well as the zeroth order arguments in terms of $Y_N^j$ and $Z_N^j$.
\begin{equation}
    X_N^j = \frac{(Y_N^j)^2 - aZ_N^j}{Y_N^{j-1}}, \quad e^{ij} = \frac{iNa\,c\sin\beta}{Y_N^{j-1}}, \quad
    k_N^{ij} = \sqrt{(X_N^j)^2 + (e^{ij})^2}, \qquad
    c_0^j = Y_N^j\varepsilon, \quad \beta_0^j = \sin^{-1}\frac{Nc\sin\beta}{Y_N^{j-1}}.
\end{equation}

We list the results of the tirangulation of the $j$-th slice i.e. top-line, top-angle, opening angle and all quantities needed to calculate them. They are presented in the form as described in \eqref{eq: separated form}, separating the contributions of the corrections to the arguments from the interior values and collecting them in the argument correction coefficients:\\

We start with the second order terms to the top line of the $j$-th slice. The second order term of the top line and its argument coefficients are calculated by adding the lengths of the top lines of all segments $\Delta_1$ - $\Delta_{N-1}$\footnote{We abbreviate the segments $\Delta_{ij}$ with $\Delta_i$ since we always use different segments in the same slice $\Sigma_j$. They are also sketched in that way in Fig.~\ref{fig: scheme-main}, where for example $\Sigma_1$ stands for the first segments in all the slices, representing $\Delta_{1j}$, $\forall j\in\{1,\ldots,N\}$.} and then extracting the correction terms due to the second order terms in the arguments. The last segment $\Delta_N$ has no top line.
\begin{align}
    \bar{b}_2^j =&\ \overset{\circ}{b}{}_2^j
    + \overset{\circ}{b}{}_2^{j,c}c_2^j + \overset{\circ}{b}{}_2^{j,\beta}\beta_2^j, \quad
    \overset{\circ}{b}{}_2^j = \overset{\circ}{b}{}_3^{1,j}\varepsilon + \sum_{i=2}^{N-2}\overset{\circ}{b}{}_3^{ij}\varepsilon
    + \overset{\circ}{b}{}_3^{N-1,j}\varepsilon, \qquad j\in\{1,\ldots,N\} \\
    \overset{\circ}{b}{}_2^{j,c} \coloneq& \frac{1}{Y_N^j}\left[ X_N^j + \frac{I_N(e^{1,j})^2}{NX_N^j} \right], \quad
    \overset{\circ}{b}{}_2^{j,\beta} \coloneq -\frac{c_0^je^{1,j}}{Y_N^j}, \quad
    I_N \coloneq \sum_{i=2}^{N-2}\frac{1}{i+1} = \psi(N) + \gamma - \frac{3}{2}, \quad
    \psi(x) = \frac{d}{dx}\ln\Gamma(x), \label{eq: bo}
\end{align}
where $\psi$ denotes the digamma function and $\gamma$ here is the Euler constant.\\

 The first and the second last segments, $\Delta_1$ and $\Delta_{N-1}$ are special cases, since they are boundary problems. All top line pieces depend on the lengths of the rib-lines $a_2^{ij}$ in their segments $\Delta_i$. The terms containing curvature values in the regular cases are collected in $\mathcal{K}_b^{ij}$. It consists of a term proportional to the curvature parameter $C^{i-1,j}$, representing the constant term in the ansatz for $\hat{\alpha}_2^{ij}$, and two terms proportional to the curvature of the current $\Delta_i$ and the next slice $\Delta_{i+1}$. We pack their complicated proportionality factors into $\mathrm{X}_b^{ij}$ and $\mathrm{Z}_b^{ij}$ respectively.
\begin{align}
    \overset{\circ}{b}{}_3^{1,j} =& -\frac{e^{1,j}}{Y_N^j}\left( \frac{(c_0^j)^3e^{1,j}}{X_N^j}\left[ \frac{K_{1,j}}{3} + K_{2,j} \right] - \overset{\circ}{a}{}_3^{2,j} \right), \qquad
    \overset{\circ}{b}{}_3^{ij} = -\frac{e^{1,j}}{(i+1)Y_N^j}\left( \frac{(c_0^j)^3e^{1,j}}{6X_N^j}i\overset{\circ}{\mathcal{K}}{}_b^{ij}
    - \overset{\circ}{a}{}_3^{i+1,j} \right), \\
    \overset{\circ}{b}{}_3^{N-1,j} =& -\frac{(c_0^j)^2 \left(e^{1,j}\right)^2}{6 Y_N^j}\left( (N-1)^2K_{N-1,j}\frac{X_1^j}{X_N^j}
    + (2N-1)K_{N,j} + \frac{6}{(c_0^j)^2e^{1,j}}\overset{\circ}{a}{}_3^{N-1,j} \right), \notag \\
    \overset{\circ}{\mathcal{K}}{}_b^{ij} =&\ \frac{\overset{\circ}{C}{}^{i-1,j}}{(X_N^j)^2}
    + iK_{ij}\mathrm{X}_b^{ij} + (i+1)K_{i+1,j}\mathrm{Z}_b^{ij}, \qquad
    \mathrm{X}_b^{ij} \coloneq i + 5 + \frac{i-1}{i} + (i+3)\frac{(e^{1,j})^2}{(X_N^j)^2}, \notag \\
    \mathrm{Z}_b^{ij} \coloneq& \left( i + 2 + \frac{e^{1,j}e^{i+1,j}}{(X_N^j)^2} \right)
   \left( 1 + \frac{(e^{i+1,j})^2}{(k_N^{i+1,j})^2} \right) + \frac{(i+1)^2(X_N^j)^2 + (2i+1)e^{1,j}e^{i+1,j}}{i(k_N^{i+1,j})^2}, \qquad i\in\{2,\ldots,N-2\}. \notag
\end{align}

The opening angle of the slice is given by the opening angle of the first segment in the slice. It directly depends on the curvature value in it's segment $\Delta_1$ and the next one $\Delta_2$, and indirectly on all other curvature values in the slice via the rib-line of the second segment. We drop the $0$-index in the argument coefficients here, since we do not need them for other index values.
\begin{align}\label{eq: alphao}
    \alpha_2^{0,j} =& \overset{\circ}{\alpha}{}_2^{0,j}
    + \overset{\circ}{\alpha}{}_2^{j,c} c_2^j + \overset{\circ}{\alpha}{}_2^{j,\beta} \beta_2^j, \\
    \overset{\circ}{\alpha}{}_2^{0,j} =& \frac{(X_N^j)^2}{2(Y_N^j)^2}\left( \frac{(c_0^j)^2e^{1,j}}{3X_N^j}\left[ K_{1,j}\left( 3 + 4\frac{(e^{1,j})^2}{(X_N^j)^2} \right)
    + K_{2,j}\frac{(k_N^{2,j})^2}{(X_N^j)^2} \right] + \frac{\overset{\circ}{a}{}_3^{2,j}}{c_0^j} \right), \quad
    \overset{\circ}{\alpha}{}_2^{j,c} = \frac{a\cos\beta_0^j - X_N^j}{(Y_N^j)^2c_0^j}e^{1,j}, \quad
    \overset{\circ}{\alpha}{}_2^{j,\beta} = \frac{aZ_N^j}{(Y_N^j)^2} \notag
\end{align}

The top angle of a slice is the opposite angle to the direction angle of the next slice: $\bar{\gamma}^j = \beta^{j+1}$. It's second order term is the sum of the top angle in the last segment $\Delta_N$ and the angle opposite to the rib-line in the upper triangle of the second last segment $\hat{\Delta}_{N-1}$.\\
We note that none of the argument coefficients contain curvature values. The dependence them is given by the interior values alone and thus each slice depends on it's curvature values in the same way.
\begin{align}
  \bar{\gamma}_2^j =& \overset{\circ}{\gamma}{}_2^j + \overset{\circ}{\gamma}{}_2^{j,c}c_2^j
    + \overset{\circ}{\gamma}{}_2^{j,\beta}\beta_2^j, \qquad
    \overset{\circ}{\gamma}{}_2^j = \overset{\circ}{\hat{\alpha}}{}_2^{N-1,j} + \overset{\circ}{\gamma}{}_2^{N-1,j}, \quad
    \overset{\circ}{\gamma}{}_2^{j,c} = \frac{e^{N j}}{(Y_N^j)^2}, \quad
    \overset{\circ}{\gamma}{}_2^{j,\beta} = Nc_0^j\frac{X_N^j}{(Y_N^j)^2}, \label{eq: gammao}
\end{align}
Also $\hat{\alpha}_2^{N-1,j}$ depends on the curvature values in it's segment and the next one and on the rib-line in it's segment. Since there is no next segment to $\Delta_N$ and it contains no rib-line, the top angle of the last segment only depends on the last curvature value.
\begin{align}
    \overset{\circ}{\hat{\alpha}}{}_2^{N-1,j} =& \frac{1}{X_1^j(Y_N^j)^2}\left( \frac{c_0^j\,e^{N-1,j}}{6}\left[
    K_{N-1,j}\frac{X_1^j}{X_N^j}(X_1^j(Y_N^j)^2 - c_0^j\,e^{N-1,j}e^{1,j}) + K_{Nj}c_0^j(e^{1,j})^2\left( \frac{N(c_0^j)^2 - a^2}{(Y_1^j)^2} - 1 \right) \right] \right. \notag \\
    &\left.+ (X_N^j)^2\overset{\circ}{a}{}_3^{N-1,j} \vphantom{\frac{X_1^j}{X_N^j}}\right), \qquad
    \overset{\circ}{\gamma}{}_2^{N-1,j} = \frac{1}{(Y_1^j)^2}\left( \frac{K_{Nj}}{6}c_0^je^{1,j}\left[ (Y_1^j)^2
    + c_0^jX_1^j \right] \right).
\end{align}

The contribution of the argument coefficients of the recursion parameters were already included into in the ones of the rib-lines and those into the argument coefficients presented above. Since they are not required on their own, we leave them away here. We list the interior values though, since they are required to carry out the triangulation for finite $N$.\\
We saw, that all quantities we want to calculate ultimately depend on the length's of the rib-lines, which is why we call them that way. We again collect all terms containing curvature values in $\mathcal{K}_a^{ij}$, which has the same structure as $\mathcal{K}_b^{ij}$. In difference to the top lines $b_2^{ij}$ the rib-lines contain all $\mathcal{K}_a^{nj}$ of the following segments and with that all their curvature values: $K_{nj}$, $n\in\{i,\ldots,N-1\}$. We also observe, that the special case of the last rib-line $a^{N-1,j}$ appears in all other rib-lines.
\begin{align}
    \overset{\circ}{a}{}_3^{ij} =& \frac{(c_0^j)^3e^{1,j}}{6X_N^j}\sum_{n=i}^{N-2}\frac{i}{n+1}\overset{\circ}{\mathcal{K}}{}_a^{nj}
    + \frac{i}{N-1}\overset{\circ}{a}{}_3^{N-1,j}, \qquad
    \overset{\circ}{\mathcal{K}}{}_a^{nj} = \frac{\overset{\circ}{C}\,^{n-1,j}}{(X_N^j)^2}
    + n\,K_{nj}\mathrm{X}^{nj} + (n+1)K_{n+1,j}\mathrm{Z}^{nj}, \\
    \mathrm{X}^{ij} \coloneq& \frac{1}{(k_N^{ij})^2}\left[ \left(5-\frac{1}{i}\right)(X_N^j)^2
    + i^2(i+3)\frac{(e^{1,j})^4}{(X_N^j)^2} + (5i^2 + 3)(e^{1,j})^2 \right], \qquad
    \mathrm{Z}^{ij} \coloneq 1 + 2(i+1)\frac{(e^{1,j})^2}{(X_N^j)^2}, \\
    i\in&\{2,\ldots,N-2\}, \qquad j\in\{1,\ldots,N\} \notag
\end{align}
The last rib-lines are proportional to a term with the structure of a $\mathcal{K}_a^{ij}$ consistent with the expectation that they should only have one, since there are no next ones and they should not depend on themselves. Since they are special cases as quantities on the boundary, the factors $\chi^j$ and $\zeta^j$, which are the analogues to $\mathrm{X}^{ij}$ and $\mathrm{Z}^{ij}$, are not consistent with the regular case.
\begin{align}
    \overset{\circ}{a}{}_3^{N-1,j} =& \frac{(N-1)(c_0^j)^2X_1^j}{6X_N^j\left[ (N-1)c_0^j + X_1^j \right]}\left( \frac{c_0^j\,e^{1,j}}{(X_N^j)^2}\overset{\circ}{C}{}^{N-2,j}
    + K_{N-1,j}\chi^j + K_{N,j}\zeta^j \right), \\
    \chi^j =& (N-1)\frac{e^{1,j}}{X_N^j}\left( a^2 + N(c_0^j)^2 + (N+1)a\,c_0^j\cos\beta_0^j + \frac{X_1^j}{X_N^j}(Y_N^j)^2 \right)
    \notag \\
    &+ \left( (Y_N^j)^2\frac{(k_N^{N-1,j})^2 + (X_N^j)^2}{(k_N^{N-1,j})^2(X_N^j)^2} + \frac{N-2}{(k_N^{N-1,j})^2}\left[
    \frac{(X_N^j)^2}{N-1} + (2N-1)(e^{1,j})^2 \right] \right)c_0^j\,e^{N-1,j} \\
    \zeta^j =& \frac{e^{1,j}}{X_1^j}\left( 2a^2 + N(c_0^j)^2 + (2N+1)a\,c_0^j\cos\beta_0^j \right), \qquad j\in\{1,\ldots,N\}
\end{align}

All quantities we discussed before depend in some way on the recursion parameter $C^{ij}$. The contributions of the other two recursion parameters $A^{ij}$ and $D^{ij}$ are contained in $\mathcal{K}^{ij}$ and do not appear explicitly anymore. In the $C^{ij}$ the curvature values are added in the opposite direction, containing all $K_{nj}$ from the first one up to the $i$-th one. This way all quantities in the end depend on all curvature values in their slice $\Sigma_j$.
\begin{align}
    \overset{\circ}{C}{}^{1,j} =&\, K_{1,j}\left[ (X_N^j)^2 + 2(Y_N^j)^2 \right], \qquad
    \overset{\circ}{C}{}^{ij} = \frac{2}{i+1}\overset{\circ}{C}{}^{1,j} + \sum_{k=2}^i \frac{k}{i+1}\mathcal{K}^{kj},
    &\qquad i\in&\{2,\ldots,N-2\}, \notag \\
    \mathcal{K}^{ij} =&\, K_{ij}\left[ 2i(X_N^j)^2 + i\left( 3 + \frac{(X_N^j)^2}{(k_N^{ij})^2} \right)(Y_N^j)^2
    + \left( 1 + i(i^2-1)\frac{(X_N^j)^2}{(k_N^{ij})^2} \right)(e^{1,j})^2 \right], &\qquad j\in&\{1,\ldots,N\}.
\end{align}

\section{Recursion across slices}\label{sec: Recursion across slices}
We see from~\eqref{eq: recursive relation} that if we write out $\bar{b}^j$ and $\bar{\gamma}^j$ as functions of $c^{j-1}$ and $\beta^{j-1}$ we get a recursion across the slices.
\begin{align}
    c_2^j = \bar{b}_2^{j-1} = \overset{\circ}{b}{}_2^{j-1}
    + \overset{\circ}{b}{}_2^{j-1,c}c_2^{j-1} + \overset{\circ}{b}{}_2^{j-1,\beta}\beta_2^{j-1}, \quad
    \beta_2^{j} = \bar{\gamma}_2^{j-1} = \overset{\circ}{\gamma}{}_2^{j-1}
    + \overset{\circ}{\gamma}{}_2^{j-1,c}c_2^{j-1} + \overset{\circ}{\gamma}{}_2^{j-1,\beta}\beta_2^{j-1}
\end{align}
We iteratively apply the recursion rule until we arive at $j=1$ and collect the terms proportional to $\overset{\circ}{b}{}_2^n$ and $\overset{\circ}{\gamma}{}_2^n$ in the recursion substitutes $O_n^j$, $Q_n^j$, $V_n^j$ and $W_n^j$. That way we can write $\bar{b}_2^j$ and $\bar{\gamma}_2^j$ explicitly, when we know the recursion substitutes.
\begin{equation}\label{eq: b_2^j}
    \bar{b}_2^j = c_2^{j+1} = \overset{\circ}{b}\,_2^{j}
    + \sum_{n=1}^{j-1}\left( O_n^j\overset{\circ}{b}\,_2^{j-n} + Q_n^j\overset{\circ}{\gamma}\,_2^{j-n} \right),
    \qquad \bar{\gamma}_2^j = \beta_2^{j+1} = \overset{\circ}{\gamma}{}_2^j
    + \sum_{n=1}^{j-1}\left( V_n^j\overset{\circ}{b}{}_2^{j-n} + W_n^j\overset{\circ}{\gamma}{}_2^{j-n} \right).
\end{equation}
Thus the recursion is now packed in these substitutes and we solve it in the Supplemental Material, where we find:
\begin{equation}\label{eq: Onj&Qnj_explicit}
    O_n^j = \sum_{\underset{\textit{\small even}}{\imath=0,}}^n \prod_{l=1}^\frac{\imath}{2} S_l
    \cdot \frac{Y_N^j}{Y_N^{j-m_\imath}} \hat{P}_0^{m_\imath-1}, \quad m_0 \coloneq n, \qquad
    Q_n^j = -c\,e^1 \sum_{m_1=0}^{n-1} \frac{P_{m_1+1}^{n-1}}{Y_N^{j-m_1}} \sum_{\underset{\textit{\small odd}}{\imath=0,}}^n
    \prod_{l=1}^\frac{\imath-1}{2} S_{l+\frac{1}{2}} \cdot \frac{Y_N^j}{Y_N^{j-m_\imath}} \hat{P}_0^{m_\imath-1},
\end{equation}
and the recursion substitutes of $\hat{\gamma}_2^j$ show a similar structure, if we swap their role:
\begin{equation}\label{eq: Vnj&Wnj_explicit}
    V_n^j = \frac{Nf_N}{Y_N^{j-n}} \sum_{m_1=0}^{n-1} \frac{\hat{P}_{m_1+1}^{n-1}}{(Y_N^{j-m_1})^2}
    \sum_{\underset{\textit{\small odd}}{\imath=0}}^n \prod_{l=1}^\frac{\imath-1}{2} T_{l+\frac{1}{2}} \cdot P_0^{m_\imath-1}, \qquad
    W_n^j = \sum_{\underset{\textit{\small even}}{\imath=0,}}^n \prod_{l=1}^\frac{\imath}{2} T_l \cdot P_0^{m_\imath-1}, \quad m_0=n.
\end{equation}
To simplify the expressions we introduced the following quantities.
\begin{align}
    f_N \coloneq& Nce^1, \quad \tilde{Q}_n^j \coloneq -\frac{1}{ce^1}Q_n^j \quad
    \Rightarrow \quad Q_n^j = -ce^1\tilde{Q}_n^j \\
    P_m^n \coloneq& \prod_{l=m}^n\overset{\circ}{\gamma}{}_2^{j-l,\beta} = \prod_{l=m}^n\left( 1 - \frac{aZ_N^{j-l}}{(Y_N^{j-l})^2} \right),
    \qquad \hat{P}_m^n \coloneq \prod_{l=m}^n\left( 1 - \frac{aZ_N^{j-l}}{(Y_N^{j-l})^2} + \frac{\Lambda_N}{(Y_N^{j-l})^2L_N^{j-l}} \right)
\end{align}
And the nested sum operators are given by:
\begin{equation}\label{eq: S_l}
    S_l \coloneq -\frac{f_N^2}{Y_N^{j-m_{2(l-1)}}}
    \sum_{m_{2l-1}=1}^{m_{2(l-1)}-1} \frac{\hat{P}_{m_{2l-1}+1}^{m_{2(l-1)}-1}}{(Y_N^{j-m_{2l-1}})^2}
    \sum_{m_{2l}=0}^{m_{2l-1}-1} \frac{P_{m_{2l}+1}^{m_{2l-1}-1}}{Y_N^{j-m_{2l}}}
\end{equation}
\begin{equation}\label{eq: T_l}
    T_l \coloneq -c\,e^1 \sum_{m_{2l-1}=1}^{m_{2(l-1)}-1} \frac{P_{m_{2l-1}+1}^{m_{2(l-1)}-1}}{Y_N^{j-m_{2l-1}}}
    \frac{Nf_N}{Y_N^{j-m_{2l-1}}} \sum_{m_{2l}=0}^{m_{2l-1}-1} \frac{\hat{P}_{m_{2l+1}}^{m_{2l-1}-1}}{(Y_N^{j-m_{2l}})^2}
\end{equation}
With that we achieved our first goal of calculating the triangulation explicitly. For the top line and top angle we can just pick the ones of the last slice $\bar{b}_2^N$ and $\bar{\gamma}_2^N$ and we can use these recursions to express the opening angle of the $j$-th slice:
\begin{align}\label{eq: alpha_2^j}
    \alpha_2^{0,j} =& \overset{\circ}{\alpha}{}_2^{0,j} + \overset{\circ}{\alpha}{}_2^{j,c}c_2^j
    + \overset{\circ}{\alpha}{}_2^{j,\beta}\beta_2^j \\
    =& \overset{\circ}{\alpha}{}_2^{0,j} + \overset{\circ}{\alpha}{}_2^{j,c}\left\{ \overset{\circ}{b}\,_2^{j-1}
    + \sum_{n=1}^{j-2}\left( O_n^{j-1}\overset{\circ}{b}\,_2^{j-1-n} + Q_n^{j-1}\overset{\circ}{\gamma}\,_2^{j-1-n} \right) \right\}
    + \overset{\circ}{\alpha}{}_2^{j,\beta}\left\{ \overset{\circ}{\gamma}{}_2^{j-1}
    + \sum_{n=1}^{j-2}\left( V_n^{j-1}\overset{\circ}{b}{}_2^{j-1-n} + W_n^{j-1}\overset{\circ}{\gamma}{}_2^{j-1-n} \right) \right\},  \notag
\end{align}
from which we can calculate the opening angle of the triangle, by summing over all slices.\\

\section{Product integrals}\label{sec: Product integrals}
If we take the limit of $N\to\infty$ then the procucts $P_m^n$ and $\hat{P}_m^n$ become infinite products i.e., include arbitrarily many multiplications, since $m$ can be as small as $1$, while $n$ can be as big as $N$. We find a way to calculate their limits, using their analogy to series, which converge to integrals.\\

In the case of a series, we sum over a null sequence. The series converges to a finite value if the sequence converges to $0$ fast enough.
\begin{equation}
    \sum_{n=0}^\infty a_n \to a = \lim_{N\to\infty}S_N\in (-\infty,\infty), \quad S_N = \sum_{n=0}^N a_n, \qquad
    \text{if} \ a_n \overset{n\to\infty}{\longrightarrow} 0 \ \text{sufficiently fast} 
\end{equation}
As we saw in the previous section the integral works a bit differently: it is governed by an underlying function instead of a series i.e., the series changes according to the function in each step.
\begin{equation}
    \sum_{n=0}^N\underbrace{\frac{a_n(N)}{N}}_{\to0} \overset{N\to\infty}{\longrightarrow} I\in (-\infty,\infty),
    \qquad a_n(N) = f\left(\frac{n}{N}\right)
\end{equation}
In this case the increasing number of values we sum over is balanced by the $\frac{1}{N}$-factor, which makes the values smaller as $N$ increases.\\

In the case of the products $P_m^n$, which we encountered in the previous section, we see an analogy to this concept. These are products of the form:
\begin{equation}
    \lim_{N\to\infty}\prod_{n=0}^{N-1}\left(\vphantom{\frac{a_n(N)}{N}}\right. 1 + \underbrace{\frac{a_n(N)}{N}}_{\to0} \left.\vphantom{\frac{a_n(N)}{N}}\right),
    \quad a_n(N) = f\left(\frac{n}{N}\right)
\end{equation}
We wrote the product in a way such that we multiply $1$'s with small corrections instead of adding $0$'s with small corrections. Thus, we conclude, that while a Riemann integral is the limit of adding an increasing number of values increasingly close to $0$, the neutral element to addition, we multiply an increasing number of values increasingly close to $1$, the neural element to multiplication, in the case of these products. It thus makes sense to consider them as product integrals. We discuss their convergence in the Supplemental Material.
\begin{definition}[Product Integral] The product integral of a function $f$ on the interval $[a,b]$ is given by the limit
\begin{equation}
    \mathcal{P} f dx \coloneq \lim_{N\to\infty}\prod_{n=0}^N\left( 1 + \frac{f(x_n)}{N} \right),
\end{equation}
with $x_0 = a$ and $x_N = b$.
\end{definition}

What allows us to calculate some integrals analytically is the fundamental theorem of calculus. It relates a function describing the result (primitive function) to the integrand via an operation (derivative for ex.). This allows us to act this operation on functions we know from which we can learn what integral of the result will be.\\

In the Supplemental Material we derive an analogue to the fundamental theorem of calculus for product integrals and find, that the product integral can be calculated by:
\begin{equation}
    \underset{a}{\overset{b}{\mathcal{P}}} f dx = \frac{F(b)}{F(a)},
\end{equation}
using the primitive function $F$, which is defined as follows:
\begin{definition}[Product Primitive Function]
A product primitive function $F$ to a real smooth integrand $f$ is a function that satisfies:
\begin{equation}
    F = \mathcal{P}f \quad :\Leftrightarrow \quad \frac{F'(x)}{F(x)} = f(x).
\end{equation}
\end{definition}

We get a very similar relation as in the case of sums, with the difference, that we divide the derivative with the function. This unsurprisingly leads to a multiplicative behaviour of the integration constant.
\begin{equation}
    F = \mathcal{P}f \quad \Rightarrow \quad c\,F = \mathcal{P}f, \quad \forall\,c\in\mathbb{R}
\end{equation}

\section{Limit of the second order triangulation}\label{sec: Limit of the second order triangulation}
After calculating the limits of the product via product integration we can successively convert the nested sums into integrals and arrive at $n$-layered nested integrals, which repeatedly integrate over the same function. We find a general expression for this type of nested integrals in the Supplemental Material.
\begin{equation}
    J_\imath = \int_0^{x_\imath} f(x_{\imath-1})\int_0^{x_{\imath-1}} f(x_{\imath-2}) \ldots \int_0^{x_1} f(x_0)\,dx_0 \ldots dx_{\imath-2}\,dx_{\imath-1} = \sum_{s=1}^\imath\frac{(-1)^{\imath-s}}{(\imath-s)!}F^{\imath-s}(0)[I_s]_0^{x_\imath}
\end{equation}
We can see, that the expression is very similar to the exponential function. When we simplify the asymptotic expressions for $\bar{b}_2^N$, $\bar{\gamma}_2^N$ and $\bar{\alpha}_2^N$ we obtained with these methods, we can indeed identify the power series of sines and cosines. This allows us to finally calculate the limit of the second order terms of the fundamental solution explicitly. In this limit the approximation with a curvature step function becomes precise and the asymptotic expressions of the fundamental solution converge to the exact result.\\
On the sphere we know, that the excess angle is related to the surface of the spherical triangle. So, one could expect a surface integral over the curvature field to be involved int these quantities. For the top line and the top angle it almost looks that way, but there are additional functions in the integrand and we integrate over index variables instead of the side lengths.
\begin{align}\label{eq: b_2}
    b_2\varepsilon^2 =&\ \bar{b}_2^N\varepsilon^2
    \overset{N\to\infty}{\longrightarrow} -\frac{(c\,e^1)^2}{y(1)}
    \int_0^1 \int_0^1 k^2\, K[k,1-n] \,dk\, n\, dn
\end{align}
\begin{align}\label{eq: gamma_2}
    \gamma_2\varepsilon^2 = \bar{\gamma}_2^N\varepsilon^2
    \overset{N\to\infty}{\longrightarrow} c\,e^1 \int_0^1 \left( 1 - n\frac{ a\,z(1) }{ y^2(1) } \right) \int_0^1 k^2\, K[k,1-n] \,dk\,dn
\end{align}
For the opening angle we can clearly see, that this relation does not hold in general as we get triple nested integrals.
\begin{align}\label{eq: alpha_2}
    \alpha_2\varepsilon^2 =& \bar{\alpha}_2^N\varepsilon^2 = \sum_{j=1}^N\alpha_2^{0,j}\varepsilon^2 \\
    \overset{N\to\infty}{\longrightarrow}& c\,e^1\int_0^1 \int_0^1 \frac{1}{n^2}\int_0^n k^2\,K[k,j] dk\,dn
    + \int_0^j \left[ \frac{a\,z(j)}{y^2(j)} + n\frac{ (c\,e^1)^2 - a^2z^2(j) }{y^4(j)} \right]
    \int_0^1 k^2\,K[k,j-n] dk\,dn\, dj \notag
\end{align}

The curvature field in terms of the index variables is related to the Gaussian curvature field expressed in FNC by:
\begin{equation}
    K[i,j] \coloneq K\left( i\,y(j), \,\sin^{-1}\frac{j\,a\sin\beta}{y(j)} \right)
    \in C^\infty([0,1]^2;\mathbb{R}).
\end{equation}

\section{Generalization of the cosine- and sine-laws to varying curvature fields}\label{sec: sine- & cosine-laws}
Our method allows to generalize the cosine- and sine-laws to varying Gaussian curvature fields. In this work we calculated the second order term of their power series expansion for strictly positive curvature fields.\\
We compare our results to the established cosine- and sine-laws and the metric, to put them into context.\\

The metric is often seen as the natural generalization of Pythagoras theorem. We apply it to infinitesimal distances and allow factors to vary from $\delta_{ij} \to g_{ij}$:
\begin{equation}
    c^2 = a^2 + b^2 \quad \longrightarrow \quad ds^2 = g_{ij}dx^idx^j
\end{equation}

Another generalization is to allow for non-orthogonal triangles i.e. $\beta \overset{i.g.}{\neq} \frac{\pi}{2}$. This leads us to the cosine law:
\begin{equation}
    c^2 = a^2 + b^2 - 2ab\cos\gamma
\end{equation}
This can be achieved by a metric, when choosing a skew basis. Each basis treats a different infinitesimal triangle, as it gives us and instant of the cosine law for one specific angle $\beta$, so that the metric has to be translated into a different basis, when we change the angle $\beta$.
\begin{equation}
    ds^2 = dx^2 + dy'^2 - 2\underbrace{\cos\beta}_{g_{xy}}\,dx\,dy', \quad e_y' = \cos\beta\,e_x + \sin\beta\,e_y, \quad
    e_x = \binom{1}{0}, \quad e_y = \binom{0}{1}
\end{equation}

On the one hand the cosine- and sine-laws only cover flat space and the spherical version can only deal with constant curvature. Thus the metric version is more general. On the other hand just having the metric version without an analogue to the sine-law can lead into trouble, since one is stuck with the cosine. We encountered such a case when we calculated the redshift contribution from the variation of the pulsar rotation in \cite{PTA1}.\\

The addition of the sine laws gives us more flexibility, since it involves all angles and side-lengths in a triangle:
\begin{equation}
    \frac{a}{\sin\alpha} = \frac{b}{\sin\beta} = \frac{c}{\sin\gamma}
\end{equation}

The cosine- and sine-laws can be extended to the case of constant curvature, which amounts to triangles on the sphere and pseudo-sphere.
\begin{equation}
    \cos_Kc = \cos_Ka\cos_Kb + K\sin_Ka\sin_Kb\cos\gamma, \qquad \frac{\sin_Ka}{\sin\alpha} = \frac{\sin_Kb}{\sin\beta} = \frac{\sin_Kc}{\sin\gamma}
\end{equation}
Where the generalization of the cosine and sine to surfaces of constant curvature is used.\\

Here we can see, that the cosine- and sine-laws head into a different direction. They describe relations in a large triangle and are thus fully non-linear and do not require a bases as opposed to the metric formalism, where our choice of basis at each point forms infinitesimal triangles on which we use the cosine law for a flat space.\\

Our results go further in the direction of generalizing the non-linear relations in a large triangle by relating the top-line, top-angle and opening angle to the base-line, side-line and base-angle for an arbitrary curvature field.\\
To better compare our results to the original cosine- and sine-laws we swap the role of $b$ and $c$, $\beta$ and $\gamma$:
\begin{align}
    c =& C(K(p);\gamma,a,b) \approx C^{(2)}(K(l,\varphi);\gamma,a,b)
    \approx y - \frac{f^2}{y}\int_0^1 n \int_0^1 k^2\, K[k,1-n] \,dk\,dn \label{eq: C^(2)(a,b)} \\
    \beta =& SC(K(p);\gamma,b,a) \approx SC^{(2)}(K(l,\varphi);\gamma,b,a)
    = \sin^{-1}\left( \frac{b}{y}\sin\gamma \right)
    + f\int_0^1 \left( 1 - n\frac{ a\,z}{ y^2 } \right) \int_0^1 k^2\,
    K[k,1-n] \,dk\,dn \label{eq: SC^(2)(b,a)} \\
    \alpha =&\ SC(K(p);\gamma,a,b) \approx SC^{(2)}(K(l,\varphi);\gamma,a,b) \label{eq: SC^(2)(a,b)} \\
    \approx&\ \sin^{-1}\left( \frac{a}{y}\sin\gamma \right)
    + f\int_0^1 \int_0^1 \frac{1}{n^2}\int_0^n k^2\,K[k,j] dk\,dn
    + \int_0^j \left[ \frac{a\,z(j)}{y^2(j)} + n\frac{ f^2 - a^2z^2(j) }{y^4(j)} \right]
    \int_0^1 k^2\,K[k,j-n] dk\,dn\, dj, \notag
\end{align}
with
\begin{equation}
    f = ab\sin\gamma, \quad y = \sqrt{a^2 + b^2 - 2ab\cos\gamma}, \quad z = a - b\cos\gamma.
\end{equation}
\begin{figure}[h!]
    \centering\includegraphics[width=.4\linewidth]{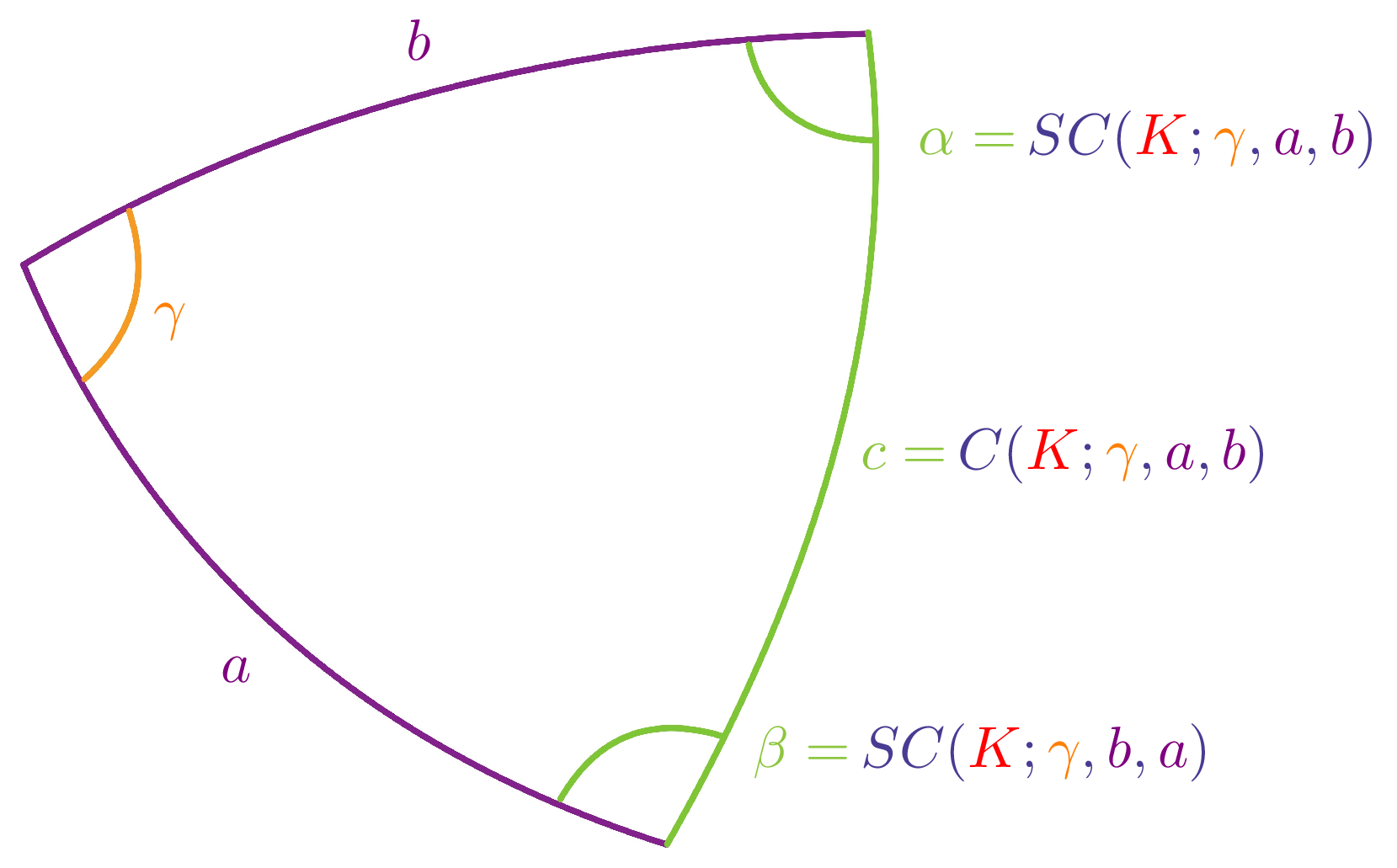}
    \caption{A sketch of the spherical triangle with the labelling choice used int this section.}
    \label{fig: Spherical_Triangle_2}
\end{figure}

We note an asymmetry arising from the varying curvature fields, which a priori can be completely asymmetric.\\
These expressions are valid for a strictly positive Gaussian curvature field $K\in C^\infty(\mathcal{M}^2;\mathbb{R}_+)$ on a $2$-dimensional Riemannian (sub)-manifold and approximate the quantities up so second order.\\
From the discussion in Sec.~\ref{subsec: Curvature parametrization} we know that this suffices to deal with $n$-dimensional manifold, when we systematically divide the problem into $2$-dimensional geodesic sub-manifolds.\\

We restrict our calculations to $\beta\in\left[0,\frac{\pi}{2}\right)$ and thus the results are shown to hold in this parameter range. The test-plots in the previous section and the comparison with $\mathcal{S}^2$ in the Supplemental Material indicate however, that the results are valid on the full parameter range $[0,2\pi)$. This has to be proven properly, but that is beyond the scope of this work.\\
An attempt at that could be made using the procedure outlined in Appendix~\ref{sec: geod. eq. test recipe}.\\

\section{Conclusions}
We introduced the geodesic flow bundle to describe a geometry in a more direct way, then the metric does. We found, that the structure of the directions at a point can be described via the geodesic flow bundle on the unit-sphere, and that there is a natural way to define the angular labels, by subsequently embedding lower dimensional spheres into the next higher dimensional ones until we reach $\mathcal{S}^{n-1}$, where $n$ is the dimension of the manifold.\\

We introduced the special case of normal chart, which preserves distances from the centre and angles at the centre and consists purely of geometrically meaningful quantities.\\

When trying to understand the relation of the geodesic flow bundle to the curvature of a geometry we found that infinitesimal spherical triangles can serve as generators of a geometry. From this we concluded that the sectional curvature, a value on an infinitesimal disc around a point, properly characterizes the degrees of freedom a curvature distribution has. This led us to introduce the great circle functions (i.e. functions on the great circles of the $n$-sphere) to get to a parametrization of a generic curvature field, which is compatible to our formalism and includes all degrees of freedom. Since the geodesic flow bundle contains more information, then the metric i.e., it includes all geodesics of the manifold in arclength parametrized form, we can calculate the metric from the geodesic flow bundle. Doing the converse is only possible, if one manages to solve the geodesic equations for a generic case. We calculated the metric from the GFB for the example in the Supplemental Material, to compare our parametrization of curvature fields to the Riemann tensor. From this example we conclude that there are geometries, which the metric cannot describe and thus curvature degrees of freedom, which are not included in the Riemann tensor.\\

We restricted ourselves to the special case of plane geodesic flow bundles, but the spherical triangulation, with which we calculate the geodesic flow bundle from a curvature field, only requires the unique existence of the geodesics that form the triangles. To treat more generic geometries one would have to rework the angular structure.\\

Since we restrict ourselves to positive curvature the only non-trivial exactly known example, we can test against is the $2$-sphere. It has turned out to be a very effective test since it has so far failed whenever we made a mistake in the calculation. Another effective test, which frequently showed us when something was wrong was plotting the triangles. When the triangles on the plot match up, we know, that the method is consistent up to the precision of the plot. To make the error better visible one can always push the limits by plotting extreme cases.\\

We completed the first step in our attempt to solve the Einstein equations. We related the curvature distributions of a subclass of Riemannian geometries to the geodesic flow bundle and derived an explicit formula to calculate the second order term to every geodesic in the manifold.\\
Higher order terms could be derived using the same methodology. It is not clear however, whether the recursions will be solvable. One could also get to terms on which our methods to calculate the limits do not apply.\\
The extension to negative curvature is a safer bet, we expect some signs to change. Though changes in the signs could lift or cause cancellations which in a mistake has proven to lead to a very different problem in the limit calculation.\\
To apply this method to general relativity we need to extend this formalism to the pseudo-Riemannian case, which we are currently working on.\\
Another problem which we did not have time to look into jet, is how to connect the energy momentum tensor to a curvature field given in the parametrization scheme we derived in this work. The Ricci tensor itself is not compatible with the geodesic flow bundle formalism and translation from the metric to geodesic flow bundle formalism is in general not possible.\\

Finally we present our main result, the second order terms of the top line, top angle and opening angle of a geodesic triangle. They determine the geodesics up to second order from Gaussian curvature in a geometric way and are essentially the second order term of the generalization of the cosine and sine laws to a positive varying curvature field.\\

To calculate this result it was necessary to introduce a notion of product integral and derive a main theorem of calculus for them. This allows us to calculate asymptotic limits of infinite products in closed (analytic) form.

\section*{Acknowledgements}
\addcontentsline{toc}{section}{Acknowledgements}
A.B. is supported by the Forschungskredit of the University of Zurich Grant No. FK-21-083 S.T. is supported by Swiss National Science Foundation Grant No. 200020 182047. Symbolic manipulation as well as some of the plots have been done using Mathematica \cite{Mathematica}.

\appendix

\section{Comparison of the geodesic flow from a point with the exponential map}\label{sec: Exp. map}
Our construction is similar to the exponential map, whose definition we give here briefly for comparison. We use the definition 2.51 given in the lecture script by M. Burger and S. Tornier~\cite{Burger-Skript} in section 2.3 for this purpose.\\

Let $(\mathcal{M},g)$ be a Riemannian manifold, with connection $\nabla$. Then 
\begin{align}
    \begin{matrix}
    \forall x\in\mathcal{M} \quad \exists \text{ a neighbourhood } U\subset\mathcal{M} \ \wedge \
    \varepsilon, \delta > 0: \qquad
    C: & \{(p,v)\in TU | \Vert v\Vert<\varepsilon\}\times(-\delta,\delta) & \to & \mathcal{M} \\
    & ((p,v),t) & \mapsto & c_{(p,v)}(t)
    \end{matrix}
\end{align}
is a well defined smooth map, where $c_{(p,v)}$ is the unique geodesic satisfying $c_{(p,v)}(0) = p$ and $\dot{c}_{(p,v)}(0) = v$.\\
When $\Vert v\Vert$ is chosen small enough, then $\delta$ can be bigger than $1$, so we can define a subset of the tangent bundle,
\begin{equation}
    \Omega \coloneq \{(p,v)\in T\mathcal{M}|c_{(p,v)} \text{ is defined on } (-\delta,\delta),\ \delta>1 \}
\end{equation}
on which we can define the exponential map:\\

\textbf{Definition:}
\begin{equation}
    \begin{matrix}
    \exp: & \Omega & \to & \mathcal{M}, \quad \exp((p,v)) = c_{(p,v)}(1)
    \end{matrix}
\end{equation}

We do not want a map involving tangent spaces, since tangent spaces are abstract constructs which a priory have no geometrical meaning unless one imposes a metric on them. Our goal is to only use geometrically meaningful quantities as far as possible. Moreover, we would mix up the directions with the distances, just so we would have to separate them again. Also, there are in general not many restrictions on the parametrization of the geodesics, when defining the exponential map. We however, need it to be parametrized by arclength.\\
The intention here is to construct a local isomorphism between the tangent bundle and the manifold constructed by geodesics. We however, want to collect a family of curves which we consider as geodesics by definition to give our manifold a geometric structure instead of using the metric.

\section{Branches of the curvature Parametrization}\label{sec: Branches}
There are three cases of triangles which all yield equivalent equations:
\begin{align}
    &\alpha_2 \geqslant \vartheta_p \quad \Rightarrow \quad \alpha_2 \geqslant 0: &\quad
    &\frac{\sin\varphi_K}{\sin(\alpha_2-\vartheta_p)} = \frac{\sin\varphi_p}{\sin(\pi-\vartheta_K)} = \frac{\sin\tau}{\sin\vartheta_p} \label{eq: alpha>=theta} \\
    &\alpha_2 < \vartheta_p, \quad \alpha_2 < 0: &\quad
    &\frac{\sin\varphi_K}{\sin(\pi+\alpha_2-\vartheta_p)} = \frac{\sin\varphi_p}{\sin(-\vartheta_K)} = \frac{\sin\tau}{\sin\vartheta_p} \label{eq: alpha<theta & <0} \\
    &\alpha_2 < \vartheta_p, \quad \alpha_2 > 0: &\quad
    &\frac{\sin(2\pi-\varphi_K)}{\sin(\vartheta_p-\alpha_2)} = \frac{\sin\varphi_p}{\sin\vartheta_K}
    = \frac{\sin\tau}{\sin(\pi-\vartheta_p)} \label{eq: alpha<theta & >0}
\end{align}
The signs in the first equations all cancel, only the second equation in~\eqref{eq: alpha<theta & <0} has a different sign. We square this equation, and thus get rid of the sign change, and add it to the squared cosine law, to get an equation from which we can get $\varphi_K$.
\begin{align}
    &1 = \sin^2\tau + \cos^2\tau = \frac{\sin^2\vartheta_p\sin^2\varphi_K}{\sin^2(\alpha_2-\vartheta_p)}
    + \left( \cos\varphi_p\cos\varphi_K + \sin\varphi_p\sin\varphi_K\cos\vartheta_p \vphantom{\sqrt{2}}\right)^2 \\
    &\Leftrightarrow \quad A\sin^2\varphi_K + B\sin(2\varphi_K) = C \\
    &A = \frac{\sin^2\varphi_p}{\sin^2(\alpha_2-\vartheta_p)} + \sin^2\varphi_p\left( \cos^2\vartheta_p + 1 \right) - 1, \qquad
    B = \frac{1}{2}\sin(2\varphi_p)\cos\vartheta_p, \qquad C = \sin^2\varphi_p
\end{align}
This equations has two solutions which are both defined piece-wise.
\begin{align}
    \varphi_K = \begin{cases}
        \tan^{-1}\left( \frac{C}{ B + \sqrt{B^2 + (A-C)C} } \right), & B \geqslant 0 \\
        \tan^{-1}\left( \frac{C}{ B - \sqrt{B^2 + (A-C)C} } \right) + \pi, & B < 0
    \end{cases}, \label{eq: varphi_K} \qquad
    \varphi_K = \begin{cases}
        \tan^{-1}\left( \frac{C}{ B - \sqrt{B^2 + (A-C)C} } \right), & B \geqslant 0 \\
        \tan^{-1}\left( \frac{C}{ B + \sqrt{B^2 + (A-C)C} } \right), & B < 0
    \end{cases}
\end{align}
The one on the left represents our choice of domain $\varphi_K\in[0,\pi)$, the one on the left returns negative angles.\\
Since $\vartheta_p\in\left(-\frac{\pi}{2},\frac{\pi}{2}\right)$ we can reduce the condition on $B$ into one on $\varphi_p$:
\begin{align}
    B \geqslant 0 \quad \Leftrightarrow \quad \varphi_p\in\left[0,\frac{\pi}{2}\right)\bigcup\left[\pi,\frac{3\pi}{2}\right)
    \qquad \wedge \qquad
    B \leqslant 0 \quad \Leftrightarrow \quad \varphi_p\in\left[\frac{\pi}{2},\pi\right)\bigcup\left[\frac{3\pi}{2},2\pi\right)
\end{align}
From the plots in Figure~\ref{fig: Cases} we see, that the denominator $B+\sqrt{B^2+(A-C)C}$ of the solution is positive, wherever $B\geqslant0$ and $B-\sqrt{B^2+(A-C)C}$ is negative wherever $B\leqslant0$. We plot the denominators on the domains of the angular position coordinates $\varphi_p$ and $\vartheta_p$, for a sample of tilt parameters $\alpha_2$.
\begin{figure}[h!]
    \begin{minipage}{\linewidth}
        \centering\includegraphics[width=\linewidth]{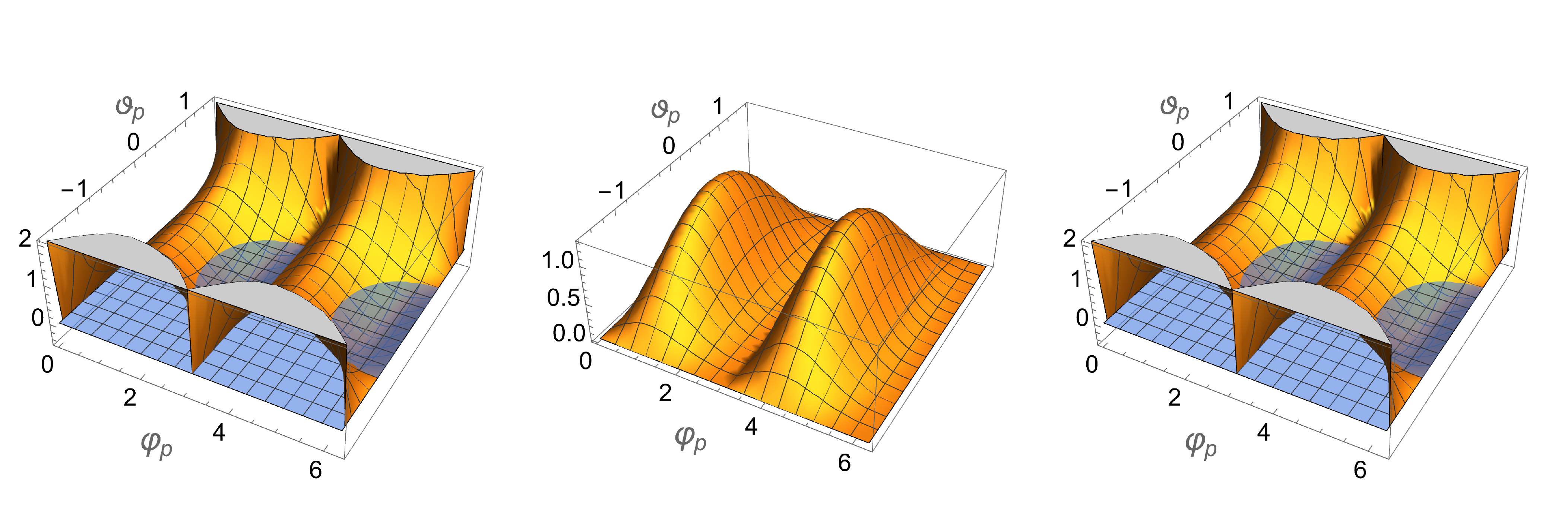}
    \end{minipage}
    \begin{minipage}{\linewidth}
        \centering\includegraphics[width=\linewidth]{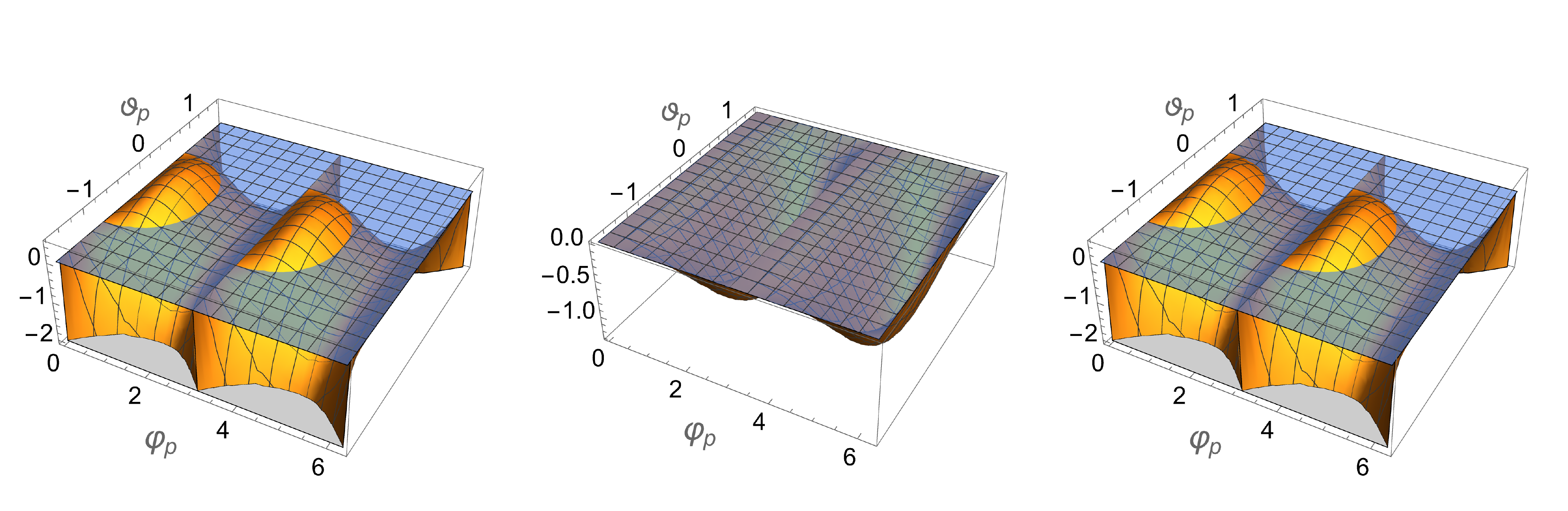}
    \end{minipage}
    \caption{We plot the denominator $B+\sqrt{B^2+(A-C)C}$ on the top and $B-\sqrt{B^2+(A-C)C}$ on the bottom as a function of $\varphi_p$ and $\vartheta_p$ in orange, for $\alpha_2=-\frac{\pi}{2},\,0,\,\frac{\pi}{2}$ from left to right. The zero plane is colored in blue.}
    \label{fig: Cases}
\end{figure}
Since the denominator in the case of $B\geqslant0$ is positive we use the principle branch of arctan and the side branch $\tan^{-1}(x)+\pi$ for the case of $B\leqslant0$, where the denominator is negative. This results in a map:
\begin{equation}
    \begin{matrix}
        \varphi_K: & [0,2\pi)\times\left(-\frac{\pi}{2}\right]\times\left(-\frac{\pi}{2}\right] & \to & [0,\pi) \\
        & (\varphi_p,\vartheta_p,\alpha_2) & \mapsto & \varphi_K = \begin{cases} \varphi_K^+, & B \geqslant 0 \\
        \varphi_K^-, & B <0 \end{cases}
    \end{matrix},
\end{equation}
Where the $\varphi_p$ domain is folded:
\begin{align}
    \begin{matrix}
        \varphi_K^+: & \left[0,\frac{\pi}{2}\right)\bigcup\left[\pi,\frac{3\pi}{2}\right) & \to & \left[0,\frac{\pi}{2}\right) \\
        & \varphi_p & \mapsto & \varphi_K^+
    \end{matrix}, \qquad
    \begin{matrix}
        \varphi_K^-: & \left[\frac{\pi}{2},\pi\right)\bigcup\left[\frac{3\pi}{2},2\pi\right) & \to & \left[\frac{\pi}{2},\pi\right) \\
        & \varphi_p & \mapsto & \varphi_K^-
    \end{matrix}
\end{align}
This is consistent with the fact, that a plane rotated by $\pi$ coincides with itself.\\

The other curvature plane parameter $\vartheta_K$ can now be calculated from the sine-law and $\varphi_K$:
\begin{equation}
    \vartheta_K = \sin^{-1}\left(\frac{\sin\varphi_p\sin(\alpha_2-\vartheta_p)}{\sin\varphi_K(\varphi_p,\vartheta_p,\alpha_2)}\right)
\end{equation}
The principle branch of the arcsine nicely coincides with the domain of $\vartheta_K$ in this case.\\

So, we ultimately parametrize the curvature field $K$ with the location $p\in\mathcal{M}$ in faithful normal coordinates $\mathcal{N}(p) = (l,\varphi_p,\vartheta_p)$ and the orientation parameters $\varphi_K,\vartheta_K$ of the geodesic plane, in which the geodesic in question lies.
\begin{equation}
    K_p(\varphi_K,\vartheta_K)
\end{equation}

\section{Parallel transport}\label{sec: Parallel transport}
It seems to be a priory not clear, how we would compare directions at two different points on a manifold. Especially if we use tangent spaces to represent them. To deal with this we would usually introduce the notion of covariant derivative or equivalently a parallel transport. This introduces an additional structure and thus additional degrees of freedom. If we work with the geodesic flow bundle however we see that there is no choice to be made. The parallel transport is already built into the flow bundle structure.\\

In classical Euclidean geometry we understand parallel as two straight lines which have the same distance between each other everywhere. Parallelly transporting a straight line along another straight line would mean, that we do not change the angle between the two as we do the transport. When we use straight lines as representants of directions we can conclude that we transport a direction parallelly along a straight line, if and only if we leave the angle between the direction and the straight line invariant.\\
We translate this concept into curved geometries, by replacing the straight lines with geodesics, which are locally straight lines. To make use of our angular structure we start w.l.o.g. form the origin $o$ and preserve the direction $\hat{\Omega}(o)$ along the geodesic to $p = \gamma_{o,\hat{\Omega}_\alpha}(l)$, by keeping the angular labels constant.
\begin{equation}
    P_{o,p,\gamma_{o,\hat{\Omega}_\alpha}}(\hat{\Omega}_\beta) = \hat{\Omega}_\beta
\end{equation}
For example in Fig.~\ref{fig: geodesic flow bundle on S^2} the black line is the geodesic from $p$ in direction $\hat{0}$, which is the direction of the original geodesic from $o$ transported along $\gamma_{o,\alpha}$ and thus encloses the angle $\alpha$ with this geodesic at $o$ as well as at $p$.\\

If we would try to define a parallel transport along non-geodesics in this formalism, we would need another flow bundle, since we need a curve between any two points on the manifold to parallelly transport between this pair of points. This would mean however that we just introduced two conflicting flow bundles and in that case the way we would identify directions at different points would be inconsistent with the geometry we imposed on the manifold.\\

To parallelly transport a direction along a non-geodesic we would just transport along geodesics between increasingly many points on the curve i.e. approximate the curve with geodesics.
\begin{align}
    \{x_i\}_{i=0}^N\subset c(I)\in C^\infty(I;M), \quad I\subset\mathbb{R}, \quad x_0 = a, \quad x_N = b
\end{align}
At each of these points we get a kink, and we pick up a change in the angle, since we continue with a geodesic that does not continue in the same direction. If we take the limit to infinitely many sampling points, forming an infinitely dense sampling of the curve everywhere, these changes become infinitely small, but we have infinitely many of them, the kinks vanish, and the approximation coincides with the smooth non-geodesic.\\
\begin{align}
    \lim_{N\to\infty}\bigcup_{i=0}^N\gamma_{x_i,\Omega_i}(\lambda) = c(\lambda), \quad
    \gamma_{x_i,\Omega_i}(\lambda_i) = x_{i+1} \quad \forall i<N \ \wedge \ \gamma_{x_N,\Omega_N}(\lambda_N) = b
\end{align}

\begin{figure}[h!]
    \centering\includegraphics[width=0.6\linewidth]{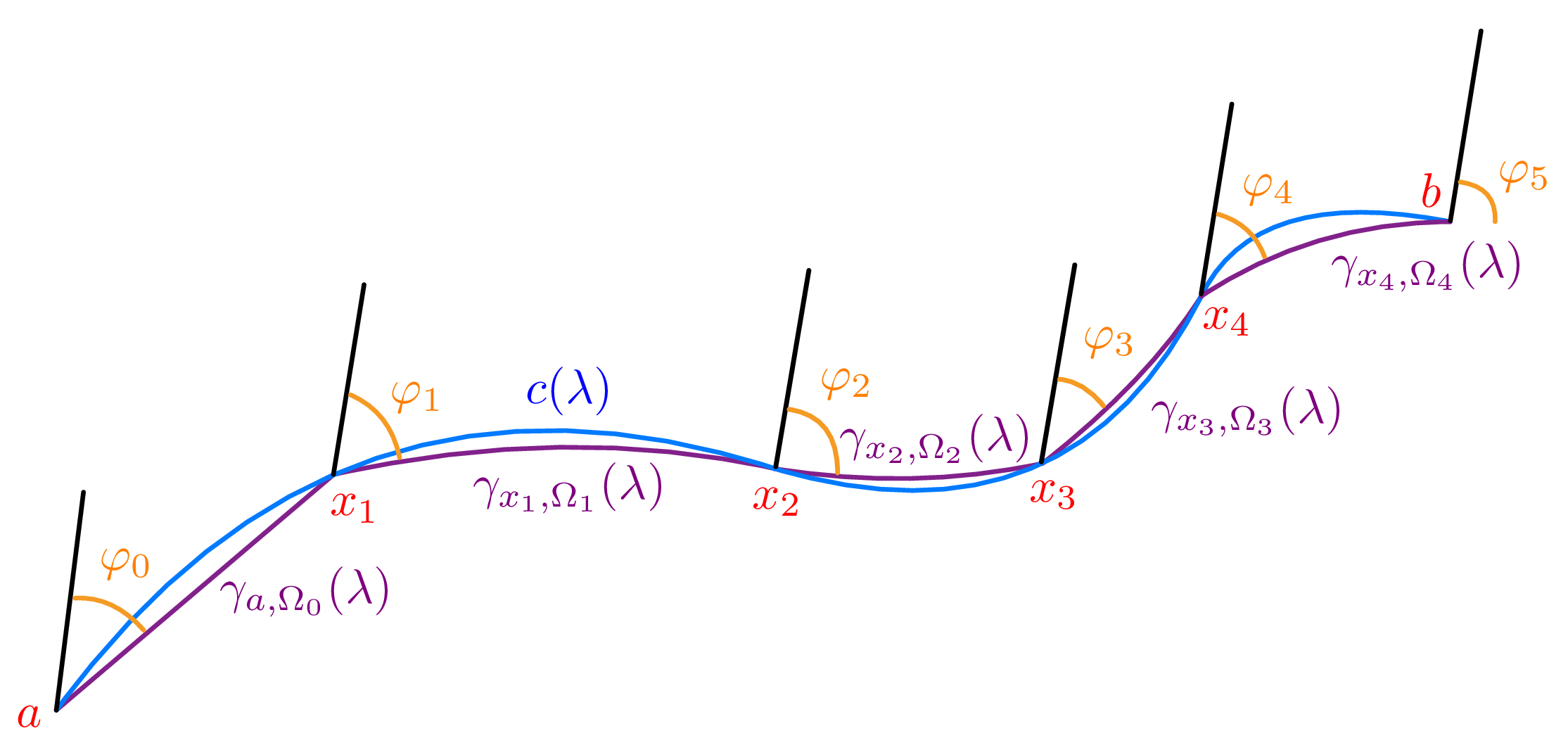}
    \caption{A sketch of how we connect sample points $x_i$ (red) of a in general non-geodesic curve $c(\lambda)$ (blue) with geodesics $\gamma_{x_i,\Omega_i}$ (purple). We transport a direction (black) piece-wise along the geodesics and observe, that the angles $\varphi_i$ (orange) change at the sample points.}
    \label{fig:my_label}
\end{figure}

We can use the inverse of the exponential map to connect the parallel transport to the classical formalism using tangent spaces:
\begin{equation}
    \forall X\in T_{x}\mathcal{M} \quad \exists\ \Omega\in\mathbb{S}^{n-1} \ \wedge \ \alpha\in\mathbb{R}: X =  \left.d_\lambda \gamma_{x,\Omega}(\alpha\lambda)\right|_{\lambda=0}
\end{equation}
\begin{equation}
    P_{c,a,b}(X) \coloneq P_{c,a,b}(\left.d_\lambda \gamma_{x,\Omega}(\alpha\lambda)\right|_{\lambda=0})
    \coloneq \lim_{N\to\infty} \left( \overset{N}{\underset{i=0}{\circ}} P_{\gamma_{x_i,\Omega_i},0,\lambda_i} \right)(\left.d_\lambda \gamma_{x,\Omega}(\alpha\lambda)\right|_{\lambda=0})
\end{equation}
Extending the angle $\Omega$ and scale factor $\alpha$ to fields, the covariant derivative can then be defined as:
\begin{equation}
    \frac{DX}{d\lambda}(\lambda) \coloneq \lim_{\varepsilon\to0} \frac{P_{c,\lambda,\lambda+\varepsilon}(X(c(\lambda))) - X(c(\lambda))}{\varepsilon}, \quad
    \forall X\in\Gamma(T\mathcal{M}), \quad \Omega(x) \ \wedge \ \alpha(x): X(x) = \left. d_\lambda\gamma_{x,\Omega(x)}(\alpha(x)\lambda) \right|_{\lambda=0}
\end{equation}
Then we can define the connection over the integral curves:
\begin{equation}
    \nabla_Y X \coloneq \frac{DX}{d\lambda}, \text{ with } c\in C^\infty(I;\mathcal{M}) \text{ s.t. }
    \dot{c}(\lambda) = Y(c(\lambda)) \quad \forall X,Y\in\Gamma(T\mathcal{M})
\end{equation}

Since the Levi-Civita connection can be directly calculated from the metric and is thus entirely determined by the geometry, it is reasonable to assume, that it coincides with our notion of parallel transport.
Under the assumption, that this is the case, we conclude that all other connections introduce additional structure on to the tangent spaces, which is unrelated to geometry. Also note that in our formalism the tangent spaces are not necessarily present. We can introduce them, but we do not need them.\\

We note a difference to the metric formalism in that here we already have a natural notion of parallel transport form the geodesic flow bundle and we then have to derive the corresponding covariant derivative from that. In the metric formalism it is natural, to start with the connection since we can calculate the connection coefficients of the Levi-Civita connection (Christoffel symbols) directly from the metric. And then we derive the parallel transport from there.\\

From Sec.~\ref{sec: Curvature} we know, that the connection is not needed, to relate the geodesic flow bundle to curvature and is thus of limited importance in this formalism. Though, it can be used to relate the notion of curvature in the GFB formalism to the Riemann tensor. However, calculating the connection from the parallel transport in the flow bundle formalism is a non-trivial task and goes beyond the scope of this work. \\
A more practical way, to compare the flow bundle formalism to the metric one, is to calculate the metric from the geodesic flow bundle in FNC, which we do in Sec.~\ref{sec: metric}. One can then calculate the connection coefficients in FNC. We give a recipe in Appendix~\ref{sec: geod. eq. test recipe} for an analytic consistency check of our main result with the geodesic equations in FNC, using the Levi Civita connection. Carrying this out in full detail however is a combersome and difficult task and is not presented here.\\

\section{A recipe for a consistency check with the geodesic equations in faithful normal coordinates.}\label{sec: geod. eq. test recipe}
Without loss of generality we can restrict ourselves to the $l,\varphi$-plane, where the metric in FNC is given by:
\begin{equation}
    g_{ij}(p) = \begin{pmatrix} 1 & 0 \\ 0 & g_{\varphi\varphi}(p) \end{pmatrix}, \qquad
    g_{\varphi\varphi}(p) = \frac{1 - \dot{l}^2(K,p,\frac{\pi}{2};0)}{\dot{\varphi}(K,p,\frac{\pi}{2};0)}, \qquad
    \mathcal{N}(p) = (l_p,\varphi_p)
\end{equation}

We spend the main part of this paper calculating the fundamental solution and conclude in section~\ref{sec: sine- & cosine-laws}, equations~\eqref{eq: C^(2)(a,b)} and~\eqref{eq: SC^(2)(a,b)} that it is given by:
\begin{align}
    l&(K,p,\beta;\lambda) \approx C^{(2)}(K(l',\varphi'+\varphi_p);\pi-\beta,\lambda,l_p)
    \approx y - \frac{f^2}{y}\int_0^1 n \int_0^1 k^2\, K[k,1-n] \,dk\,dn \\
    \varphi&(K,p,\beta;\lambda) \approx SC^{(2)}(K(l',\varphi'+\varphi_p);\pi-\beta,\lambda,l_p) \\
    &\approx \sin^{-1}\left( \frac{\lambda}{y}\sin\beta \right)
    + f\int_0^1 \int_0^1 \frac{1}{n^2}\int_0^n k^2\,K[k,j] dk\,dn
    + \int_0^j \left[ \frac{\lambda\,z(j)}{y^2(j)} + n\frac{ f^2 - \lambda^2z^2(j) }{y^4(j)} \right]
    \int_0^1 k^2\,K[k,j-n] dk\,dn\, dj, \notag
\end{align}
with
\begin{equation}
    f = \lambda\,l_p\sin\beta, \quad y = \sqrt{\lambda^2 + l_p^2 + 2\lambda\,l_p\cos\beta}, \quad
    z = \lambda + l_p\cos\beta.
\end{equation}
The shift in the angle argument in the curvature field by $\varphi_p$ comes from the fact, that we use the results for the flow on the original geodesic for an arbitrary point $p$ and we thus have to rotate the triangle by $\varphi_p$.

We can then calculate the Christoffel symbols up to second order. Taking the symmetries into account we find three independent non-zero connection coefficients for the Levi-Civita connection:
\begin{align}
    \Gamma^\varphi_{\varphi l} = \frac{g_{\varphi\varphi,l}}{2g_{\varphi\varphi}}, \quad
    \Gamma^\varphi_{\varphi\varphi} = \frac{g_{\varphi\varphi,\varphi}}{2g_{\varphi\varphi}}, \quad
    \Gamma^l_{\varphi\varphi} = -\frac{1}{2}g_{\varphi\varphi,l}, \qquad
    g_{\varphi\varphi,l} = \frac{\partial}{\partial l_p}g_{\varphi\varphi}(l_p,\varphi_p), \quad
    g_{\varphi\varphi,\varphi} = \frac{\partial}{\partial\varphi}g_{\varphi\varphi}(l_p,\varphi_p).
\end{align}

Abbreviating $l(K,p,\beta;\lambda)$ and $\varphi(K,p,\beta;\lambda)$ with $l(\lambda)$ and $\varphi(\lambda)$ the geodesic equation then read:
\begin{align}
    &\ddot{\gamma}^i + \Gamma^i_{jk}(\gamma)\dot{\gamma}^j\dot{\gamma}^k = 0; \qquad
    \gamma^i(\lambda) = \mathcal{N}(\gamma_{p,\beta}(\lambda))^i = \binom{l(\lambda)}{\varphi(\lambda)} \\
    &\ddot{l}(\lambda) - \frac{1}{2}\Gamma^l_{\varphi\varphi}(l(\lambda),\varphi(\lambda))\,\dot{l}^2(\lambda) = 0, \qquad
    \ddot{\varphi}(\lambda) + 2\Gamma^\varphi_{\varphi l}(l(\lambda),\varphi(\lambda))\,
    \dot{\varphi}(\lambda)\dot{l}(\lambda) + \Gamma^\varphi_{\varphi\varphi}(l(\lambda),\varphi(\lambda))\,
    \dot{\varphi}^2(\lambda) = 0
\end{align}
The integral in $\varphi(K,p,\beta;\lambda)$ is in general not so easy to calculate. One would have to investigate, if there is a good expansion of the curvature field $K(l',\varphi')$ for which the integral is solvable. Working this out is beyond the scope of this work.

\bibliographystyle{apsrev4-1}
\bibliography{GeodFBRef}

\end{document}


\maketitle

\tableofcontents

\section{Examples of Geodesic Flow Bundles and their Faithful Normal Charts}\label{sec: Examples}
We demonstrate the concept of the geodesic flow bundle on two examples, the Manhattan metric and the $2$-sphere. We use the faithful normal chart to visualize the geodesics. Since we can plot the geodesic flow bundle onto the $2$-sphere embedded in $\mathbb{R}^3$ and since it is a non-trivial example, it is the most helpful example to understand how this structure works.

\subsection{The Geodesic Flow Bundle of the Manhattan metric}
The geodesic flow bundle on the euclidean plane $E\simeq\mathbb{R}^2$, in Cartesian coordinates is given by:
\begin{equation}
    \Phi_E(q;\alpha,\lambda) = \binom{c}{0} + \lambda\binom{\cos\beta}{\sin\beta},
\end{equation}
restricting to the $x$-axis $q=(c,0)$ for simplicity.

Now we can construct a flow bundle on $\mathbb{R}^2$ which induces the Manhattan / maximum metric:
\begin{equation}
    d(\Vec{x},\Vec{y}) = \max\{|x_1-y_1|,|x_2-y_2|\}
\end{equation}
We leave the angles invariant $\varphi\mapsto\varphi$ but change the distances $r=\sqrt{x^2+y^2}\mapsto l=\max\{|x|,|y|\}$.
\begin{align}
    &\Phi_M((c,0);\beta,\lambda') = l\binom{\cos\varphi}{\sin\varphi}, &\,
    &l(c,\beta,\lambda') = \max\{|c+\lambda'\cos\beta|,|\lambda'\sin\beta|\}, \\
    &\varphi(c,\beta,\lambda') = \sin^{-1}\frac{ac\sin\beta}{\sqrt{(\lambda')^2+c^2+2\lambda' c\cos\beta}}, &\,
    &\lambda' = \lambda\max\{|\cos\beta|,|\sin\beta|\}
\end{align}
We need to transform the parameter to make sure, that if $\Phi_E(q;\beta,\lambda) = p$ then $\Phi_M(q;\beta,\lambda') = p$ ends at the same point in the new geometry.
\begin{figure}[h!]
    \centering
    \includegraphics[width=\linewidth]{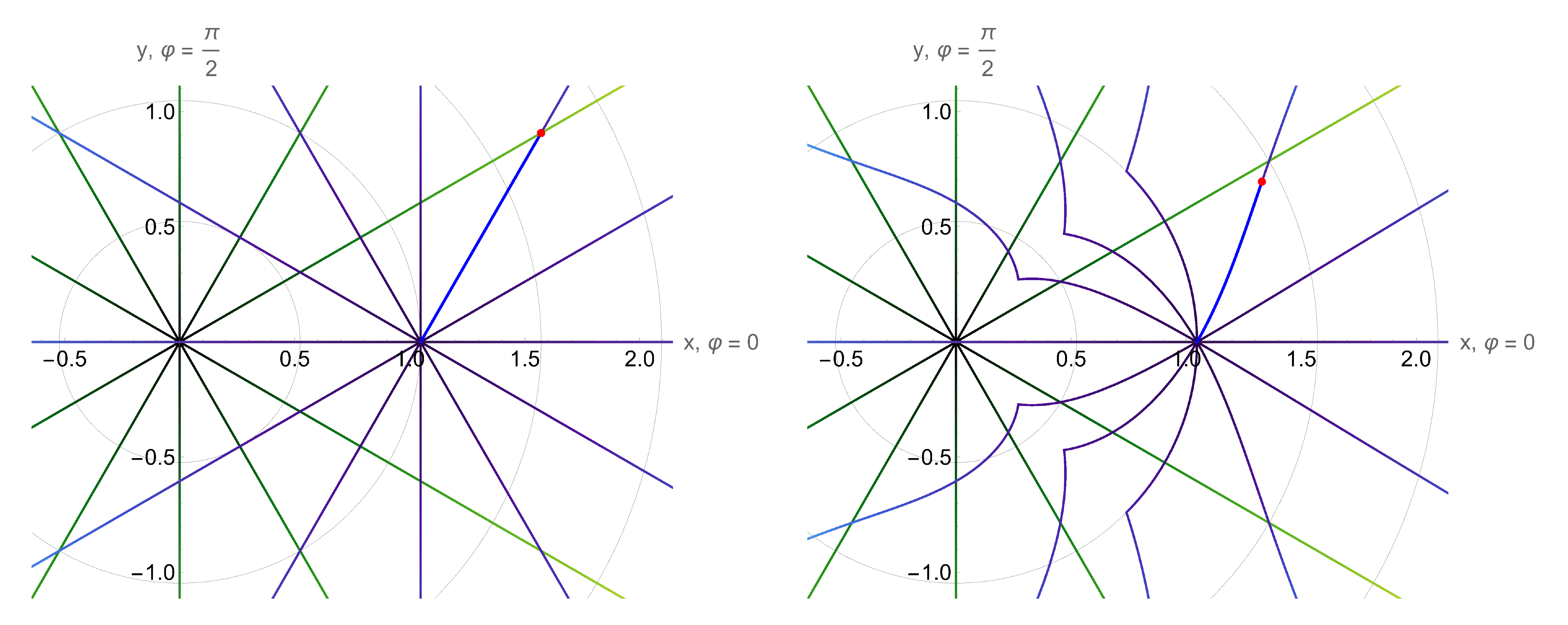}
    \caption{We plot the geodesic flow from $o = (0,0)$ in green and from $q = (c,0)$ in blue of the Euclidean plane on the left and for the Manhattan metric on the right, where $c = \frac{\pi}{3}$. For comparison we highlight the geodesic $\gamma_{q,\beta}(\lambda)$ (blue, thick), with $\beta = \frac{\pi}{3}$ until $p$ (in red) at $\lambda = \frac{\pi}{3}$. We can see in this example, that we can also describe non-smooth metrics in this formalism.}
    \label{fig:my_label}
\end{figure}

\begin{align}
    &\Phi_M(o,\alpha,0) = o \ \wedge \ \Phi_M(o,\alpha,b) = p \quad \Rightarrow \quad
    d(o,p) = l = \max\{|c+\lambda'\cos\beta|,|\lambda'\sin\beta|\} \quad \checkmark \\
    &\Phi_M(q;\beta,0) = q \ \wedge \ \Phi_M(q;\beta,\lambda') = p \quad \Rightarrow \quad
    d(q,p) = \lambda' = \max\{|\lambda\cos\beta|,|\lambda\sin\beta|\} \quad \checkmark
\end{align}

\subsection{The Geodesic Flow Bundle and Faithful Normal Coordinates of the 2-sphere $\mathcal{S}^2$}\label{sec: 2-sphere}
We choose polar coordinates on the sphere to compare the flow bundle of $\mathbb{S}^2$ embedded in $\mathbb{R}^3$ with the faithful normal coordinates.
\begin{equation}
    p(\theta,\phi) = \begin{pmatrix} x \\ y \\ z \end{pmatrix} = \begin{pmatrix} \cos\theta\cos\phi \\ \cos\theta\sin\phi \\ \sin\theta
    \end{pmatrix}, \quad \theta\in\left(-\frac{\pi}{2},\frac{\pi}{2}\right), \quad \phi\in[0,2\pi)
\end{equation}

The sphere $S^{n-1}$ is the subset $\mathbb{S}^{n-1}\subset\mathbb{R}^n$, which is invariant under the group action of $O(n)$ on $\mathbb{R}^n$ i.e. every point $p$ in $S^{n-1}$ is mapped to another point on the sphere under the action of the rotation group.
\begin{equation}
    \begin{matrix}
    \cdot: & O(n)\times\mathbb{R}^3 & \to & \mathbb{R}^3 \\
    & (R,x) & \mapsto & R\cdot x
    \end{matrix}, \quad R\cdot x\in\mathbb{S}^{n-1} \quad \forall R\in O(n) \ \forall x\in\mathbb{S}^{n-1},
    \quad \text{i.e.} \quad O(n)\cdot\mathbb{S}^{n-1}=\mathbb{S}^{n-1}
\end{equation}
Thus we can generate a geodesic by acting a rotation $R\in O(n)$ onto a origin point $o$. We choose our origin at $o = (1,0,0)$ and pick rotations around the $z$-axis to trace out our first geodesic. This geodesic coincides with the half equator (the green line continued by the blue one in Fig.~\ref{Supp-fig: geodesic flow bundle on S^2}).
\begin{equation}
    R_z(\phi) = \begin{pmatrix} \cos\phi & -\sin\phi & 0 \\ \sin\phi & \cos\phi & 0 \\
    0 & 0 & 1 \end{pmatrix} \in O(n); \qquad \gamma_{o,0}(\lambda) = R_z(\lambda)\cdot o
    = \begin{pmatrix} \cos\lambda \\ \sin\lambda \\ 0 \end{pmatrix} = p(0,\lambda), \quad \lambda\in(0,\pi)
\end{equation}

We then use the $O(n)$-action again to generate the flow from $o$ (green gradient lines in Fig.~\ref{Supp-fig: geodesic flow bundle on S^2}), acting rotations around the $x$-axis on the original geodesic.
\begin{align}
    R_x(\varphi) = \begin{pmatrix} 1 & 0 & 0 \\ 0 & \cos\varphi & -\sin\varphi \\
    0 & \sin\varphi & \cos\varphi \end{pmatrix}\in O(n); \quad
    \gamma_{o,\varphi}(\lambda) = R_x(\varphi)\cdot\gamma_{o,0}(\lambda) = \begin{pmatrix}
    \cos\lambda \\ \cos\varphi\sin\lambda \\ \sin\varphi\sin\lambda \end{pmatrix},
\end{align}
with $\lambda\in(0,\pi)$ and $\alpha\in[0,2\pi)$.\\

The geodesic flow from $o$ defines the faithful normal coordinates at $o$. A point $p\in\mathcal{M}$ which can be reach by a geodesic from $o$ is labelled with the direction and arclength of that geodesic:
$p = \gamma_{o,\varphi}(l) \ \mapsto \ p(\varphi,l)$
\begin{equation}
    \begin{matrix}
    \mathcal{N}: & \mathbb{S}^2\subset\mathbb{R}^3 & \to & \mathbb{R}^2 \\
    & p & \mapsto & (l,\varphi)(p)
    \end{matrix}, \quad  l\in\mathbb{R}_+, \ \varphi\in[0,2\pi) \ : \quad \gamma_{o,\varphi}(l) = p
\end{equation}

The geodesic flow from any point appears as straight lines which display distances correctly in the faithful normal coordinates at that point.

Next we rotate the flow from $o$ around the $z$-axis to generate the flow from the original geodesic (blue gradient lines in Fig.~\ref{Supp-fig: geodesic flow bundle on S^2}).
\begin{equation}
    \gamma_{q(c),\beta}(\lambda) = R_z(c)\cdot\gamma_{o,\beta}(\lambda)
    = \begin{pmatrix}
    \cos{c}\cos\lambda - \cos\beta \sin{c}\sin\lambda \\
    \cos\beta\cos{c}\sin\lambda + \sin{c}\cos\lambda \\
    \sin\beta\sin\lambda
    \end{pmatrix}, \quad q(c) = \gamma_{o,0}(c)
\end{equation}
To find the expression of $\gamma_{q,\beta}$ in the $o$-chart $\mathcal{N}_o$ (faithful normal coordinates at $o$) we need to solve the equation:
\begin{equation}
    \gamma_{q(c),\beta}(\lambda) = p(l,\varphi) = \gamma_{o,\varphi}(l),
\end{equation}
for $\varphi$ and $l$ in terms of $c,\beta$ and $\lambda$.
\begin{align}\label{eq: fundamental solution K=1}
    &\begin{pmatrix}
    \cos{c}\,\cos\lambda - \cos\beta \sin{c}\,\sin\lambda \\
    \cos\beta\cos{c}\,\sin\lambda + \sin{c}\,\cos\lambda \\
    \sin\beta\sin\lambda
    \end{pmatrix} \overset{!}{=}
    \begin{pmatrix} \cos\lambda \\ \cos\varphi\sin\lambda \\ \sin\varphi\sin\lambda \end{pmatrix}
    \quad \Rightarrow \quad \mathcal{N}_o(\gamma_{q,\beta}(\lambda))
    = \binom{l(c,\beta,\lambda)}{\varphi(c,\beta,\lambda)}, \\
    &l(c,\beta,\lambda) = \cos^{-1}\left( \cos\lambda\cos{c} + \sin\lambda\sin{c}\,\cos(\pi-\beta) 
    \vphantom{\sqrt{2}}\right), \\
    &\varphi(c,\beta,\lambda) = \sin^{-1}\left( \frac{\sin{\lambda}\,\sin\beta}{\sqrt{1-\left( \cos{\lambda}\cos{c} + \sin{\lambda}\sin{c}\,\cos(\pi-\beta)  \right)^2}} \right)
\end{align}

\begin{figure}[h!]
    \centering
    \includegraphics[width=\linewidth]{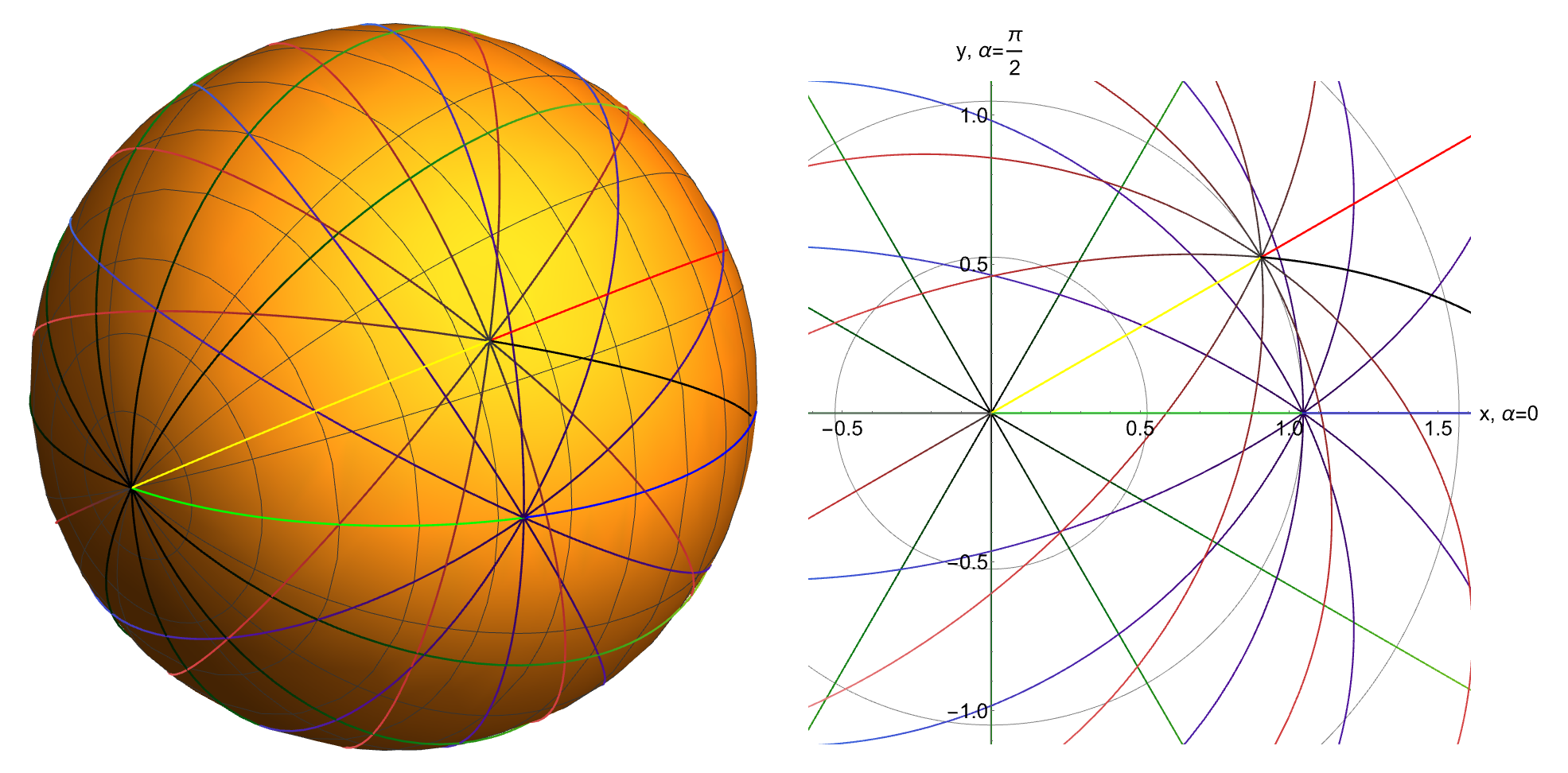}
    \caption{
    We plot a collection of curves from the geodesic flow bundle on $\mathcal{S}^2$ embedded in its ambient space $\mathbb{R}^3$ on the left and in the faithful normal chart at $o$ on the right. The green thick line is the original geodesic $\gamma_{o,0}$ in this example and is continued by the blue thick one $\gamma_{q,0}$ which is thus also labelled with $0$. The angle between the original geodesic and the yellow thick line at $o$ is $\alpha = \frac{\pi}{6}$ and thus the yellow thick line $\gamma_{o,\frac{\pi}{6}}$ and its continuation, the red thick line $\gamma_{o,\frac{\pi}{6}}$ have that angular label. The black thick line $\gamma_{p,0}$ is the one with the label $0$ and thus the angle between it and the red thick one is $\frac{\pi}{6}$ as well.
    }
    \label{Supp-fig: geodesic flow bundle on S^2}
\end{figure}

Due to the symmetry of the sphere it is straight forward to generalize the flow from a point $q$ on the original geodesic to the one from a generic point $p\in\mathcal{M}$:
\begin{equation}
    p = \gamma_{o,\varphi_p}(l_p); \quad \mathcal{N}_o(\gamma_{p,\beta}(\lambda)) = \binom{l(p,\beta,\lambda)}{\varphi(p,\beta,\lambda)}
    = \binom{b(l_p,\beta,\lambda)}{\varphi_p + \alpha(l_p,\beta,\lambda)}
\end{equation}

This geodesic flow describes a flow from a generic point $p$ which is represented, by the red gradient lines in Fig.~\ref{Supp-fig: geodesic flow bundle on S^2}. The functions $b$ and $\alpha$ coincide with $l$ and $\varphi$ in Eq.~\eqref{eq: fundamental solution K=1}. We see, that in $2$ dimensions the only change in the coordinate functions $l$ and $\varphi$ when we move to a generic point is the constant shift $\varphi_p$, which rotates the flow from $q$ to $p$.\\

We can then calculate the metric from the derivatives of $l$ and $\varphi$ with respect to arclength $\lambda$ at $\lambda=0$ of the geodesic heading out in increasing $\varphi$-direction at $p$, i.e. $\beta=\frac{\pi}{2}$.
\begin{equation}
    \dot{l}\left(l_p,\frac{\pi}{2},0\right) = 0, \quad
    \dot{\varphi}\left(l_p,\frac{\pi }{2},0\right) = \frac{1}{\sin{b}} \quad
    \Rightarrow \quad g_{\varphi\varphi} = \frac{1-\dot{l}^2(l_p,\frac{\pi}{2};0)}{\dot{\varphi}^2(l_p,\frac{\pi}{2};0)} = \sin^2l_p
\end{equation}
\begin{equation}
    \Rightarrow \quad g(p) = \begin{pmatrix} 1 & 0 \\ 0 & \sin^2l_p \end{pmatrix}
\end{equation}

\subsection{Example on the 3-sphere}\label{subsec: Ex. 3-sphere}
To get the flow from $q$ in faithful normal coordinates we use the $2$-dimensional case and tilt the triangle as described in the main text.
\begin{equation}
    \gamma_{q,\hat{\Omega}_\beta}(\lambda) = l(l_p,\beta_1;\lambda)\hat{\Omega}(\varphi(l_p,\beta_1;\lambda),\vartheta)
    = \gamma_{o,\hat{\Omega}}(l),
    \quad \vartheta = \beta_2, \text{ for } K(p) = 1, \quad \mathcal{N}_o(\gamma_{q,\hat{\Omega}_\beta}(\lambda)) = \begin{pmatrix}
    b(l_p,\beta_1;\lambda) \\ \alpha(l_p,\beta_1;\lambda) \\ \beta_2
    \end{pmatrix}
\end{equation}
We now rotate the entire triangle to another point $p = (l_p,\varphi_p,\vartheta_p)_\mathcal{N}$ in a way, which is consistent with the angular coordinates. To find the normal coordinates of $\gamma_{p,\hat{\Omega}_\beta}(\lambda)$ we need to solve the following equation.
\begin{equation}
    \gamma_{p,\hat{\Omega}_\beta}(\lambda) = R_{x_2,x_3}(\vartheta_p)\circ R_{x_1,x_2}(\varphi_p)\cdot\gamma_{q,\hat{\Omega}_\beta}(\lambda) \overset{!}{=} l\,\hat{\Omega}(\varphi,\vartheta)
    = \gamma_{o,\hat{\Omega}}(l)
\end{equation}
We see, that this is only a problems of angles at $o$ and thus will be generally valid, independent of the curvature involved. The dependence on the curvature field lies in the functions of the fundamental solution $b$  and $\alpha$. The solution to the angular problem is:
\begin{align}
    &\varphi(\hat{\Omega}_p,\hat{\Omega}_\alpha) = \tan^{-1}\left( \frac{\sqrt{(1-\text{Cosa}_2)(1+\text{Cosa}_2)}}{\text{Cosa}_2} \right), \quad
    \vartheta(\hat{\Omega}_p,\hat{\Omega}_\alpha)= \tan^{-1}\left( 
    \frac{ \sin\vartheta_p\widetilde{\text{Cosa}}_2 + \cos\vartheta_p\sin\alpha_1\sin\alpha_2 }
    { \cos\vartheta_p\widetilde{\text{Cosa}}_2 - \sin\vartheta_p\sin\alpha_1\sin\alpha_2 } \right), \notag \\
    &\hat{\Omega}_p = \hat{\Omega}(\varphi_p,\vartheta_p), \quad \hat{\Omega}_\alpha = \hat{\Omega}(\alpha_1,\alpha_2), \qquad \alpha_1 = \alpha(c,\beta_1,\lambda), \quad \alpha_2 = \beta_2 \\
    &\text{Cosa}_2 \coloneq \cos\alpha_1\cos\varphi_p - \sin\alpha_1\sin\varphi_p\cos\alpha_2, \quad
    \widetilde{\text{Cosa}}_2 \coloneq \cos\alpha_1\sin\varphi_p + \sin\alpha_1\cos\varphi_p\cos\alpha_2
\end{align}

To calculate the metric we need the geodesics $\gamma_{p,\hat{\varphi}}(\lambda)$ and $\gamma_{p,\hat{\vartheta}}(\lambda)$, which head out in $\varphi$-, and $\vartheta$-direction respectively from $p$ at every point $p\in\mathcal{M}$. To keep the map from directions to $\mathbb{S}^2$ consistent across the manifold i.e. satisfy condition~\eqref{eq: consistency}, we need to rotate directions at $p$ back to $q$, to determine the parameters $\beta_1$ and $\beta_2$:
\begin{equation}
    \hat{\Omega}(\beta_1,\beta_2) = R_{x_1,x_2}^{-1}(\varphi_p)
    \circ R_{x_2,x_3}^{-1}(\vartheta_p)\cdot\hat{\Omega}_\beta
\end{equation}
The parameter $\beta_1$ is directly related to the angle of the triangle at $q$ via $\beta = \pi-\beta_1$ and $\beta_2$ describes the tilt of the triangle with respect to the rotated original plane, rotated from $q$ to $p$.

To calculate the metric we need geodesics, which head out in $\varphi$- and $\vartheta$-direction from any point $p\in\mathcal{M}$. Increasing one angular variable at a point in the faithful normal chart creates a flow of coordinate lines:
\begin{equation}
    \mathcal{N}_o(p) = \mathcal{N}(\gamma_{o,\hat{\Omega}}(l)) = \mathcal{N}(l\,\hat{\Omega}(\varphi,\vartheta)) = (l,\varphi,\vartheta)_\mathcal{N}
\end{equation}
\begin{align}
    \phi_\mathcal{N}^\varphi(p) = l_p\,\hat{\Omega}(\varphi_p+\varphi,\vartheta_p), \quad
    \phi_\mathcal{N}^\vartheta(p) = l_p\,\hat{\Omega}(\varphi_p,\vartheta_p+\vartheta)
\end{align}
We calculate the tangent vector field of coordinate line flows:
\begin{align}
    \partial_\varphi(p) &= \dot{\phi}_\mathcal{N}^\varphi(p) = l_p\frac{\partial\hat{\Omega}}{\partial\varphi}(\varphi_p,\vartheta_p) = l_p\underbrace{\begin{pmatrix}
    -\sin\varphi_p \\ \cos\vartheta_p\cos\varphi_p \\ \sin\vartheta_p\cos\varphi_p
    \end{pmatrix}}_{\coloneq\,\hat{\varphi}(p)}, \\
    \partial_\vartheta(p) &= \dot{\phi}_\mathcal{N}^\vartheta(p) = l_p\frac{\partial\hat{\Omega}}{\partial\vartheta}(\varphi_p,\vartheta_p) = l_p\underbrace{\begin{pmatrix}
    0 \\ -\sin\vartheta_p\sin\varphi_p \\ \cos\vartheta_p\sin\varphi_p
    \end{pmatrix}}_{\coloneq\,\vec{\vartheta}(p)}
\end{align}
The two tangent vector fields are orthogonal to each other at every point $\hat{\varphi}(p)\cdot\vec{\vartheta}(p) = 0,$ $\forall p\in\mathcal{M}$. The vector field $\hat{\varphi}$ is a unit vector field, but $\vec{\vartheta}$ is not. To obtain an element of $\mathbb{S}^2$, which we use to parametrize the geodesic flow with, we need to normalize $\vec{\vartheta}$. Especially because it also vanishes at $q$ where we apply our solution.
\begin{equation}
    \hat{\vartheta}(p) = \frac{1}{\sin\varphi_p}\vec{\vartheta}(p)
\end{equation}
We calculate the $\beta_1$ and $\beta_2$ parameters of $\hat{\varphi}$, by solving:
\begin{align}
    \hat{\Omega}(\beta_1,\beta_2) \overset{!}{=} R_{x_1,x_2}^{-1}(\varphi_p)\circ R_{x_2,x_3}^{-1}(\vartheta_p)\cdot\hat{\varphi}(\varphi_p,\vartheta_p) \quad
    \Rightarrow \quad \beta_1 = \frac{\pi}{2}, \quad \beta_2 = 0
\end{align}
Then the components of $\gamma_{p,\hat{\varphi}}(\lambda)$ are given by:
\begin{align}
    \varphi\left( \hat{\Omega}_p,\hat{\Omega}\left( \alpha\left(c,\frac{\pi}{2};\lambda\right),0 \right)\right)
    &= \tan^{-1}\left( \frac{\sqrt{1-\cos^c\cos^2\lambda}}{\sqrt{1+\sin^2c\cot^2\lambda}} \frac{ \sqrt{\cos^2\varphi_0 + \sin{c}\cot{\lambda}(\sin(2\varphi_p) + \sin{c}\sin^2\varphi_p\cot{\lambda})} }{\cos\varphi_p\cos\lambda\sin{c} - \sin\varphi_p\sin\lambda} \right) \notag \\
    \vartheta\left( \hat{\Omega}_p,\hat{\Omega}\left( \alpha\left(c,\frac{\pi}{2};\lambda\right),0 \right)\right)
    &= \vartheta_p, \quad l\left(c,\frac{\pi}{2};\lambda\right) = \tan^{-1}\left( \frac{\sqrt{1-\cos^2c\cos^2\lambda}}{\cos{c}\cos{\lambda}} \right)
\end{align}
To calculate the metric component we need the tangent vector at $p$ in faithful normal coordinates, i.e. the derivatives of the components. It is trivial to confirm, that this tangent vector indeed has a vanishing $\theta$-component, since it is constant.
\begin{align}
    \dot{l}(p,\hat{\varphi};0) = \left.d_{\lambda} l\left(c,\frac{\pi}{2};\lambda\right)\right|_{\lambda=0} = 0, \quad
    \dot{\varphi}(p,\hat{\varphi};\lambda)
    = \left.d_{\lambda} \varphi\left( \hat{\Omega}_p,\hat{\Omega}\left( \alpha\left(c,\frac{\pi}{2};\lambda\right),0 \right)\right) \right|_{\lambda=0} = \frac{1}{\sin{c}}
\end{align}
And thus using our formula for the metric components Eq.~\eqref{eq: metric} we get:
\begin{equation}
    g_{\varphi\varphi}(p) = \frac{1-\dot{l}^2(p,\hat{\varphi};0)}{\dot{\varphi}^2(p,\hat{\varphi};0)} = \sin^2c
\end{equation}

Repeating the process for $\gamma_{p,\hat{\vartheta}}(\lambda)$ we get:
\begin{align}
    &\hat{\Omega}(\beta_1,\beta_2) \overset{!}{=} R_{x_1,x_2}^{-1}(\varphi_p)\circ R_{x_2,x_3}^{-1}(\vartheta_p)\cdot\hat{\vartheta} \quad
    \Rightarrow \quad \beta_1 = \frac{\pi}{2}, \quad \beta_2 = \frac{\pi}{2}, \\
    &\varphi\left( \hat{\Omega}_p,\hat{\Omega}\left( \alpha\left(c,\frac{\pi}{2};\lambda\right), \frac{\pi}{2} \right) \right) = \tan^{-1}\left( \frac{\sqrt{1-\cos^2c\cos^2\lambda}}{\sin{c}\cos\varphi_p\cos\lambda} \sqrt{\frac{1+\sin^2c\sin^2\varphi_p\cot^2\lambda}{1+\sin^2c\cot^2\lambda}} \right), \\
    &\vartheta\left( \hat{\Omega}_p,\hat{\Omega}\left( \alpha\left(c,\frac{\pi}{2};\lambda\right), \frac{\pi}{2} \right) \right) = \tan^{-1}\left( \frac{\sin\vartheta_p\sin{c}\sin\varphi_p\cot\lambda + \cos\vartheta_p}{\cos\vartheta_p\sin{c}\sin\varphi_p\cot\lambda - \sin\vartheta_p} \right)
\end{align}
And finally the components of the tangent vector $\dot{\gamma}_{p,\hat{\vartheta}}(0)$ at $p$ are:
\begin{align}
    \dot{l}(p,\hat{\vartheta};0) = 0, \quad \dot{\varphi}(p,\hat{\vartheta};0) = 0, \quad \dot{\vartheta}(p,\hat{\vartheta};0) = \frac{1}{\sin{c}\sin\varphi_p}
\end{align}
which leads to the metric component:
\begin{equation}
    g_{\vartheta\vartheta}(p) = \frac{1-\dot{l}^2(p,\hat{\vartheta};0)}{\dot{\vartheta}^2(p,\hat{\vartheta};0)} = \sin^2c\sin^2\varphi_p
\end{equation}

Now, that we have all components we can write down the metric in faithful normal coordinates and calculate the curvature tensors, to check whether our result for this example makes sense.
\begin{equation}
    g(c,\varphi,\vartheta) = \begin{pmatrix}
    1 & 0 & 0 \\ 0 & \sin^2c & 0 \\ 0 & 0 & \sin^2c\,\sin^2\varphi_p
    \end{pmatrix}
\end{equation}

The non-vanishing Christoffel symbols are:
\begin{align}
    \Gamma^i_{kl} = \frac{1}{2}g^{ij}( g_{jk,l} + g_{jl,k} - g_{kl,j} ); \qquad
    \Gamma^c_{\varphi\varphi} &= -\cos{c}\sin{c}, &\quad \Gamma^c_{\vartheta\vartheta} &= -\cos{c}\sin{c}\sin^2\varphi, &\quad
    \Gamma^\varphi_{\vartheta\vartheta} &= -\cos\varphi\sin\varphi \notag \\
    \Gamma^\varphi_{c\varphi} &= \cot{c}, &\quad \Gamma^\vartheta_{c\vartheta} &= \cot{c}, &\quad \Gamma^\vartheta_{\varphi\vartheta} &= \cot{\varphi}
\end{align}

Using these we can calculate the components of the Riemann tensor. Again, we list the non-vanishing ones:
\begin{align}
& & R^i_{jkl} &= -\Gamma^i{ik,l} + \Gamma^i_{jl,k} - \Gamma^a_{jk}\Gamma^i_{al} + \Gamma^a_{jl}\Gamma^i_{ak}; \\
R^c_{\varphi c\varphi} &= -R^c_{\varphi\varphi c} = \sin^2c, &\quad R^c_{\vartheta c\vartheta} &= -R^c{\vartheta\vartheta c} = \sin^2c\sin^2\varphi, &\quad R^\varphi_{c\varphi c} &= -R^\varphi_{cc\varphi} = 1, \\
R^\varphi_{\vartheta\varphi\vartheta} &= -R^\varphi_{\vartheta\vartheta\varphi} = \sin^2c\sin^2\varphi &\quad 
R^\vartheta_{c\vartheta c} &= -R^\vartheta_{cc\vartheta} = 1, &\quad
R^\vartheta_{\varphi\vartheta\varphi} &= -R^\vartheta_{\varphi\varphi\vartheta} = \sin^2c
\end{align}

We contract the Riemann tensor to obtain the Ricci tensor and then contract again to get the Riemann curvature scalar:
\begin{align}
    R_{jl} = R^i_{jil} = 2\begin{pmatrix}
    1 & 0 & 0 \\ 0 & \sin^2c & 0 \\ 0 & 0 & \sin^2c\sin^2\varphi
    \end{pmatrix}, \qquad R = R^i_i = g^{ij}R_{ij} = 6
\end{align}

And we get the expected result for the unit $3$-sphere. For comparison the general formula for the $n$-sphere with radius $r$ is:
\begin{equation}
    R\left(\mathcal{S}^n(r)\right) = \frac{n(n-1)}{r^2}
\end{equation}
\begin{figure}[h!]
\begin{minipage}{.5\linewidth}
    \centering\includegraphics[width=\linewidth]{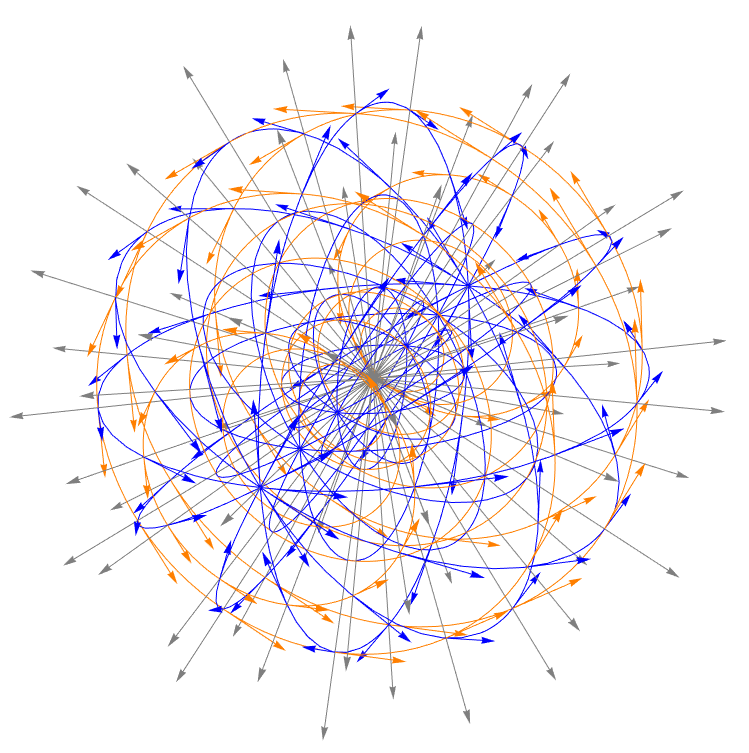}
\end{minipage}
\begin{minipage}{.5\linewidth}
    \centering\includegraphics[width=\linewidth]{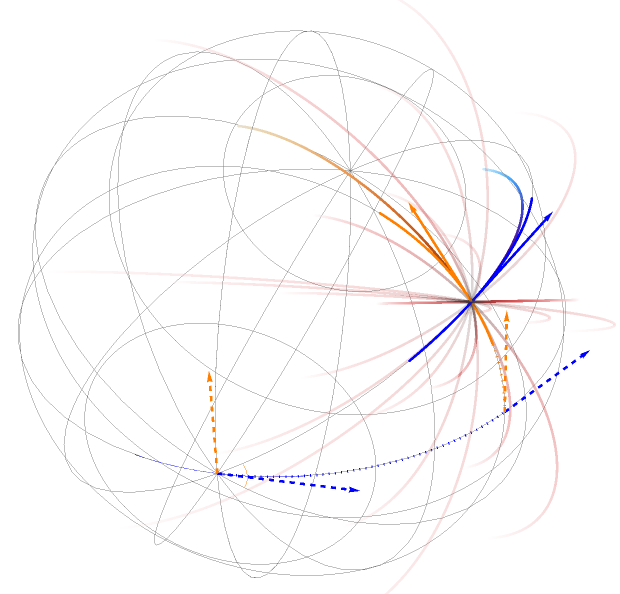}
\end{minipage}
    \caption{We show a sample of the basis vectors we use to calculate the metric on the left and visuallize how we rotate the vectors at $p$ to $q$, where the angular parameters $\beta_1$ and $\beta_2$ are defined.
    The $\varphi$-coordinate line at $p$ is drawn in orange and the $\vartheta$ one in blue. The geodesics heading out in the same direction as those two are drawn in orange and blue gradients respectively. Their tangent vectors at $p$ are the solid orange and blue arrows. A sample of rest of the flow from $p$ is plotted in translucent red gradients. The rotated vectors are dashed.}
    \label{fig: FNC basis vector fields}
\end{figure}

\subsection{Constant curvature on each central plane}\label{subsec: NotASphere}
From the main text we know, that if we want to pick a curvature field we can pick an arbitrary point w.l.o.g. $o$ and then choose a curvature field on every geodesic surface through $o$ in a smooth way. We again restrict ourselves to geometries where the geodesic surfaces appear as planes in faithful normal coordinates. We choose the following curvature field, where the curvature on each such geodesic plane is constant:
\begin{equation}
    K_p(\varphi_K,\vartheta_K) = \cos^2\varphi_K\cos^2(2\vartheta_K), \quad \forall p\in\mathcal{M}
\end{equation}
Although these geodesic surfaces, labelled by $\varphi_K$ and $\vartheta_K$, appear as planes in the faithful normal coordinates at $o$, they are $2$-spheres of different radii.\\
In this choice of curvature field, the curvature varies in different directions in the same way at each point and is thus not position but purely direction dependent.\\
To compare these curvature degrees of freedom with the Riemann tensor we try to calculate the metric. Since the curvature field is the same at each point it suffices to look at the point $q(c) = \gamma_{o,\hat{0}}(c)$. In the main text we saw that we need the geodesics heading out in the directions of increasing $\varphi$ and $\vartheta$ at each point. We determined already in Sec.~\ref{subsec: Ex. 3-sphere} that the direction labels for $\hat{\varphi}$ are $\beta_1 = \frac{\pi}{2},\, \beta_2 = 0$ and the ones for $\hat{\vartheta}$ are $\beta_1 = \frac{\pi}{2},\, \beta_2 = \frac{\pi}{2}$. But the geodesics heading out in these directions agree with the unit sphere since the curvature is given by:
\begin{align}
    &\varphi_p = \vartheta_p = 0, \quad \alpha_2 = \beta_2 \\
    \Rightarrow \quad &\varphi_K(0,0,0) = 0, &\quad &\vartheta_K(0,0,0) = 0, &\quad &K_p\left(0,0\right) = 1, &\text{in} \ \hat{\varphi} \text{-direction}, \\
    \Rightarrow \quad &\varphi_K\left(0,0,\frac{\pi}{2}\right) = 0, &\quad &\vartheta_K\left(0,0,\frac{\pi}{2}\right) = \frac{\pi}{2}, &\quad &K_p\left(0,\frac{\pi}{2}\right) = 1, &\text{in} \ \hat{\vartheta} \text{-direction}.
\end{align}
using the functions $\varphi_K(\varphi_p,\vartheta_p,\alpha_2)$ and $\vartheta_K(\varphi_p,\vartheta_p,\alpha_2)$ we derived in the main text.\\

Thus we find, that the geodesic planes labeled by $(0,0)$ and $\left(0,\frac{\pi}{2}\right)$ have constant curvature $1$, thus the geodesics in these surfaces are the ones of the unit $2$-sphere and when we use their derivatives to calculate the metric components we get the metric of the unit sphere.\\

But if we rotate the basis by $\frac{\pi}{4}$ we get instead:
\begin{align}
    &\varphi_K\left(0,0,\frac{\pi}{4}\right) = 0, \quad \vartheta_K\left(0,0,\frac{\pi}{4}\right) = \frac{\pi}{4}, \quad K_p\left(0,\frac{\pi}{4}\right) = 0, &\text{in} \ \hat{\varphi} \text{-direction}, \\
    &\varphi_K\left(0,0,-\frac{\pi}{4}\right) = 0, \quad \vartheta_K\left(0,0,-\frac{\pi}{4}\right) = -\frac{\pi}{4}, \quad K_p\left(0,-\frac{\pi}{4}\right) = 0, &\text{in} \ \hat{\vartheta} \text{-direction},
\end{align}
keeping in mind, that the domain of $\alpha_2$ is between $-\frac{\pi}{2}$ and $\frac{\pi}{2}$ and thus we have $\alpha_2=-\frac{\pi}{4}$ for the $\hat{\vartheta}$-direction instead of $\frac{3\pi}{4}$, which both label tha same plane.\\

With a choice of basis that lies in these two planes we get the flat metric, because these planes are indeed the only real planes through $o$ in this geometry.\\
We conclude, that the metric and thus also the Riemann tensor is not able to pick up this degree of freedom.\\

In Fig.~\ref{fig: Not a sphere} we see that the flow of this geometry agrees with the sphere in the directions $0,\,\frac{\pi}{2},\,\pi,\,\frac{3\pi}{2}$ but in the other directions the blue geodesics do not cover the grey ones of the sphere since the curvature in those directions is lower and thus the geodesics are bent less.
\begin{figure}[h!]
    \begin{minipage}{.5\linewidth}
        \centering\includegraphics[width=\linewidth]{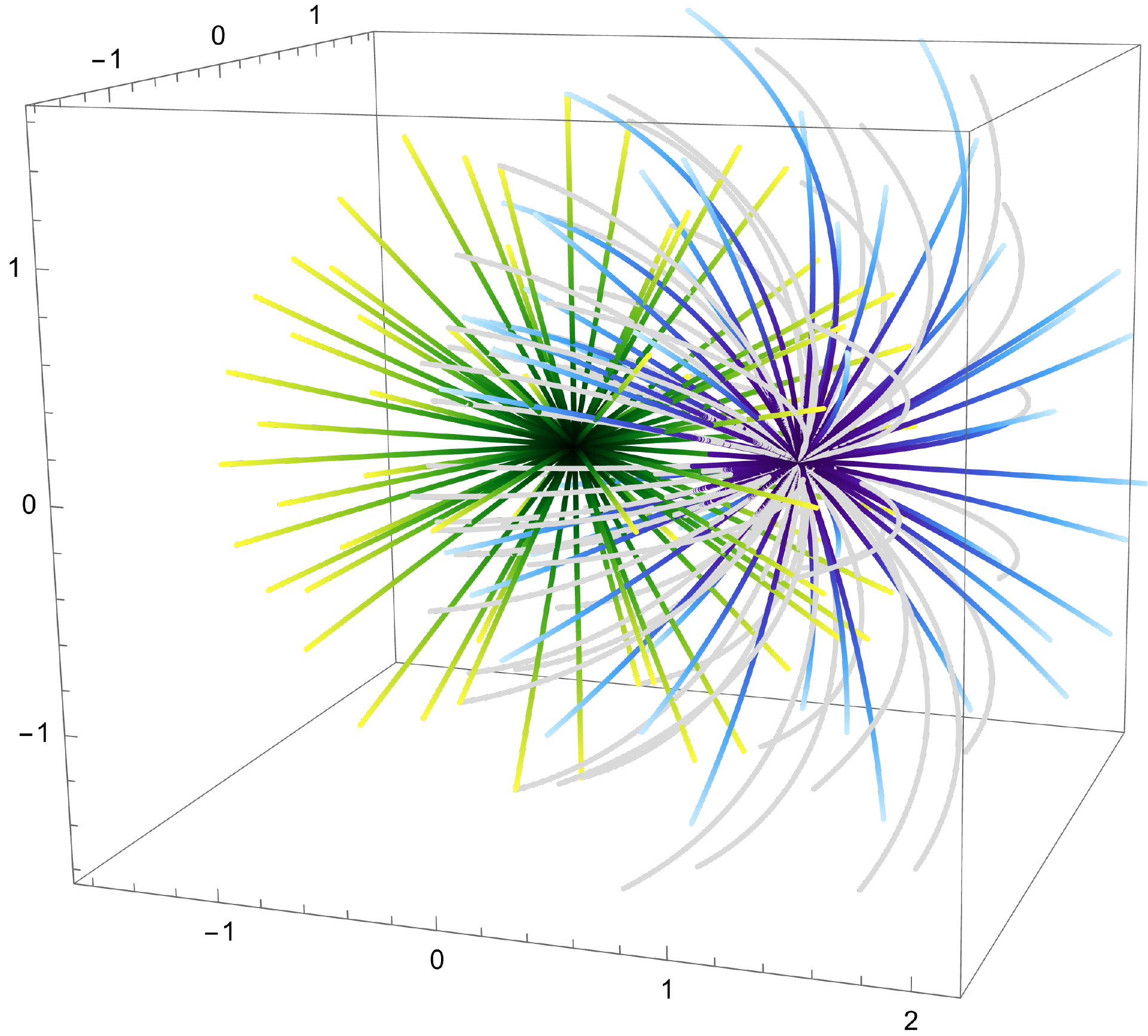}
    \end{minipage}
    \begin{minipage}{.5\linewidth}
        \centering\includegraphics[width=0.9\linewidth]{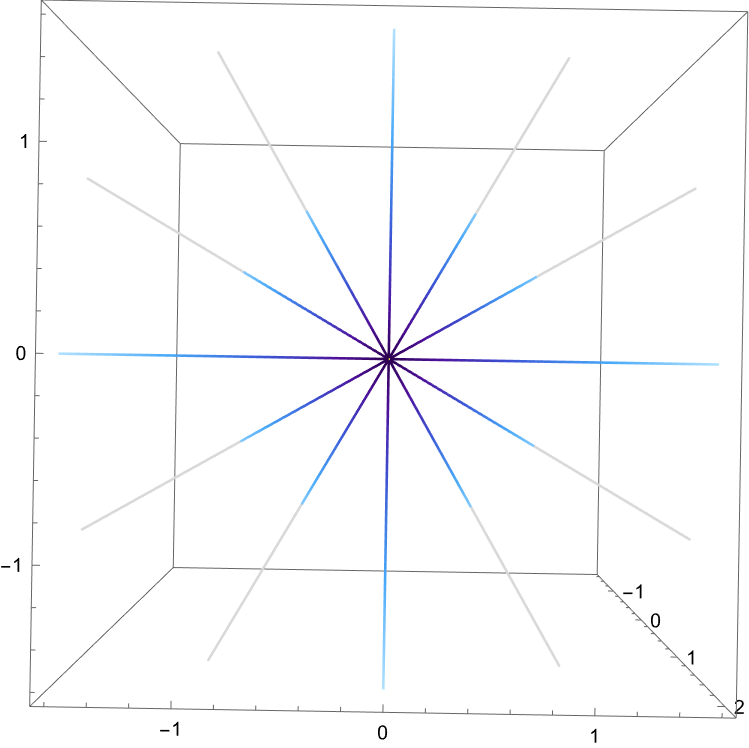}
    \end{minipage}
    \caption{A 3D plot of the geodesic flow from the origin point in green and from the point $q = (\frac{\pi}{3},0,0)$ in blue, with the flow of the $3$-sphere in gray for comparison. A side view from the right hand side is provided on the right.}
    \label{fig: Not a sphere}
\end{figure}

\subsection{$n$-dimensional flat space}
To deal with the $n$-dimensional case we first need the tangent vector fields $\hat{\vartheta}_i$, which we gain by taking derivatives of $\hat{\Omega}^(n-1)$:
\begin{equation}
    \hat{\Omega}^{(n-1)} = \begin{pmatrix} \cos\varphi \\
    \cos\vartheta_{n-2}\cos\vartheta_{n-3}\ldots\cos\vartheta_1\sin\varphi \\
    \sin\vartheta_1\sin\varphi \\ \sin\vartheta_2\cos\vartheta_1\sin\varphi \\
    \sin\vartheta_3\cos\vartheta_2\cos\vartheta_1\sin\varphi \\ \vdots \\
    \sin\vartheta_{n-2}\cos\vartheta_{n-3}\ldots\cos\vartheta_1\sin\varphi
    \end{pmatrix}, \quad
    \hat{\varphi}(p) = \hat{\Omega}^{(n-1)}_{,\varphi} = \begin{pmatrix} -\sin\varphi \\
    \cos\vartheta_{n-2}\cos\vartheta_{n-3}\ldots\cos\vartheta_1\cos\varphi \\
    \sin\vartheta_1\cos\varphi \\ \sin\vartheta_2\cos\vartheta_1\cos\varphi \\
    \sin\vartheta_3\cos\vartheta_2\cos\vartheta_1\cos\varphi \\ \vdots \\
    \sin\vartheta_{n-2}\cos\vartheta_{n-3}\ldots\cos\vartheta_1\cos\varphi
    \end{pmatrix}
\end{equation}
\begin{align}
    \vec{\vartheta}_i(p) = \hat{\Omega}^{(n-1)}_{,\vartheta_1} = \begin{pmatrix} 0 \\
    -\cos\vartheta_{n-2}\cos\vartheta_{n-3}\ldots\cos\vartheta_2\sin\vartheta_1\sin\varphi \\
    \cos\vartheta_1\sin\varphi \\ -\sin\vartheta_2\sin\vartheta_1\sin\varphi \\
    -\sin\vartheta_3\cos\vartheta_2\sin\vartheta_1\sin\varphi \\ \vdots \\
    -\sin\vartheta_{n-2}\cos\vartheta_{n-3}\ldots\cos\vartheta_2\sin\vartheta_1\sin\varphi
    \end{pmatrix}, \quad
    \vec{\vartheta}_i(p) = \hat{\Omega}^{(n-1)}_{,\vartheta_i} = \begin{pmatrix} 0 \\
    -\cos\vartheta_{n-2}\ldots\sin\vartheta_i\ldots\cos\vartheta_1\sin\varphi \\
    \sin\vartheta_1\sin\varphi \\ \vdots \\
    \cos\vartheta_i\ldots\cos\vartheta_1\sin\varphi \\ \vdots \\
    -\sin\vartheta_{n-2}\ldots\sin\vartheta_i\ldots\cos\vartheta_1\sin\varphi
    \end{pmatrix}
\end{align}
and then normalize them:
\begin{align}
    \hat{\varphi}\cdot\hat{\varphi} =& \sin^2\varphi + \cos^2\vartheta_{n-2}\cos^2\vartheta_{n-3}\ldots\cos^2\vartheta_1\cos^2\varphi + \sin^2\vartheta_1\cos^2\varphi
    + \sin^2\vartheta_2\cos^2\vartheta_1\cos\varphi \notag \\
    &+ \sin^2\vartheta_3\cos^2\vartheta_2\cos^2\vartheta_1\cos^2\varphi + \ldots
    +\sin^2\vartheta_{n-2}\cos\vartheta_{n-3}\ldots\cos^2\vartheta_1\cos^2\varphi \notag \\
    =& \sin^2\varphi + (\cos^2\vartheta_{n-2} + \sin^2\vartheta_{n-2})\cos^2\vartheta_{n-3}\ldots\cos^2\vartheta_1\cos^2\varphi + \sin^2\vartheta_1\cos^2\varphi
    + \sin^2\vartheta_2\cos^2\vartheta_1\cos\varphi \notag \\
    &+ \sin^2\vartheta_3\cos^2\vartheta_2\cos^2\vartheta_1\cos^2\varphi + \ldots
    + \sin^2\vartheta_{n-3}\cos^2\vartheta_{n-4}\ldots\cos^2\vartheta_1\cos^2\varphi \notag \\
    =& \sin^2\varphi + (\cos^2\vartheta_{n-3} + \sin^2\vartheta_{n-3})\cos^2\vartheta_{n-4}\ldots\cos^2\vartheta_1\cos^2\varphi + \sin^2\vartheta_1\cos^2\varphi
    + \sin^2\vartheta_2\cos^2\vartheta_1\cos\varphi \notag \\
    &+ \sin^2\vartheta_3\cos^2\vartheta_2\cos^2\vartheta_1\cos^2\varphi + \ldots
    + \sin^2\vartheta_{n-4}\cos^2\vartheta_{n-5}\ldots\cos^2\vartheta_1\cos^2\varphi \notag \\
    \vdots& \notag \\
    =& \sin^2\varphi + (\cos^2\vartheta_3 + \sin^2\vartheta_3)\cos^2\vartheta_2\cos^2\vartheta_1\cos^2\varphi + \sin^2\vartheta_1\cos^2\varphi + \sin^2\vartheta_2\cos^2\vartheta_1\cos\varphi \notag \\
    =& \sin^2\varphi + (\cos^2\vartheta_2 + \sin^2\vartheta_2)\cos^2\vartheta_1\cos^2\varphi + \sin^2\vartheta_1\cos^2\varphi \notag \\
    =& \sin^2\varphi + (\cos^2\vartheta_1 + \sin^2\vartheta_1)\cos^2\varphi
    = \sin^2\varphi + \cos^2\varphi = 1,
\end{align}
\begin{align}
    \vec{\vartheta}_1\cdot\vec{\vartheta}_1 =& 0 + \cos^2\vartheta_{n-2}\cos^2\vartheta_{n-3}\ldots\cos^2\vartheta_2\sin^2\vartheta_1\sin^2\varphi
    + \cos^2\vartheta_1\sin^2\varphi + \sin^2\vartheta_2\sin^2\vartheta_1\sin^2\varphi \notag \\
    &+ \sin^2\vartheta_3\cos^2\vartheta_2\sin^2\vartheta_1\sin^2\varphi + \ldots
    + \sin^2\vartheta_{n-2}\cos^2\vartheta_{n-3}\ldots\cos^\vartheta_2\sin^2\vartheta_1\sin^2\varphi \notag \\
    =& (\cos^2\vartheta_{n-2} + \sin^2\vartheta_{n-2})\cos^2\vartheta_{n-3}\ldots\cos^2\vartheta_2\sin^2\vartheta_1\sin^2\varphi
    + \cos^2\vartheta_1\sin^2\varphi + \sin^2\vartheta_2\sin^2\vartheta_1\sin^2\varphi \notag \\
    &+ \sin^2\vartheta_3\cos^2\vartheta_2\sin^2\vartheta_1\sin^2\varphi + \ldots
    + \sin^2\vartheta_{n-3}\ldots\cos^\vartheta_2\sin^2\vartheta_1\sin^2\varphi \notag \\
    \vdots \notag \\
    =& (\cos^2\vartheta_3 + \sin^2\vartheta_3)\cos^2\vartheta_2\sin^2\vartheta_1\sin^2\varphi
    + \cos^2\vartheta_1\sin^2\varphi + \sin^2\vartheta_2\sin^2\vartheta_1\sin^2\varphi \notag \\
    =& (\cos^2\vartheta_2 + \sin^2\vartheta_2)\sin^2\vartheta_1\sin^2\varphi
    + \cos^2\vartheta_1\sin^2\varphi = (\sin^2\vartheta_1 + \cos^2\vartheta_1)\sin^2\varphi = \sin^2\varphi \\
    \Rightarrow \quad \hat{\vartheta}_1 =& \frac{1}{\sin\varphi}\vec{\vartheta}_1,
\end{align}
\begin{align}
    \vec{\vartheta}_i\cdot\vec{\vartheta}_i =& \cos^2\vartheta_{n-2}\ldots\sin^2\vartheta_i\ldots\cos^2\vartheta_1\sin^2\varphi + \sin^2\vartheta_1\sin^2\varphi
    + \ldots + \cos^2\vartheta_i\ldots\cos^2\vartheta_1\sin^2\varphi + \ldots \notag \\
    &+ \sin^2\vartheta_{n-2}\ldots\sin^2\vartheta_i\ldots\cos^2\vartheta_1\sin^2\varphi \notag \\
    =& (\cos^2\vartheta_{n-2} + \sin^2\vartheta_{n-2})\cos^2\vartheta_{n-3}\ldots\sin^2\vartheta_i\ldots\cos^2\vartheta_1\sin^2\varphi + \sin^2\vartheta_1\sin^2\varphi
    + \ldots + \cos^2\vartheta_i\ldots\cos^2\vartheta_1\sin^2\varphi + \ldots \notag \\
    &+ \sin^2\vartheta_{n-3}\ldots\sin^2\vartheta_i\ldots\cos^2\vartheta_1\sin^2\varphi \notag \\
    \vdots \notag \\
    =& (\cos^2\vartheta_{i+1} + \sin^2\vartheta_{i+1})\sin^2\vartheta_i\ldots\cos^2\vartheta_1\sin^2\varphi + \sin^2\vartheta_1\sin^2\varphi
    + \ldots + \cos^2\vartheta_i\ldots\cos^2\vartheta_1\sin^2\varphi \notag \\
    =& (\sin^2\vartheta_i + \cos^2\vartheta_i)\cos^2\vartheta_{i-1}\ldots\cos^2\vartheta_1\sin^2\varphi + \sin^2\vartheta_1\sin^2\varphi
    + \ldots + \sin^2\vartheta_{i-1}\ldots\cos^2\vartheta_1\sin^2\varphi \notag \\
    =& (\cos^2\vartheta_{i-1} + \sin^2\vartheta_{i-1})\cos^2\vartheta_{i-2}\ldots\cos^2\vartheta_1\sin^2\varphi + \sin^2\vartheta_1\sin^2\varphi
    + \ldots + \sin^2\vartheta_{i-2}\ldots\cos^2\vartheta_1\sin^2\varphi \notag \\
    \vdots \notag \\
    =& (\cos^2\vartheta_1 + \sin^2\vartheta_1)\sin^2\varphi = \sin^2\varphi \quad
    \Rightarrow \quad \hat{\vartheta}_i = \frac{1}{\sin\varphi}\vec{\vartheta}_i.
\end{align}

Next, we calculate the angular parameters $\beta_i$, by solving the following equation: 
\begin{align}
    \hat{\Omega}(\beta_1,\ldots,\beta_{n-1}) \overset{!}{=} R_{x_1,x_2}^{-1}(\varphi)\circ R_{x_2,x_3}^{-1}(\vartheta_1)\circ\ldots\circ R_{x_2,x_n}^{-1}(\vartheta_{n-2})\cdot\hat{\vartheta}_i = \left. \frac{1}{\sin\varphi}\vec{\vartheta}_i \right|_{\varphi=0,\vartheta_i=0} \\
    \begin{pmatrix} \cos\beta_1 \\
    \cos\beta_{n-1}\cos\beta_{n-2}\ldots\cos\beta_2\sin\beta_1 \\
    \sin\beta_2\sin\beta_1 \\ \sin\beta_3\cos\beta_2\sin\beta_1 \\
    \sin\beta_4\cos\beta_3\cos\beta_2\sin\beta_1 \\ \vdots \\
    \sin\beta_{n-1}\cos\beta_{n-2}\ldots\cos\beta_2\sin\beta_1
    \end{pmatrix} \overset{!}{=}
    \begin{pmatrix} 0 \\
    -\cos0\ldots\sin0\ldots\cos0 \\
    \sin0 \\ \vdots \\
    \cos0\ldots\cos0 \\ \vdots \\
    -\sin0\ldots\sin0\ldots\cos0
    \end{pmatrix} = \begin{pmatrix} 0 \\ 0 \\ 0 \\ \vdots \\ 1 \\ \vdots \\ 0 \end{pmatrix}
    \begin{matrix}  \\  \\  \\ \vphantom{\vdots} \\ \leftarrow i \\ \vphantom{\vdots} \\ \vphantom{0} \end{matrix} \\
    \cos\beta_1 \overset{!}{=} 0 \quad \Rightarrow \quad \beta_1 = \frac{\pi}{2}, \quad
    \sin\beta_2\sin\frac{\pi}{2} = 0 \quad \Rightarrow \quad \beta_2 = 0, \quad
    \sin\beta_3\cos0\sin\frac{\pi}{2} = 0 \quad \Rightarrow \quad \beta_3 = 0, \ldots \notag \\
    \sin\beta_{i+1}\cos0\ldots\cos0\sin\frac{\pi}{2} = 1 \quad \Rightarrow \quad \beta_{i+1} = \frac{\pi}{2}
\end{align}
All other equations contain a $\cos\beta_{i+1} = \cos\frac{\pi}{2} = 0$ and are thus trivially satisfied i.e. any choice will do. For simplicity we pick $\beta_j = 0$ for $j>i+1$.

Now we are ready to calculate the components of the coordinate line tangent vectors. The fundamental solution $b$ and $\alpha$ for a flat space is given by:\\
\begin{align}
    b(c,\beta_1;\lambda) = \sqrt{c^2 + \lambda^2 + 2c\lambda\cos\beta_1}, \quad
    \alpha(c,\beta_1;\lambda) = \tan^{-1}\left( \frac{\lambda\sin\beta_1}{c + \lambda\cos\beta_1} \right)
\end{align}
The distance function $l=b$ is the same in all directions and thus we can calculate it's derivative for all cases:
\begin{equation}
    \dot{l}\left(c,\frac{\pi}{2};0\right) = \left. \frac{\lambda + c\cos\beta_1}{\sqrt{c^2 + \lambda^2 + 2c\lambda\cos\beta_1}} \right|_{\beta_1=\frac{\pi}{2},\lambda=0} = 0
\end{equation}

Now we would need the functions $\varphi(\hat{\Omega}_p,\hat{\Omega}_\alpha)$ and $\vartheta_i(\hat{\Omega}_p,\hat{\Omega}_\alpha)$. To obtain these we would have to solve
\begin{equation}
    R_{x_2,x_n}^{-1}(\vartheta_{n-2,p})\circ\ldots\circ R_{x_2,x_3}^{-1}(\vartheta_{1,p})\circ R_{x_1,x_2}^{-1}(\varphi_p)\cdot\hat{\Omega}(\alpha_1,\ldots,\alpha_{n-1})
    \overset{!}{=} \hat{\Omega}(\varphi,\vartheta_1,\ldots,\vartheta_{n-2})
\end{equation}
in $n$ dimensions. This is beyond the scope of this work and we thus continue from here with $\mathbb{R}^3$, where we calculate $\varphi$ and $\vartheta$:
\begin{align}
    \dot{\varphi}(p,\hat{\varphi};\lambda)
    = \left.d_{\lambda} \varphi\left( \hat{\Omega}_p,\hat{\Omega}\left( \alpha\left(c,\frac{\pi}{2};\lambda\right),0 \right)\right) \right|_{\lambda=0} = \frac{1}{c}, \quad
    \dot{\vartheta}(p,\hat{\varphi};\lambda)
    = \left.d_{\lambda} \vartheta\left( \hat{\Omega}_p,\hat{\Omega}\left( \alpha\left(c,\frac{\pi}{2};\lambda\right),\frac{\pi}{2} \right)\right) \right|_{\lambda=0} = \frac{1}{c\sin\varphi_p}
\end{align}
We get the metric:
\begin{equation}
    g = \begin{pmatrix} 1 & 0 & 0 \\ 0 & c^2 & 0 \\ 0 & 0 & c^2\sin^2\varphi_p \end{pmatrix},
\end{equation}
the non-vanishing Chritoffel symbols are:
\begin{align}
    \Gamma^c_{\varphi\varphi} = -c, \quad \Gamma^c_{\vartheta\vartheta} = -c\sin^2\varphi, \quad
    \Gamma^\varphi_{c\varphi} = \frac{1}{c}, \quad \Gamma^\varphi_{\vartheta\vartheta} = -\cos\varphi\sin\varphi, \quad \Gamma^\vartheta_{c\vartheta} = \frac{1}{c}, \quad \Gamma^\vartheta_{\varphi\vartheta} = \cot\varphi
\end{align}
The Riemann tensor vanishes, as expected for flat space.

\section{Procedure}\label{sec: Procedure}
The basic idea is, to approximate the Gaussian curvature field with a step function, such that we divide the geodesic triangle into $N^2$ smaller geodesic triangles with constant curvature. We then use trigonometry on the sphere in each of the triangles to calculate the opening angle and top line up to second order and then take the limit $N\to\infty$ where the approximation becomes precise.

\begin{enumerate}
    \item Solve the first slice for finite N.
    \begin{enumerate}
        \item We solve the trigonometry problems within each segment.
        \begin{enumerate}
            \item Identify and solve the different types of trigonometry problems in a general manner creates template formulas.
            \item Solve the $0$-th order problem.
            \item Insert the $0$-th order results into the template formulas and calculate the $2$-nd order solution to each segment.
        \end{enumerate}
        \item Link up the segments of the first slice.
        \begin{enumerate}
            \item Quantities of a segment are dependent on the ones of the next segment from which we derive recursion rules.
            \item The recursion can be solved by a continued fraction.
        \end{enumerate}
        \item Backwards inserting.
        \begin{enumerate}
            \item Quantities of a segment depend on quantities of its successor. Starting with the last one we successively insert all dependencies.
            \item We solve the resulting recursion.
        \end{enumerate}
    \end{enumerate}
    \item Extend the solution of the first slice to the case of the j-th one.
    \begin{enumerate}
        \item Correction terms in the arguments generate additional terms in the template formulas.
        \item We generalize the substitutes we used to simplify our results.
        \item We update the results of the first slice and collect the additional terms in the argument correction coefficients.
    \end{enumerate}
    \item Link the slices together to form the full triangle.
    \begin{enumerate}
        \item Quantities of a slice depend on the ones of its predecessor, which leads to another recursion.
        \item The solution to this recursion contains sums and products which become series and infinite products when $N$ tends to infinity.
    \end{enumerate}
    \item Prepare the mathematical tools for the limit calculations.
    \begin{enumerate}
        \item We convert the asymptotic series to integrals.
        \item We introduce the notion of product integrals and derive an analogue to the fundamental theorem of calculus to calculate the infinite products.
        \item The nested sums lead to layered integrals. We derive a solution in terms of the primitive function to the repeatedly occurring integrand.
    \end{enumerate}
    \item Take the limit of $N$ to infinity.
    \begin{enumerate}
        \item The solution for finite $N$ contains many layers of substitutions. We successively determine and insert the asymptotic behaviour of the substitutes starting with the innermost.
        \item Identifying the power series of sine and cosine we calculate the asymptotic behaviour of the recursion substitutes.
        \item Putting everything together and comparing the asymptotic growth as $N$ tends to infinity, we finally calculate the limit.
    \end{enumerate}
\end{enumerate}

\section{The functions $C$ and $SC$}\label{sec: C, SC}
Most of the calculations discussed in the main text are repeated applications of the cosine- and sine-law for a constantly curved surface with radius $R$.
\begin{equation}
    \cos_Rc = \cos_Ra\cos_Rb + \frac{1}{R^2}\sin_Ra\sin_Rb\cos\gamma, \qquad \frac{\sin_Ra}{\sin\alpha} = \frac{\sin_Rb}{\sin\beta} = \frac{\sin_Rc}{\sin\gamma}, \quad
    \alpha,\beta,\gamma\notin\pi\mathbb{Z}
\end{equation}
If $R\in\mathbb{R}$ the surface is the $2$-sphere. The radius can also be purely imaginary $R\in i\mathbb{R}$, in which case the surface is a pseudo sphere. The curvature of the surface is the inverse square of the radius which leads to a negative real value, in case of the pseudo sphere and the curved versions of sine and cosine are defined by:\\
\begin{minipage}{0.08\linewidth}
    \begin{equation*}
        K = \frac{1}{R^2};
    \end{equation*}
\end{minipage}
\begin{minipage}{0.92\linewidth}
    \begin{align}\label{eq: cosR}
    \cos_Ra &\coloneq \cos\frac{a}{R} = \sum_{n=0}^\infty\frac{(-1)^n}{(2n)!}\left(\frac{a}{R}\right)^{2n} = \sum_{n=0}^\infty\frac{(-K)^n a^{2n}}{(2n)!}
    = \begin{cases} \cos\left(\sqrt{K}a\right), & K > 0 \\ \cosh\left(\sqrt{|K|}a\right), & K < 0 \end{cases} \eqcolon \cos_Ka, \\
    \sin_Ra &\coloneq R\sin\frac{a}{R} = R\sum_{n=0}^\infty\frac{(-1)^n}{(2n+1)!}\left(\frac{a}{R}\right)^{2n+1}
    = \sum_{n=0}^\infty\frac{(-K)^n a^{2n+1}}{(2n+1)!} = \begin{cases} \frac{\sin\left(\sqrt{K}a\right)}{\sqrt{K}}, & K > 0 \\
    \frac{\sinh\left(\sqrt{|K|}a\right)}{\sqrt{|K|}}, & K < 0 \end{cases} \eqcolon \sin_Ka \notag
    \end{align}
\end{minipage}
We observe, that the power series captures both the negative and the positive curvature case i.e. $K\in\mathbb{R}\backslash\{0\}$ in one expression.\\

To relate the curved cosine- and sine law to the flat space version, we need to use the series expansions and take the limit $R\to\infty$.
\begin{align}
    \lim_{R\to\infty}K &= 0 \quad \Rightarrow \quad R\to\infty \quad \Leftrightarrow \quad K\to0\\
    \cos_Kc &= 1 - \frac{c^2}{2}K + \mathcal{O}(K^2) = \left( 1 - \frac{a^2}{2}K + \mathcal{O}(K^2) \right)\left( 1 - \frac{b^2}{2}K + \mathcal{O}(K^2) \right)
    + K\left( a + \mathcal{O}(K) \right)\left( b + \mathcal{O}(K) \right)\cos\gamma \notag \\
    &= 1 - \frac{a^2}{2}K - \frac{b^2}{2}K + Ka\,b\cos\gamma + \mathcal{O}(K^2) = \cos_Ka\cos_Kb + K\sin_Ka\sin_Kb\cos\gamma \notag \\
    \Rightarrow \quad c^2 &= a^2 + b^2 - a\,b\cos\gamma, \quad K = 0
\end{align}
\begin{equation}
    \frac{\sin_Ka}{\sin\alpha} = \frac{a + \mathcal{O}(K)}{\sin\alpha} = \frac{b + \mathcal{O}(K)}{\sin\beta}
    = \frac{\sin_Kb}{\sin\beta} \quad
    \overset{K\to0}{\longrightarrow} \quad \frac{a}{\sin\alpha} = \frac{b}{\sin\beta}
\end{equation}
We see, that we can not take the limit of $K\to0$ on the curvature sine and cosine and then use them in the cosine law, because the relevant term is of order $K$. For the sine law this works: $\sin_Ka = a$.\\

In the case of constant curvature these laws suffice to solve the triangulation problem discussed in the previous section. Problem 1, for $N=1$ describes the triangle exactly. We use the cosine law, to calculate $b$, which in this case coincides with $d^1,1$:
\begin{equation}
    b(c,\beta,a) = \cos_K^{-1}\left( \cos_Ka\cos_Kc + K\sin_Ka\sin_Kc\cos(\pi-\beta) \vphantom{\sqrt{K}}\right),
    \quad K \neq 0
\end{equation}
We can then use $b$ to calculate $\alpha$, which coincides with $\alpha^{0,1}$, using the sine law.
\begin{equation}
    \alpha(c,\beta,a) = \sin^{-1}\left( \frac{\sin_Ka}{\sin_Kb(c,\beta,a)}\sin\beta \right)
\end{equation}

In fact the vast majority of the functions can be expressed in terms of these two functions, which we from now on call $C$ "inverse cosine-law" and $SC$ "inverse sine- on the inverse cosine-law".
\begin{align}
    C(K;\gamma,a,b) \coloneq& \cos_K^{-1}\left( \cos_Ka\cos_Kb + K\sin_Ka\sin_Kb\,\cos\gamma \vphantom{\sqrt{K}}\right),
     &K \neq 0 \notag \\
    =& \frac{1}{\sqrt{K}}\cos^{-1}\left( \cos\sqrt{K}a\cos\sqrt{K}b + \sin\sqrt{K}a\sin\sqrt{K}b\,\cos\gamma \right),  &K > 0\\
    =& \frac{1}{\sqrt{|K|}}\cosh^{-1}\left( \cosh\sqrt{|K|}a\cosh\sqrt{|K|}b - \sinh\sqrt{|K|}a\sinh\sqrt{|K|}b\cos\gamma \right),
    &K < 0 \\
    SC(K;\gamma,a,b) \coloneq& \sin^{-1}\left( \frac{\sin_Ka\,\sin\gamma}{\sin_K\left( \cos_K^{-1}( \cos_Ka\cos_Kb + K\sin_Ka\sin_Kb\,\cos\gamma ) \vphantom{\sqrt{K}}\right)} \right) \notag \\
    =& \sin^{-1}\left( \frac{\sin\sqrt{K}a\,\sin\gamma}{\sqrt{1-\left( \cos\sqrt{K}a\cos\sqrt{K}b + \sin\sqrt{K}a\sin\sqrt{K}b\,\cos\gamma  \right)^2}} \right),
    &K > 0 \\
    =& \sin^{-1}\left( \frac{\sinh\sqrt{|K|}a\,\sin\gamma}{\sqrt{\left( \cosh\sqrt{|K|}a\cosh\sqrt{|K|}b - \sinh\sqrt{|K|}a\sinh\sqrt{|K|}b\,\cos\gamma  \right)^2-1}} \right),
    &K < 0
\end{align}
where we used, that $\sin(\cos^{-1}{y}) = \sqrt{1-y^2}, \ y\in(-1,1)$ and $\sinh(\cosh^{-1}y) = \sqrt{y^2-1}, \ y>1$.\\
\begin{figure}[h!]
    \centering\includegraphics[width=.4\linewidth]{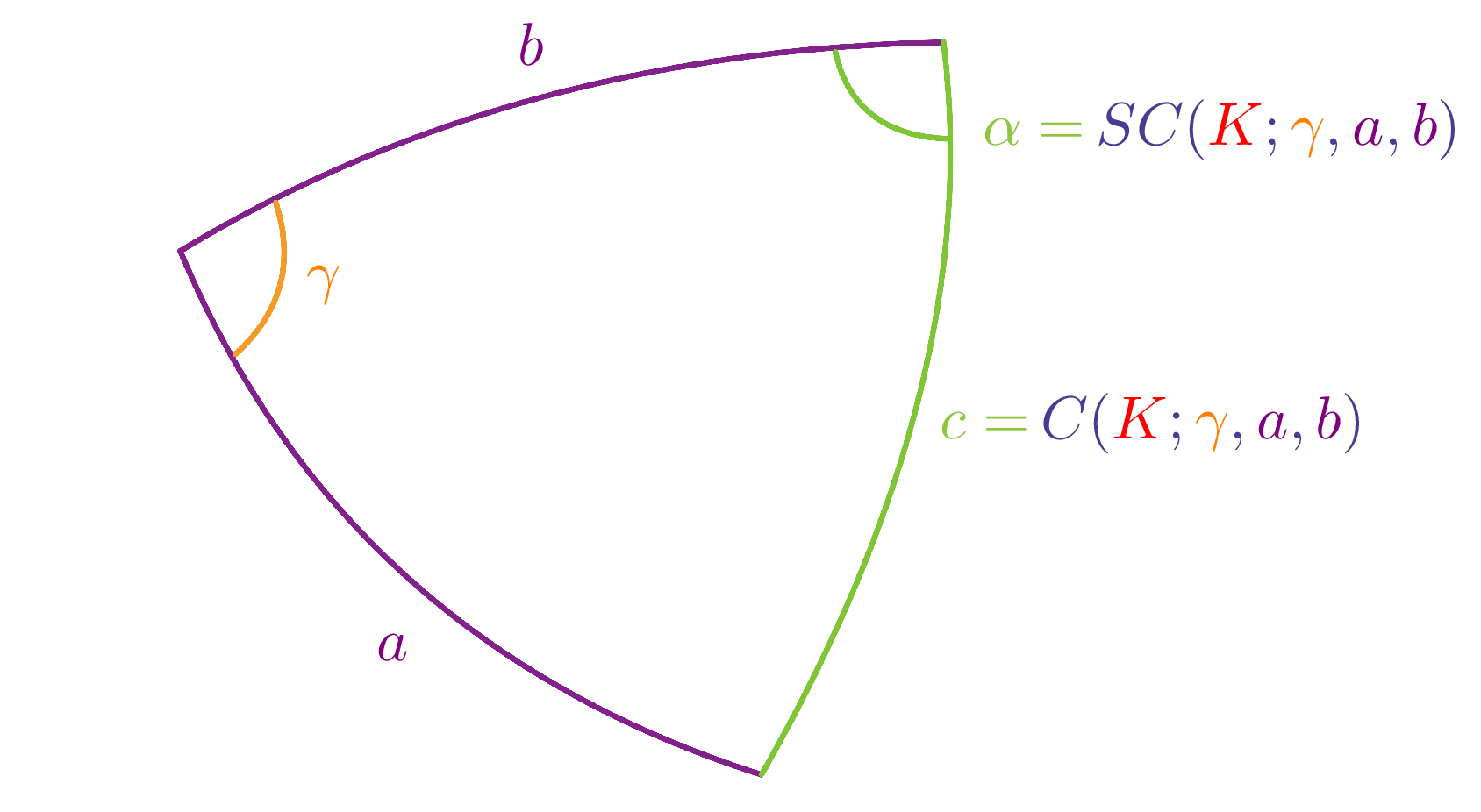}
    \caption{Sketch of the spherical triangle and the labelling we use in this section.}
    \label{fig: Spherical_Triangle_1}
\end{figure}

To more easily memorize, what these functions do: $\gamma$ is always the angle between the sides $a$ and $b$, $C$ calculates the length of the side opposite to $\gamma$ and $SC$ calculates the angle opposite to the first side-length argument. We also note, that $C$ is symmetric, while $SC$ is not. If we swap $a$ and $b$ in the $SC$-function, we calculate the angle $\beta$ opposite to $b$ instead of $\alpha$ opposite to $a$.

It is straight forward to check by inserting the expressions from~\eqref{eq: cosR}, that the inverses of the curvature cosine and sine are given by:
\begin{align}
    \cos_K^{-1}y = \begin{cases} \frac{\cos^{-1}y}{\sqrt{K}}, & K > 0 \\
    \frac{\cosh^{-1}y}{\sqrt{|K|}}, & K < 0 \end{cases} \quad \text{and} \quad
    \sin_K^{-1}y = \begin{cases} \frac{\sin^{-1}\left(\sqrt{K}y\right)}{\sqrt{K}}, & K > 0 \\
    \frac{\sinh^{-1}\left(\sqrt{|K|}y\right)}{\sqrt{|K|}}, & K < 0 \end{cases}
\end{align}
We can use the power series expression for arc-sine and hyperbolic, which we derive in Appendix~\ref{sec: Arcfunc} and the fact, that cosine can be seen as a shifted sine, to write the curved inverses in a power series:
\begin{align}
    \sin_K^{-1}y = \frac{1}{\sqrt{K}}\sum_{k=0}^\infty \frac{(2k)!}{(2^kk!)^2} \frac{(\sqrt{K}y)^{2k+1}}{2k+1}
    = \sum_{k=0}^\infty \frac{(2k)!K^k}{(2^kk!)^2} \frac{y^{2k+1}}{2k+1}, \quad y\in\left[-\frac{1}{\sqrt{K}},\frac{1}{\sqrt{K}}\right]
\end{align}
\begin{align}
    \cos x =& \sin\left(x + \frac{\pi}{2}\right) \quad \Rightarrow \quad \cos^{-1}y = \frac{\pi}{2} - \sin^{-1}y; \\ \cos_K^{-1}y =& \frac{1}{\sqrt{K}}\left( \frac{\pi}{2} - \sin^{-1}y \right) = \frac{1}{\sqrt{K}}\left( \frac{\pi}{2} - \sum_{k=0}^\infty \frac{(2k)!}{(2^kk!)^2} \frac{y^{2k+1}}{2k+1} \right),
    \quad y\in[-1,1]
\end{align}
Since arc-sine is defined on $[-1,1]$ the curved version needs to satisfy $|\sqrt{K}y|<1$, whilst the domain of the arc-cosine remains the same.\\
In the hyperbolic case we use the identity $\cosh^{-1}y = \sinh^{-1}\sqrt{y^2-1}$ to reuse the hyperbolic arc-sine series and we do not get a strict power series for the hyperbolic arc-cosine.
\begin{align}
    \sinh^{-1}y &= \sum_{k=0}^\infty \frac{(-1)^k(2k)!}{(2^kk!)^2} \frac{y^{2k+1}}{2k+1}, \quad y\in(-1,1), \\
    \cosh^{-1}y &= \sqrt{y^2-1}\sum_{k=0}^\infty \frac{(-1)^k(2k)!}{(2^kk!)^2} \frac{(y^2-1)^k}{2k+1}, \quad y\in(1,\sqrt{2})
\end{align}

The cosine law itself, written out in terms of the conventional cosine and sine functions, implies that $C$ is well defined for all possible parameter values:
\begin{align}
    &\cos\sqrt{K}c = \cos\sqrt{K}a\cos\sqrt{K}b + \sin\sqrt{K}a\sin\sqrt{K}b\,\cos\gamma\in [-1,1], \quad K > 0  \\
    &\Rightarrow \quad
    \exists\, C(K;\gamma,a,b), \quad \forall K\in\mathbb{R}_{>0}, \ \gamma\in[0,2\pi), \ a,b\in\mathbb{R}_+
\end{align}
In the same way the existence of $SC$ follows from:
\begin{align}
    \sin\alpha = \frac{\sin\sqrt{K}a}{\sin\sqrt{K}b}\sin\beta\in[-1,1] \quad \Rightarrow \quad
    \exists\,SC(K;\beta,a,c), \quad \forall K\in\mathbb{R}_{>0}, \ \beta\in[0,2\pi), \ a,c\in\mathbb{R}_+
\end{align}

\begin{align}
    \cosh\sqrt{|K|}c = \cosh\sqrt{|K|}a\cosh\sqrt{|K|}b + \sinh\sqrt{|K|}a\sinh\sqrt{|K|}b\,\cos\gamma\in\mathbb{R}_+, \quad K < 0
\end{align}

The functions $C$ and $SC$ are always well defined and the power series always converge for small enough $a,b$ and $c$, but a problem occurs, if the quantity they are supposed to describe does not lie in their image set.\\
To include angles $\alpha > \frac{\pi}{2}$ one would have to find the turning point $a_{turn}$ for which $SC$ assumes $\frac{\pi}{2}$, which corresponds to the maximum of the sine:
\begin{align}\label{eq: SC-branch}
    &SC(K;\gamma,a_{turn},b) \overset{!}{=} \frac{\pi}{2} \quad \rightarrow \quad a_{turn}(K;\gamma,b); \quad
    SC(K;\gamma,a,b) = 
    \begin{cases}
        \sin^{-1}\left( \frac{\sin_Ka}{\sin_Kc}\sin\gamma \right) & a \leqslant a_{turn} \\
        \pi - \sin^{-1}\left( \frac{\sin_Ka}{\sin_Kc}\sin\gamma \right) & a > a_{turn}
    \end{cases}.
\end{align}
When we expand to second order this equation would be solvable. We do encounter a case in Sec.~\ref{subsubsec: Problem 2, Types 2 and 3}, where we clearly have the second case and thus need to use the side-branch. If not explicitly stated otherwise, $SC$ always refers to the principle branch (first case).\\
Since distances arc-lengths of the curves are always positive the arc-cosine covers all occurring cases.
\begin{equation}
    c = C(K;\gamma,a,b)\in[0,\pi], \qquad \begin{matrix} \cos^{-1}: & [-1,1] & \to & [0,\pi] \end{matrix}
\end{equation}

\subsection{Expansion for small triangles}\label{sec: Expansion for small triangles}
When we cut a curved triangle into tiny pieces as described above then the most generic problem is evaluating the $C$ and $SC$ function on expanded distances and angles
\begin{equation}\label{eq: variable expansion}
    a = \sum_{n=1}^\infty a_n\varepsilon^n, \quad b = \sum_{n=1}^\infty b_n\varepsilon^n, \quad \gamma = \sum_{n=0}^\infty \gamma_n\varepsilon^n,
\end{equation}
where $\varepsilon=\frac{1}{N}$ is the global scale parameter of the triangle. We note, that if we make a triangle arbitrarily small the side-lengths will tend to 0, whilest the angles tend to a finite non-zero value as long as the triangle is not degenerate. Therefore the leading order term of the side-lengts $a$, $b$ and $c$ i.e. the $C$-function must be of first order and come with an $\varepsilon$. The leading order of the angles $\alpha$, $\beta$ and $\gamma$ are of zeroth order and thus their series and the one of the $SC$-function starts at $n=0$ instead of $n=1$.\\

We want to get a the full power series in powers of $\varepsilon$ of the $C$ and $SC$ functions which describe them precisely within the convergence radius, to make sure that no terms are forgotten. We calculate these, by first expanding quantity by quantity and then stitching them together via Cauchy product rule.
\begin{equation}
    \sum_{n=0}^\infty a_n \cdot \sum_{n=0}^\infty b_n = \sum_{n=0}^\infty \sum_{k=0}^n a_k b_{n-k}
\end{equation}
Our first concern is,  that power of the expansion parameter is at least the main index of the series of nested sums. This way we can cut the series off at a given $N$ and be sure, that we have all terms up to this order. It will then be easy to drop away the even higher terms as we go along.\\
Trying to follow this strategy we immediately run into a problem because we get powers of power series:
\begin{equation}
    \cos_Ka = \sum_{n=0}^\infty\frac{(-K)^n a^{2n}}{(2n)!}
    = \sum_{n=0}^\infty\frac{(-K)^n}{(2n)!}\left( \sum_{k=1}^\infty a_k\varepsilon^k \right)^{2n}
\end{equation}
We calculate the first few powers of the power series of $a$ to understand the structure:
\begin{align}\label{eq: first powers of power series}
    \left( \sum_{k=1}^\infty a_k\varepsilon^k \right)^0 &= 1, \qquad \left( \sum_{k=1}^\infty a_k\varepsilon^k \right)^1 = \sum_{k=1}^\infty a_k\varepsilon^k \\
    \left( \sum_{k=1}^\infty a_k\varepsilon^k \right)^2 &= \sum_{k=0}^\infty a_{k+1}\varepsilon^{k+1}
    \sum_{i=0}^\infty a_{i+1}\varepsilon^{i+1} = \sum_{k=0}^\infty \sum_{i=0}^k a_{i+1}a_{k-i+1} \varepsilon^{k+2} \\
    \left( \sum_{k=1}^\infty a_k\varepsilon^k \right)^3 &= \sum_{k=0}^\infty \left( \sum_{i=0}^k a_{i+1}a_{k-i+1} \varepsilon^{k+2} \right) \sum_{j=0}^\infty a_{j+1}\varepsilon^{j+1}
    = \sum_{k=0}^\infty \sum_{j=0}^n \left( \sum_{i=0}^j a_{i+1}a_{j-i+1} \varepsilon^{j+2} \right) a_{k-j+1}\varepsilon^{k-j+1} \notag \\
    &= \sum_{k=0}^\infty \sum_{j=0}^k \sum_{i=0}^j a_{i+1}a_{j-i+1}a_{k-j+1}\varepsilon^{k+3}
\end{align}
We see, that the $\varepsilon$-power is shifted by the power of the power series we can capture this structure by introducing the substitutes $A_k^n$:
\begin{equation}
    \left( \sum_{k=1}^\infty a_k\varepsilon^k \right)^n = \sum_{k=0}^\infty A_k^n \varepsilon^{k+n}
\end{equation}
Comparing with~\eqref{eq: first powers of power series} we can write the coefficients up to $n=3$ down explicitly:
\begin{align}
    A_k^0 = \delta_{0k}, \quad A_k^1 = a_{k+1}, \quad A_k^2 = \sum_{i=0}^k a_{i+1}a_{k-i+1}, \quad
    A_k^3 = \sum_{j=0}^k \sum_{i=0}^j a_{i+1}a_{j-i+1}a_{k-j+1}
\end{align}
When we insert this back into the power series for the constant curvature cosine, we see, that the power shift by $n$ accomplishes the condition we required above.
\begin{equation}
    \cos_Ka = \sum_{n=0}^\infty \frac{(-K)^n}{(2n)!} \sum_{k=0}^\infty A_k^{2n} \varepsilon^{k+2n}
    = \sum_{n=0}^\infty \sum_{k=0}^n \frac{(-K)^k}{(2k)!}A_{n-k}^{2k} \varepsilon^{n+k}
\end{equation}
Doing another Cauchy product we can calculate the term
\begin{align}
    \cos_Ka\cos_Kb = \sum_{n=0}^\infty \sum_{k=0}^n \frac{(-K)^k}{(2k)!}A_{n-k}^{2k} \varepsilon^{n+k} \cdot
    \sum_{i=0}^\infty \sum_{j=0}^i \frac{(-K)^j}{(2j)!}B_{i-j}^{2j} \varepsilon^{i+j} = \sum_{n=0}^\infty \sum_{i=0}^n \sum_{k=0}^i \sum_{j=0}^{n-i} \frac{(-K)^{k+j}}{(2k)!(2j)!} A_{i-k}^{2k} B_{n-i-j}^{2j} \varepsilon^{n+k+j}
\end{align}

The same method can be applied to the constant curvature sine terms:
\begin{align}
    \sin_Ka &= \sum_{n=0}^\infty\frac{(-K)^n a^{2n+1}}{(2n+1)!}
    = \sum_{n=0}^\infty\frac{(-K)^n}{(2n+1)!} \left( \sum_{k=1}^\infty a_k\varepsilon^k \right)^{2n+1}
    = \sum_{n=0}^\infty\frac{(-K)^n}{(2n+1)!} \sum_{k=0}^\infty A_k^{2n+1}\varepsilon^{k+2n+1} \notag \\
    &= \sum_{n=0}^\infty \sum_{k=0}^n \frac{(-K)^{k}}{(2n+1)!}  A_{n-k}^{2n+1}\varepsilon^{n+k+1} \\
    \sin_Ka\sin_Kb &= \sum_{n=0}^\infty\sum_{i=0}^n\sum_{k=0}^i\sum_{j=0}^{n-i} \frac{(-K)^{k+j}}{(2k+1)!(2j+1)!}
    A_{i-k}^{2k+1} B_{n-i-j}^{2j+1} \varepsilon^{n+k+j+2}
\end{align}

The $\cos\gamma$ is going to be a bit more complicated since $\gamma$ starts at $0$-th order.
\begin{equation}
    \left( \sum_{k=0}^\infty \gamma_k\varepsilon^k \right)^n = \sum_{k=0}^\infty \Gamma_k^n \varepsilon^k, 
    \qquad \cos\gamma = \sum_{n=0}^\infty \sum_{k=0}^n \frac{(-1)^k}{(2k)!} \Gamma_{n-k}^{2k} \varepsilon^{n-k}
\end{equation}
If we now cut the series off at fore example $n = 3$ we miss several $0$-th order terms, since for each $n$ there is a $k=n$, similar for $1$-st order and so on. Thus, we need to split off the $0$-th order term $\gamma_0$ and define the coefficients $\Gamma_k^n$ in the same way as the $A_k^n$ and $B_k^n$:
\begin{equation}
    \gamma = \gamma_0 + \delta\gamma = \gamma_0 + \sum_{n=1}^\infty\gamma_n\varepsilon^n
\end{equation}
\begin{align}
    \cos\gamma &= \cos(\gamma_0 + \delta\gamma) = \cos\gamma_0\cos\delta\gamma - \sin\gamma_0\sin\delta\gamma \notag \\
    &= \cos\gamma_0 \sum_{n=0}^\infty \frac{(-1)^n}{(2n)!} \left( \sum_{k=1}^\infty \gamma_k\varepsilon^k \right)^{2n} - \sin\gamma_0 \sum_{n=0}^\infty \frac{(-1)^n}{(2n+1)!} \left( \sum_{k=1}^\infty \gamma_k\varepsilon^k \right)^{2n+1} \notag \\
    &= \cos\gamma_0 \sum_{n=0}^\infty \sum_{k=0}^n \frac{(-1)^k}{(2k)!} \Gamma_{n-k}^{2k} \varepsilon^{n+k}
    - \sin\gamma_0 \sum_{n=0}^\infty \sum_{k=0}^n \frac{(-1)^k}{(2k+1)!} \Gamma_{n-k}^{2k+1} \varepsilon^{n+k+1} \notag \\
    &= \sum_{n=0}^\infty \sum_{k=0}^n (-1)^k\left( \frac{\cos\gamma_0}{(2k)!} \Gamma_{n-k}^{2k}
    - \frac{\sin\gamma_0}{(2k+1)!} \Gamma_{n-k}^{2k+1}\varepsilon \right) \varepsilon^{n+k}
\end{align}

Now we are ready to multiply the second term together
\begin{align}
    \sin_Ka&\sin_Kb\cos\gamma \notag \\
    =& \sum_{n=0}^\infty\sum_{i=0}^n\sum_{k=0}^i\sum_{j=0}^{n-i} \frac{(-K)^{k+j}}{(2k+1)!(2j+1)!} A_{i-k}^{2k+1} B_{n-i-j}^{2j+1} \varepsilon^{n+k+j+2} 
    \cdot \sum_{m=0}^\infty \sum_{l=0}^m (-1)^l\left( \frac{\cos\gamma_0}{(2l)!} \Gamma_{m-l}^{2l}
    - \frac{\sin\gamma_0}{(2l+1)!} \Gamma_{m-l}^{2l+1}\varepsilon \right) \varepsilon^{m+l} \notag \\
    =& \sum_{n=0}^\infty \sum_{m=0}^n \sum_{i=0}^m\sum_{k=0}^i\sum_{j=0}^{n-i}
    \frac{(-K)^{k+j}}{(2k+1)!(2j+1)!} A_{i-k}^{2k+1} B_{m-i-j}^{2j+1}
    \sum_{l=0}^{n-m} (-1)^l\left( \frac{\cos\gamma_0}{(2l)!} \Gamma_{n-m-l}^{2l} - \frac{\sin\gamma_0}{(2l+1)!} \Gamma_{n-m-l}^{2l+1}\varepsilon \right) \varepsilon^{n+k+j+l+2}
\end{align}
and merge the two into the series of the constant curvature cosine of $c$:
\begin{align}
    &\cos_Kc = \cos_Ka\cos_Kb + K\sin_Ka\sin_Kb\cos\gamma 
    = \sum_{n=0}^\infty\left\{ \sum_{i=0}^n \sum_{k=0}^i \sum_{j=0}^{n-i} \frac{(-K)^{k+j}}{(2k)!(2j)!} A_{i-k}^{2k} B_{n-i-j}^{2j} \varepsilon^{k+j} \right. \\
    &\left.+
    K\sum_{m=0}^n \sum_{i=0}^m\sum_{k=0}^i\sum_{j=0}^{n-i} \sum_{l=0}^{n-m}
    \frac{(-K)^{k+j}}{(2k+1)!(2j+1)!} A_{i-k}^{2k+1} B_{m-i-j}^{2j+1}
    (-1)^l\left( \frac{\cos\gamma_0}{(2l)!} \Gamma_{n-m-l}^{2l} - \frac{\sin\gamma_0}{(2l+1)!} \Gamma_{n-m-l}^{2l+1}\varepsilon \right) \varepsilon^{k+j+l+2} \right\}\varepsilon^n \notag \\
    &\eqcolon \sum_{n=0}^\infty \tilde{C}_n(\varepsilon)\varepsilon^n
\end{align}
We note, that the term $\tilde{C}_n(\varepsilon)$ in the curly brackets is not the $n$-th order coefficient since it contains terms of higher but not lower order. We denote the proper $n$-th order coefficient of $\cos_Kc$ with $C_n$. The relation between the two substitutes is given by:
\begin{equation}
    \cos_Kc \eqcolon \sum_{n=0}^\infty C_n\varepsilon^n, \qquad C_n = \sum_{k=0}^n\tilde{C}^{(n-k)}_k \label{eq: C_n(tildeC_n)}
\end{equation}

When we insert the arguments in our substitutes we will see, that $C_0 = 1$ and $C_1 = 0$ and we again face the problem that our series starts at $0$. This time it will require some more trickery to write the expression in such a way, that we can apply the addition theorem for the arcus cosine for the case of positive curvature $K>0$:\\
\begin{equation}
    \cos^{-1}x + \cos^{-1}y = \begin{cases} \cos^{-1}(xy-\sqrt{1-x^2}\sqrt{1-y^2}), & x + y \geqslant 0 \\
    2\pi - \cos^{-1}(xy-\sqrt{1-x^2}\sqrt{1-y^2}), & x + y < 0 \end{cases} \quad
    \overset{x=y}{\Rightarrow} \quad 2\cos^{-1}x = \begin{cases} \cos^{-1}(2x^2-1), & x \geqslant 0 \\
    2\pi - \cos^{-1}(2x^2-1), & x < 0 \end{cases} \notag
\end{equation}
\begin{align}
    c &= C(K;\gamma,a,b) = \cos_K^{-1}(\cos_Ka\cos_Kb + K\sin_Ka\sin_Kb\cos\gamma)
    = \cos_K^{-1}\left( \sum_{k=0}^\infty C_k\varepsilon^k \right) 
    \overset{C_0 = 1}{=} \frac{1}{\sqrt{K}}\cos^{-1}\left( 1 + \sum_{n=0}^\infty C_n\varepsilon^n \right),
\end{align}
using the identity
\begin{equation}
    \cos^{-1}(-x) = \pi - \cos^{-1}x, \qquad \text{we get:} \qquad c = \frac{1}{\sqrt{K}}\left[ \pi - \cos^{-1}
    \left(\vphantom{\sqrt{\sum_{n=1}^\infty}}\right.\right.
    2{\underbrace{\sqrt{-\frac{1}{2}\sum_{n=1}^\infty C_n\varepsilon^n}}_{x}}^2 -1
    \left.\left.\vphantom{\sqrt{\sum_{n=1}^\infty}}\right) \right].
\end{equation}
We check, that $x$ is in fact greater then zero:
\begin{equation}
    \forall c > 0: \quad \cos_Kc = \cos\sqrt{K}c = 1 + \sum_{n=1}^\infty C_n\varepsilon^n < 1 \quad \Rightarrow
    \quad \sum_{n=1}^\infty C_n\varepsilon^n < 0 \quad \Rightarrow \quad x > 0
\end{equation}
Now we use the appropriate addition theorem and recall that $\cos^{-1}y = \frac{\pi}{2} - \sin^{-1}y$:
\begin{align}
    c =& \frac{1}{\sqrt{K}}\left[ \pi - 2\cos^{-1}\sqrt{-\frac{1}{2}\sum_{n=1}^\infty C_n\varepsilon^n} \right]
    = \frac{2}{\sqrt{K}}\sin^{-1}\sqrt{-\frac{1}{2}\sum_{n=1}^\infty C_n\varepsilon^n}
    = \frac{2}{\sqrt{K}}\sum_{k=0}^\infty \frac{(2k)!}{(2^kk!)^2}\frac{1}{2k+1}\sqrt{-\frac{1}{2}\sum_{n=1}^\infty C_n\varepsilon^n}^{2k+1} \notag \\
    =& \frac{\sqrt{2}}{\sqrt{K}}\sqrt{-\sum_{n=1}^\infty C_n\varepsilon^n}\sum_{k=0}^\infty \frac{(2k)!}{2^{3k}(k!)^2}\frac{(-1)^k}{2k+1} \left(\sum_{n=1}^\infty C_n\varepsilon^n\right)^k
    = \frac{\sqrt{2}}{\sqrt{K}}\sqrt{-\sum_{n=1}^\infty C_n\varepsilon^n}\sum_{k=0}^\infty \frac{(2k)!}{2^{3k}(k!)^2}\frac{(-1)^k}{2k+1} \sum_{n=0}^\infty \overline{C}_n^k\varepsilon^{n+k} \notag \\
    =& \frac{\sqrt{2}}{\sqrt{K}}\sqrt{-\sum_{n=1}^\infty C_n\varepsilon^n} \sum_{n=0}^\infty \sum_{k=0}^n \frac{(2k)!}{2^{3k}(k!)^2}\frac{(-1)^k}{2k+1} \overline{C}_{n-k}^k\varepsilon^n
\end{align}
To expand the square root, we use the binomial series:
\begin{equation}
    (1 + z)^\alpha = \sum_{k=0}^\infty \binom{\alpha}{k} z^k, \quad \binom{\alpha}{k} \coloneq \frac{(\alpha)_k}{k!}, \quad (\alpha)_k \coloneq \prod_{j=0}^{k-1}(\alpha-j), \quad \forall\alpha\in\mathbb{C}
    \quad \forall z\in\mathring{B}_1(0)\subset\mathbb{C}
\end{equation}
\begin{align}
    \sqrt{-\sum_{n=1}^\infty C_n\varepsilon^n} \overset{C_1=0}{=}& \sqrt{-\sum_{n=2}^\infty C_n\varepsilon^n}
    = \sqrt{-C_2\varepsilon^2}\left( 1 + \frac{1}{C_2}\sum_{n=1}^\infty C_{n+2}\varepsilon^n \right)^\frac{1}{2}
    = \sqrt{-C_2}\varepsilon \sum_{k=0}^\infty \binom{\frac{1}{2}}{k} \left( \frac{1}{C_2}\sum_{n=1}^\infty C_{n+2}\varepsilon^n \right)^k \notag \\
    =& \sqrt{-C_2}\varepsilon \sum_{k=0}^\infty \binom{\frac{1}{2}}{k} \frac{1}{C_2^k} \sum_{n=1}^\infty \widehat{C}_n^k\varepsilon^{n+k}
    = \sqrt{-C_2} \sum_{n=1}^\infty \sum_{k=0}^n \binom{\frac{1}{2}}{k} \frac{1}{C_2^k} \widehat{C}_{n-k}^k\varepsilon^{n+1}
\end{align}
One additional Cauchy product will produce us the series of first function $C$ we are looking for:
\begin{align}
    C(K;\gamma,a,b) =& \frac{\sqrt{2}}{\sqrt{K}} \sqrt{-C_2} \sum_{n=1}^\infty \sum_{k=0}^n \binom{\frac{1}{2}}{k} \frac{1}{C_2^k} \widehat{C}_{n-k}^k\varepsilon^{n+1} \sum_{n=0}^\infty \sum_{k=0}^n \frac{(2k)!}{2^{3k}(k!)^2}\frac{(-1)^k}{2k+1} \overline{C}_{n-k}^k\varepsilon^n \notag \\
    =& \frac{\sqrt{-2C_2}}{\sqrt{K}} \sum_{n=0}^\infty\sum_{k=0}^n\sum_{i=0}^k\sum_{j=0}^{n-k} \binom{\frac{1}{2}}{i}
    \frac{(2j)!}{2^{3j}(j!)^2} \frac{(-1)^j}{2j+1} \frac{\widehat{C}_{k-i}^i\overline{C}_{n-k-j}^j}{C_2^i} \varepsilon^{n+1} = \sum_{n=1}^\infty c_n\varepsilon^n = c
\end{align}
And indeed the expansion for $c$ starts at first order as expected for a side-length.\\

We prepare ourselves to expand the argument of the $SC$ functions arcus sine by introducing jet more substitutes:
\begin{align}
    \sin_Ka =& \sum_{n=0}^\infty \sum_{k=0}^n \frac{(-K)^k}{(2k+1)!} A_{n-k}^{2k+1} \varepsilon^{n+k+1}
    \eqcolon \sum_{n=1}^\infty \mathcal{A}_n(\varepsilon)\varepsilon^n \notag \\
    \sin_Kc =& \sum_{n=0}^\infty \frac{(-K)^n}{(2n+1)!} \left( \sum_{k=1}^\infty c_k \varepsilon^k \right)^{2n+1}
    = \sum_{n=0}^\infty \sum_{k=0}^n \frac{(-K)^k}{(2k+1)!} C_{n-k}^{2k+1} \varepsilon^{k+n+1}
    \eqcolon \sum_{n=1}^\infty \mathcal{C}_n(\varepsilon) \varepsilon^n
\end{align}
We suppress the epsilon dependence, when we expand the fraction of constant curvature sines, using the binomial series again:
\begin{align}
    \frac{\sin_Ka}{\sin_Kc} \overset{\frac{:\varepsilon}{:\varepsilon}}{=}& \frac{\mathcal{A}_1 + \sum_{n=1}^\infty\mathcal{A}_{n+1}\varepsilon^n}
    {\mathcal{C}_1 + \sum_{n=1}^\infty\mathcal{C}_{n+1}\varepsilon^n}
    = \sum_{n=0}^\infty \mathcal{A}_{n+1} \varepsilon^n \frac{1}{\mathcal{C}_1}\left( 1 + \frac{1}{\mathcal{C}_1}\sum_{n=1}^\infty\mathcal{C}_{n+1}\varepsilon^n \right)^{-1}
    = \sum_{n=0}^\infty \mathcal{A}_{n+1} \varepsilon^n \frac{1}{\mathcal{C}_1} \sum_{k=0}^\infty \binom{-1}{k}
    \left( \frac{1}{\mathcal{C}_1} \sum_{l=1}^\infty \mathcal{C}_{l+1}\varepsilon^l \right)^k \notag \\
    =& \sum_{n=0}^\infty \mathcal{A}_{n+1} \varepsilon^n \frac{1}{\mathcal{C}_1} \sum_{k=0}^\infty \binom{-1}{k} \frac{1}{\mathcal{C}_1^k} \sum_{l=1}^\infty D_l^k\varepsilon^{l+k}
    = \sum_{n=0}^\infty \mathcal{A}_{n+1} \varepsilon^n \frac{1}{\mathcal{C}_1} \sum_{k=0}^\infty\sum_{l=0}^k
    \binom{-1}{k} \frac{D_{k-l}^l}{\mathcal{C}_1^l} \varepsilon^k \notag \\
    =& \sum_{n=0}^\infty \sum_{k=0}^n \sum_{l=0}^k \binom{-1}{k} \frac{\mathcal{A}_{n-k+1}D_{k-l}^l}{\mathcal{C}_1^{l+1}} \varepsilon^n
\end{align}

The sine of $\gamma$ works the same as it's cosine:
\begin{align}
    \sin\gamma =& \sin(\gamma_0 + \delta\gamma) = \sin\gamma_0\cos\delta\gamma + \cos\gamma_0\sin\delta\gamma 
    =& \sum_{n=0}^\infty \sum_{k=0}^n (-1)^k\left( \frac{\sin\gamma_0}{(2k)!}\Gamma_{n-k}^{2k} + \frac{\cos\gamma_0}{(2k+1)!}\Gamma_{n-k}^{2k+1}\varepsilon \right)\varepsilon^{n+k}
\end{align}

Then the argument of the arc sine can be multiplied together with the Cauchy product as usual:
\begin{align}
    &\frac{\sin_Ka}{\sin_Kc}\sin\gamma = \sum_{n=0}^\infty \sum_{k=0}^n \sum_{l=0}^k \binom{-1}{k} \frac{\mathcal{A}_{n-k+1}D_{k-l}^l}{\mathcal{C}_1^{l+1}} \varepsilon^n
    \sum_{i=0}^\infty \sum_{j=0}^i (-1)^j\left( \frac{\sin\gamma_0}{(2j)!}\Gamma_{i-j}^{2j} + \frac{\cos\gamma_0}{(2j+1)!}\Gamma_{i-j}^{2j+1}\varepsilon \right)\varepsilon^{i+j} \\
    &= \sum_{n=0}^\infty \sum_{i=0}^n \sum_{k=0}^i \sum_{l=0}^k \binom{-1}{l} \frac{\mathcal{A}_{i-k+1}D_{k-l}^l}{\mathcal{C}_1^{l+1}} \sum_{j=0}^{n-i} (-1)^j\left( \frac{\sin\gamma_0}{(2j)!}\Gamma_{n-i-j}^{2j} + \frac{\cos\gamma_0}{(2j+1)!}\Gamma_{n-i-j}^{2j+1}\varepsilon \right) \varepsilon^{n+j} 
    \eqcolon \sum_{n=0}^\infty \tilde{S}_n(\varepsilon)\varepsilon^n \eqcolon \sum_{n=0}^\infty S_n \varepsilon^n \notag
\end{align}
As with the $\tilde{C}_n(\varepsilon)$ and $C_n$ the $\tilde{S}_n(\varepsilon)$ is the substitute we can extract from the series directly, while the $S_n$ denotes the true $n$-th order coefficient.\\

We could not find a way to apply the addition theorem for the arc-sine or something analogue to it. So, we resorted to the semi explicit Taylor series. The disadvantage being, that we cannot write down the $n$-th derivative of arc-sine.
\begin{align}
    SC(K;\gamma,a,b) =& \sin^{-1}\left( \frac{\sin_Ka}{\sin_KC(K;\gamma,a,b)}\sin\gamma \right)
    = \sin^{-1}\left( \frac{\sin_Ka}{\sin_Kc}\sin\gamma \right)
    = \sin^{-1}\left( S_0 + \sum_{k=1}^\infty S_k\varepsilon^k \right) \\
    =& \sum_{n=0}^\infty (\sin^{-1})^{(n)}(S_0)\left( \sum_{k=1}^\infty S_k\varepsilon^k \right)^n
    = \sum_{n=0}^\infty (\sin^{-1})^{(n)}(S_0) \sum_{k=0}^\infty S_k^n \varepsilon^{n+k}
    = \sum_{n=0}^\infty \sum_{k=0}^n (\sin^{-1})^{(k)}(S_0) S_{n-k}^k \varepsilon^n = \alpha \notag
\end{align}

Now, that we have the full series expansion of both functions we are looking for we can make ourselves a list of terms which we need to expand $SC$ up to second order and $C$ up to third order in $\varepsilon$, by going backwards through the substitutes until we hit the input variables. This way we can avoid forgetting terms, when guessing up to which order the input variables have to be expanded.
\begin{align}
    \alpha \approx& \sum_{n=0}^2 \alpha_n\varepsilon^n + \mathcal{O}(\varepsilon^n) \approx \sum_{n=0}^2 \sum_{k=0}^n (\sin^{-1})^{(k)}(S_0) S_{n-k}^k \varepsilon^n + \mathcal{O}(\varepsilon^3) \notag \\
    \approx& \sin^{-1}(S_0) + (\sin^{-1})'(S_0)S_1\varepsilon
    + \{ (\sin^{-1})'(S_0)S_2 + (\sin^{-1})''(S_0)S_1^2 \}\varepsilon^2 + \mathcal{O}(\varepsilon^3)
\end{align}
So, we need $S_0$, $S_1$ and $S_2$. We can calculate them by picking the terms with appropriate order from the $\tilde{S}_n$ substitutes:
\begin{equation}
    S_n = \sum_{k=0}^n \tilde{S}_k^{(n-k)} \quad \Rightarrow \quad S_0 = \tilde{S}_0^{(0)}, \quad
    S_1 = \tilde{S}_0^{(1)} + \tilde{S}_1^{(0)}, \quad
    S_2 = \tilde{S}_0^{(2)} + \tilde{S}_1^{(1)} + \tilde{S}_2^{(0)}
\end{equation}
Thus, we need $\tilde{S}_0$ up to second order, $\tilde{S}_1$ up to first and only the zeroth order term of $\tilde{S}_2$:
\begin{align}
    \tilde{S}_n(\varepsilon) =& \sum_{i=0}^n \sum_{k=0}^i \sum_{l=0}^k \binom{-1}{l} \frac{\mathcal{A}_{i-k+1}D_{k-l}^l}{\mathcal{C}_1^{l+1}} \sum_{j=0}^{n-i} (-1)^j\left( \frac{\sin\gamma_0}{(2j)!}\Gamma_{n-i-j}^{2j} + \frac{\cos\gamma_0}{(2j+1)!}\Gamma_{n-i-j}^{2j+1}\varepsilon \right) \varepsilon^j; \\
    \tilde{S}_0 =& \frac{\mathcal{A}_1}{\mathcal{C}_1}D_0^0\left( \sin\gamma_0\Gamma_0^0
    + \cos\gamma_0\Gamma_0^1\varepsilon \right) \\
    \tilde{S}_1 \approx& \left\{ \frac{\mathcal{A}_1}{\mathcal{C}_1}D_0^0\Gamma_1^0 + \left[ \frac{\mathcal{A}_2}{\mathcal{C}_1}D_0^0 + \frac{\mathcal{A}_1}{\mathcal{C}_1}D_1^0 - \frac{\mathcal{A}_1}{\mathcal{C}_1^2}D_0^0 \right]\Gamma_0^0 \right\}\sin\gamma_0 \notag \\
    &+ \left\{ \frac{\mathcal{A}_1}{\mathcal{C}_1}D_0^0\left[ \cos\gamma_0\Gamma_1^1 - \frac{1}{2}\sin\gamma_0\Gamma_0^2 \right] + \left[ \frac{\mathcal{A}_2}{\mathcal{C}_1}D_0^0 + \frac{\mathcal{A}_1}{\mathcal{C}_1}D_1^0 - \frac{\mathcal{A}_1}{\mathcal{C}_1^2}D_0^1 \right]\cos\gamma_0\Gamma_0^1 \right\}\varepsilon + \mathcal{O}(\varepsilon^2) \\
    \tilde{S}_2 \approx& \left\{ \frac{\mathcal{A}_1}{\mathcal{C}_1}D_0^0\Gamma_2^0 + \left[ \frac{\mathcal{A}_2}{\mathcal{C}_1}D_0^0 + \mathcal{A}_1\left( \frac{D_1^0}{\mathcal{C}_1} - \frac{D_0^1}{\mathcal{C}_1^2} \right) \right]\Gamma_1^0 \right.\notag\\
    &\left.+ \left[ \frac{\mathcal{A}_3}{\mathcal{C}_1}D_0^0
    + \mathcal{A}_2\left( \frac{D_1^0}{\mathcal{C}_1} - \frac{D_0^1}{\mathcal{C}_1^2} \right)
    + \mathcal{A}_1\left( \frac{D_2^0}{\mathcal{C}_1} - \frac{D_1^1}{\mathcal{C}_1^2}
    + \frac{D_0^2}{\mathcal{C}_1^3} \right) \right]\Gamma_0^0 \right\}\sin\gamma_0 + \mathcal{O}(\varepsilon)
\end{align}

Next we resubstitute the $D_l^k$ with the $\mathcal{C}_n$:
\begin{equation}
    \left( \sum_{n=1}^\infty \mathcal{C}_{n+1}\varepsilon^n \right)^k = \sum_{l=0}^\infty D_l^k \varepsilon^{l+k} \quad \Rightarrow \quad D_l^0 = \delta_{l0}, \quad D_l^1 =  \mathcal{C}_{l+2}, \quad
    D_l^2 = \sum_{k=0}^l \mathcal{C}_{k+2} \mathcal{C}_{l-k+2}
\end{equation}
Inserting this back into the $\tilde{S}$ expansions we get:
\begin{align}\label{eq: tS_n}
    \tilde{S}_0 =& \frac{\mathcal{A}_1}{\mathcal{C}_1}\left( \sin\gamma_0\Gamma_0^0
    + \cos\gamma_0\Gamma_0^1\varepsilon \right) \\
    \tilde{S}_1 \approx& \left\{ \frac{\mathcal{A}_1}{\mathcal{C}_1}\Gamma_1^0 + \left[ \frac{\mathcal{A}_2}{\mathcal{C}_1} - \frac{\mathcal{A}_1\mathcal{C}_2}{\mathcal{C}_1^2} \right]\Gamma_0^0 \right\}\sin\gamma_0 + \left\{ \frac{\mathcal{A}_1}{\mathcal{C}_1}\left[ \cos\gamma_0\Gamma_1^1 - \frac{1}{2}\sin\gamma_0\Gamma_0^2 \right] + \left[ \frac{\mathcal{A}_2}{\mathcal{C}_1}
    - \frac{\mathcal{A}_1\mathcal{C}_2}{\mathcal{C}_1^2} \right]\cos\gamma_0\Gamma_0^1 \right\}\varepsilon + \mathcal{O}(\varepsilon^2) \\
    \tilde{S}_2 \approx& \left\{ \frac{\mathcal{A}_1}{\mathcal{C}_1}\Gamma_2^0 + \left[ \frac{\mathcal{A}_2}{\mathcal{C}_1} - \frac{\mathcal{A}_1\mathcal{C}_2}{\mathcal{C}_1^2} \right]\Gamma_1^0
    + \left[ \frac{\mathcal{A}_3}{\mathcal{C}_1} - \frac{\mathcal{A}_2\mathcal{C}_2}{\mathcal{C}_1^2}
    + \mathcal{A}_1\left( \frac{\mathcal{C}_2^2}{\mathcal{C}_1^3} - \frac{\mathcal{C}_3}{\mathcal{C}_1^2} \right) \right]\Gamma_0^0 \right\}\sin\gamma_0 + \mathcal{O}(\varepsilon)
\end{align}

We hit the input variables of $a$, when we resubstitute $\mathcal{A}_n$ with $A_k^n$ and then that with $a_n$. The $\mathcal{A}$ and $\mathcal{C}$ are also $\varepsilon$-dependent and no proper order coefficients. The $\mathcal{A}_1$ appears in all $\tilde{S}_n$, $\mathcal{A}_2$ only from $\tilde{S}_1$ on and $\mathcal{A}_2$ in $\tilde{S}_2$. So, we need $\mathcal{A}_1$ up to second, $\mathcal{A}_2$ up to first order and only the zeroth order term of $\mathcal{A}_3$.
\begin{align}
    \mathcal{A}_n(\varepsilon) =& \sum_{k=0}^{n-1} \frac{(-K)^k}{(2k+1)!} A_{n-k-1}^{2k+1} \varepsilon^k; \quad
    \mathcal{A}_1 = A_0^1, \quad \mathcal{A}_2 = A_1^1 - \frac{K}{6}A_0^3\varepsilon, \quad \mathcal{A}_3 \approx A_2^1 + \mathcal{O}(\varepsilon) \\
    a^n =& \left( \sum_{k=1}^\infty a_k\varepsilon^k \right)^n = \sum_{k=0}^\infty A_k^n \varepsilon^{k+n}; \\
    A_k^0 =& \delta_{k0}, \quad A_k^1 = a_{k+1}, \quad A_k^2 = \sum_{l=0}^k a_{l+1} a_{k-l+1}, \quad
    A_k^3 = \sum_{l=0}^k\sum_{j=0}^i a_{j+1} a_{i-j+1} a_{k-i+1} \\
    \Rightarrow \quad \mathcal{A}_1 =& a_1, \quad \mathcal{A}_2 = a_2 - \frac{K}{6}a_1^3\varepsilon \quad
    \mathcal{A}_3 \approx a_3 + \mathcal{O}(\varepsilon)
\end{align}
Doing the same for the $\mathcal{C}$'s we get:
\begin{align}
    \mathcal{C}_n(\varepsilon) =& \sum_{k=0}^{n-1} \frac{(-K)^k}{(2k+1)!} C_{n-k-1}^{2k+1} \varepsilon^k, \quad
    c^n = \left( \sum_{k=1}^\infty c_k\varepsilon^k \right)^n = \sum_{k=0}^\infty C_k^n \varepsilon^{k+n}; \\
    \Rightarrow \quad \mathcal{C}_1 =& c_1 \quad \mathcal{C}_2 = c_2 - \frac{K}{6}c_1^3\varepsilon \quad
    \mathcal{C}_3 \approx c_3 + \mathcal{O}(\varepsilon)
\end{align}

With the only difference, that the $c_n$ are not input variables but $n$-th order coefficients of the $C$-function. And thus we continue by expanding it up to third order, since we saw, that we need $c_3$:
\begin{equation}
    c \approx \sum_{n=1}^3 c_n \varepsilon^n + \mathcal{O}(\varepsilon^4)
    \approx \frac{\sqrt{-2C_2}}{\sqrt{K}} \sum_{n=0}^2\sum_{k=0}^n\sum_{i=0}^k\sum_{j=0}^{n-k} \binom{\frac{1}{2}}{i}
    \frac{(2j)!}{2^{3j}(j!)^2} \frac{(-1)^j}{2j+1} \frac{\widehat{C}_{k-i}^i\overline{C}_{n-k-j}^j}{C_2^i} \varepsilon^{n+1} + \mathcal{O}(\varepsilon^4)
\end{equation}

We observe, that the largest value $i$ and $j$ can assume is 2. So, we need the following substitutes:
\begin{align}
    \left( \sum_{k=1}^\infty C_k \varepsilon^k \right)^n = \sum_{k=0}^\infty \overline{C}_k^0 \varepsilon^{k+n}; \qquad \overline{C}_k^0 = \delta_{k0}, \quad \overline{C}_k^1 = C_{k+1}, \quad
    \overline{C}_k^2 = \sum_{l=0}^k C_{l+1} C_{k-l+1} \\
    \left( \sum_{k=1}^\infty C_{k+2} \varepsilon^k \right)^n = \sum_{k=0}^\infty \widehat{C}_k^n \varepsilon^{k+n}; \qquad \widehat{C}_k^0 = \delta_{k0}, \quad \widehat{C}_k^1 = C_{k+3}, \quad
    \widehat{C}_k^3 = \sum_{l=0}^k C_{l+3} C_{k-l+3}
\end{align}

After summing to $2$ and using the substitutions above we get:
\begin{equation}
    c \approx \frac{\sqrt{-2C_2}}{\sqrt{K}}\left( \varepsilon
    - \left[ \frac{C_1}{12} - \frac{C_3}{2C_2} \right]\varepsilon^2
    - \left[ \frac{C_2}{12} - \frac{2C_1^2}{160} + \frac{C_1C_3}{24C_2} - \frac{C_4}{2C_2}
    + \frac{C_3^2}{8C_2^2} \right]\varepsilon^3 + \mathcal{O}(\varepsilon^4) \right)
\end{equation}

As with the $S_n$ we calculate the $C_n$ from the $\tilde{C}_n(\varepsilon)$:
\begin{align}\label{eq: C_n(tC_n)}
    C_n = \sum_{k=0}^n \tilde{C}_k^{(n-k)}; \qquad C_0 =& \tilde{C}_0^{(0)}, \quad
    C_1 = \tilde{C}_0^{(1)} + \tilde{C}_1^{(0)}, \quad
    C_2 = \tilde{C}_0^{(2)} + \tilde{C}_1^{(1)} + \tilde{C}_2^{(0)}, \\
    C_3 =& \tilde{C}_0^{(3)} + \tilde{C}_1^{(2)} + \tilde{C}_2^{(1)} + \tilde{C}_3^{(0)}, \quad
    C_4 = \tilde{C}_0^{(4)} + \tilde{C}_1^{(3)} + \tilde{C}_2^{(2)} + C_3^{(1)} + C_4^{(0)}
\end{align}
\begin{align}\label{eq: tC_n}
    \tilde{C}&_n = \sum_{i=0}^n \sum_{k=0}^i \sum_{j=0}^{n-i} \frac{(-K)^{k+j}}{(2k)!(2j)!} A_{i-k}^{2k} B_{n-i-j}^{2j} \varepsilon^{k+j} \\
    &+ K\sum_{m=0}^n \sum_{i=0}^m\sum_{k=0}^i\sum_{j=0}^{n-i} \sum_{l=0}^{n-m}
    \frac{(-K)^{k+j}}{(2k+1)!(2j+1)!} A_{i-k}^{2k+1} B_{m-i-j}^{2j+1}
    (-1)^l\left( \frac{\cos\gamma_0}{(2l)!} \Gamma_{n-m-l}^{2l} - \frac{\sin\gamma_0}{(2l+1)!} \Gamma_{n-m-l}^{2l+1}\varepsilon \right) \varepsilon^{k+j+l+2} \notag
\end{align}
Thus, the relevant $\tilde{C}$ up to required order are:
\begin{align}
    \tilde{C}_0 =& A_0^0B_0^0 + KA_0^1B_0^1\cos\gamma_0\Gamma_0^0\varepsilon^2
    - KA_0^1B_0^1\sin\gamma_0\Gamma_0^1\varepsilon^3 \\
    \tilde{C}_1 \approx& A_0^0B_1^0 + A_1^0B_0^0 - \frac{K}{2}(A_0^0B_0^2 + A_0^2B_0^0)\varepsilon
    + K(A_0^1B_0^1\Gamma_1^0 + [A_0^1B_1^1 + A_1^1B_0^1]\Gamma_0^1)\cos\gamma_0\varepsilon^2 \notag \\
    &- K\left\{ \left( \frac{1}{2}A_0^1B_0^1\Gamma_0^2 + \frac{K}{6}[A_0^3B_0^1 + A_0^1B_0^3]\Gamma_0^0 \right)\cos\gamma_0 + (A_0^1B_0^1\Gamma_1^1 + [A_0^1B_1^1 + A_1^1B_0^1]\Gamma_0^1)\sin\gamma_0 \right\}\varepsilon^3 + \mathcal{O}(\varepsilon^4) \\
    \tilde{C}_2 \approx& A_0^0B_2^0 + A_1^0B_1^0 + A_2^1B_0^0
    - \frac{K}{2}\{ A_0^0B_1^2 + A_0^2B_1^0 + A_1^0B_0^2 + A_1^2B_0^0 \}\varepsilon \notag \\
    &+ K\left\{ \frac{K}{24}[A_0^0B_0^4 + 6A_0^2B_0^2 + A_0^4B_0^0] + (A_0^1B_0^1\Gamma_2^0 + [A_0^1B_1^1 + A_1^1B_0^1]\Gamma_1^0 + [A_0^1B_2^1 + A_1^1B_1^1 + A_2^1B_0^1]\Gamma_0^0)\cos\gamma_0 \right\}\varepsilon^2 \notag \\
    &+ \mathcal{O}(\varepsilon^3) \\
    \tilde{C}_3 \approx& A_0^0B_3^0 + A_1^0B_2^0 + A_2^0B_1^0 + A_3^0B_0^0
    - \frac{K}{2}\{ A_0^0B_2^2 + A_0^2B_2^0 + A_1^0B_1^2 + A_1^2B_1^0 + A_2^0B_0^2 + A_2^2B_0^0 \}\varepsilon
    + \mathcal{O}(\varepsilon^2) \\
    \tilde{C}_4 \approx& A_0^0B_4^0 + A_1^0B_3^0 + A_2^0B_2 + A_3^0B_1^0 + A_4^0B_0^0 + \mathcal{O}(\varepsilon)
\end{align}

The $B_k^n$ and $\Gamma_k^n$ behave the same way as the $A_k^n$. We need the $A$ and $B$'s up to $n=4$ and the $\Gamma$'s up to $n=2$.
\begin{align}
    b^n = \left( \sum_{k=1}^\infty b_k\varepsilon^k \right)^n \eqcolon& \sum_{k=0}^\infty B_k^n\varepsilon^{k+n}; \qquad B_k^0 = \delta_{k0}, \quad B_k^1 = b_{k+1}, \quad B_k^2 = \sum_{l=0}^k b_{l+1} b_{k-l+1}, \quad
    B_k^3 = \sum_{i=0}^k \sum_{j=0}^i b_{j+1} b_{i-j+1} b_{k-i+1}, \notag \\
    \left( \sum_{k=1}^\infty b_k\varepsilon^k \right)^4
    =& \sum_{k=1}^\infty b_k\varepsilon^k \left( \sum_{k=1}^\infty b_k\varepsilon^k \right)^3
    = \sum_{n=1}^\infty b_n\varepsilon^n \sum_{k=0}^\infty \underbrace{\sum_{i=0}^k \sum_{j=0}^i b_{j+1} b_{i-j+1} b_{k-i+1}}_{B_k^3} \varepsilon^{n+4} \\
    =& \sum_{n=1}^\infty \sum_{k=0}^n \sum_{i=0}^k \sum_{j=0}^i b_{j+1} b_{i-j+1} b_{k-i+1} b_{n-k+1} \varepsilon^{n+4} \quad \Rightarrow \quad B_n^4 = \sum_{k=0}^n \sum_{i=0}^k \sum_{j=0}^i b_{j+1} b_{i-j+1} b_{k-i+1} b_{n-k+1} \notag \\
    \delta\gamma^n = \left( \sum_{k=1}^\infty \gamma_k\varepsilon^k \right)^n \eqcolon& \sum_{k=0}^\infty \Gamma_k^n\varepsilon^{k+n}; \qquad \Gamma_k^0 = \delta_{k0}, \quad \Gamma_k^1 = \gamma_{k+1}, \quad
    \Gamma_k^2 = \sum_{l=0}^k \gamma_{l+1} \gamma_{k-l+1}
\end{align}

Now, that we re-substituted everything down to the input variables, we can start inserting them and go through the process in forward direction again.\\
We start by picking the right terms from the $\tilde{C}$'s according to~\eqref{eq: C_n(tildeC_n)} and insert the input variables:
\begin{align}
    C_0 =& \tilde{C}_0^{(0)} = A_0^0B_0^0 = 1, \qquad
    C_1 = \tilde{C}_0^{(1)} + \tilde{C}_1^{(0)} = 0 + A_0^0\cancelto{0}{B_1^0} + \cancelto{0}{A_1^0}B_0^0 = 0 \\
    C_2 =& \tilde{C}_0^{(2)} + \tilde{C}_1^{(1)} + \tilde{C}_2^{(0)}
    = KA_0^1B_0^1\cos\gamma_0\Gamma_0^0 - \frac{K}{2}(A_0^0B_0^2 + A_0^2B_0^0)
    + A_0^0\cancelto{0}{B_2^0} + \cancelto{0}{A_1^0B_1^0} + \cancelto{0}{A_2^0}B_0^0 \notag \\
    =& -\frac{K}{2}\left[ a_1^2 + b_1^2 - a_1b_1\cos\gamma_0 \right] \\
    C_3 =& -K\left[ b_1b_2 + a_1a_2 - (a_1b_2 + a_2b_1)\cos\gamma_0 + a_1b_1\gamma_1\sin\gamma_0 \right] \\
    C_4 =& K\left\{ \frac{K}{24}\left( a_1^4 + 6a_1^2b_1^2 + b_1^4 - 4[a_1^3b_1 + a_1b_1^3]\cos\gamma_0 \vphantom{\sqrt{2}}\right) - \frac{1}{2}\left(a_2^2 + 2a_1a_3 + 2b_1b_3 + b_3^2\vphantom{\sqrt{2}}\right) \right.\notag\\
    &\left.+ \left( a_1b_3 + a_2b_2 + a_3b_1 - \frac{a_1b_1}{2}\gamma_1^2 \right)\cos\gamma_0
    - \left( [a_1b_2 + a_2b_1]\gamma_1 + a_1b_1\gamma_2 \vphantom{\sqrt{2}}\right)\sin\gamma_0 \right\}
\end{align}
As we can see from Eq.~\eqref{eq: tC_n} the $\tilde{C}$ only depend on $A_k^n$, $B_k^n$ and $\Gamma_k^n$, where no assumption on the $C$'s is used and thus we confirmed, that $C_0 = 1$ and $C_1 = 0$, which we used to expand the arc-cosine.\\
To write the $C$-function in a compact way, we introduce the following substitutions:
\begin{equation}
    y \coloneq \sqrt{-\frac{2C_2}{K}} = \sqrt{a_1^2 + b_1^2 - 2a_1b_1\cos\gamma_0}, \quad
    \Rightarrow \quad 2C_2 = -Ky^2, \qquad x \coloneq \frac{C_3}{K}, \qquad z \coloneq \frac{C_4}{K}
\end{equation}
\begin{align}\label{eq: C(K,gamma,a,b)_expanded}
    C(K;\gamma,a,b) \approx& \frac{\sqrt{-2C_2}}{\sqrt{K}}\left( \varepsilon
    + \frac{C_3}{2C_2}\varepsilon^2
    - \left[ \frac{C_2}{12} - \frac{C_4}{2C_2}
    + \frac{C_3^2}{8C_2^2} \right]\varepsilon^3 + \mathcal{O}(\varepsilon^4) \right) \notag \\
    \approx& y\left( \varepsilon - \frac{x}{y^2}\varepsilon^2 + \left[ \frac{Ky^2}{24} - \frac{z}{y^2}
    - \frac{x^2}{2y^2} \right]\varepsilon^3 + \mathcal{O}(\varepsilon^4) \right) \approx c
\end{align}

To calculate the second order expansion of $SC$ we rewrite the $\mathcal{C}$'s in terms of $y$, $x$ and $z$:
\begin{equation}
    \mathcal{C}_1 = c_1 = y \qquad
    \mathcal{C}_2 = c_2 - \frac{K}{6}c_1^3\varepsilon = -\frac{x}{y} - \frac{K}{6}y^3\varepsilon \qquad
    \mathcal{C}_3 \approx c_3 + \mathcal{O}(\varepsilon)
    \approx y\left[ \frac{Ky^2}{24} - \frac{z}{y^2} - \frac{x^2}{2y^2} \right] + \mathcal{O}(\varepsilon)
\end{equation}
Next, we need the $\tilde{S}$ from~\eqref{eq: tS_n} in terms of the input variables or $y$, $x$, $z$, since we know these explicitly:
\begin{align}
    \tilde{S}_0 =& \frac{\mathcal{A}_1}{\mathcal{C}_1}( \sin\gamma_0\cancelto{1}{\Gamma_0^0}
    + \cos\gamma_0\Gamma_0^1\varepsilon ) = \frac{a_1}{y}( \sin\gamma_0 + \gamma_1\cos\gamma_0\varepsilon ) \\
    \tilde{S}_1 \approx& \frac{1}{y}\left\{ a_2 - \frac{a_1c_2}{y} \right\}\sin\gamma_0
    + \frac{1}{y}\left\{ \frac{K}{6}\left[ \frac{a_1c_1^3}{y} - a_1^3 \right]\sin\gamma_0
    + \left( a_1\gamma_2 + \left[ a_2 - \frac{a_1c_2}{y} \right]\gamma_1 \right)\cos\gamma_0
    - \frac{a_1}{2}\gamma_1^2\sin\gamma_0 \right\}\varepsilon + \mathcal{O}(\varepsilon^2) \notag \\
    \tilde{S}_2 \approx& \frac{1}{y}\left\{ a_3 - \frac{a_1c_2 + a_1c_3}{y} + \frac{a_1c_2^2}{y^2} \right\}\sin\gamma_0 + \mathcal{O}(\varepsilon)
\end{align}
Then we can assemble the $S_n$ just as we did with the $C_n$:
\begin{align}
    S_0 =& \tilde{S}_0^{(0)} = \frac{a_1}{y}\sin\gamma_0 \\
    S_1 =& \tilde{S}_0^{(1)} + \tilde{S}_1^{(0)}
    = \frac{1}{y}\left( \left[ a_2 + \frac{a_1x}{y^2} \right]\sin\gamma_0 + a_1\gamma_1\cos\gamma_0 \right) \\
    S_2 =& \frac{1}{y}\left\{ \frac{K}{6}\left[ \frac{3a_1y^2}{4} - a_1^3 \right]\sin\gamma_0
    + \left( a_1\gamma_2 + \left[ a_2 + \frac{a_1x}{y^2} \right]\gamma_1 \right)\cos\gamma_0
    + \left( a_3 + \frac{a_2x}{y^2} + a_1\left[ \frac{z}{y^2} + \frac{3x^2}{2y^4} - \frac{\gamma_1^2}{2} \right] \right)\sin\gamma_0 \right\}
\end{align}

We leave the $S_n$ as substitutes, since they are now given explicitly and there is not much that can be simplified if we would insert them. With that, the expansion of the $SC$-function to second order in $\varepsilon$ is finally given by:
\begin{align}\label{eq: SC(K,gamma,a,b)_expanded}
    SC(K;\gamma,a,b) \approx& \sin^{-1}(S_0) + (\sin^{-1})'(S_0)S_1\varepsilon
    + \left\{ (\sin^{-1})'(S_0)S_2 + (\sin^{-1})''(S_0)S_1^2 \vphantom{\sqrt{2}}\right\}\varepsilon^2 + \mathcal{O}(\varepsilon^3) \notag \\
    \approx& \sin^{-1}(S_0) + \frac{S_1}{\sqrt{1-S_0^2}}\varepsilon
    + \left\{ \frac{S_2}{\sqrt{1-S_0^2}} + \frac{yS_1^2}{\sqrt{1-S_0^2}^3} \right\}\varepsilon^2 + \mathcal{O}(\varepsilon^3)
\end{align}

\begin{figure}[h!]
    \centering
    \includegraphics[width=\linewidth]{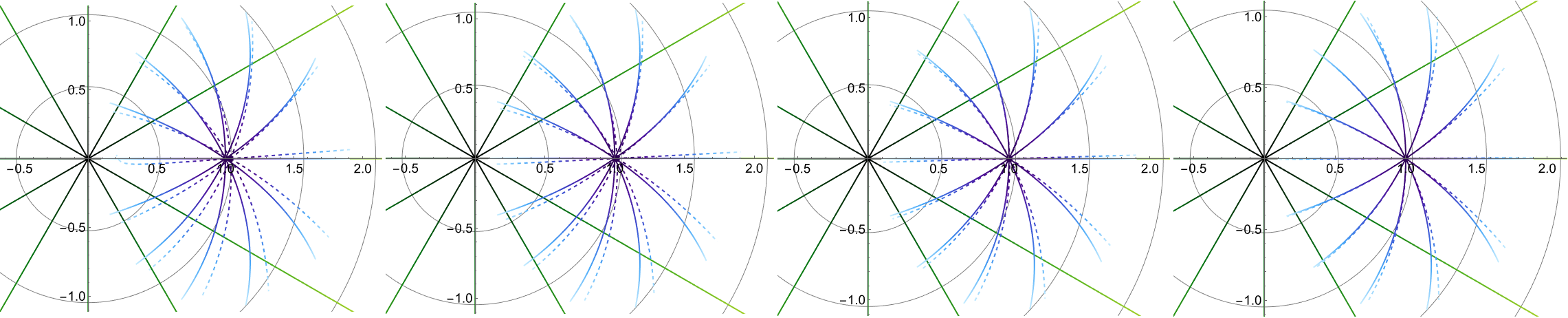}
    \caption{We approximate the geodesic flow from $q(c)$, $c=\frac{\pi}{3}$ on the $2$-sphere using the formulas above (dashed) and compare it to the exact one. We increase the number of cuts $N=\varepsilon^{-1}=3,5,10,20$, to check that this approximation converges to the flow on the sphere up to second order. I.e. we test our approximation, with finesse $\varepsilon$, of the second order approximation of the flow bundle on $\mathcal{S}^2$.}
    \label{fig: Approx-Test}
\end{figure}

\section{A first Slice of the Solution}\label{sec: A first Slice of the Solution}
In this section we calculate the opening angle $\alpha^0$ and side-length of the top line $\bar{b}$ of the first slice $\Sigma_1$. For ease of notation we suppress the slice index $j=1$ in this section.\\
We outlined in the main text how we divide the procedure into problems and subtypes of these problems. We will first derive a general solution to these (template formulas) and then use the template formulas to solve the slice order by order. It is not directly possible to solve the second order explicitly and we instead get two subsequent recursions.\\
In the Sec.~\ref{subsec: The recursion parameters} we solve the first and more complicated coupled recursion of three recursion parameters. We then conclude this section by piecing everything together, solve the second recursion over the rib-lines and discuss the solution to the first slice up to second order.

\subsection{Solving Problems of different Types}\label{subsec: Solving Problems of different Types}
We now use the expansions of the cosine- and sine-laws derived in Sec.~\ref{sec: Expansion for small triangles} on the segments in the first slice, which we established in the main text. To make the procedure easier to understand, we color the side-lengths and angles according to the context in which they are used in the different problems.\\
The segments are constructed to approximate the curvature field with a step-function, which is constant on each segment and we thus use red for segments. Quantities, which can be expressed directly or are already determined in terms of input variables will be black. For most problems we have to few information to solve them directly and we will instead solve them in dependence of still undetermined quantities. We will mark these unknowns, which
we will treat indiscriminately from input variables within a single problem, with blue. Green is used for functions of unknowns, while dark green is for functions of such functions. For example $\hat{\alpha}^1(\alpha^1,d^1)$ is not used in segment $\Delta_1$, where we would view $\alpha^1$ and $d^1$ as unknowns, but in segment $\Delta^2$, where we express the two as functions on $a$. Which brings us to the special case $a$ which is the key to solving "Problem 2". It is first used as an unknown and after solving the $\delta$-equation, it becomes a function of $\alpha^i$ and $d^i$. We use turquoise for it and light blue for quantities such as delta which we are a priory not interested in and merely use them to set up equations to solve for other quantities.
\begin{figure}[h!]
    \centering\includegraphics[width=\linewidth]{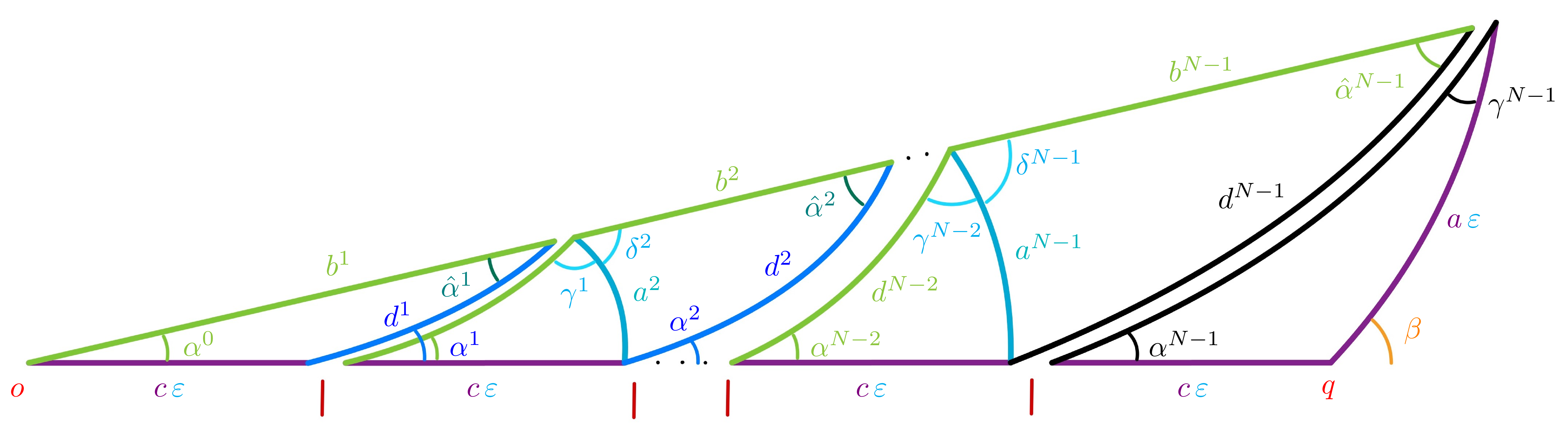}
    \caption{A sketch of the segmentation of the first slice.}
    \label{fig: First_Slice}
\end{figure}

\subsubsection{Problem 0}
First we need to determine the curvature values with which we approximate the Gaussian curvature field on our segments. We pick the values at the lower left edge of each segment:
\begin{equation}
    K_{i,1} = K(q_{i-1,1}) = K((i-1)c\varepsilon,0), \quad
    q_{i,1} = \gamma_{o,\bar{\alpha}^{0,0}}(ic^1\varepsilon) \quad i\in\{1,\ldots,N\},
\end{equation}

\subsubsection{Problem 1}\label{sec: Problem 1}
The segment $\Delta_N$ is a triangle of which we have two side-lengths $a\varepsilon$ and $c\varepsilon$ and the angle between them $\pi-\beta$, see Fig.~\ref{fig: Problem 1}. So, we can solve it directly.
\begin{align}
    \left.\begin{matrix}
        d(\varepsilon) = C(K;\pi-\beta,a\varepsilon,c\varepsilon) \\
        \alpha(\varepsilon) = SC(K;\pi-\beta,a\varepsilon,c\varepsilon)
    \end{matrix}\right\} \quad \Rightarrow \quad
    \begin{matrix}
    a_1 \mapsto a, & a_2 \mapsto 0, & a_3 \mapsto 0 \\
    b_1 \mapsto c, & b_2 \mapsto 0, & b_3 \mapsto 0 \\
    \gamma_0 \mapsto \pi-\beta, & \gamma_1 \mapsto 0, & \gamma_2 \mapsto 0
    \end{matrix}
\end{align}

The substitute $y$ will still be relevant in future calculation as it is the first order term of the $C$-function and appears in many other terms.
\begin{align}
    y = \sqrt{a^2 + c^2 + ac\cos\beta}
\end{align}
\begin{align}
    d_1 &= y, &\quad d_2 &= 0, &\quad d_3 &= -\frac{K}{6y}a^2c^2\sin^2\beta \\
    \alpha_0 &= \sin^{-1}\left(\frac{a}{y}\sin\beta\right), &\quad \alpha_1 &= 0, &\quad
    \alpha_2 &= K\frac{ a^2(4c^2 - 3y^2) + y^2(c^2 + 3y^2) + 2ac(a^2 + c^2)\cos\beta }
    {24y^2\sqrt{y^2 - a^2\sin^2\beta}}a\sin\beta
\end{align}

In the case of $\gamma$ the roles of $a$ and $c$ are swapped.
\begin{equation}
    \gamma(\varepsilon) = SC(K;\pi-\beta,c\varepsilon,a\varepsilon) \quad
    \Rightarrow \quad a_1 \mapsto c, \quad b_1 \mapsto a
\end{equation}

\begin{align}
    \gamma_0 = \sin^{-1}\left(\frac{c}{y}\sin\beta\right), \quad \gamma_1 = 0, \quad 
    \gamma_2 = K\frac{(a + c\cos\beta)\left(y^2 + c(c + a\cos\beta)\right)}
    {6y^2\sqrt{y^2 - c^2\sin^2\beta}} c\,a\sin\beta
\end{align}
As we can see, the first order terms $d_2$, $\alpha_1$ and $\gamma_1$ all vanish.
\begin{figure}[h!]
    \centering\includegraphics[width=.4\linewidth]{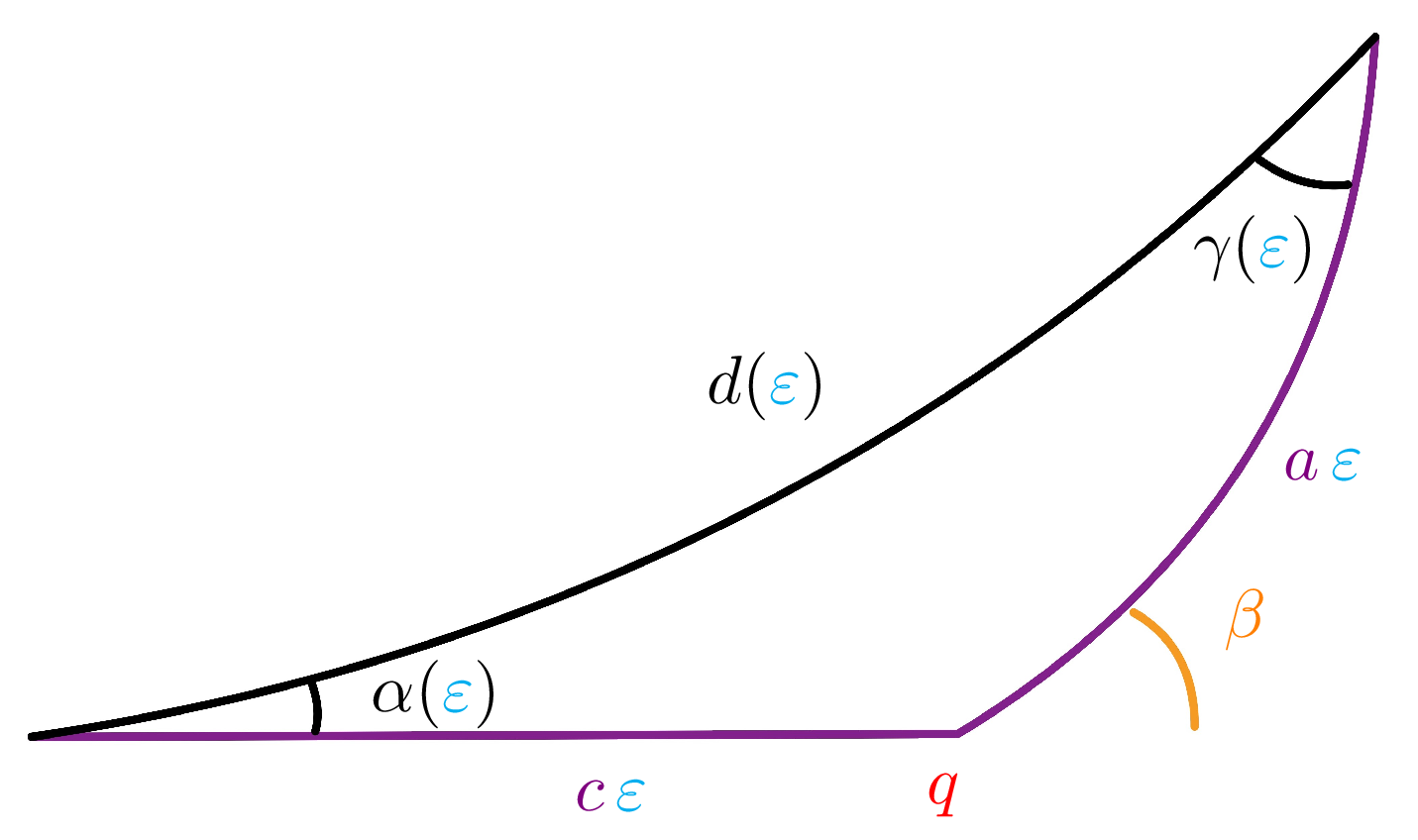}
    \caption{A sketch of Problem 1.}
    \label{fig: Problem 1}
\end{figure}

\subsubsection{Problem 2, Type 1}\label{sec: Problem 2, Type 1}
The segment $\Delta_1$ is also a triangle (Fig.~\ref{fig: Problem 2, Type 1}). We express the angles $\alpha_o$, $\hat{\alpha}$ and the side length $b$ in terms of the still unknown $\alpha$ and $d$. We assume, that the first order terms $\alpha_1$ and $d_2$ vanish.\\
\begin{align}
    \left.\begin{matrix}
    b(\varepsilon) = C(K;\pi-\alpha(\varepsilon),c\varepsilon,d(\varepsilon)) \\
    \hat{\alpha}(\varepsilon) = SC(K;\pi-\alpha(\varepsilon),c\varepsilon,d(\varepsilon))
    \end{matrix}\right\} \quad \Rightarrow \quad
    \begin{matrix}
    a_1 \mapsto c, & a_2 \mapsto 0, & a_3 \mapsto 0 \\
    b_1 \mapsto d_1, & b_2 \mapsto 0, & b_3 \mapsto d_3 \\
    \gamma_0 \mapsto \pi-\alpha_0, & \gamma_1 \mapsto 0, & \gamma_2 \mapsto -\alpha_2
    \end{matrix}
\end{align}

The substitute $y$ in this case takes the form:
\begin{equation}
    y=\sqrt{c^2 + d_1^2 + 2c\,d_1\cos\alpha_0}
\end{equation}

We get the following coefficients:
\begin{align}
    b_1 &= y & b_2 &= 0, & b_3 &= -\frac{1}{y}\left(
    \frac{K}{6}c^2d_1^2\sin^2\alpha_0 - d_3(d_1 + c\cos\alpha_0) + c\,d_1\alpha_2\sin\alpha_0 \right) \\
    \hat{\alpha}_0 &= \sin^{-1}\left(\frac{c}{y}\sin\alpha_0\right), &
    \hat{\alpha}_1, &= 0 &
    \hat{\alpha}_2 &= \frac{c}{y^2}\left( \frac{K}{6}d_1\sin\alpha_0\left[y^2 + c(c + d_1\cos\alpha_0)\right]
    + \alpha_2(c + d_1\cos\alpha_0) - d_3\sin\alpha_0 \right) \notag
\end{align}

For the angle $\alpha^0$ the again the arguments are swapped:
\begin{equation}
    \alpha^0(\varepsilon) = SC(K;\pi-\alpha(\varepsilon),d(\varepsilon),c\varepsilon), \quad
    \Rightarrow \quad a_1 \mapsto d_1, \quad a_3 \mapsto d_3, \quad b_1 \mapsto c,
    \quad b_3 \mapsto 0
\end{equation}

and we get the coefficients:
\begin{align}
    \alpha_0^0 &= \sin^{-1}\left( \frac{d_1}{y}\sin\alpha_0 \right), \quad
    \alpha_1^0 = 0, \\
    \alpha_2^0 &= \frac{1}{y^2}\left( \frac{K}{6}c\,d_1 \left[y^2 + d_1(d_1 + c\cos\alpha_0)\right]\sin\alpha_0
    + \alpha_2\,d_1(d_1 + c\cos\alpha_0) + d_3\,c\sin\alpha_0 \right) \notag
\end{align}

\begin{figure}[h!]
    \centering\includegraphics[width=.6\linewidth]{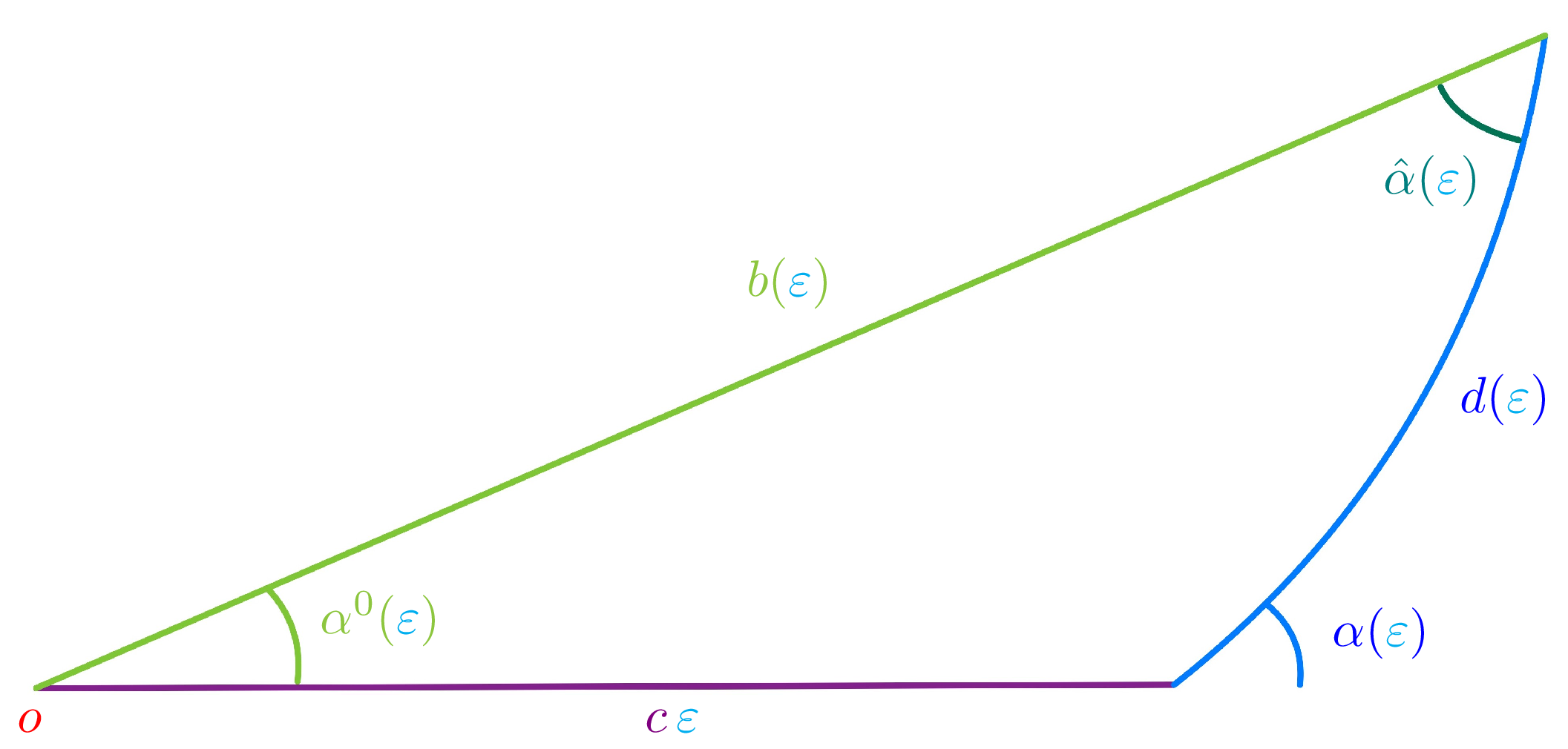}
    \caption{A sketch of Problem 2, Type 1}
    \label{fig: Problem 2, Type 1}
\end{figure}

\subsubsection{Problem 2, Types 2 and 3}\label{subsubsec: Problem 2, Types 2 and 3}
The two previous cases covered the boundary problems, where we had triangles. In this section we cover the recurring case, where we have squares, which we divide into two triangles (Fig.~\ref{fig: Problem 2, Types 2 & 3}). This makes the solution more complicated, because we cannot straight forwardly apply the cosine- and sine-laws. It seems at first glance that we have to few information, to solve it at all.\\
We have the same input $c$, $d$ and $\alpha$, but we are now supposed to calculate $2$ side-lengths and $2$ angles. Starting with the lower left triangle we just have the base-line $c$ and a right angle, so we would have to calculate more quantities from the upper right triangle to solve it. But if we look at the upper right triangle, we are missing the base-line. The additional piece of information, we need to solve this problem, comes from the fact that the top-sine of the slice $\Sigma_1$ is a straight line in faithful normal coordinates. This allows us set up an equation, describing the angle $\delta$ from both sides, and solve it for $a$. With that everything is determined in terms of $d$ and $\alpha$ as we wanted.
\begin{figure}[h!]
    \centering\includegraphics[width=.7\linewidth]{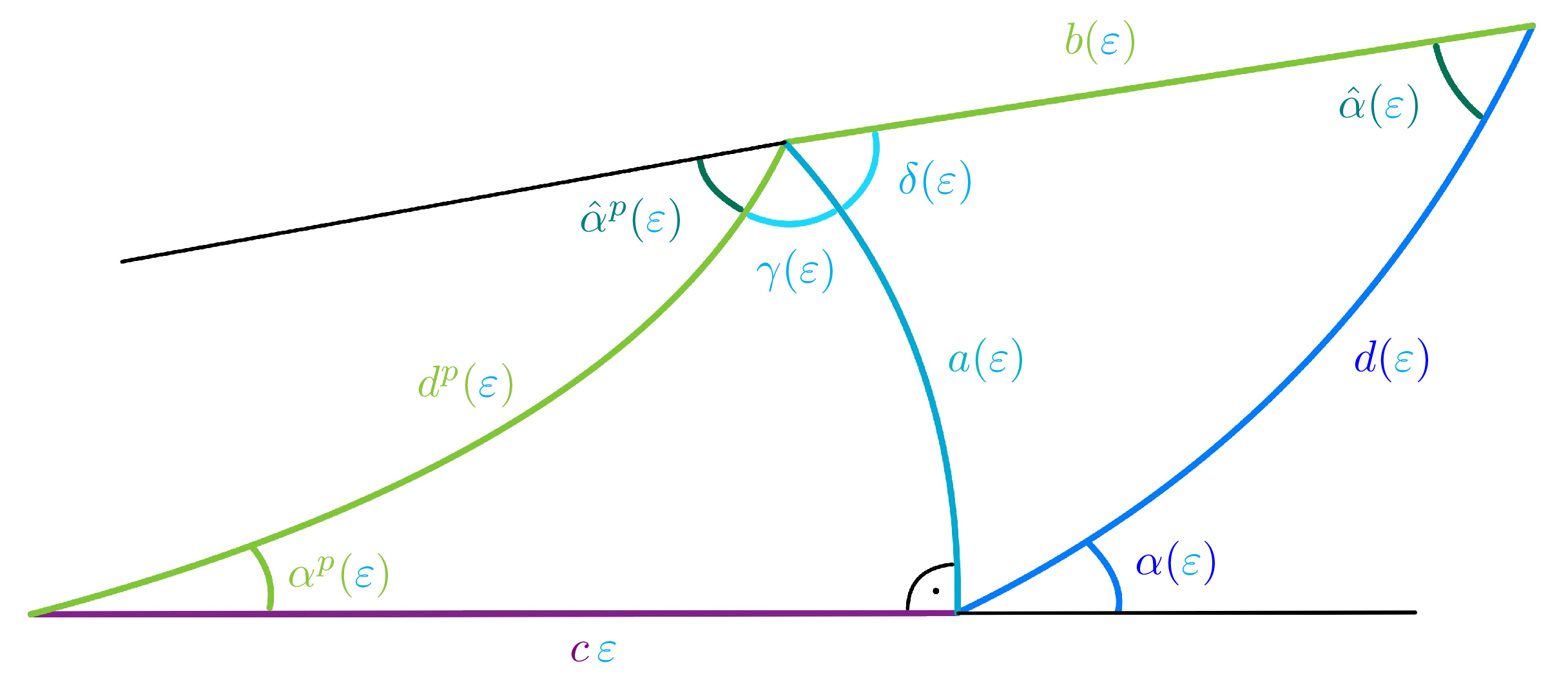}
    \caption{A sketch of the Problem 2, Types 2 and 3.}
    \label{fig: Problem 2, Types 2 & 3}
\end{figure}

We start by determining the previous angle $\alpha^p$ and side-length $d^p$, which we used as unknown variables in the previous segment $\Delta_{n-1}$, in terms of $a$, which is for the first part of this problem regarded as the new unknown. And we again assume, that the first order of the input variables vanish i.e. $a_2 = 0$ in this case:
\begin{align}
    \left.\begin{matrix}
        d^p(a(\varepsilon)) = C(K;\frac{\pi}{2},a(\varepsilon),c\varepsilon) \\
        \alpha^p(a(\varepsilon)) = SC(K;\frac{\pi}{2},a(\varepsilon),c\varepsilon)
    \end{matrix}\right\} \quad \Rightarrow \quad
    \begin{matrix}
    a_1 \mapsto a_1, & a_2 \mapsto 0, & a_3 \mapsto a_3 \\
    b_1 \mapsto c, & b_2 \mapsto 0, & b_3 \mapsto 0 \\
    \gamma_0 \mapsto \frac{\pi}{2}, & \gamma_1 \mapsto 0, & \gamma_2 \mapsto 0
    \end{matrix}
\end{align}

In this problem we are going to have two versions of the substitute $y$.
\begin{equation}
    y_1 = \sqrt{a_1^2 + c^2}, \qquad y_2 = \sqrt{a_1^2 + d_1^2 - 2a_1d_1\sin\alpha_0}
\end{equation}

The coefficients for $\alpha^p(a(\varepsilon))$ and $d^p(a(\varepsilon))$, dependent on $a(\varepsilon)$ are:
\begin{align}
    d^p_1 &= y_1, & d^p_2 &= 0, & d^p_3 &= -\frac{a_1}{y_1}\left( \frac{K}{6}c^2a_1 - a_3 \right) \\
    \alpha^p_0 &= \sin^{-1}\frac{a_1}{y}, & \alpha^p_1 &= 0, &
    \alpha^p_2 &= \frac{c}{y^2}\left( \frac{K}{6}a_1(y^2 + a_1^2) + a_3 \right)
\end{align}

Next we describe the angle $\delta$ from both sides. We use the fact, that the angle $\hat{\alpha}^p$ from the previous segment, the angle $\gamma$ of the lower left triangle, which we calculate via sine law using $a$, $c$ and $\alpha^p$, of the current segment and $\delta$ add up to $\pi$. On the other hand we can calculate $\delta$ using the sine law in the upper right triangle using $a$, $d$ and $\alpha$. We can then solve this equation for $a$ to express it in terms of $\alpha$ and $d$.\\
The angle $\delta$ is larger than $\frac{\pi}{2}$ for sufficiently thin slices and thus we have to use the side-branch of the $SC$-function, Eq.~\eqref{eq: SC-branch}. The sketch Fig.~\ref{fig: Problem 2, Types 2 & 3} suggests otherwise but the plot in Fig.~\ref{fig: 1st Slice Test-plot} offers a better comparison.
\begin{equation}
    \delta = \pi - SC\left(K;\frac{\pi}{2},c\varepsilon,a(\varepsilon)\right) - \hat{\alpha}^p(a(\varepsilon))
    = \pi - SC\left(K;\frac{\pi}{2}-\alpha(\varepsilon),d(\varepsilon),a(\varepsilon)\right)
\end{equation}

We get an equation for every order which becomes linear, when one inserts the solution of the previous order equation, as it is always the case in perturbation theory. Since there is no lower order, the $0$-th order equation will remain non-linear.
\begin{align}
    \delta^{(0)} &= \pi - \sin^{-1}\frac{c}{y_1} - \hat{\alpha}^p_0
    \overset{!}{=} \pi - \sin^{-1}\left(\frac{d_1}{y_2}\right)\cos\alpha_0 \qquad
    \delta^{(1)} = 0 \overset{!}{=} 0 \\
    \delta^{(2)} &= -\frac{1}{y_1^2}\left( \frac{K}{6}a_1c(y_1^2 + c^2) - c a_3 \right) - \hat{\alpha}^p_2 \notag \\
    &\overset{!}{=} \frac{1}{y_2^2}\left(
    \frac{K}{6}a_1d_1\left[ y_2^2 + d_1(d_1 - a_1\sin\alpha_0) \right]\cos\alpha_0
    + (a_1d_3 - a_3d_1)\cos\alpha_0 + d_1\alpha_2(d_1 - a_1\sin\alpha_0) \right)
\end{align}

If we drop the assumption on $a_2$ i.e. $a_2 \overset{i.g.}{\neq} 0$, for the first order equation we get:
\begin{equation}
    \delta^{(1)} = a_2\frac{c}{a_1^2 + c^2} \overset{!}{=}
    - a_2\frac{d_1\cos\alpha_0}{a_1^2 + d_1^2 - 2a_1d_1\sin\alpha_0}
\end{equation}
The terms, to which $a_2$ is multiplied in this equation are in general not proportional to each other and thus $a_2$ has to be $0$, to satisfy the equation. So, we confirmed the assumption we made in the beginning of this section.

Having taken care of the core problem, we now calculate the quantities we need to solve the slice problem.
\begin{align}
    \left.\begin{matrix}
        b(\varepsilon) = C(K;\frac{\pi}{2}-\alpha(\varepsilon),a(\varepsilon),d(\varepsilon)) \\
        \hat{\alpha}(\varepsilon) = SC(K;\frac{\pi}{2}-\alpha(\varepsilon),a(\varepsilon),d(\varepsilon))
    \end{matrix}\right\} \quad \Rightarrow \quad
    \begin{matrix}
    a_1 \mapsto a_1, & a_2 \mapsto 0, & a_3 \mapsto a_3 \\
    b_1 \mapsto d_1, & b_2 \mapsto 0, & b_3 \mapsto d_3 \\
    \gamma_0 \mapsto \frac{\pi}{2} - \alpha_0, & \gamma_1 \mapsto 0, & \gamma_2 \mapsto -\alpha_2
    \end{matrix}
\end{align}
The coefficients of $b(\varepsilon)$ and $\hat{\alpha}(\varepsilon)$ expressed in terms of $d(\varepsilon)$ and $\alpha(\varepsilon)$ are:
\begin{align}
    b_1 &= y_2, \qquad b_2 = 0 \notag \\
    b_3 &= -\frac{1}{y_2}\left( \frac{K}{6}a_1^2d_1^2\cos^2\alpha_0 - d_3(d_1 - a_1\sin\alpha_0)
    + a_1d_1\alpha_2\cos\alpha_0 - a_3(a_1 - d_1\sin\alpha_0) \right) \\
    \hat{\alpha}_0 &= \sin^{-1}\left(\frac{a_1}{y_2}\cos\alpha_0\right), \qquad \hat{\alpha}_1 = 0 \notag \\
    \hat{\alpha}_2 &= \frac{1}{y_2^2}\left(
    \frac{K}{6}a_1d_1\left[ y_2^2 + a_1(a_1 - d_1\sin\alpha_0) \right]\cos\alpha_0
    + (a_3d_1 - a_1d_3)\cos\alpha_0 + a_1\alpha_2(a_1 - d_1\sin\alpha_0) \right)
\end{align}
We observe, that also here the first order terms vanish under the assumption, that $\alpha_1$ and $d_1$ vanish.

\subsubsection{The first order problem indeed vanishes}
We saw, that in every type of every problem the first order results vanished, if the first order of the input variable vanished. Now arguing the procedure backwards, the higher order terms of the input variables in problem one are non existent since the input variables are precisely $a\varepsilon$ and $c\varepsilon$. But the first order results of Problem 1 are the input variables $\alpha$ and $d$ of Problem 2 Type 3 and thus the vanishing of the results in this case, which are the input variables for Problem 2 Type 2, is confirmed. This again confirms the vanishing of the first order in Problem 2 Type 2 and following the same logic from right to left in the sketch Fig.~\ref{fig: First_Slice} we see that also the input variables of Problem 2 Type 1 have vanishing first order coefficients and thus we confirm that the solution to the first order problem is, that all first order coefficients vanish.

\subsection{The $0$-th order problem}
We do a perturbation to linearize the problem, so we can solve it, but the zeroth order will remain non-linear. So, we will not apply the same procedure to it as we do with the second order and instead solve it directly. Which we can, since it is a triangle on a flat surface (Euclidean plane). Fig.~\ref{fig: 0-th order scheme} shows a sketch of the $0$-th order problem.\\
\begin{figure}[h!]
    \centering\includegraphics[width=.8\linewidth]{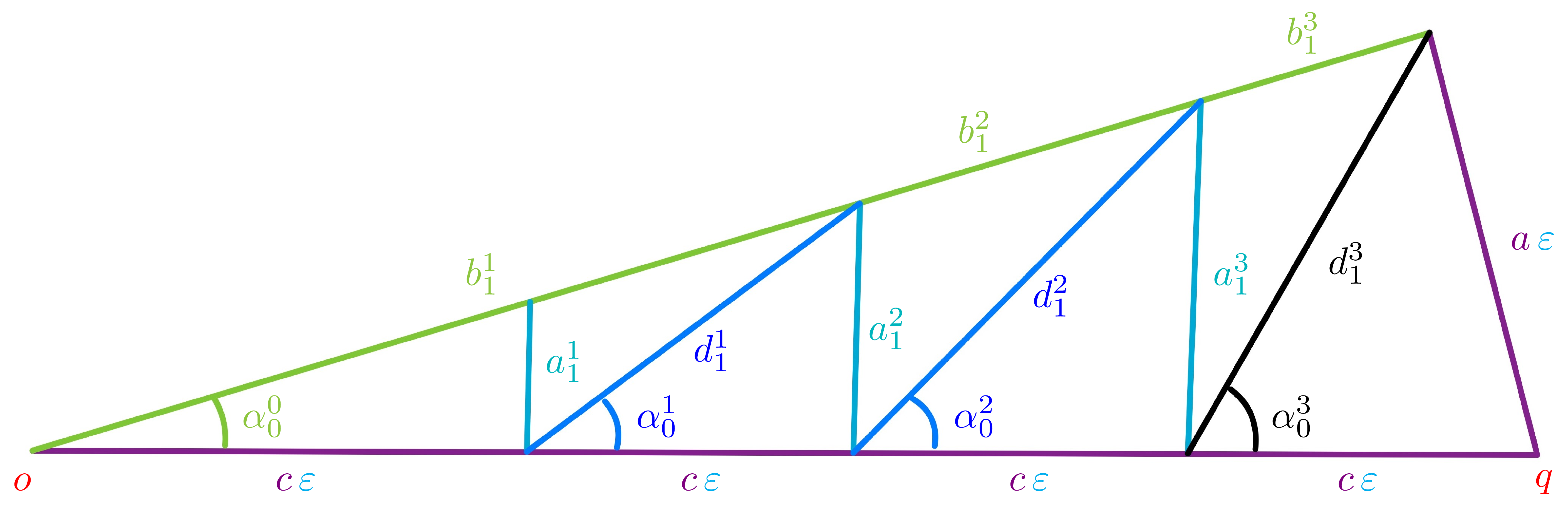}
    \caption{We sketch the zeroth order problem as an easy reference for the labelling conventions. We observe, that all lengths are of first order in $\varepsilon$, while all angles are of zeroth order.}
    \label{fig: 0-th order scheme}
\end{figure}

We will first treat the more generic case and postpone the special cases to the subsequent subsections.

\subsubsection{General 0-th order solution}\label{sec: general 0-th order sol}
We start with the quantities we actually want to calculate as end results $\bar{b}_0\,$, $\alpha^0_0$ of the first slice $\Sigma_1$ and will then use these to calculate the 0-th order terms of the intermediate results, which we need to calculate the higher order terms of $\bar{b}$ and $\alpha^0$.\\
We only need a straight forward application of the cosine-law to get $b_1$ and using the result for $b_1$ we calculate $\alpha_0$ via sine-law.
\begin{equation}
    y \coloneq \sqrt{a^2 + N^2c^2 + 2Nac\cos\beta}, \qquad
    \bar{b}_0 = \sum_{i=1}^{N-1}b_1^i\varepsilon = y\,\varepsilon, \qquad
    \alpha^0_0 = \sin^{-1}\left(\frac{a}{y}\sin\beta\right),
\end{equation}
where we used, that $\cos(\pi-\beta) = -\cos\beta$, which we will use quite often along with $\sin(\pi-\beta) = \sin\beta$.\\

We can use the Euclidean sine-law again to calculate the $a_1^i$. Since we divided the base-line in $N$ equidistant parts, all $b_1^i$ are the same, except of the last one. The $\tilde{b}_1^i$ are defined without merging the first two triangles together to segment $\Delta_1$ and are for $i>2$ related to the $b_1^i$ we otherwise use by: $\tilde{b}_1^i = b_1^{i-1}$.
\begin{align}
    \frac{a_1^i}{\sin\alpha_0^0}=\frac{\sum_{j=1}^i \tilde{b}_1^j}{\sin\frac{\pi }{2}}, \qquad \tilde{b}_1^i\cos\alpha_0^0 = c, \quad \forall i<N \qquad
    \Rightarrow \quad a_1^i = ic\frac{\sin\alpha_0^0}{\cos\alpha_0^0}=ic\tan\alpha_0^0
\end{align}
\begin{equation}
    \Rightarrow \quad a_1^i = \frac{i\,c\,a\sin\beta}{Nc + a\cos\beta}, \qquad \forall i \in \{1,..,N-1\}
\end{equation}
Subsequently using the Euclidean cosine- and then sine-law we can determine $d_1^i$ and then $\alpha_0^i$:
\begin{equation}
    y_1  = \sqrt{\left(a_1^{i+1}\right)^2 + c^2}, \qquad d_1^i = y_1, \qquad
    \alpha_0^i = \sin^{-1}\frac{a_1^{i+1}}{y_1}
\end{equation}
\begin{equation}
    \Rightarrow \quad d_1^i = c\frac{\sqrt{(Nc + a\cos\beta)^2 + (i+1)^2a^2\sin^2\beta}}{Nc + a\cos\beta},
    \quad \alpha_0^i = \sin^{-1}\frac{(i+1)a\sin\beta}{\sqrt{(Nc + a\cos\beta)^2 + (i+1)^2a^2\sin^2\beta}}
\end{equation}
This holds $\forall i < N-1$ i.e. for all triangles except of the last one, which is treated in Sec.~\ref{sec: DN P1}.\\

Having determined the $d_1^i$ and $\alpha_0^i$ we can use them to calculate the $b_1^i$ and $\hat{\alpha}_0^i$:
\begin{equation}
    y_2 = \sqrt{\left(a_1^i\right)^2 + \left(d_1^i\right)^2 - 2a_1^id_1^i\cos\left(\frac{\pi}{2} - \alpha_0^i\right)}, \qquad b_1^i = y_2 = \frac{c}{\cos\alpha_0}, \qquad
    \hat{\alpha}_0^i = \sin^{-1}\left( \frac{a_1^i}{y_2}\sin\left(\frac{\pi}{2}-\alpha_0^i\right) \right)
\end{equation}
\begin{align}
    \Rightarrow \quad b_1^i &= c\frac{\sqrt{a^2 + N^2c^2 + 2Na\,c\cos\beta}}{Nc + a\cos\beta}, \quad \forall i\in\{2,\ldots,N-2\} \notag \\
    \hat{\alpha}_0^i &= \sin^{-1}\frac{i\,a(Nc + a\cos\beta)\sin\beta}{\sqrt{a^2 + N^2c^2 + 2Na\,c\cos\beta}\sqrt{(Nc + a\cos\beta)^2 
    + (i+1)^2a^2 \sin^2\beta}}, \qquad \forall i\in\{2,\ldots,N-2\}
\end{align}
We cover the special case of the first triangle in Sec.~\ref{sec: D1 P2 T1}. The result holds however for $\tilde{b}_1^1$ and $\tilde{b}_1^2$.\\

\subsubsection{$\Delta_N$: Problem 1}\label{sec: DN P1}
For this example, we can directly use the $0$-th order of this Problem, which we derived in last section:
\begin{equation}
    y = \sqrt{a^2 + c^2 + 2ac\cos\beta}, \qquad d_1^{N-1} = y, \qquad
    \alpha_0^{N-1} = \sin^{-1}\left(\frac{a}{y}\sin\beta\right), \qquad
    \gamma_0^{N-1} = \sin^{-1}\left(\frac{c}{y}\sin\beta\right)
\end{equation}
\begin{equation}
    d_1^{N-1} = \sqrt{a^2 + c^2 + 2ac\cos\beta}, \quad
    \alpha_0^{N-1} = \sin^{-1}\frac{a\sin\beta}{\sqrt{a^2 + c^2 + 2ac\cos\beta}}, \quad
    \gamma_0^{N-1} = \sin^{-1}\frac{c\sin\beta}{\sqrt{a^2 + c^2 + 2ac\cos\beta}}
\end{equation}
This result is independent of $N$ and always describes the last triangle.

\subsubsection{$\Delta_{N-1}$: Problem 2, Type 3}
The lower triangle $\overset{\vee}{\Delta}_{N-1}$ works the same as $\overset{\vee}{\Delta}_i$, $i\in\{2,..,N-2\}$. Thus $a_1^i$ can be used for $i = N-1$:\\
\begin{equation}
    a_1^{N-1} = \left.a_1^i\right|_{i=N-1} = \frac{(N-1)\,c\,a\sin\beta}{Nc + a\cos\beta}
\end{equation}
The other quantities which are calculated in the upper triangle $\hat{\Delta}_{N-1}$ change, due to the angle being $\beta$ instead of $\beta_i = \frac{\pi}{2}$. The $d_1^{N-1}$ and $\alpha_0^{N-1}$ we got from $\Delta_N$. So, we can use them to calculate $b_1^{N-1}$ and $\hat{\alpha}_0^{N-1}$.
\begin{align}
    b_1^{N-1} = \sqrt{a^2 + N^2c^2 + 2Na\,c\cos\beta}\frac{c + a\cos\beta}{Nc + a\cos\beta}, \quad
    \hat{\alpha}_0^{N-1} = \sin^{-1}\frac{(N-1)c\,a\sin\beta}
    {\sqrt{a^2 + c^2 + 2a\,c\cos\beta}\sqrt{a^2 + N^2c^2 + 2Na\,c\cos\beta}}
\end{align}

\subsubsection{$\Delta_1$: Problem 2, Type 1}\label{sec: D1 P2 T1}
When we glue the first two triangles together we get another triangle. Thus we create a slightly different case then above, to simplify the problem:
\begin{equation}
    y = \sqrt{c^2 + \left(d_1^1\right)^2 - 2cd_1^1\cos\left(\pi -\alpha_0^1\right)}, \qquad b_1^1 = y, \qquad \gamma_0^1 = \sin^{-1}\left(\sin\left(\pi - \alpha_0^1\right)\frac{c}{y}\right)
\end{equation}
\begin{align}
    \Rightarrow \quad b_1^1 &= 2c\frac{\sqrt{a^2 + N^2c^2 + 2Nac\cos\beta}}{Nc + a\cos\beta} = 2y_2, \notag \\
    \hat{\alpha}_0^1 &= \sin^{-1}\frac{(Nc + a\cos\beta)a\sin\beta}{\sqrt{a^2 + N^2c^2 + 2Nac\cos\beta}
    \sqrt{(Nc + a\cos\beta)^2 + 4a^2\sin^2\beta}}
\end{align}
The merging only affects the definition of $b^1 \coloneq \tilde{b}^1 + \tilde{b}^2$ but not $\hat{\alpha}^1$, thus $\gamma_0^1 = \hat{\alpha}_0^1$.\\

The side length $d_1^1$ and angle $\alpha_0^1$ are no special cases and in fact also $\alpha_0^0$ is consistent with $\alpha_0^i$ for $i=0$:
\begin{equation}\label{eq: alpha_0^i, i->0}
    \left.\alpha_0^i\right|_{i=0} = \sin^{-1}\frac{a\sin\beta}{\sqrt{(Nc + a\cos\beta)^2 + a^2\sin^2\beta}}
    \overset{\eqref{eq: kN1 identity}}{=} \sin^{-1}\frac{a\sin\beta}{\sqrt{a^2 + N^2c^2 + 2Na\,c\cos\beta}} = \alpha_0^0
\end{equation}

\subsection{The $2$-nd order problem}\label{subsec: The $2$-nd order problem}
We will now insert the zeroth and first order (which are all zero) solutions into the second order terms and simplify them. Since there are often reoccurring terms we can define the following substitutes, to simplify our results.
\begin{align}\label{eq: substitutes}
    &X_N \coloneq Nc + a\cos\beta, \quad e^i \coloneq i\,a\sin\beta, \quad Z_N = a + Nc\cos\beta\\
    &Y_N \coloneq \sqrt{a^2 + N^2c^2 + 2Nac\cos\beta}, \quad k_N^i = \sqrt{X_N^2 + (e^i)^2}
    \quad \Rightarrow \quad Y_N = k_N^1 \label{eq: substitute relation}
\end{align}

We check, that the relation in~\eqref{eq: substitute relation} holds:
\begin{align}\label{eq: kN1 identity}
    k_N^1 = \sqrt{N^2c^2 + 2Nac\cos\beta + a^2\cos^2\beta + a^2\sin^2\beta} = \sqrt{N^2c^2 + 2Nac\cos\beta + a^2} = Y_N.
\end{align}

\subsubsection{$\Delta_N$: Problem 1}
In this most simple case there is not much to adjust and we can directly apply the second order results from Sec.~\ref{sec: Problem 1}, simplify some terms and rewrite them, using the new substitutes.
\begin{align}
    d_3^{N-1} =& -\frac{K_Nc^2a^2\sin^2\beta}{6\sqrt{a^2 + c^2 + 2a\,c\cos\beta}}
    = -K_N\frac{c^2(e^1)^2}{6\,Y_1}, \\
    \alpha_2^{N-1} =& \frac{K_N}{6}\left( 1 + \frac{a(a + c\cos\beta)}{a^2 + c^2 + 2a\,c\cos\beta}
    \right)c\,a\sin\beta = \frac{K_N}{6}\left( 1 + \frac{aZ_1}{Y_1^2} \right)c\,e^1 \\
    \gamma_2^{N-1} =& \frac{K_N}{6}\left( 1 + \frac{c(c + a\cos\beta)}{a^2 + c^2 + 2a\,c\cos\beta}
    \right)a\,c\sin\beta = \frac{K_N}{6}\left( 1 + \frac{cX_1}{Y_1^2} \right)c\,e^1
\end{align}

\subsubsection{$\Delta_1$: Problem 2, Type 1}
In this case we do not get the results explicitly in terms of the input variables $a$, $c$ and $\beta$, but as functions of the as of yet unknown variables $d_3^1$ and $\alpha_2^1$. We therefore have to write the results from Sec.~\ref{sec: Problem 2, Type 1} more specifically.
\begin{align}
    y &= \sqrt{c^2 + (d_1^1)^2 + 2c\,d_1^1\cos\alpha_0^1} = b_1^1; \notag \\
    b_3^1(\alpha_2^1,d_3^1) &= -\frac{1}{y}\left(
    \frac{K_1}{6}c^2(d_1^1)^2\sin^2\alpha_0^1 - d_3^1(d_1^1 + c\cos\alpha_0^1)
    + \alpha_2^1\,c\,d_1^1\sin\alpha_0^1 \right) \\
    \hat{\alpha}_2^1(\alpha_2^1,d_3^1) &= \frac{c}{y^2}\left( \frac{K_1}{6}d_1^1\sin\alpha_0^1\left[y^2 + c(c + d_1^1\cos\alpha_0^1)\right] - d_3^1\sin\alpha_0^1 + \alpha_2^1(c + d_1^1\cos\alpha_0^1) \right) \\
    \alpha_2^0(\alpha_2^1,d_3^1) &= \frac{1}{y^2}\left( \frac{K_1}{6}c\,d_1^1 \left[y^2 + d_1^1(d_1^1
    + c\cos\alpha_0^1)\right]\sin\alpha_0^1 + d_3^1\,c\sin\alpha_0^1
    + \alpha_2^1\,d_1^1(d_1^1 + c\cos\alpha_0^1) \right)
\end{align}

Now we can use the zeroth order solutions $d_1^1$ and $\alpha_0^1$ to linearize the problem and simplify the expressions.
\begin{align}
    b_3^1(\alpha_2^1,d_3^1) =& -\frac{1}{\sqrt{a^2 + N^2c^2 + 2Nac\cos\beta}}
        \left( \frac{K_1a^2c^3\sin^2\beta}{3(Nc + a\cos\beta)} + ac\alpha_2^1\sin\beta
        - d_3^1\frac{(Nc + a\cos\beta)^2 + 2a^2\sin^2\beta}{\sqrt{(Nc + a\cos\beta)^2 + 4a^2\sin^2\beta}} \right), \notag \\
     =& -\frac{1}{Y_N}\left( \frac{K_1c^3(e^1)^2}{3X_N} + \alpha_2^1c\,e^1
        - d_3^1\frac{X_N^2 + 2(e^1)^2}{k_N^2} \right) \\
    \hat{\alpha}_2^1(\alpha_2^1,d_3^1) =& \frac{(Nc + a\cos\beta)^2}{a^2 + N^2c^2 + 2Na\,c\cos\beta}
        \left( \frac{K_1c^2a\sin\beta}{6(Nc + a\cos\beta)} \left[
        3 + \frac{2\,a^2\sin^2\beta}{(Nc + a\cos\beta)^2} \right]
        \vphantom{\left[\frac{d_3^1}{\sqrt{(Nc + a\cos\beta)^2 + 4\,a^2\sin^2\beta}}\right]}\right.\notag \\
        &\left. + \frac{\alpha _2^1}{2} - \frac{a\sin\beta d_3^1}{2\,c\sqrt{(Nc + a\cos\beta)^2 + 4a^2\sin^2\beta}} \right) \notag \\
     =& \frac{X_N^2}{Y_N^2}\left( \frac{K_1c^2e^1}{6X_N} \left[ 3 + \frac{2(e^1)^2}{X_N^2} \right]
        + \frac{\alpha_2^1}{2} - \frac{e^1}{2}\frac{d_3^1}{ck_N^2} \right) \\
    \alpha_2^0(\alpha_2^1,d_3^1) =& \frac{(Nc + a\cos\beta)^2}{2(a^2 + N^2c^2 + 2Na\,c\cos\beta)}
        \left( \frac{K_1a\,c^2\sin\beta}{3(Nc + a\cos\beta)}
        \left[ 3 + \frac{4a^2\sin^2\beta}{(Nc + a\cos\beta)^2} \right]
        \vphantom{\left[\frac{d_3^1}{\sqrt{(Nc+a\cos\beta)^2+4a^2\sin^2\beta}}\right]}\right.\notag\\
     &\left. + \frac{d_3^1a\sin\beta}{c\sqrt{(Nc + a\cos\beta)^2 + 4a^2\sin^2\beta}}
        + \alpha_2^1\left( 1 + \frac{2a^2\sin^2\beta}{(Nc + a\cos\beta)^2} \right) \right) \notag \\
     =& \frac{X_N^2}{2Y_N^2}\left( \frac{K_1c^2\,e^1}{3X_N}\left[ 3 + 4\frac{(e^1)^2}{X_N^2} \right]
        + \frac{d_3^1e^1}{c\,k_N^2} + \alpha_2^1\left( 1 + 2\frac{(e^1)^2}{X_N^2} \right) \right)
\end{align}

\subsubsection{$\Delta_i$: Problem 2, Type 2}
First we calculate the second order term of $d$ and $\alpha$ of the previous segment in terms of $a$, using the expressions for $d_3^p$ and $\alpha_2^p$ we derived in Sec.~\ref{subsubsec: Problem 2, Types 2 and 3}:
\begin{equation}
    y_1 = \sqrt{(a_1^i)^2 + c^2}, \quad
    d_3^{i-1}(a_3^i) = -\frac{a_1^i}{y_1}\left( \frac{K_i}{6}c^2a_1^i - a_3^i \right), \quad
    \alpha_2^{i-1}(a_3^i) = \frac{c}{y_1^2}\left( \frac{K_i}{6}a_1^i\left[ y_1^2 + (a_1^i)^2 \right] + a_3^i \right)
\end{equation}

After inserting the lower order solutions we get $d_3^{i-1}$ and $\alpha_2^{i-1}$ as functions of $a_3^i$:
\begin{align}
    d_3^{i-1}(a_3^i) =& \frac{i\,a\sin\beta}{\sqrt{(Nc + a\cos\beta)^2 + i^2a^2\sin^2\beta}}
    \left( a_3^i - \frac{i K_ic^3\,a\sin\beta}{6(Nc + a\cos\beta)} \right)
    = \frac{e^i}{k_N^i}\left( a_3^i - \frac{K_ic^3e^i}{6X_N} \right) \\
    \alpha_2^{i-1}(a_3^i) =& \frac{(Nc + a\cos\beta)^2}{(Nc + a\cos\beta)^2 + i^2a^2\sin^2\beta}
    \left( \frac{a_3^i}{c} + i\,K_ic^2\,a\sin\beta\frac{(Nc + a\cos\beta)^2 + 2i^2a^2\sin^2\beta}
    {6(Nc + a\cos\beta)^3} \right) \notag \\
    =& \frac{X_N^2}{(k_N^i)^2}\left( \frac{a_3^i}{c}
    + \frac{K_i}{6}c^2e^i\frac{(k_N^i)^2 + (e^i)^2}{X_N^3} \right)
\end{align}

This type of problem is the one, which is iteratively applied. This means, that for a general segment $\Delta_i$ the $\hat{\alpha}^{i-1}$ of the previous segment is used in the $\delta$-equation to solve for $a$, which is then ultimately used, to calculate $\hat{\alpha}^i$, which will then appear in the next segments $\delta$-equation.\\
This means, that $\hat{\alpha}^1$ will become a function of a function of a function and so on: $\hat{\alpha}^1(\alpha^1(\alpha^2(\ldots),d^2(\ldots)),d^1(\ldots))$ and that the expressions for $\hat{\alpha}^i$ would become increasingly more complicated the further right we go, since we would essentially pack the entire problem of $N$ segments into one function until we can finally insert the last piece and solve everything at once. Whilst we can brute force this for a given $N$, this is not feasible for a general $N$, since we would need a general form for $\hat{\alpha}^i$ for a generic $i$. But, to calculate the limit of $N\to\infty$, we need the solution for a general $N$.\\
We can obtain such a general form by choosing a linear ansatz and adjusting the some prefactors for convenience, which we found by trial and comparison with brute forcing $N=2$ and $N=3$:
\begin{align}\label{eq: Ansatz}
    \hat{\alpha}_2^i(\alpha_2^i,d_3^i) =& \frac{(Nc + a\cos\beta)^2}{a^2 + N^2c^2 + 2Na\,c\cos\beta}\left(
    \frac{c^2a\sin\beta}{6(Nc + a\cos\beta)^3}C^i + A^i\alpha_2^i
    - \frac{D^i d_3^i}{c\sqrt{(Nc + a\cos\beta)^2 + (i+1)^2a^2\sin^2\beta}} \right) \notag \\
    =& \frac{X_N^2}{Y_N^2}\left( \frac{c^2e^1}{6X_N^3}C^i + A^i\alpha_2^i - \frac{D^i d_3^i}{c\,k_N^{i+1}} \right)
\end{align}
Since we are using perturbation theory all higher quantities and equations will be linear in the variable in that order and $\hat{\alpha}_2^i$ can thus always be written in this form.\\

Inserting the second order term of the unknowns $d$ and $\alpha$ of the previous segment, which we just calulated above into this ansatz, we get $\hat{\alpha}_2^{i-1}$ of the previous segment as a function of $a_3^i$.
\begin{align}
    \hat{\alpha}_2^{i-1}&(a_3^i) = \frac{X_N^2}{Y_2^2}\left( \frac{c^2e^1}{6X_N^3}C^{i-1}
    + A^{i-1}\alpha_2^{i-1}(a_3^i) - \frac{D^{i-1}}{c\,k_N^i}d_3^{i-1}(a_3^i) \right) \notag \\
    =& \frac{(Nc + a\cos\beta)^2}{a^2 + N^2c^2 + 2Na\,c\cos\beta}\left(
    \frac{c^2a\sin\beta}{6(Nc + a\cos\beta)^3}C^{i-1} \right.\notag\\
    &\left. + K_ic^2i\,a\sin\beta\frac{A^{i-1}\left[ (Nc + a\cos\beta)^2 + 2i^2a^2\sin^2\beta\right]
    + D^{i-1}i\,a\sin\beta}{6(Nc + a\cos\beta) \left[ (Nc + a\cos\beta)^2 + i^2a^2\sin^2\beta \right]}
    + \frac{a_3^i}{c}\frac{A^{i-1}(Nc + a\cos\beta)^2 - D^{i-1}i\,a\sin\beta}{(Nc + a\cos\beta)^2
    + i^2a^2\sin^2\beta} \right) \notag \\
    =& \frac{X_N^2}{Y_N^2}\left( \frac{c^2e^1}{6X_N^3}C^{i-1}
    + K_ic^2e^i\frac{A^{i-1}[(k_N^i)^2 + (e^i)^2] + D^{i-1}e^i}{6X_N(k_N^i)^2}
    + \frac{a_3^i}{c}\frac{A^{i-1}X_N^2 - D^{i-1}e^i}{(k_N^i)^2} \right)
\end{align}

Now we solve the second order $\delta$-equation to obtain $a_3^i$:
\begin{align}
    \delta^{(2)} =& -\frac{1}{y_1^2}\left( \frac{K_i}{6}a_1^ic(y_1^2 + c^2) - c\,a_3^i \right)
    - \hat{\alpha}_2^{i-1} \notag \\
    \overset{!}{=}& \frac{1}{y_2^2}\left(
    \frac{K_i}{6}a_1^id_1^i\left[ y_2^2 + d_1^i(d_1^i - a_1^i\sin\alpha_0^i) \right]\cos\alpha_0^i
    + (a_1^id_3^i - a_3^id_1^i)\cos\alpha_0^i + d_1^i\alpha_2^i(d_1^i - a_1^i\sin\alpha_0^i) \right), \\
    y_1 =& \sqrt{(a_1^i)^2 + c^2}, \qquad
    y_2 = \sqrt{(a_1^i)^2 + (d_1^i)^2 - 2a_1^id_1^i\sin\alpha_0^i} \notag
\end{align}

After inserting the zeroth order solutions along with $\hat{\alpha}_2^{i-1}(a_3^i)$, simplifying and substituting the equation looks as follows:
\begin{align}
    \delta^{(2)} =& \frac{X_N^2}{(k_N^i)^2}\left( \frac{a_3^i}{c}
    - \frac{K_ic^2e^i}{6X_N^3}\left[ X_N^2 + (k_N^i)^2 \right] \right) \notag \\
    &- \frac{X_N^2}{Y_N^2}\left( \frac{c^2e^1}{6X_N^3}C^{i-1}
    + K_ic^2e^i\frac{A^{i-1}[(k_N^i)^2 + (e^i)^2] + D^{i-1}e^i}{6X_N(k_N^i)^2}
    + \frac{a_3^i}{c}\frac{A^{i-1}X_N^2 - D^{i-1}e^i}{(k_N^i)^2} \right) \notag \\
    \overset{!}{=}& \frac{X_N^2}{Y_N^2}\left(
    \frac{K_ic^2e^i}{6X_N^3}\left[ 2X_N^2 + (2+i)(e^1)^2 \right]
    + \alpha_2^i\frac{X_N^2 + e^{i+1}e^1}{X_N^2} + \frac{d_3^ie^i}{c\,k_N^{i+1}} - \frac{a_3^i}{c} \right)
\end{align}

We see immediately, that we can divide $X_N^2$ out. Next we move all terms containing $a_3^i$ to the left hand side and all other terms to the right hand side:
\begin{align}
    a_3^i&\frac{Y_N^2 - A^{i-1}X_N^2 + D^{i-1}e^i + (k_N^i)^2}{c\,Y_N^2(k_N^i)^2}
    \overset{!}{=} \frac{1}{Y_N^2(k_N^i)^2}\left\{ \frac{c^2e^1}{6X_N}\left[ \frac{(k_N^i)^2}{X_N^2}C^{i-1}
    \right.\right.\notag\\
    &\left.\left. + iK_i\left( \frac{(k_N^i)^2\left[ 2X_N^2 + (2+i)(e^1)^2 \right]
    + Y_N^2\left[2X_N^2 + (e^i)^2 \right]}{X_N^2} + D^{i-1}e^i + A^{i-1}\left[X_N^2 + 2(e^i)^2 \right] \right) \right] \right.\notag\\
    &\left. + (k_N^i)^2\left[ \frac{\alpha_2^i}{X_N^2}\left[ X_N^2 + (i+1)(e^1)^2 \right]
    + \frac{d_3^ie^i}{c\,k_N^{i+1}} \right] \right\}
\end{align}
Finally we get $a_3^i$ dependent on $d_3^i$, $\alpha_2^i$ and the recursion parameters $C^{i-1}$, $A^{i-1}$ and $D^{i-1}$:
\begin{align}
    a_3^i(\alpha_2^i,d_3^i) =& \frac{c}{Y_N^2 - A^{i-1}X_N^2 + D^{i-1}e^i + (k_N^i)^2}\left\{ \frac{c^2 e^1}{6 X_N}\left[
    \frac{(k_N^i)^2}{X_N^2}C^{i-1} \right.\right.\notag\\
    &\left.\left. + iK_i\left( \frac{(k_N^i)^2\left[ 2X_N^2 + (i+2)(e^1)^2 \right]
    + Y_N^2\left[ 2X_N^2 + (e^i)^2 \right]}{X_N^2} + D^{i-1}e^i + A^{i-1}\left[ X_N^2 + 2(e^i)^2 \right] \right) \right] \right.\notag\\
    &\left. + (k_N^i)^2\left[ \frac{\alpha _2^i}{X_N^2}\left[ X_N^2 + (i+1)(e^1)^2 \right]
    + \frac{d_3^ie^i}{c\,k_N^{i+1}} \right] \right\}
\end{align}

We then move on to calculate $b_3^i$ inserting the zeroth order expressions and $a_3^i$ into:
\begin{align}
    b_3^i(\alpha_2^i,d_3^i) = -\frac{1}{y_2}\left( \frac{K_i}{6}(a_1^i)^2(d_1^i)^2\cos^2\alpha_0^i
    - d_3^i(d_1^i - a_1^i\sin\alpha_0^i) + a_1^id_1^i\alpha_2^i\cos\alpha_0^i
    - a_3^i(a_1^i - d_1^i\sin\alpha_0^i) \right),
\end{align}
which gives us:
\begin{align}
    b_3^i(\alpha_2^i,d_3^i) =& -\frac{1}{Y_N}\left(
    \frac{c\,e^1}{e^1(iD^{i-1} + (i^2 + 1)e^1) - (A^{i-1} - 2)X_N^2}\left\{ \frac{c^2e^1}{6X_N}\left[
    \frac{(k_N^i)^2}{X_N^2}C^{i-1} \right.\right.\right.\notag\\
    &\left.\left.\left. + iK_i\left( \left[ (k_N^i)^2 + (e^i)^2 \right]A^{i-1} + e^iD^{i-1} + Y_N^2
    + \frac{(k_N^i)^2}{X_N^2}\left[ 3Y_N^2 + e^1e^i \right] \right) \right] \right.\right.\notag\\
    &\left.\left. + (k_N^i)^2\left[ \frac{X_N^2 + e^1e^{i+1}}{X_N^2}\alpha_2^i + \frac{e^id_3^i}{c\,k_N^{i+1}} \right] \right\} + \frac{K_ic^3(e^i)^2}{6X_N} + i\alpha_2^ic\,e^1
    - d_3^i\left[ k_N^{i+1} - \frac{e^ie^{i+1}}{k_N^{i+1}} \right] \right)
\end{align}

Having exercised through almost the entire segment we finally come to $\hat{\alpha}_2^i$ of this segment:
\begin{align}
    \hat{\alpha}_2^i(\alpha_2^i,d_3^i) = \frac{1}{y_2^2}\left(
    \frac{K_i}{6}a_1^id_1^i\left[ y_2^2 + a_1^i(a_1^i - d_1^i\sin\alpha_0^i) \right]\cos\alpha_0^i
    + (a_3^id_1^i - a_1^id_3^i)\cos\alpha_0^i + a_1^i\alpha_2^i(a_1^i - d_1^i\sin\alpha_0^i) \right)
\end{align}

As with $b_3^i$ just before we insert all quantities, which we already calculated, especially $a_3^i$:
\begin{align}
    \hat{\alpha}_2^i&(\alpha_2^i,d_3^i) = \frac{X_N^2}{Y_N^2}\left\{ \frac{c^2e^1}{6X_N^3}\left[ \frac{(k_N^i)^2C^{i-1}}
    {Y_N^2 - A^{i-1}X_N^2 + D^{i-1}e^i + (k_N^i)^2} \right.\right.\notag\\
    &\left.\left. + iK_i\left( \frac{ (k_N^i)^2\left[ (3+A^{i-1})X_N^2 + (i+3)(e^1)^2 \right]
    + X_N^2\left[ D^{i-1}e^i + A^{i-1}(e^i)^2 + Y_N^2 \right] }
    {Y_N^2 - A^{i-1}X_N^2 + D^{i-1}e^i + (k_N^i)^2} + \left[ X_N^2 - (i-1)(e^1)^2 \right] \right) \right] 
    \right.\notag\\
    &\left. - \left( i(e^1)^2 - \frac{ (k_N^i)^2\left[ X_N^2 + (i+1)(e^1)^2 \right] }
    { Y_N^2 - A^{i-1}X_N^2 + D^{i-1}e^i + (k_N^i)^2 } \right)\frac{\alpha_2^i}{X_N^2}
    + \left( \frac{(k_N^i)^2}{Y_N^2 - A^{i-1}X_N^2 + D^{i-1}e^i + (k_N^i)^2} - 1 \right)
    \frac{d_3^ie^i}{c\,k_N^{i+1}} \right\}
\end{align}

Comparing with the ansatz~\eqref{eq: Ansatz} we find recursion rules for $C$, $A$ and $D$. Since the lengthy denominator occurs a lot we substitute it with $B^i$, which becomes part of the recursion rule:
\begin{align}
    B^i \coloneq& Y_N^2 - A^{i-1}X_N^2 + D^{i-1}e^i + (k_N^i)^2, \\
    A^i =& \frac{1}{X_N^2}\left( (k_N^i)^2\frac{X_N^2 + (i+1)(e^1)^2}{B^i} - i(e^1)^2 \right), \quad
    D^i = e^i\left( 1 - \frac{(k_N^i)^2}{B^i} \right), \\
    C^i =& \frac{(k_N^i)^2}{B^i}C^{i-1} \\
    &+ iK_i\left( 
    \frac{ X_N^2\left( \left[ (k_N^i)^2 + (e^i)^2 \right]A^{i-1} + e^iD^{i-1} \right)
    + (k_N^i)^2\left[ 3Y_N^2 + i(e^1)^2 \right] + X_N^2Y_N^2 }{B^i} + \left[ X_N^2 - (i-1)(e^1)^2\right] \right) \notag
\end{align}

\subsubsection{$\Delta_{N-1}:$ Problem 2, Type 3}
The upper right triangle $\hat{\Delta}_{N-1}$ in the second last segment $\Delta_{N-1}$ is a special case, since as mentioned in the $0$-th order case, it's neighbouring triangle has in general an angle $(\pi-\beta)$ instead of $\frac{\pi}{2}$. Since the lower left triangle $\overset{\vee}{\Delta}_{N-1}$ is not affected by this change we can use the $d_3^p(a_3)$ and $\alpha_2^p(a_3)$ from the previous section with $i=N-1$:
\begin{align}
    d_3^{N-2}(a_3^{N-1}) =& \frac{e^{N-1}}{k_N^{N-1}}\left( a_3^{N-1} - \frac{K_{N-1}c^3e^{N-1}}{6X_N} \right) \\
    \alpha_2^{N-2}(a_3^{N-2}) =& \frac{X_N^2}{(k_N^{N-1})^2}\left( \frac{a_3^{N-1}}{c}
    + \frac{K_{N-1}}{6}c^2e^{N-1}\frac{(k_N^{N-1})^2 + (e^{N-1})^2}{X_N^3} \right)
\end{align}

This means, that also only the right hand side of the $\delta$-equation changes.
\begin{align}
    \delta^{(2)} =& \frac{1}{y_2^2}\left(
    \frac{K_{N-1}}{6}a_1^{N-1}d_1^{N-1}\left[ y_2^2 + d_1^{N-1}(d_1^{N-1} - a_1^{N-1}\sin\alpha_0^{N-1}) \right]\cos\alpha_0^{N-1} \right.\notag\\
    &\left.\qquad + (a_1^{N-1}d_3^{N-1} - a_3^{N-1}d_1^{N-1})\cos\alpha_0^{N-1} + d_1^{N-1}\alpha_2^{N-1}(d_1^{N-1} - a_1^{N-1}\sin\alpha_0^{N-1}) \vphantom{\frac{K_{N-1}}{6}}\right), \\
    y_2 =& \sqrt{(a_1^{N-1})^2 + (d_1^{N-1})^2 - 2a_1^{N-1}d_1^{N-1}\sin\alpha_0^{N-1}} \notag
\end{align}
But we know $\alpha_2^{N-1}$ and $d_3^{N-1}$ already from Problem 1. When we insert them, together with the zeroth order solutions and simplify, we get the following equation:
\begin{align}
    \delta^{(2)} =& \frac{X_N^2}{(k_N^{N-1})^2}\left( \frac{a_3^{N-1}}{c}
    - \frac{K_{N-1}c^2e^{N-1}}{6X_N^3}\left[ X_N^2 + (k_N^{N-1})^2 \right] \right) \notag \\
    &- \frac{X_N^2}{Y_N^2}\left( \frac{c^2e^1}{6X_N^3}C^{N-2}
    + K_{N-1}c^2e^{N-1}\frac{A^{N-2}[(k_N^{N-1})^2 + (e^{N-1})^2] + D^{N-2}e^{N-1}}{6X_N(k_N^{N-1})^2}
    + \frac{a_3^{N-1}}{c}\frac{A^{N-2}X_N^2 - D^{N-2}e^{N-1}}{(k_N^{N-1})^2} \right) \notag \\
    \overset{!}{=}& \frac{cX_N^2}{X_1Y_N^2}\left\{ \frac{e^1}{6X_N}\left[
    (N-1)K_{N-1}\frac{X_1}{X_N}\left( a^2+Nc^2 + (N+1)a\,c\cos\beta + \frac{X_1}{X_N}Y_N^2 \right)
    \right.\right.\notag\\
    &\left.\left. + K_N \left(2 a^2+N c^2+(2 N+1) a c \cos (\beta )\right) \vphantom{\frac{X_1}{X_N}}\right] 
    -\frac{a_3^{N-1}}{c} \right\}
\end{align}

When we solve for $a_3^{N-1}$ in the same manner as in the previous section we get:
\begin{align}
    a_3^{N-1} =& \frac{c^2X_1(k_N^{N-1})^2}{ 6X_N\left[ c(k_N^{N-1})^2
    + X_1\left( Y_N^2 - A^{N-2}X_N^2 + D^{N-2}e^{N-1} \right) \right] }\left\{ \frac{e^1}{X_1} \right.\notag\\
    &\left. \cdot \left[ (N-1)K_{N-1}\frac{X_1}{X_N}\left( a^2 + Nc^2 +(N+1)a\,c\cos\beta + \frac{X_1}{X_N}Y_N^2 \right) + K_N\left( 2a^2 + Nc^2 + (2N+1)a\,c\cos\beta \right) \right] \right.\notag\\
    &\left. + \frac{c\,e^1}{X_N^2}C^{N-2} + K_{N-1}\frac{c\,e^{N-1}}{(k_N^{N-1})^2}\left[
    Y_N^2\left( 1 + \frac{(k_N^{N-1})^2}{X_N^2} \right)
    + A^{N-2}\left[ 2(k_N^{N-1})^2 - X_N^2 \right] + D^{N-2}e^{N-1} \right] \right\}
\end{align}

When we insert all calculated quantities into
\begin{align}
    b_3^{N-1} =& -\frac{1}{y_2}\left( \frac{K_{N-1}}{6}(a_1^{N-1})^2(d_1^{N-1})^2\cos^2\alpha_0^{N-1}
    - d_3^{N-1}(d_1^{N-1} - a_1^{N-1}\sin\alpha_0^{N-1}) \right.\notag\\
    &\left. + a_1^{N-1}d_1^{N-1}\alpha_2^{N-1}\cos\alpha_0^{N-1}
    - a_3^{N-1}(a_1^{N-1} - d_1^{N-1}\sin\alpha_0^{N-1}) \vphantom{\frac{K_{N-1}}{6}}\right),
\end{align}
we can directly calculate $b_3^{N-1}$ in terms of the recursion parameters and input variables only:
\begin{align}
    b_3^{N-1} =& -\frac{c^2}{6 Y_N}\frac{X_1}{X_N}\left\{ \frac{(k_N^{N-1})^2e^1}
    {X_1\left( Y_N^2 - A^{N-2}X_N^2 + D^{N-2}e^{N-1} \right) + c(k_N^{N-1})^2}\left[ \frac{ce^1}{X_N^2}C^{N-2}
    \right.\right.\notag\\
    & + K_{N-1}e^{N-1}\left( c\frac{A^{N-2}X_N^2 + D^{N-2}e^{N-1} + 2A^{N-2}(e^{N-1})^2}{(k_N^{N-1})^2} + \frac{a^2 + Nc^2 + (N+1)a\,c\cos\beta}{X_N} \right.\notag\\
    &\left.\left. + \frac{Y_N^2}{X_N}\left[ \frac{X_1+c}{X_N} + \frac{cX_N}{(k_N^{N-1})^2} \right] \right) + K_Ne^1\frac{2a^2 + Nc^2 + (2N+1)a\,c\cos\beta}{X_1}\right] \notag \\
    &\left.+ K_{N-1}(e^{N-1})^2 + K_N(2 N-1)(e^1)^2\frac{X_N}{X_1} \right\}
\end{align}
We will see later, that leaving the expression dependent on $a_3^{N-1}$ is more comfortable to handle:
\begin{align}
    b_3^{N-1}(a_3^{N-1}) = -\frac{X_N}{X_1Y_N}\left( \frac{K_{N-1}}{6}c^2(e^{N-1})^2\frac{X_1^2}{X_N^2}
    + (2N-1)\frac{K_N}{6}c^2(e^1)^2\frac{X_1}{X_N} - e^1\left[ \frac{(N-1)c}{X_N}-1 \right]a_3^{N-1} \right)
\end{align}

We again insert all previously calculated quantities into
\begin{align}
    \hat{\alpha}_2^{N-1}(\alpha_2^{N-1},d_3^{N-1}) =& \frac{1}{y_2^2}\left(
    \frac{K_{N-1}}{6}a_1^{N-1}d_1^{N-1}\left[ y_2^2 + a_1^{N-1}(a_1^{N-1}
    - d_1^{N-1}\sin\alpha_0^{N-1}) \right]\cos\alpha_0^{N-1} \right.\notag\\
    &\left. + (a_3^{N-1}d_1^{N-1} - a_1^{N-1}d_3^{N-1})\cos\alpha_0^{N-1}
    + a_1^{N-1}\alpha_2^{N-1}(a_1^{N-1} - d_1^{N-1}\sin\alpha_0^{N-1}) \vphantom{\frac{K_{N-1}}{6}}\right),
\end{align}
which leads us to full form only dependent on the recursion parameters and initial variables:
\begin{align}
    \hat{\alpha}_2^{N-1} =& \frac{1}{6 Y_N^2}\left\{ \frac{c^3X_N(k_N^{N-1})^2e^1}
    {X_1\left(Y_N^2 - A^{N-2}X_N^2 + D^{N-2}e^{N-1}\right) + c(k_N^{N-1})^2}\left[
    \frac{C^{N-2}}{X_N^2} \right.\right.\notag\\
    & + (N-1)K_{N-1}\left( \frac{1}{cX_N}\left[ a^2 + Nc^2 + (N+1)a\,c\cos\beta
    + \frac{X_1}{X_N}Y_N^2 \right] + Y_N^2\frac{X_N^2 + (k_N^{N-1})^2}{X_N^2(k_N^{N-1})^2} \right.\notag\\
    &\left.\left. + \frac{ A^{N-2}\left[ (k_N^{N-1})^2 + (e^{N-1})^2 \right] + D^{N-2}e^{N-1} }{(k_N^{N-1})^2} \right) + \frac{K_N}{c X_1}\left( 2a^2 + Nc^2 + (2N+1)a\,c\cos\beta \right) \right] \notag\\
    &\left. + cK_{N-1}\frac{e^{N-1}}{X_N}(X_1 Y_N^2-c\,e^1 e^{N-1})
    + (N-1)K_N\frac{(e^1)^3c^2}{X_1}\left( \frac{cX_N - a(a + c\cos\beta)}{Y_1^2} - 1 \right) \right\}
\end{align}
and in dependence of $a_3^{N-1}$ for convenience:
\begin{align}
    \hat{\alpha}_2^{N-1}(a_3^{N-1}) =& \frac{1}{X_1Y_N^2}\left( \frac{ce^{N-1}}{6}\left[
    K_{N-1}\frac{X_1}{X_N}(X_1Y_N^2 - c\,e^{N-1}e^1) + K_Nc(e^1)^2\left( \frac{Nc^2 - a^2}{Y_1^2} - 1 \right) \right] + X_N^2a_3^{N-1} \right)
\end{align}

\subsection{The recursion parameters}\label{subsec: The recursion parameters}
To make sure, that the expressions we use are well defined we need to understand the index range of the recursion parameters. The angle $\hat{\alpha}_2^i$ appears in the segments $\Delta_1,\ldots,\Delta_N-2$, thus the recursion formulas are valid for $i\in\{2,\ldots,N-2\}$. Since the $B^i$ depend on the previous $A^{i-1}$ and $D^{i-1}$ it's index range is shortened by one: $i\in\{3,\ldots,N-2\}$.\\
To calculate the first $B^i$ we need the first $A^i$ and $D^i$ which we can read out from the first $\hat{\alpha}^i$, which is unsurprisingly:
\begin{equation}
    \hat{\alpha}_2^1(\alpha_2^1,d_3^1) = \frac{X_N}{2Y_N^2}\left( \frac{K_1}{3}c^2e^1\left[1 + 2\frac{Y_N^2}{X_N^2}\right]
        + X_N\left[\alpha_2^1 - \frac{d_3^1e^1}{c\,k_N^2}\right] \right)
    = \frac{X_N^2}{Y_N^2}\left( \frac{c^2e^1}{6X_N^3}K_1\left[X_N^2 + 2Y_N^2\right]
        + \frac{1}{2}\alpha_2^1 - \frac{e^1d_3^1}{2c\,k_N^2} \right) \notag
\end{equation}

Bringing it into standard form and comparing with the ansatz~\eqref{eq: Ansatz} we get the first recursion parameters:
\begin{equation}
    C^1 = K_1\left[X_N^2 + 2Y_N^2\right], \qquad A^1 = \frac{1}{2}, \qquad D^1 = \frac{e^1}{2},
\end{equation}

and thus $B^2$ is:
\begin{equation}
    B^2 = Y_N^2 - A^1X_N^2 + D^1e^2 + (k_N^2)^2 = 3\left[\frac{X_N^2}{2} + 2(e^1)^2\right],
\end{equation}
\begin{align}
    B^i \coloneq& Y_N^2 - A^{i-1}X_N^2 + D^{i-1}e^i + (k_N^i)^2, \quad i\in\{2,\ldots,N-2\} \label{eq: Bi}\\
    A^i =& \frac{1}{X_N^2}\left( (k_N^i)^2\frac{X_N^2 + (i+1)(e^1)^2}{B^i} - i(e^1)^2 \right), \quad
    D^i = e^i\left( 1 - \frac{(k_N^i)^2}{B^i} \right), \quad i\in\{2,\ldots,N-2\} \label{eq: Ai&Di}\\
    C^i =& \frac{(k_N^i)^2}{B^i}C^{i-1} \\
    &+ iK_i\left( 
    \frac{ X_N^2\left( \left[ (k_N^i)^2 + (e^i)^2 \right]A^{i-1} + e^iD^{i-1} \right)
    + (k_N^i)^2\left[ 3Y_N^2 + i(e^1)^2 \right] + X_N^2Y_N^2 }{B^i} + \left[ X_N^2 - (i-1)(e^1)^2\right] \right), \notag\label{eq: Ci} \\
    i\in&\{2,\ldots,N-2\} \notag
\end{align}

When we look at the recursion parameters~\eqref{eq: Bi},\eqref{eq: Ai&Di} and~\eqref{eq: Ci}, it looks quite convoluted, since $A^i$ and $D^i$ are dependent on $B^i$ which itself is dependent on $A^{i-1}$ and $D^{i-1}$. We could just get rid of $B^i$ using it's definition~\eqref{eq: Bi}. Then $A^i$ and $D^i$ both depend on $A^{i-1}$ and $D^{i-1}$, while $C^{i}$ depends on $C^{i-1}$, $A^{i-1}$ and $D^{i-1}$. It may seem at first, that the introduction of $B^i$ complicated the issue, although it's use simplifies the expressions. We can decouple the recursion rules~\eqref{eq: Bi},\eqref{eq: Ai&Di} and~\eqref{eq: Ci} however, by instead inserting the rules for $A^i$ and $D^i$ into the rule for $B^i$:
\begin{align}
    B^i \coloneq& Y_N^2 - A^{i-1}X_N^2 + D^{i-1}e^i + (k_N^i)^2, \quad A^i = \frac{1}{X_N^2}\left( (k_N^i)^2\frac{X_N^2 + (i+1)(e^1)^2}{B^i} - i(e^1)^2 \right), \quad
    D^i = e^i\left( 1 - \frac{(k_N^i)^2}{B^i} \right) \notag \\
    =& Y_N^2 - (k_N^{i-1})^2\frac{X_N^2 + i(e^1)^2}{B^{i-1}} + (i-1)(e^1)^2 + e^ie^{i-1}\left( 1 - \frac{(k_N^{i-1})^2}{B^{i-1}} \right) + (k_N^i)^2 \qquad \left|\begin{matrix} Y_N^2 = X_N^2 + (e^1)^2 \\
    (k_N^i)^2 = X_N^2 + i^2(e^1)^2 \end{matrix}\right. \notag \\
    =& X_N^2 + \left( \cancel{1} + \cancel{i}-\cancel{1} + i^2-\cancel{i} + i^2 \right)(e^1)^2 + X_N^2
    - \frac{(k_N^{i-1})^2\left[ X_N^2 + \left( \cancel{i} + i^2-\cancel{i} \right)(e^1)^2 \right]}{B^{i-1}} \notag \\
    =& 2(k_N^i)^2 - \frac{(k_N^{i-1})^2(k_N^i)^2}{B^{i-1}}
    = (k_N^i)^2\left( 2 - \frac{(k_N^{i-1})^2}{B^{i-1}} \right), \quad i\in\{3,\ldots,N-2\}
\end{align}
If we start with some $i$ and then just insert the expression for $B^{i-1}$ and so on, we see that the $B^i$'s are finite continued fractions, which become infinite when we take the limit $N\to\infty$ and $i\to N$ i.e. when we consider one of the last segments.
\begin{align}
    B^i =& (k_N^i)^2\left( 2 - \frac{(k_N^{i-1})^2}{ (k_N^{i-1})^2\left( 
    2 - \frac{(k_N^{i-2})^2}{B^{i-2}} \right) } \right)
    = \ldots = (k_N^i)^2\left(\vphantom{\begin{matrix}1\\1\\1\end{matrix}}\right.
    2 - \frac{\cancel{(k_N^{i-1})^2}}{ \cancel{(k_N^{i-1})^2}\left(\vphantom{\begin{matrix}1\\1\\1\end{matrix}}\right. 
    2 - \frac{\cancel{(k_N^{i-2})^2}}{ \cancel{(k_N^{i-2})^2}\left(\vphantom{\begin{matrix}1\\1\\1\end{matrix}}\right. 
    2 - \frac{\cancel{(k_N^{i-3})^2}}{ \ddots \cancel{(k_N^3)^2}\left( 2 - \frac{(k_N^2)^2}{B^2} \right) } 
    \left.\vphantom{\begin{matrix}1\\1\\1\end{matrix}}\right) }
    \left.\vphantom{\begin{matrix}1\\1\\1\end{matrix}}\right) } \left.\vphantom{\begin{matrix}1\\1\\1\end{matrix}}\right) \notag \\
    =& (k_N^i)^2\left(\vphantom{\begin{matrix}1\\1\\1\end{matrix}}\right.
    2 - \frac{1}{ 2 - \frac{1}{  2 - \frac{1}{ \ddots 2 - \frac{(k_N^2)^2}{B^2} } } } \left.\vphantom{\begin{matrix}1\\1\\1\end{matrix}}\right)
\end{align}
We can simplify this by introducing
\begin{align}
    b^i \coloneq \frac{B^i}{(k_N^i)^2} = 2 - \frac{(k_N^{i-1})^2}{B^{i-1}} = 2 - \frac{1}{b^{i-1}}
    = 2 - \frac{1}{2 - \frac{1}{2 - \frac{1}{\ddots 2-\frac{1}{b^2}}}}.
\end{align}
After calculating the first few examples one can see, that they follow a very simple rule.\\
\begin{proposition}
\begin{equation}
    b^i = \frac{i+1}{i}, \qquad i \geqslant 2
\end{equation}
\end{proposition}

\begin{proof}
\begin{flalign}
&\text{Induction start:} &\quad b^2 &= \frac{B^2}{(k_N^2)^2} = \frac{3}{2}\frac{X_N^2 + 4(e^1)^2}{X_N^2 + (e^2)^2} = \frac{3}{2} \qquad \checkmark&& \\
&\text{Induction step:} &\quad b^{i+1} &= 2 - \frac{1}{b^i} = 2 - \frac{i}{i+1} = \frac{2(i+1) - i}{i+1} = \frac{i+2}{i+1} \qquad \checkmark&&
\end{flalign}
\end{proof}

Then the other recursion parameters become:
\begin{align}
    A^i =& \frac{1}{X_N^2}\left( \frac{X_N^2 + (i+1)(e^1)^2}{b^i} - i(e^1)^2 \right)
    = \frac{i}{i+1}, \quad
    D^i = e^i\left( 1 - \frac{1}{b^i} \right) = e^i\left( \frac{i+1 - i}{i+1} \right) = \frac{e^i}{i+1} \notag \\
    C^i =& \frac{i}{i+1}C^{i-1} \\
    &+ \underbrace{iK_i\left( 
    \frac{1}{b^i}\left[  \frac{X_N^2}{(k_N^i)^2}\left( \left[ (k_N^i)^2 + (e^i)^2 \right]A^{i-1} + e^iD^{i-1} \right)
    + 3Y_N^2 + i(e^1)^2 + \frac{X_N^2Y_N^2}{(k_N^i)^2}  \right]
    + X_N^2 - (i-1)(e^1)^2 \right)}_{\textcolor{gray}{\mathcal{K}^i}} \notag \\
    =& \frac{1}{b^i}\left(\vphantom{\frac{X_N^2}{(k_N^i)^2}}\right. C^{i-1}
    + \underbrace{K_i\left[ 2iX_N^2 + i\left( 3 + \frac{X_N^2}{(k_N^i)^2} \right)Y_N^2
    + \left( 1 + \frac{i(i^2-1)X_N^2}{(k_N^i)^2} \right)(e^1)^2 \right]}_{\mathcal{K}^i}
    \left.\vphantom{\frac{X_N^2}{(k_N^i)^2}}\right)
\end{align}

We have $A^i$ and $D^i$ in explicit form now but there is still a recursion remaining in $C^i$. To solve this recursion we simplified the $C^i$ by introducing the substitute $\mathcal{K}^i$ to replace the constant term (not dependent on $C^j$) containing the curvature dependence. This way we get a more generic form of the problem and remove details, which are not relevant for the recursion. We can solve this recursion the same way we solved the $b^i$ by directly inserting the recursion rule into itself until the structure becomes obvious:
\begin{align}
    C^i =& \frac{1}{b^i}\left( C^{i-1} + \mathcal{K}^i \right)
    = \frac{1}{b^i}\left( \frac{1}{b^{i-1}}\left( C^{i-2} + \mathcal{K}^{i-1} \right) + \mathcal{K}^i \right)
    = \frac{1}{b^i}\left( \frac{1}{b^{i-1}}\left( \ldots
    \frac{1}{b^3}\left( \frac{1}{b^2}\left( C^1 + \mathcal{K}^2 \right) + \mathcal{K}^3 \right) \ldots
    + \mathcal{K}^{i-1} \right) + \mathcal{K}^i \right) \notag \\
    =& \frac{1}{b^ib^{i-1}\dots b^3b^2}C^1 + \frac{1}{b^ib^{i-1}\dots b^3b^2}\mathcal{K}^2 + \ldots
    + \frac{1}{b^ib^{i-1}}\mathcal{K}^{i-1} + \frac{1}{b^i}\mathcal{K}^i
    = C^1\prod_{j=2}^i\frac{1}{b^j} + \sum_{k=2}^i\mathcal{K}^k \prod_{j=k}^i\frac{1}{b^j}
\end{align}

\begin{align}
    \prod_{j=k}^i\frac{1}{b^j} = \frac{k}{k+1}\frac{k+1}{k+2}\cdot\ldots\cdot\frac{i}{i+1} = \frac{k}{i+1} \quad
    \Rightarrow \quad C^i = \frac{2}{i+1}C^1 + \sum_{k=2}^i \frac{k}{i+1}\mathcal{K}^k
\end{align}
Since the curvature term $\mathcal{K}^k$ comes from the recursion formula of $C^i$ it's range is $k\in\{2,\ldots,N-2\}$ and we check, that all $k$ appearing in valid $C^i$ are within range.

\subsection{Spherical triangulation of the $1$-st slice}\label{subsec: Spherical triangulation of the $1$-st slice}
With the recursion parameters in explicit form we can start backwards inserting in the following manner:
\begin{align}
    &\Delta_N, \ \text{Problem 1:} & &\to & &d_3^{N-1}, \quad \alpha_2^{N-1}, \quad \gamma_2^{N-1}
    \notag \\
    &\Delta_{N-1}, \ \text{Problem 2, Type 3:} & &\to & &a_3^{N-1}(C^{N-2},A^{N-2},D^{N-2}) \notag \\
    & & &\to & &b_3^{N-1}(a_3^{N-1}), \quad \hat{\alpha}_2^{N-1}(a_3^{N-1}), \quad d_3^{N-2}(a_3^{N-1}),
    \quad \alpha_2^{N-2}(a_3^{N-1}) \notag \\
    &\Delta_{N-2}, \ \text{Problem 2, Type 2:} & &\to & &a_3^{N-2}(C^{N-3},A^{N-3},D^{N-3};d_3^{N-2},\alpha_2^{N-2}), \quad b_3^{N-2}(C^{N-3},A^{N-3},D^{N-3};d_3^{N-1},\alpha_2^{N-1}) \notag \\
    & & &\to & &d_3^{N-3}(a_3^{N-2}), \quad \alpha_2^{N-3}(a_3^{N-2}) \notag \\
    &\vdots \notag \\
    &\Delta_i, \ \text{Problem 2, Type 2:} & &\to & &a_3^i(C^{i-1},A^{i-1},D^{i-1};d_3^i,\alpha_2^i), \quad b_3^i(C^{i-1},A^{i-1},D^{i-1};d_3^i,\alpha_2^i) \notag \\
    & & &\to & &d_3^{i-1}(a_3^i), \quad \alpha_2^{i-1}(a_3^i) \notag \\
    &\vdots \notag \\
    &\Delta_2, \ \text{Problem 2, Type 2:} & &\to & &a_3^2(C^1,A^1,D^1;d_3^2,\alpha_2^2), \quad b_3^2(C^1,A^1,D^1;d_3^2,\alpha_2^2) \quad \to \quad d_3^1(a_3^2), \quad \alpha_2^1(a_3^2) \notag \\
    &\Delta_1, \ \text{Problem 2, Type 1:} & &\to & &b_3^1(d_3^1,\alpha_2^1), \quad \alpha_2^0(d_3^1,\alpha_2^1)
\end{align}

Now that we solved the parameter recursion, we would like to rewrite our results, so we can calculate them more directly. We start by simplifying $a_3^{N-1}$ using $B^{N-2}$ and introducing the substitutes $\chi^{N-1}$ and $\zeta^{N-1}$:
\begin{align}
    a_3^{N-1} =& \frac{(N-1)c^2X_1}{6X_N\left[ (N-1)c + X_1 \right]}\left( \frac{c\,e^1}{X_N^2}C^{N-2}
    + K_{N-1}\chi + K_N\zeta \right) \\
    \chi =& (N-1)\frac{e^1}{X_N}\left( a^2 + Nc^2 + (N+1)a\,c\cos\beta + \frac{X_1}{X_N}Y_N^2 \right) \notag \\
    &+ \left( Y_N^2\frac{(k_N^{N-1})^2 + X_N^2}{(k_N^{N-1})^2X_N^2} + \frac{N-2}{(k_N^{N-1})^2}\left[
    \frac{X_N^2}{N-1} + (2N-1)(e^1)^2 \right] \right)c\,e^{N-1} \\
    \zeta =& \frac{e^1}{X_1}( 2a^2 + Nc^2 + (2N+1)a\,c\cos\beta )
\end{align}

There is no point in inserting the expression for $a_3^{N-1}$ into $b_3^{N-1}$ or $\hat{\alpha}_2^{N-1}$, since we will calculate the asymptotic behaviour of the substitutes first and then work ourselves through the layers of substitutions.
\begin{align}
    b_3^{N-1} =& -\frac{c^2 \left(e^1\right)^2}{6 Y_N}\left( (N-1)^2K_{N-1}\frac{X_1}{X_N}
    + (2N-1)K_N + \frac{6}{c^2e^1}a_3^{N-1} \right) \\
    \hat{\alpha}_2^{N-1}(a_3^{N-1}) =& \frac{1}{X_1Y_N^2}\left( \frac{c\,e^{N-1}}{6}\left[
    K_{N-1}\frac{X_1}{X_N}(X_1Y_N^2 - c\,e^{N-1}e^1) + K_Nc(e^1)^2\left( \frac{Nc^2 - a^2}{Y_1^2} - 1 \right) \right] + X_N^2a_3^{N-1} \right)
\end{align}

By solving the recursion parameters we got one step further, we can calculate the quantities in the segment $\Delta_{N-1}$ as well. But only one step, since as we can see above, that the quantities in the segment $\Delta_i$ are dependent on the ones in segment $\Delta_{i-1}$. This still does not allow us to write an explicit expression for a quantity in $\Delta_i$.\\
A first step to solve this problem is to derive a recursion formula for $a_3^i$ by exercising this procedure from one segment to the next one similar to what we did with $\hat{\alpha}_2^i$. Except that we don't need an ansatz here.\\
When we insert $d_3^i$ and $\alpha_2^i$ into $a_3^i$ we already get the desired result:
\begin{align}
    a_3^i(a_3^{i+1}) =& a_3^i(d_3^i(a_3^{i+1}),\alpha_2^i(a_3^{i+1}))
    = \frac{(k_N^i)^2}{B^i}\left( \frac{c^3e^1}{6X_N}\mathcal{K}_a^i + a_3^{i+1} \right)
    = \frac{1}{b^i}\left( \frac{c^3e^1}{6X_N}\mathcal{K}_a^i + a_3^{i+1} \right), \\
    \mathcal{K}_a^i \coloneq& \frac{C^{i-1}}{X_N^2} + i\,K_i\mathrm{X}^i + (i+1)K_{i+1}\mathrm{Z}^i \\
    \mathrm{X}^i \coloneq& \frac{1}{(k_N^i)^2}\left[ \left(5-\frac{1}{i}\right)X_N^2
    + i^2(i+3)\frac{(e^1)^4}{X_N^2} + (5i^2 + 3)(e^1)^2 \right], \qquad
    \mathrm{Z}^i \coloneq 1 + 2(i+1)\frac{(e^1)^2}{X_N^2}.
\end{align}

This is essentially the same recursion problem as the one for $C^i$, so we can solve it in the same way. This time the index is increasing, so the recursion stops at the last term, which is $a_3^{N-1}$:
\begin{align}
    a_3^i =& \frac{1}{b^i}\left( \frac{c^3e^1}{6X_N}\mathcal{K}_a^i + \frac{1}{b^{i+1}}\left( \frac{c^3e^1}{6X_N}\mathcal{K}_a^{i+1} + a_3^{i+2} \right) \right) \notag \\
    =& \frac{1}{b^i}\left( \frac{c^3e^1}{6X_N}\mathcal{K}_a^i + \frac{1}{b^{i+1}}\left( \frac{c^3e^1}{6X_N}\mathcal{K}_a^{i+1} + \frac{1}{b^{i+2}}\left( \ldots \frac{1}{b^{N-2}}\left( \frac{c^3e^1}{6X_N}\mathcal{K}_a^{N-2} + a_3^{N-1} \right) \ldots \right) \right) \right) \notag \\
    =& \frac{c^3e^1}{6X_N}\frac{\mathcal{K}_a^i}{b^i} + \frac{c^3e^1}{6X_N}\frac{\mathcal{K}_a^{i+1}}{b^ib^{i+1}}
    + \ldots + \frac{c^3e^1}{6X_N}\frac{\mathcal{K}_a^{N-2}}{b^ib^{i+1}\dots b^{N-2}} + \frac{a_3^{N-1}}{b^ib^{i+1}\dots b^{N-2}} \notag \\
    =& \frac{c^3e^1}{6X_N}\sum_{n=i}^{N-2}\mathcal{K}_a^n \prod_{j=i}^n\frac{1}{b^j}
    + a_3^{N-1}\prod_{j=i}^{N-2}\frac{1}{b^j}
    = \frac{c^3e^1}{6X_N}\sum_{n=i}^{N-2}\frac{i}{n+1}\mathcal{K}_a^n
    + \frac{i}{N-1}a_3^{N-1}
\end{align}

Now that we have $a_3^i$ explicitly we can get the explicit at simplified form of $b_3^i$, by inserting $d_3^i$ and $\alpha_2^i$:\\
\begin{align}
    b_3^i(a_3^{i+1}) =& b_3^i(d_3^i(a_3^{i+1}),\alpha_2^i(a_3^{i+1}))
    = -\frac{e^1}{(i+1)Y_N}\left( \frac{c^3e^1}{6X_N}i\mathcal{K}_b^i
    + a_3^{i+1} \right), \\
    \mathcal{K}_b^i \coloneq& \frac{C^{i-1}}{X_N^2} + iK^i\mathrm{X}_b^i
    + (i+1)K^{i+1}\mathrm{Z}_b^i, \\
    \mathrm{X}_b^i \coloneq& i + 5 + \frac{i-1}{i} + (i+3)\frac{(e^1)^2}{X_N^2} \notag \\
    \mathrm{Z}_b^i \coloneq& \left( i + 2 + \frac{e^1e^{i+1}}{X_N^2} \right)
   \left( 1 + \frac{(e^{i+1})^2}{(k_N^{i+1})^2} \right) + \frac{(i+1)^2X_N^2 + (2i+1)e^1e^{i+1}}{i(k_N^{i+1})^2}
\end{align}

The only thing we need to do in segment $\Delta_i$ at this point is inserting the expressions for $d_3^1$ and $\alpha_2^1$ in dependence of $a_3^2$ into $b_3^1$ and $\alpha_2^0$:
\begin{align}
    b_3^1 =& -\frac{e^1}{Y_N}\left( \frac{c^3e^1}{X_N}\left[ \frac{K_1}{3} + K_2 \right] - a_3^2 \right), \qquad
    \alpha_2^0 = \frac{X_N^2}{2Y_N^2}\left( \frac{c^2e^1}{3 X_N}\left[
    K_1\left( 3 + \frac{4(e^1)^2}{X_N^2} \right) + K_2\frac{(k_N^2)^2}{X_N^2} \right]
    + \frac{a_3^2}{c} \right)
\end{align}

\subsubsection{$0$-th order solution}
The zeroth order side-length of the top-line $\bar{b}_0$ and opening angle $\alpha_0$ of the slice are the first quantities we calculated in Sec.~\ref{sec: general 0-th order sol}. We summarize the $0$-th order results in this section and check their consistency.\\
Each piece $b_1^i$ comes with an $\varepsilon$, but if we add $N$ of them together this cancels and we get a $0$-th order result.
\begin{align}
    \bar{b}_0 =& b_1^1\varepsilon + \sum_{i=2}^{N-2}b_1^i\varepsilon + b_1^{N-1}\varepsilon \notag \\
    =& 2c\frac{\sqrt{a^2 + N^2c^2 + 2Na\,c\cos\beta}}{Nc + a\cos\beta}\varepsilon
    + \sum_{i=2}^{N-2}c\frac{\sqrt{a^2 + N^2c^2 + 2Na\,c\cos\beta}}{Nc + a\cos\beta}\varepsilon
    + (c + a\cos\beta)\frac{\sqrt{a^2 + N^2c^2 + 2Na\,c\cos\beta}}{Nc + a\cos\beta}\varepsilon \notag \\
    =& \sqrt{a^2 + N^2c^2 + 2Na\,c\cos\beta}\,\varepsilon = Y_N\varepsilon
\end{align}
We can check, that the limit $N\to\infty$ becomes $c$ as we would expect, when we make the slice infinitely slim and the top-line converges to the base-line:
\begin{equation}
    \lim_{N\to\infty}\bar{b}_0^1 = \lim_{N\to\infty}\sqrt{\frac{a^2}{N^2} + c^2 + \frac{2}{N}a\,c\cos\beta}
    = \sqrt{c^2} \overset{c>0}{=} c \qquad \checkmark
\end{equation}
The limiting value of $\alpha_0^0$ is zero as expected:
\begin{equation}
    \alpha_0^0 = \sin^{-1}\frac{e^1}{Y_N}, \qquad \lim_{N\to\infty}\alpha_0^0
    = \lim_{N\to\infty}\sin^{-1}\left( \frac{a\sin\beta}{\sqrt{a^2 + N^2c^2 + 2Na\,c\cos\beta}} \right)
    = \sin^{-1}0 = 0 \qquad \checkmark
\end{equation}

To calculate the angle $\beta_2$ for the next slice and thus for the recursion across slices we need it's vertical angle, which is the sum of the two adjacent angles $\hat{\alpha}_0^{N-1}$ and $\gamma_0^{N-1}$:
\begin{align}
    \hat{\alpha}_0^{N-1} =& \sin^{-1}\frac{c\,e^{N-1}}{Y_1Y_N}, \qquad
    \gamma_0^{N-1} = \sin^{-1}\left( \frac{c}{Y_1}\sin\beta \right), \qquad
    \hat{\alpha}_0^{N-1},\ \gamma_0^{N-1} < \frac{\pi}{2} \\
    \beta_0^2 =& \hat{\alpha }_0^{N-1} + \gamma _0^{N-1}
    = \sin^{-1}\left( \frac{Nc}{Y_N}\sin\beta \right), \qquad \lim_{N\to\infty}\beta_0^2 = \beta \qquad \checkmark
\end{align}
The limit shows, that if $N$ tends to infinity, the difference between angles of subsequent slices in this case $\beta_0^1$ and $\beta_0^2$ vanishes, as we would expect.

\subsubsection{$2$-nd order solution}\label{subsubsec: $2$-nd order solution}
We list the results for all quantities we need to calculate the top line, top angle and opening angle of the first slice $\Sigma_1$. We use these expressions for the plot in Fig.~\ref{fig: 1st Slice Test-plot} with which we check the consistency of our results.\\

The second order term of the top line is calculated by adding the lengths of the top lines of all segments $\Delta_1$ - $\Delta_{N-1}$. The last segment $\Delta_N$ has no top line.
\begin{equation}
    \bar{b}_2 = b_3^1\,\varepsilon + \sum_{i=2}^{N-2}b_3^i\,\varepsilon + b_3^{N-1}\varepsilon,
\end{equation}
They all depend on the lengths of the rib-lines $a_2^i$ in their segments $\Delta_i$. The first and the second last segments, $\Delta_1$ and $\Delta_{N-1}$ are special cases, since they are boundary problems. The terms containing curvature values in the regular cases are collected in $\mathcal{K}_b^i$. It consists of a term proportional to the curvature parameter $C^{i-1}$, representing the constant term in the ansatz for $\hat{\alpha}_2^i$, and two terms proportional to the curvature of the current $\Delta_i$ and the next slice $\Delta_{i+1}$. We pack their complicated proportionality factors into $\mathrm{X}_b^i$ and $\mathrm{Z}_b^i$ respectively.
\begin{align}
    b_3^1 =& -\frac{e^1}{Y_N}\left( \frac{c^3e^1}{X_N}\left[ \frac{K_1}{3} + K_2 \right] - a_3^2 \right), \quad
    b_3^{N-1} = -\frac{c^2 \left(e^1\right)^2}{6 Y_N}\left( (N-1)^2K_{N-1}\frac{X_1}{X_N}
    + (2N-1)K_N + \frac{6}{c^2e^1}a_3^{N-1} \right), \notag \\
    b_3^i =& -\frac{e^1}{(i+1)Y_N}\left( \frac{c^3e^1}{6X_N}i\mathcal{K}_b^i
    - a_3^{i+1} \right), \qquad
    \mathcal{K}_b^i \coloneq \frac{C^{i-1}}{X_N^2} + iK_i\mathrm{X}_b^i
    + (i+1)K_{i+1}\mathrm{Z}_b^i, \notag \\
    \mathrm{X}_b^i \coloneq& i + 5 + \frac{i-1}{i} + (i+3)\frac{(e^1)^2}{X_N^2}, \quad
    \mathrm{Z}_b^i \coloneq \left( i + 2 + \frac{e^1e^{i+1}}{X_N^2} \right)
    \left( 1 + \frac{(e^{i+1})^2}{(k_N^{i+1})^2} \right)
    + \frac{(i+1)^2X_N^2 + (2i+1)e^1e^{i+1}}{i(k_N^{i+1})^2}, \notag \\
    i\in&\{2,\ldots,N-2\}.
\end{align}

The opening angle of the slice is given by the opening angle of the first segment in the slice. It directly depends on the curvature value in it's segment $\Delta_1$ and the next one $\Delta_2$ and the rib-line of the second segment:
\begin{align}
    \alpha_2^0 =& \frac{X_N^2}{2Y_N^2}\left( \frac{c^2e^1}{3 X_N}\left[
    K_1\left( 3 + \frac{4(e^1)^2}{X_N^2} \right) + K_2\frac{(k_N^2)^2}{X_N^2} \right]
    + \frac{a_3^2}{c} \right).
\end{align}

The top angle of the first slice is the opposite angle to the direction angle of the second slice: $\bar{\gamma}^1 = \beta^2$. It's second order term is the sum of the top angle in the last segment $\Delta_N$ and the angle opposite to the rib-line in the upper triangle of the second last segment $\hat{\Delta}_{N-1}$.
\begin{equation}
    \bar{\gamma}_2^1 = \hat{\alpha}_2^{N-1} + \gamma_2^{N-1}.
\end{equation}
Also $\hat{\alpha}_2^{N-1}$ depends on the curvature values in it's segment and the next one and on the rib-line in it's segment. Since there is no next segment to $\Delta_N$ and it contains no rib-line, the top angle of the last segment only depends on the last curvature value.
\begin{align}
    \hat{\alpha}_2^{N-1} =& \frac{1}{X_1Y_N^2}\left( \frac{c\,e^{N-1}}{6}\left[
    K_{N-1}\frac{X_1}{X_N}(X_1Y_N^2 - c\,e^{N-1}e^1) + K_Nc(e^1)^2\left( \frac{Nc^2 - a^2}{Y_1^2} - 1 \right) \right]
    + X_N^2a_3^{N-1} \right), \\
    \gamma_2^{N-1} =& \frac{K_N}{6}\left( 1 + \frac{cX_1}{Y_1^2} \right)c\,e^1.
\end{align}

We saw, that all quantities we want to calculate ultimately depend on the length's of the rib-lines, which is where the name comes from. We again collect all terms containing curvature values in $\mathcal{K}_a^i$, which has the same structure as $\mathcal{K}_b^i$. In difference to the top lines $b_2^i$ the rib-lines contain all $\mathcal{K}_a^n$ of the following segments and with that all their curvature values: $K_{n}$, $n\in\{i,\ldots,N-1\}$. We also observe, that the special case of the last rib-line $a^{N-1}$ appears in all other rib-lines.
\begin{align}
    a_3^i =& \frac{c^3e^1}{6X_N}\sum_{n=i}^{N-2}\frac{i}{n+1}\mathcal{K}_a^n + \frac{i}{N-1}a_3^{N-1}, \qquad
    \mathcal{K}_a^i \coloneq \frac{C^{i-1}}{X_N^2} + i\,K_i\mathrm{X}^i
    + (i+1)K_{i+1}\mathrm{Z}^i, \\
    \mathrm{X}^i \coloneq& \frac{1}{(k_N^i)^2}\left[ \left(5-\frac{1}{i}\right)X_N^2
    + i^2(i+3)\frac{(e^1)^4}{X_N^2} + (5i^2 + 3)(e^1)^2 \right], \quad
    \mathrm{Z}^i \coloneq 1 + 2(i+1)\frac{(e^1)^2}{X_N^2}, \qquad i\in\{2,\ldots,N-2\}. \notag
\end{align}
The last rib-lines is proportional to a term with the structure of a $\mathcal{K}_a^i$ consistent with the expectation that is should only have one, since there is no next one and it should not depend on itself. Since it is a special case as a quantity on the boundary, factors $\chi$ and $\zeta$, which are the analogues to $\mathrm{X}^i$ and $\mathrm{Z}^i$, are not consistent with the regular case.
\begin{align}
    a_3^{N-1} =&\, \frac{(N-1)c^2X_1}{6X_N\left[ (N-1)c + X_1 \right]}\left( \frac{c\,e^1}{X_N^2}C^{N-2}
    + K_{N-1}\chi + K_N\zeta \right), \qquad \zeta = \frac{e^1}{X_1}( 2a^2 + Nc^2 + (2N+1)a\,c\cos\beta ), \notag \\
    \chi =&\, (N-1)\frac{e^1}{X_N}\left( a^2 + Nc^2 + (N+1)a\,c\cos\beta + \frac{X_1}{X_N}Y_N^2 \right) \notag \\
    &+ \left( Y_N^2\frac{(k_N^{N-1})^2 + X_N^2}{(k_N^{N-1})^2X_N^2} + \frac{N-2}{(k_N^{N-1})^2}\left[
    \frac{X_N^2}{N-1} + (2N-1)(e^1)^2 \right] \right)c\,e^{N-1}.
\end{align}

All quantities we discussed before depend in some way on the recursion parameter $C^i$. The contributions of the other two recursion parameters $A^i$ and $D^i$ are contained in $\mathcal{K}^i$ and do not appear explicitly anymore. In the $C^i$ the curvature values are added in the opposite direction, containing all $K_n$ from the first one up to the $i$-th one. This way all quantities in the end depend on all curvature values in their slice.
\begin{align}
    C^1 =&\, K_1\left[X_N^2 + 2Y_N^2\right], \qquad 
    C^i = \frac{2}{i+1}C^1 + \sum_{k=2}^i \frac{k}{i+1}\mathcal{K}^k \\
    \mathcal{K}^i =&\, K_i\left[ 2iX_N^2 + i\left( 3 + \frac{X_N^2}{(k_N^i)^2} \right)Y_N^2
    + \left( 1 + \frac{i(i^2-1)X_N^2}{(k_N^i)^2} \right)(e^1)^2 \right], \qquad i\in\{2,\ldots,N-2\}
\end{align}

We check the consistency of these expressions for the first slice and the case of finite $N$ and constant curvature $K$ on various values for the parameters $a, c$ and $\beta$ by plotting the first slice including rib-lines as we show in Fig.~\ref{fig: 1st Slice Test-plot}. We use the second order terms to the $C$- and $SC$-functions~\eqref{eq: C(K,gamma,a,b)_expanded}, \eqref{eq: SC(K,gamma,a,b)_expanded} to plot the curved lines. If everything is consistently expanded to second order, the triangle has to match up and the rib-lines have to end at the top line.\\
We cannot test the case of varying curvature this way, since \eqref{eq: C(K,gamma,a,b)_expanded}, \eqref{eq: SC(K,gamma,a,b)_expanded} are not the correct functions to draw the curved lines in this case. 
\begin{figure}[h!]
    \centering
    \includegraphics[width=\linewidth]{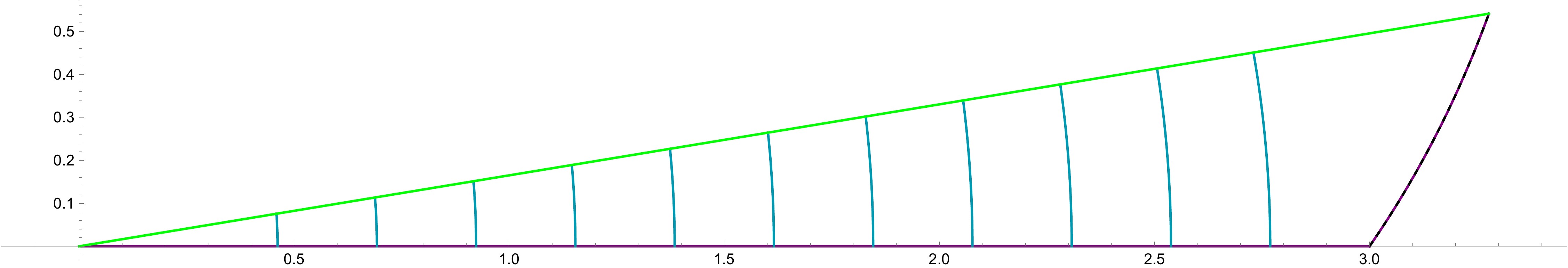}
    \caption{We test our method on the example of the unit sphere, where we expect a precise agreement with the expansions made in Sec.~\ref{sec: C, SC}. We plot the example $N=13$, $a=5$, $c=3$, $\beta=\frac{\pi}{6}$. The chosen values for $a$ and $c$ are not smaller then one in which case we would not expect this to be a good approximation. But this is not only to be seen as an approximation but rather as the second order terms of power series which as a hole are precisely valid for large values.}
    \label{fig: 1st Slice Test-plot}
\end{figure}

\section{The $j$-th slice}\label{sec: The $j$-th slice}
Now that we dealt with the first slice we have everything we need to complete the triangulation.\\
Each slice is congruent to the first one with a different baseline $c^j$ instead of $c^1 = c$ and angle $\beta^j$ instead of $\beta^1 = \beta$, with one exception. The two new arguments have higher order correction terms.\\
The idea is, to calculate the $c^j$ and $\beta^j$ using the results of the first slice in terms of the previous baseline and angle $c^{j-1}$, $\beta^{j-1}$. Thus we get a coupled recursion for the two quantities: $c^j(c^{j-1},\beta^(j-1)$, and $\beta^j(c^{j-1},\beta^{j-1})$.\\
The baseline of the $j$-th slice is the top-line of the previous slice and its angle is the same as the angle across, belonging to the previous triangle.
\begin{align}
    c^j(c^{j-1},\beta^{j-1}) =&\, \bar{b}^{j-1} = \bar{b}(K_{i,j-1};a,c^{j-1},\beta^{j-1}), \notag \\
    \beta^j(c^{j-1},\beta^{j-1}) =&\, \bar{\gamma}^{j-1} \coloneq \hat{\alpha}^{N-1,j-1} + \gamma^{N-1,j-1}
    = \bar{\gamma}(K_{i,j-1};a,c^{j-1},\beta^{j-1}),
\end{align}
\begin{equation}
    \bar{b} \coloneq \bar{b}^1 = \sum_{i=1}^{N-1}b^{i,1}, \qquad \bar{\gamma} \coloneq \bar{\gamma}^1 = \hat{\alpha}^{N-1,1} + \gamma^{N-1,1}
\end{equation}
\begin{equation}
    K_{i,j} = K(q_{i-1,j}) = K((i-1)c^j\varepsilon,\bar{\alpha}^{j-1}), \quad
    q_{i,j} = \gamma_{o,\bar{\alpha}^{j-1}}(ic^j\varepsilon) \quad i,j\in\{1,\ldots,N\},
\end{equation}
where we inserted the faithful normal coordinates of $q_{i-1,j}$.
\begin{align}
    \bar{\alpha}^{j-1} = \sum_{l=1}^{j-1} \alpha^{0,l}
    = \sum_{l=1}^{j-1} \alpha^{0,1}(K_{i,l};a,c^l,\beta^l)
\end{align}
As we can see from our results in the previous section the second order terms of $c^2$ and $\bar{\alpha}^1$ are linear in the $K_i$ and can thus be written up to second order in the following form:
\begin{align}
    c^2 \approx c_0^2 + c_2^2\varepsilon^2 = \bar{b}_0 + \sum_{i=1}^N\kappa_b^iK_{i1}\varepsilon^2, \qquad \bar{\alpha}^1 \approx \bar{\alpha}_0^1 + \bar{\alpha}_2^1\varepsilon^2
    = \alpha_0^0 + \sum_{i=1}^N\kappa_\alpha^iK_{i1}\varepsilon^2
\end{align}
We only need the curvature values $K_{i2}$ of the second slice $\Sigma_2$ to zeroth order, since they only appear in second order terms.
\begin{equation}
    K_{i2} = K((i-1)c^2\varepsilon,\bar{\alpha}^1)
    \approx K\left((i-1)\left\{\bar{b}_0\varepsilon + \sum_{i=1}^N\kappa_b^iK_{i1}\varepsilon^3\right\},
    \alpha_0^0 + \sum_{i=1}^N\kappa_\alpha^iK_{i1}\varepsilon^2\right)
    \approx K\left((i-1)\bar{b}_0\varepsilon,\alpha_0^0\right) + \mathcal{O}(\varepsilon)
\end{equation}
We see, that the curvature terms decouple from the recursion, since they are shifted in the order in $\varepsilon$ and we can calculate all curvature values we need from the zeroth order solution. Extending the zeroth order solutions from the first slice to all slices is straight forward. We again apply the cosine and sine laws the triangle, which includes all slices up to the $j$-th one and we thus have $j\,a$ instead of $a$, see Fig.~\ref{fig: 0-th order Triangulation}.
\begin{figure}[h!]
    \centering\includegraphics[width=.8\linewidth]{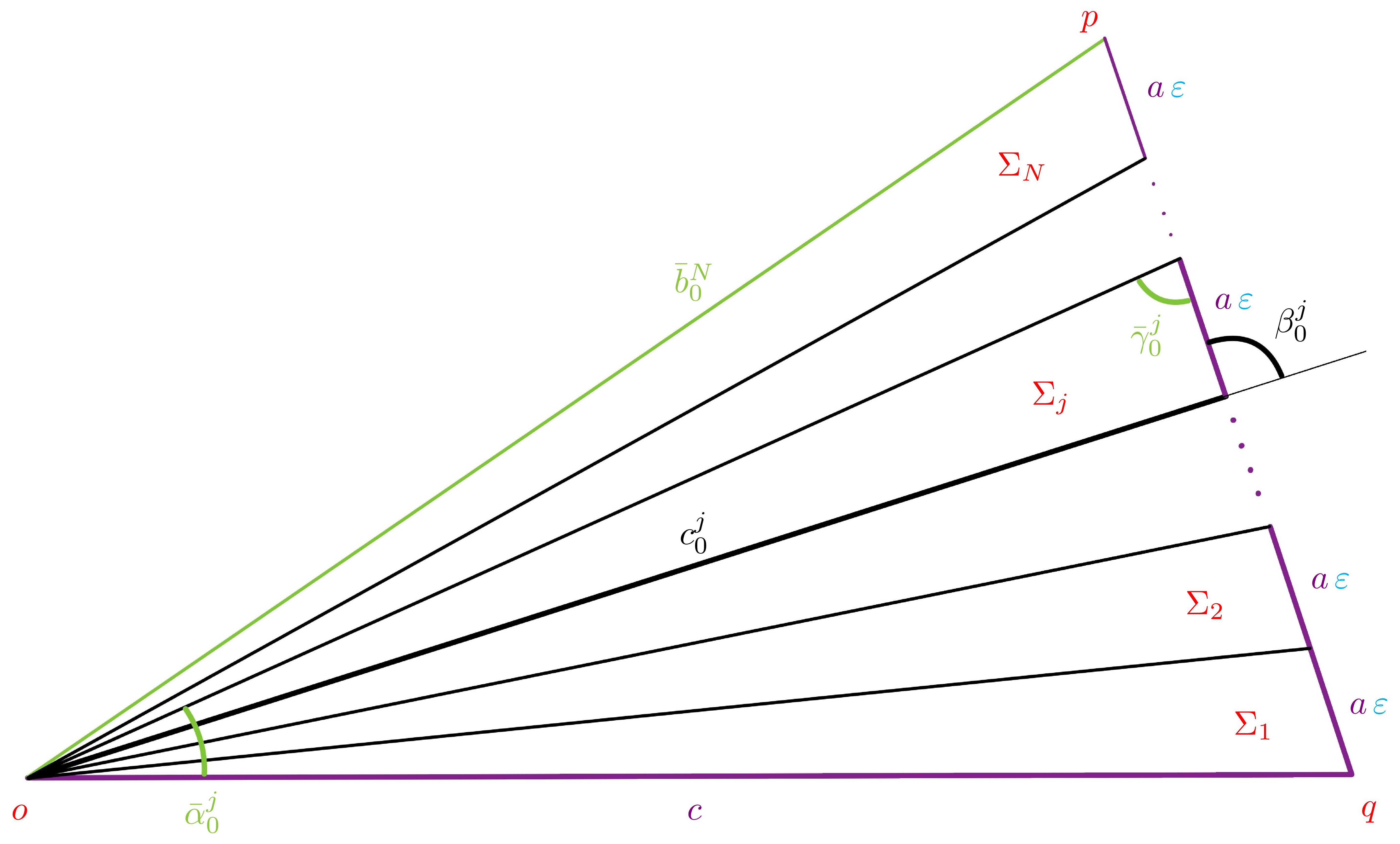}
    \caption{A sketch of the zeroth order slicing, showing the relations between the base- $c_0^j$ and top-lines $b_0^j$, direction- $\beta_0^j$ and top-angles $\gamma_0^j$ and $\bar{\alpha}_0^j$ and the opening angles of the slices $\alpha_0^j$.}
    \label{fig: 0-th order Triangulation}
\end{figure}
\begin{align}\label{eq: c_0^j -> K_ij}
    &c_0^j = \bar{b}_0^{j-1} = \sqrt{ (j-1)^2a^2 + N^2c^2 + 2N(j-1)\,a\,c\cos\beta}\,\varepsilon, \qquad
    \bar{\alpha}_0^j = \sin^{-1}\frac{j\,a\sin\beta}{\sqrt{ j^2a^2 + N^2c^2 + 2Nj\,a\,c\cos\beta}}, \\
    &\rightarrow \quad K_{ij} = K\left((i-1)c_0^j\varepsilon,\bar{\alpha}_0^{j-1}\right) \label{eq: K_ij}
\end{align}
In the same way we get:
\begin{align}\label{eq: beta_0^j}
    \beta_0^j = \bar{\gamma}_0^{j-1} = \sin^{-1}\left( \frac{Nc\sin\beta}{\sqrt{ (j-1)^2a^2 + N^2c^2 + 2N(j-1)\,a\,c\cos\beta}} \right)
\end{align}

In difference to the first slice the arguments $c^j$ and $\beta^j$ in place of $c$ and $\beta$ have higher order terms:
\begin{align}
    \bar{b} &\approx \bar{b}_0 + \bar{b}_2\varepsilon^2 + \mathcal{O}(\varepsilon^3), &\quad
    \bar{\gamma} &\approx \bar{\gamma}_0 + \bar{\gamma}_2\varepsilon^2 + \mathcal{O}(\varepsilon^3), \\
    \bar{b}_0 &= \sum_{i=1}^{N-1} b_1^{i,1}\varepsilon, \quad \bar{b}_2 = \sum_{i=1}^{N-1} b_3^{i,1}\varepsilon
    &\quad \bar{\gamma}_0 &= \hat{\alpha}_0^{N-1} + \gamma_0^{N-1}, \quad
    \bar{\gamma}_2 = \hat{\alpha}_2^{N-1} + \gamma_2^{N-1}
\end{align}
\begin{align}
    \Rightarrow \quad c^j \approx&\, \bar{b}_0\left(a,c^{j-1},\beta^{j-1}\right)
    + \bar{b}_2\left(K_{ij};a,c^{j-1},\beta^{j-1}\right)\varepsilon^2 + \mathcal{O}(\varepsilon^3)
    = c_0^j + c_2^j\varepsilon^2 + \mathcal{O}(\varepsilon^3), \\
    \beta^j \approx&\, \bar{\gamma}_0\left(a,c^{j-1},\beta^{j-1}\right)
    + \bar{\gamma}_2\left(K_{ij};a,c^{j-1},\beta^{j-1}\right)\varepsilon^2 + \mathcal{O}(\varepsilon^3)
    = \beta_0^j + \beta_2^j\varepsilon^2 + \mathcal{O}(\varepsilon^3)
\end{align}
\color{black}

The immediate way to extend the solution of the first slice to the case of the $j$-th slice would be, to substitute the arguments $c\mapsto c_0^j+c_2^j\varepsilon^2$, $\beta\mapsto\beta_0^j+\beta_2^j\varepsilon^2$ in the functions $\bar{b}$ and $\bar{\gamma}$ of the first slice and then expand them to second order in $\varepsilon$.
\begin{align}
    c^j \approx&\, \bar{b}_0(a,c_0^{j-1} + c_2^{j-1}\varepsilon^2,\beta _0^{j-1}
    + \beta _2^{j-1} \varepsilon^2)
    + \bar{b}_2(K_{i j};a,c_0^{j-1}+c_2^{j-1}\varepsilon^2,\beta_0^{j-1}
    + \beta_2^{j-1}\varepsilon^2)\varepsilon^2 + \mathcal{O}(\varepsilon^3) \notag \\
    \approx&\, \bar{b}_0(a,c_0^{j-1},\beta_0^{j-1})
    + \left\{ \left[ \partial_c\,\bar{b}_0(a,c_0^{j-1},\beta_0^{j-1}) \right]c_2^{j-1}
    + \left[ \partial_\beta\,\bar{b}_0(a,c_0^{j-1},\beta_0^{j-1}) \right]\beta_2^{j-1} + \bar{b}_2(K_{ij};a,c_0^{j-1},\beta_0^{j-1}) \right\}\varepsilon^2 + \mathcal{O}(\varepsilon^3) \notag \\
    \approx&\, \overset{\circ}{b}\,_0^{j-1}+ \left\{ \overset{\circ}{b}\,_0^{j-1,c}c_2^{j-1}
    + \overset{\circ}{b}\,_0^{j-1,\beta}\beta_2^{j-1} + \overset{\circ}{b}\,_2^{j-1} \right\}\varepsilon^2
    + \mathcal{O}(\varepsilon^3) \label{eq: b expansion} \\
    \vphantom{a}\notag\\
    \beta^j \approx&\, \bar{\gamma}_0(a,c_0^{j-1}+c_2^{j-1}\varepsilon^2,\beta_0^{j-1}+\beta_2^{j-1}\varepsilon^2)
    + \bar{\gamma}_2(K_{ij};a,c_0^{j-1}+c_2^{j-1}\varepsilon^2,\beta_0^{j-1}+\beta_2^{j-1}\varepsilon^2)
    + \mathcal{O}(\varepsilon^3) \notag \\
    \approx&\, \bar{\gamma}_0(a,c_0^{j-1},\beta_0^{j-1})
    + \left\{ \left[ \partial_c\,\bar{\gamma}_0(a,c_0^{j-1},\beta_0^{j-1}) \right]c_2^{j-1}
    + \left[ \partial_\beta\,\bar{\gamma}_0(a,c_0^{j-1},\beta_0^{j-1}) \right]\beta_2^{j-1}
    + \bar{\gamma}_2(K_{ij};a,c_0^{j-1},\beta_0^{j-1}) \right\}\varepsilon^2 + \mathcal{O}(\varepsilon^3) \notag \\
    \approx&\, \overset{\circ}{\gamma}\,_0^{j-1} + \left\{ \overset{\circ}{\gamma}\,_0^{j-1,c}c_2^{j-1}
    + \overset{\circ}{\gamma}\,_0^{j-1,\beta}\beta_2^{j-1} + \overset{\circ}{\gamma}\,_2^{j-1} \right\}\varepsilon^2 + \mathcal{O}(\varepsilon^3) \label{eq: beta expansion}
\end{align}
We will see in the following subsection however, that this does not work, since the Taylor expansion does not adequately describe the approximation scheme we constructed. To calculate the correct argument correction coefficients:
\begin{align}
    \overset{\circ}{b}\,_n^{j,c} \neq \partial_c\,\bar{b}_n(K_{ij};a,c_0^j,\beta_0^j), \quad
    \overset{\circ}{b}\,_n^{j,\beta} \neq \partial_\beta\,\bar{b}_n(K_{ij};a,c_0^j,\beta_0^j), \quad
    \overset{\circ}{\gamma}\,_n^{j,c} \neq \partial_c\bar{\gamma}_n(K_{ij};a,c_0^j,\beta_0^j), \quad
    \overset{\circ}{\gamma}\,_n^{j,\beta} \neq \partial_\beta\bar{\gamma}_n(K_{ij};a,c_0^j,\beta_0^j),
\end{align}
we need to extend the method used for the first slice to the case with second order corrections to the arguments $c$ and $\beta$ and collect the contributions to the $c_2^j$- and $\beta_2^j$-coefficients for both functions $\bar{b}_2$ and $\bar{\gamma}_2$. We abbreviate them from now on with argument coefficients.\\

Following the same procedure as for the first slice, we first need to extend the problems.

\subsection{Extended Problems}\label{subsec: Extended Problems}
In the simplified notation, dropping the $j$-index, the new arguments are $a\mapsto a$, $c\mapsto c_0 + c_2\varepsilon^2$ and $\beta\mapsto \beta_0 + \beta_2\varepsilon^2$.\\
We update the template formulas from Sec.~\ref{subsec: Solving Problems of different Types} including the additional terms arising from the perturbation of the arguments $c$ and $\beta$.\\
We calculate the additional terms arising from this perturbation in the following subsections. It can be checked, that the results we get reduce to the ones of the first slice, if we set $c_0 = c$, $c_2 = 0$, $\beta_0 = \beta$ and $\beta_2 = 0$.\\

Problem 1 is the only problem, which is directly affected by both changes and thus for $d(\varepsilon)$ and $\alpha(\varepsilon)$ the $b_3$ and $\gamma_2$ are mapped to $c_2$ and $-\beta_2$ instead of $0$.
\begin{figure}[h!]
\begin{minipage}{0.3\linewidth}
    \centering\includegraphics[width=\linewidth]{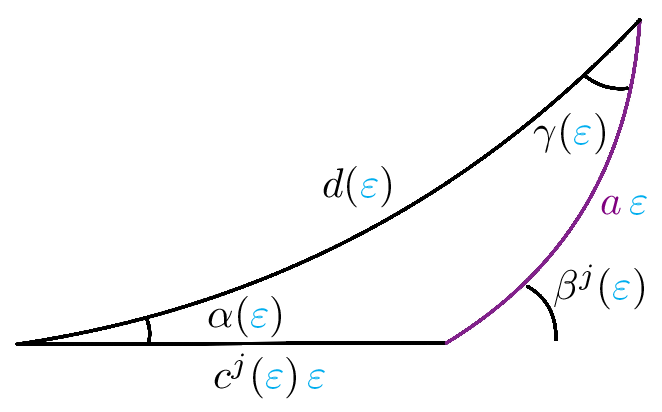}
\end{minipage}
\hfill
\begin{minipage}{0.5\linewidth}
    \begin{align}
    \left.\begin{matrix}
        d(\varepsilon) = C(K;\pi-\beta(\varepsilon),a\varepsilon,c(\varepsilon)) \\
        \alpha(\varepsilon) = SC(K;\pi-\beta(\varepsilon),a\varepsilon,c(\varepsilon))
    \end{matrix}\right\} \quad \Rightarrow \quad
    \begin{matrix}
    a_1 \mapsto a, & a_2 \mapsto 0, & a_3 \mapsto 0 \\
    b_1 \mapsto c_0, & b_2 \mapsto 0, & b_3 \mapsto c_2 \\
    \gamma_0 \mapsto \pi-\beta_0, & \gamma_1 \mapsto 0, & \gamma_2 \mapsto -\beta_2
    \end{matrix} \notag \\
    \gamma(\varepsilon) = SC(K;\pi-\beta(\varepsilon),c(\varepsilon),a\varepsilon) \quad
    \Rightarrow  \quad a_1 \mapsto c_0, \ a_3 \mapsto c_2, \quad b_1 \mapsto a, \ b_3 \mapsto 0 \notag
\end{align}
\end{minipage}
\caption{Problem 1, $j$-th slice, segment $\Delta_N$.}\label{fig: Extended Problem 1}
\end{figure}

Problem 2, Type 1 is only affected by the change to $c$.
\begin{figure}[h!]
\begin{minipage}{0.3\linewidth}
    \centering\includegraphics[width=\linewidth]{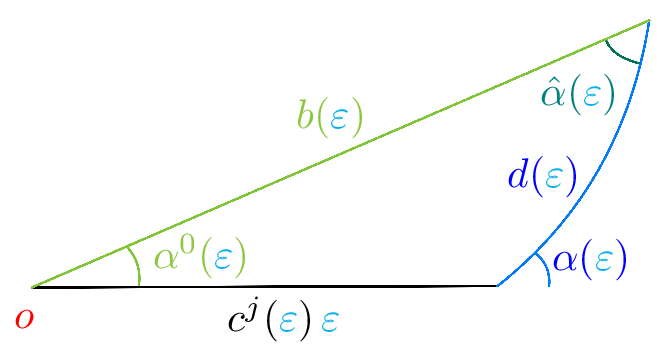}
\end{minipage}
\begin{minipage}{0.6\linewidth}
\begin{align}
    &\left.\begin{matrix}
    b(\varepsilon) = C(K;\pi-\alpha(\varepsilon),c(\varepsilon),d(\varepsilon)) \\
    \hat{\alpha}(\varepsilon) = SC(K;\pi-\alpha(\varepsilon),c(\varepsilon),d(\varepsilon))
    \end{matrix}\right\} \quad \Rightarrow \quad
    \begin{matrix}
    a_1 \mapsto c_0, & a_2 \mapsto 0, & a_3 \mapsto c_2 \\
    b_1 \mapsto d_1, & b_2 \mapsto 0, & b_3 \mapsto d_3 \\
    \gamma_0 \mapsto \pi-\alpha_0, & \gamma_1 \mapsto 0, & \gamma_2 \mapsto -\alpha_2
    \end{matrix} \notag \\
    \vphantom{a}\notag\\
    &\ \alpha^0(\varepsilon) = SC(K;\pi-\alpha(\varepsilon),d(\varepsilon),c(\varepsilon)), \quad \Rightarrow
    \begin{matrix}
    \quad a_1 \mapsto d_1, \quad a_3 \mapsto d_3 \\
    \quad b_1 \mapsto c_0,\quad b_3 \mapsto c_2
    \end{matrix} \notag
\end{align}
\end{minipage}
\caption{Problem 2, Type 1, $j$-th slice, segment $\Delta_1$.}
\label{fig: Extended Problem 2, Type 1}
\end{figure}

It may look a bit confusing, that we have a $c_0$ term which is on the same order in $\varepsilon$ as $d_1$ and the same with $c_2$ and $d_3$. This comes from the different way, the two arise. The $d(\varepsilon)$ describes a side-length of a small triangle in a segment, which we wrote as a truncated power series $d(\varepsilon) = \sum_{n=1}^3d_n\varepsilon^n$. The series has to start at $1$, since $d$ is at most of order $\mathcal{O}(\varepsilon)$. The $c^j$ however describes the entire base-line of the $j$-th segment. So, $c^j$ is of order $\mathcal{O}(1)$ and thus it's truncated power series starts at $0$: $c^j = \sum_{n=0}^2c_n\varepsilon^n$. It is then cut into $N$ pieces of equal length $\frac{c}{N} = c\varepsilon$ and thus each piece $c_0\varepsilon$ is of the same order as $d_1\varepsilon$ and so on.\\

In the Types 2 and 3 of Problem 2 only the lower left triangle $\overset{\vee}{\Delta}$ is affected by the change to $c$. And so, $b_3$ becomes $c_2$ instead of vanishing. We observe, that only the left hand side of the $\delta$-equation changes, since the right hand side only involves the upper right triangle $\hat{\Delta}$ which is not affected by this change. We thus only list the quantities obtained using $\overset{\vee}{\Delta}$:
\begin{figure}[h!]
\begin{minipage}{0.3\linewidth}
    \centering\includegraphics[width=\linewidth]{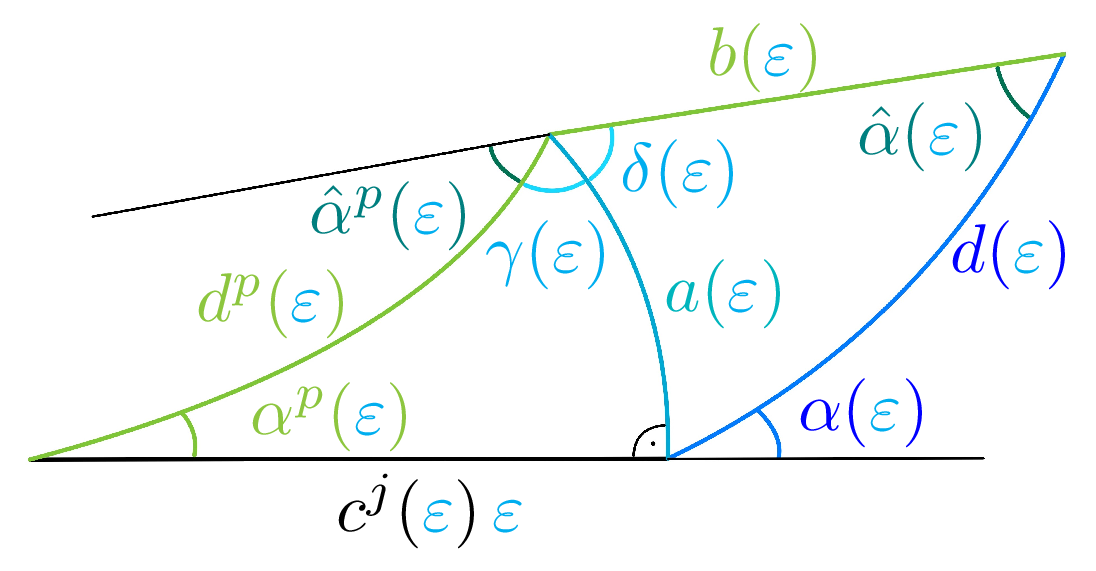}
\end{minipage}
\begin{minipage}{0.7\linewidth}
\begin{align}
    \left.\begin{matrix}
        d^p(a(\varepsilon)) = C(K;\frac{\pi}{2},a(\varepsilon),c(\varepsilon)) \\
        \alpha^p(a(\varepsilon)) = SC(K;\frac{\pi}{2},a(\varepsilon),c(\varepsilon))
    \end{matrix}\right\} \quad \Rightarrow \quad
    \begin{matrix}
    a_1 \mapsto a_1, & a_2 \mapsto 0, & a_3 \mapsto a_3 \\
    b_1 \mapsto c_0, & b_2 \mapsto 0, & b_3 \mapsto c_2 \\
    \gamma_0 \mapsto \frac{\pi}{2}, & \gamma_1 \mapsto 0, & \gamma_2 \mapsto 0
    \end{matrix} \notag \\
    \delta = \pi - SC\left(K;\frac{\pi}{2},c(\varepsilon),a(\varepsilon)\right) - \hat{\alpha}^p(a(\varepsilon))
    = \pi - SC\left(K;\frac{\pi}{2} - \alpha(\varepsilon),d(\varepsilon),a(\varepsilon)\right)
\end{align}
\end{minipage}
\caption{Problem 2, Types 2 and 3, $j$-th slice, segments $\Delta_2 - \Delta_{N-1}$}
\label{fig: Extended Problem 2, Types 2 & 3}
\end{figure}

This is the case, because no quantity from this triangle $\hat{\Delta}$ in the first slice $\Sigma_1$ is directly dependent on $c$ or $\beta$. They are indirectly dependent on $c$ which causes a mistake, if we generalize this result to the $j$-th segment, by replacing $c$ with $c_0^j + c_2^j$ and then expanding in $\varepsilon$. Since there is no direct dependence on $c^j$ in the right hand side of the $\delta$-equation, there cannot be a $c_2^j$ in these terms.
The $\hat{\alpha}_2^p(a_3)$ is a more convoluted case. At first glance it seems to have no $c_2^j$-dependence, since it is calculated from the upper right triangle $\hat{\Delta}$. The $c_2^j$ does indeed not appear in the ansatz, but it does come in via $\alpha_2^p$ and $d_3^p$ which are expressed in terms of $a_3(\varepsilon)$, to set up the $\delta$-equation, using the lower left triangle $\overset{\vee}{\Delta}$. So, $\hat{\alpha}_2$ will also have different $c_2^j$-dependence, then what we would expect from Taylor expansion.\\
This is the reason, why the direct substitution and Taylor expansion does not lead to the correct result. And thus, the only way to derive the correct expressions, is to go through the entire calculation again and extend it to the generic case of the $j$-th slice, which we summarize below.\\

We list the zeroth and second order template formulas, which are different from the first slice. Due to the additional terms proportional to $c_2$ and $\beta_2$ we cannot factor out some pre-factors as we could in the first slice and in turn the terms become longer, but more systematic. Since we did not introduce a first order perturbation $c_1\varepsilon$, nothing changes on the first order and the same argument can thus be applied to check, that the first order terms still all vanish.\\
 
\textbf{Problem 1:}
\begin{align}
    d_1 &= y, &\quad d_3 &= -\frac{1}{y}\left( \frac{K}{6}c_0^2a^2\sin^2\beta_0
    - c_2(c_0 + a\cos\beta_0) + \beta_2c_0a\sin\beta_0 \right) \notag \\
    \alpha_0 &= \sin^{-1}\left(\frac{a}{y}\sin\beta_0\right), &\quad
    \alpha_2 &= \frac{1}{y^2}\left( \frac{K}{6}c_0a\sin\beta_0 \left[ y^2 + a(a + c_0\cos\beta_0) \right]
    + \beta_2\,a(a + c_0\cos\beta_0) - c_2\,a\sin\beta_0 \right) \notag \\
    \gamma_0 &= \sin^{-1}\left(\frac{c}{y}\sin\beta\right), &\quad 
    \gamma_2 &= \frac{1}{y^2}\left( \frac{K}{6}c_0a\sin\beta_0 \left[ y^2 + c_0(c_0 + a\cos\beta_0) \right]
    + \beta_2\,c_0(c_0 + a\cos\beta_0) + c_2\,a\sin\beta_0 \right), \notag
\end{align}
\begin{equation}
    \text{with} \quad y = \sqrt{a^2 + c_0^2 + ac_0\cos\beta_0}
\end{equation}

\textbf{Problem 2, Type 1:}
\begin{align}
    b_1 &= y, \qquad \hat{\alpha}_0 = \sin^{-1}\left(\frac{c_0}{y}\sin\alpha_0\right), \qquad \alpha_0^0 = \sin^{-1}\left( \frac{d_1}{y}\sin\alpha_0 \right), \qquad y = \sqrt{c_0^2 + d_1^2 + 2c_0\,d_1\cos\alpha_0} \notag \\
    b_3 &= -\frac{1}{y}\left(
    \frac{K}{6}c_0^2d_1^2\sin^2\alpha_0 - c_2(c_0 + d_1\cos\alpha_0)
    - d_3(d_1 + c_0\cos\alpha_0) + \alpha_2c_0d_1\sin\alpha_0 \right), \notag \\
    \hat{\alpha}_2 &= \frac{1}{y^2}\left( \frac{K}{6}c_0d_1\left[ y^2
    + c_0(c_0 + d_1\cos\alpha_0) \right]\sin\alpha_0 + c_2d_1\sin\alpha_0 - d_3c_0\sin\alpha_0
    + \alpha_2c_0(c_0 + d_1\cos\alpha_0) \right), \notag \\
    \alpha_2^0 &= \frac{1}{y^2}\left( \frac{K}{6}c_0d_1\left[ y^2
    + d_1(d_1 + c_0\cos\alpha_0) \right]\sin\alpha_0 - c_2d_1\sin\alpha_0 + d_3c_0\sin\alpha_0
    + \alpha_2d_1(d_1 + c_0\cos\alpha_0) \right).
\end{align}

\textbf{Problem 2, Types 2 and 3}
\begin{align}
    d^p_1 &= y_1, \qquad \alpha^p_0 = \sin^{-1}\frac{a_1}{y}; \qquad \delta^{(0)} = \pi - \sin^{-1}\frac{c_0}{y_1} - \hat{\alpha}^p_0
    \overset{!}{=} \pi - \sin^{-1}\left(\frac{d_1}{y_2}\right)\cos\alpha_0 \notag \\
    d^p_3 &= -\frac{1}{y_1}\left( \frac{K}{6}a_1^2c_0^2 - a_1a_3 - c_0c_2 \right), \qquad
    \alpha^p_2 = \frac{1}{y_1^2}\left( \frac{K}{6}a_1c_0(y_1^2 + a_1^2) + c_0a_3 - a_1c_2 \right); \notag \\
    \delta^{(2)} &= -\frac{1}{y_1^2}\left( \frac{K}{6}a_1c_0(y_1^2 + c_0^2) - c_0a_3 + a_1c_2 \right)
    - \hat{\alpha}^p_2 \notag \\
    &\overset{!}{=} \frac{1}{y_2^2}\left(
    \frac{K}{6}a_1d_1\left[ y_2^2 + d_1(d_1 - a_1\sin\alpha_0) \right]\cos\alpha_0
    + (a_1d_3 - a_3d_1)\cos\alpha_0 + d_1\alpha_2(d_1 - a_1\sin\alpha_0) \right), \notag \\
    y_1 &= \sqrt{a_1^2 + c_0^2}, \qquad y_2 = \sqrt{a_1^2 + d_1^2 - 2a_1d_1\sin\alpha_0}
\end{align}

\subsection{Generalizing the substitutes}\label{sec: generalized substitutes}
For the $j$-th slice we need to replace all $c$ and $\beta$ with $c_0^j$ and $\beta_0^j$. We used the substitutes to simplify our expressions in the first slice. So, the best way of generlizing our results to the $j$-th case is to start with generalizing the substitutes.
\begin{equation}
    Y_N^j \coloneq \sqrt{j^2a^2 + N^2c^2 + 2jNac\cos\beta}; \qquad c_0^j = Y_N^j\varepsilon, \quad
    \beta_0^j = \sin^{-1}\frac{Nc\sin\beta}{Y_N^{j-1}}
\end{equation}

We check the consistency with the first slice:
\begin{equation}
    Y_N^1 = Y_N \quad \wedge \quad c_0^1 = Y_N^0\varepsilon = \sqrt{N^2c^2}\frac{1}{N} = c \quad \checkmark
\end{equation}

When we generalize the substitutes we will encounter the term $\cos\beta_0^j$ a couple of times and thus it is practical to prepare its simplification beforehand:
\begin{align}
    \cos\beta_0^j =& \cos\sin^{-1}\frac{Nc\sin\beta}{Y_N^{j-1}}
    = \sqrt{1-\frac{N^2c^2\sin^2\beta}{(Y_N^{j-1})^2}}
    = \sqrt{\frac{(Y_N^{j-1})^2 - N^2c^2(1-\cos^2\beta)}{(Y_N^{j-1})^2}} \notag \\
    =& \sqrt{\frac{(j-1)^2a^2 + 2(j-1)Na\,c\cos\beta + N^2c^2\cos^2\beta}{(Y_N^{j-1})^2}}
    = \sqrt{\frac{[(j-1)a + Nc\cos\beta]^2}{(Y_N^{j-1})^2}} = \frac{(j-1)a + Nc\cos\beta}{Y_N^{j-1}}
\end{align}

Now we are prepared to check, whether $Y_N^j$ is consistent with the generalization of $Y_N$ by substituting the zeroth order arguments of the first slice $\Sigma_1$ with the ones of $\Sigma_j$.
\begin{align}
    \left.Y_N\right|&_{c\mapsto c_0^j, \beta\mapsto\beta_0^j} = \sqrt{a^2 + N^2(c_0^j)^2 + 2Na\,c_0^j\cos\beta_0^j} = \sqrt{a^2 + \cancel{N^2}(Y_N^{j-1})^2\cancel{\varepsilon^2}
    + 2\cancel{N}a\cancel{Y_N^{j-1}}\cancel{\varepsilon}\frac{(j-1)a + Nc\cos\beta}{\cancel{Y_N^{j-1}}}} \notag \\
    =& \sqrt{\{1 + (j-1)^2 + 2(j-1)\}a^2 + N^2c^2 + 2\{(j-1) + 1\}Na\,c\cos\beta} \notag \\
    =& \sqrt{(1+(j-1))^2a^2 + N^2c^2 + 2jNa\,c\cos\beta} = Y_N^j \quad \checkmark
\end{align}

We continue with the generalization of $Z_N$ which becomes much more important in the $j$-th case then it is in the first slice:
\begin{align}
    \left.Z_N\right|&_{c\mapsto c_0^j, \beta\mapsto\beta_0^j} = a + Nc_0^j\cos\beta_0^j
    = a + \cancel{N}\cancel{Y_N^{j-1}}\cancel{\varepsilon}\frac{(j-1)a + Nc\cos\beta}{\cancel{Y_N^{j-1}}}
    = ja + Nc\cos\beta \eqcolon Z_N^j
\end{align}

We will see, that one can write all generalized substitutes in terms of $Y_N^j$ and $Z_N^j$. Thus it is useful, to be aware of the following identities.
\begin{align}
    Z_N^{j-n} =&\, (j-n)a + Nc\cos\beta = Z_N^j - na \label{eq: Z_n^j-n} \\
    (Y_N^{j-n})^2 =&\, (j-n)^2a^2 + N^2c^2 + 2(j-n)Nac\cos\beta
    = (Y_N^j)^2 - 2nja^2 - 2nNac\cos\beta + n^2a^2 \notag \\
    =&\, (Y_N^j)^2 - 2naZ_N^j + n^2a^2 \label{eq: Y_N^j-n} \\
    (Y_N^{j-1})^2 + aZ_N^{j-1} =&\, (Y_N^j)^2 - 2aZ_N^j + a^2 + a(Z_N^j - a) = (Y_N^j)^2 - aZ_N^j \label{eq: Y and Z}
\end{align}

Next, we generalize $X_N$, $e^i$ and $k_N^i$:
\begin{align}
    X_N^j \coloneq& \left.X_N\right|_{c\mapsto c_0^j,\,\beta\mapsto\beta_0^j} = Nc_0^j + a\cos\beta_0^j
    = \cancel{N}Y_N^{j-1}\cancel{\varepsilon} + a\frac{(j-1)a + Nc\cos\beta}{Y_N^{j-1}} = \frac{(Y_N^{j-1})^2 + aZ_N^{j-1}}{Y_N^{j-1}} \notag \\
    \overset{\eqref{eq: Y and Z}}{=}& \frac{(Y_N^j)^2 - aZ_N^j}{Y_N^{j-1}} \\
    e^{ij} \coloneq& \left.e^i\right|_{c\mapsto c_0^j,\,\beta\mapsto\beta_0^j} = i\,a\sin\beta_0^j = i\,a\sin\sin^{-1}\frac{Nc\sin\beta}{Y_N^{j-1}} = \frac{iNa\,c\sin\beta}{Y_N^{j-1}} \\
    k_N^{ij} \coloneq& \left.k_N^i\right|_{c\mapsto c_0^j,\,\beta\mapsto\beta_0^j}
    = \sqrt{(X_N^j)^2 + (e^{ij})^2}
    = \frac{1}{Y_N^{j-1}}\sqrt{[(Y_N^j)^2 - aZ_N^j]^2 + i^2N^2a^2c^2\sin^2\beta}
\end{align}

Last we need to take care of two special cases. We cannot just set $N = 1$ to calculate $Y_1^j$, because $c_0^j$ and $\beta_0^j$ also depend on $N$ and this $N$ is not $1$. The same is true for $X_1^j$.
\begin{align}
    Y_1^j \coloneq& \left.Y_1\right|_{c\mapsto c_0^j,\,\beta\mapsto\beta_0^j}
    = \sqrt{a^2 + (c_0^j)^2 + 2ac_0^j\cos\beta_0^j}
    = \sqrt{a^2 + (Y_N^{j-1})^2\varepsilon^2
    + 2a\cancel{Y_N^{j-1}}\varepsilon\frac{(j-1)a + Nc\cos\beta}{\cancel{Y_N^{j-1}}}} \notag \\
    =& \frac{1}{N}\sqrt{\{N^2 + (j-1)^2 + 2N(j-1)\}a^2 + N^2c^2 + 2\{(j-1) + N\}Na\,c\cos\beta}
    = \frac{Y_N^{N+j-1}}{N} \notag \\
    X_1^j \coloneq& \left.X_1\right|_{c\mapsto c_0^j,\,\beta\mapsto\beta_0^j}
    = c_0^j + a\cos\beta_0^j = Y_N^{j-1}\varepsilon + a\frac{(j-1)a + Nc\cos\beta}{Y_N^{j-1}}
    = \frac{(Y_N^{j-1})^2 + NaZ_N^{j-1}}{NY_N^{j-1}}
\end{align}

\subsection{The extended $2$-nd order problem}\label{sec: The ext. $2$-nd order problem}
When we bring $c^j$ and $\beta^j$ into the form of the last lines in Eqs.~\eqref{eq: b expansion} and~\eqref{eq: beta expansion} we can construct a coupled recursion for their second order terms. We saw in Sec.~\ref{subsec: Extended Problems} why the Taylor approximation does not yield the correct expansion. Thus, instead of applying the Taylor approximation as described in these equations we collect all contributions coming from the second order corrections to the arguments in the argument coefficients. We call the remnants the interior values and mark a function with a $\circ$ when we evaluate them at the zeroth order contribution only. The interior values are the functions of the first slice, with the arguments $c$ and $\beta$ substituted by the zeroth order terms of the $j$-th baseline $c_0^j$ and direction angle $\beta_0^j$. Since we are expanding, the second order terms are linearized and thus the argument coefficients do not contain second order contributions of the arguments themselves and thus get marked with a $\circ$ as well:
\begin{align}
    &f^1(c,\beta) \quad \mapsto \quad f^j(c^j,\beta^j) \approx f^1(c_0^j + c_2^j\varepsilon^2, \beta_0^j + \beta_2^j\varepsilon^2) \approx \overset{\circ}{f}{}^j + \overset{\circ}{f}{}^{j,c}c_2^j + \overset{\circ}{f}{}^{j,\beta}\beta_2^j \\
    &\text{interior values:} \quad \overset{\circ}{f}{}^j \coloneq f^1(c_0^j,\beta_0^j), \qquad \text{argument correction coefficients:} \quad \overset{\circ}{f}{}^{j,c} \neq \partial_c f^1(c,\beta)|_{c\mapsto c_0^j, \beta\mapsto \beta_0^j}
\end{align}

We work out the updated expressions from \ref{subsec: Extended Problems} explicitly, as we did in \ref{subsec: The $2$-nd order problem}, by inserting the new zeroth order terms and extracting the contributions to the argument coefficients. All $c$'s are replaced with $c_0^j$'s, the $\beta$'s with $\beta_0^j$'s and all substitutes are replaced with their generalized versions, which we derived in the previous Sec.~\ref{subsec: generalized substitutes}. When we solve the $\delta$-equation using the generalized ansatz, introducing a generalized version of the recursion substitutes
\begin{equation}
    \hat{\alpha}_2^{i,j}(\alpha_2^{i,j},d_3^{i,j}) = \frac{(X_N^j)^2}{(Y_N^j)^2}\left( \frac{(c_0^j)^2e^{1,j}}{6(X_N^j)^3}C^{ij} + A^{ij}\alpha_2^{i,j} - \frac{D^{ij} d_3^{i,j}}{c\,k_N^{i+1,j}} \right)
\end{equation}
we find that only $C^{ij}$ gets a contribution to the argument coefficients. The other recursion substitutes are generalized straight forwardly.
\begin{equation}
    C^{1,j} = K_{1,j}\left[ (X_N^j)^2 + 2(Y_N^j)^2 \right] + 3\frac{(X_N^j)^2}{(c_0^j)^3}c_2^j, \qquad
    A^{1,j} = \frac{1}{2}, \qquad D^{1,j} = \frac{e^{1,j}}{2},
\end{equation}
\begin{align}
    A^{ij} =& \frac{i}{i+1}, \qquad
    D^{ij} = \frac{e^{ij}}{i+1} \notag \\
    C^{ij} =& \frac{1}{b^i}\left(\vphantom{\frac{(X_N^j)^2}{(k_N^{ij})^2}}\right. C^{i-1,j}
    + \underbrace{K_{ij}\left[ 2i(X_N^j)^2 + i\left( 3 + \frac{(X_N^j)^2}{(k_N^{ij})^2} \right)(Y_N^j)^2
    + \left( 1 + \frac{i(i^2-1)(X_N^j)^2}{(k_N^{ij})^2} \right)(e^{1,j})^2 \right]}_{\mathcal{K}^{ij}}
    + \underbrace{\frac{6 (X_N^j)^2}{i(c_0^j)^3}}_{\mathcal{E}^{ij}}c_2^j 
    \left.\vphantom{\frac{(X_N^j)^2}{(k_N^{ij})^2}}\right)
\end{align}

The recursion is of the same form as in the first slice and we track the additional term $\mathcal{E}^{ij}c_2^j$ through the solution of the recursion:
\begin{align}
    C^{ij} =& \frac{1}{b^i}\left( C^{i-1,j} + \mathcal{K}^{ij} + \mathcal{E}^{ij}c_2^j \right)
    = \frac{1}{b^i}\left( \frac{1}{b^{i-1}}\left( C^{i-2,j} + \mathcal{K}^{i-1,j} + \mathcal{E}^{i-1,j}c_2^j \right) + \mathcal{K}^{ij} + \mathcal{E}^{ij}c_2^j \right) \notag \\
    =& \frac{1}{b^ib^{i-1}\dots b^3b^2}C^{1,j}
    + \frac{1}{b^ib^{i-1}\dots b^3b^2}\left[\mathcal{K}^{2,j} + \mathcal{E}^{2,j}c_2^j \right] + \ldots
    + \frac{1}{b^ib^{i-1}}\left[ \mathcal{K}^{i-1,j} + \mathcal{E}^{i-1,j}c_2^j \right]
    + \frac{1}{b^i}\left[ \mathcal{K}^{ij} + \mathcal{E}^{ij}c_2^j \right] \notag \\
    =& C^{1,j}\prod_{l=2}^i\frac{1}{b^l}
    + \sum_{k=2}^i\left[ \mathcal{K}^{kj} + \mathcal{E}^{kj}c_2^j \right]\prod_{l=k}^i\frac{1}{b^l}
\end{align}
and pull out the second order corrections to $c_0^j$:
\begin{align}
    C^{ij} =& \frac{2}{i+1}C^{1,j} + \sum_{k=2}^i \frac{k}{i+1}\left[ \mathcal{K}^{kj}
    + \mathcal{E}^{kj}c_2^j \right]
    = \overset{\circ}{C}{}^{1,j} + \frac{6}{i+1}\frac{(X_N^j)^2}{(c_0^j)^3}c_2^j + \sum_{k=2}^i \frac{k}{i+1}\mathcal{K}^{kj} + \sum_{k=2}^i \frac{k}{i+1}\frac{6 (X_N^j)^2}{k(c_0^j)^3}c_2^j \notag \\
    =& \overset{\circ}{C}{}^{ij} + \frac{6}{i+1}\frac{(X_N^j)^2}{(c_0^j)^3}\left( 1 + \sum_{k=2}^i 1 \right)c_2^j
    = \overset{\circ}{C}{}^{ij} + \frac{6i(X_N^j)^2}{(i+1)(c_0^j)^3}c_2^j, \qquad
    \overset{\circ}{C}{}^{ij} = \left.C^i\right|_{c\mapsto c_0^j,\ \beta\mapsto\beta_0^j}, \quad C^i = C^{i,1}
\end{align}

Next, we insert the extended recursion coefficients, to update the recursion of the rib-lines $a_3^{ij}$
\begin{align}
    a_3^{ij}(a_3^{i+1,j}) =& \frac{1}{b^i}\left( \frac{(c_0^j)^3e^{1,j}}{6X_N^j}\mathcal{K}_a^{ij} \textcolor{green!40!blue}{-} F^{ij}c_2^j
    + a_3^{i+1,j} \right), \\
    \mathcal{K}_a^{ij} \coloneq& \frac{C^{i-1,j}}{(X_N^j)^2} + i\,K_{ij}\mathrm{X}^{ij}
    + (i+1)K_{i+1,j}\mathrm{Z}^{ij}, \qquad
    F^{ij} \coloneq \frac{e^{i-1 j}}{i X_N^j} \\
    \mathrm{X}^{ij} \coloneq& \frac{1}{(k_N^{ij})^2}\left[ \left(5-\frac{1}{i}\right)(X_N^j)^2
    + i^2(i+3)\frac{(e^{1,j})^4}{(X_N^j)^2} + (5i^2 + 3)(e^{1,j})^2 \right], \qquad
    \mathrm{Z}^{ij} \coloneq 1 + 2(i+1)\frac{(e^{1,j})^2}{(X_N^j)^2},
\end{align}
and find, that they as well remain in the same form. Thus we can solve it in the same way and we collect the $c_2$ contributions in an additional term:
\begin{align}
    a_3^{ij} =& \frac{1}{b^i}\left( \frac{(c_0^j)^3e^{1,j}}{6X_N^j}\mathcal{K}_a^{ij} - F^{ij}c_2^j
    + \frac{1}{b^{i+1}}\left( \frac{(c_0^j)^3e^{1,j}}{6X_N^j}\mathcal{K}_a^{i+1,j} - F^{i+1,j}c_2^j
    + a_3^{i+2,j} \right) \right) \notag \\
    =& \frac{1}{b^i}\left( \frac{(c_0^j)^3e^{1,j}}{6X_N^j}\mathcal{K}_a^{ij} - F^{ij}c_2^j
    + \frac{1}{b^{i+1}}\left( \frac{(c_0^j)^3e^{1,j}}{6X_N^j}\mathcal{K}_a^{i+1,j} - F^{i+1,j}c_2^j
    + \frac{1}{b^{i+2}}\left( \ldots \vphantom{\frac{(c_0^j)^3e^{1,j}}{6X_N^j}}\right.\right.\right.\notag\\
    &\left.\left.\left. \ldots \frac{1}{b^{N-2}}\left( \frac{(c_0^j)^3e^{1,j}}{6X_N^j}\mathcal{K}_a^{N-2,j}
    - F^{N-2,j}c_2^j + a_3^{N-1,j} \right) \ldots \right) \right) \right) \notag \\
    =& \frac{1}{b^i}\left[ \frac{(c_0^j)^3e^{1,j}}{6X_N^j}\mathcal{K}_a^{ij} - F^{ij}c_2^j \right]
    + \frac{1}{b^ib^{i+1}}\left[ \frac{(c_0^j)^3e^{1,j}}{6X_N^j}\mathcal{K}_a^{i+1,j} - F^{i+1,j}c_2^j \right]
    + \ldots \notag \\
    &\ldots + \frac{1}{b^ib^{i+1}\dots b^{N-2}}\left[ \frac{(c_0^j)^3e^{1,j}}{6X_N^j}\mathcal{K}_a^{N-2,j}
    - F^{N-2,j}c_2^j \right] + \frac{a_3^{N-1,j}}{b^ib^{i+1}\dots b^{N-2}} \notag \\
    =& \frac{(c_0^j)^3e^{1,j}}{6X_N^j}\sum_{n=i}^{N-2}\mathcal{K}_a^{nj} \prod_{j=i}^n\frac{1}{b^j}
    - c_2^j\sum_{n=i}^{N-2}F^{nj}\prod_{j=i}^n\frac{1}{b^j} + a_3^{N-1,j}\prod_{j=i}^{N-2}\frac{1}{b^j} \notag \\
    =& \frac{(c_0^j)^3e^{1,j}}{6X_N^j}\sum_{n=i}^{N-2}\frac{i}{n+1}\mathcal{K}_a^{nj}
    + \frac{i}{N-1}a_3^{N-1,j} - c_2^j\sum_{n=i}^{N-2}\frac{i}{n+1}F^{nj}
\end{align}

\subsection{$2$-nd order solution to $\Sigma_j$}\label{subsubsec: 2-nd order of Sigma_j}
After updating all quantities, separating the contributions of the corrections to the arguments from the interior values and collecting them in the argument correction coefficients, we list the results of the $j$-th slice in the form we need them for the limit calculation:

We start with the second order terms to the top line of the $j$-th slice.
\begin{align}
    \bar{b}_2^j =&\ \overset{\circ}{b}{}_2^j
    + \overset{\circ}{b}{}_2^{j,c}c_2^j + \overset{\circ}{b}{}_2^{j,\beta}\beta_2^j, \quad
    \overset{\circ}{b}{}_2^j = \overset{\circ}{b}{}_3^{1,j}\varepsilon + \sum_{i=2}^{N-2}\overset{\circ}{b}{}_3^{ij}\varepsilon
    + \overset{\circ}{b}{}_3^{N-1,j}\varepsilon, \qquad j\in\{1,\ldots,N\} \\
    \overset{\circ}{b}{}_2^{j,c} \coloneq& \frac{1}{Y_N^j}\left[ X_N^j + \frac{I_N(e^{1,j})^2}{NX_N^j} \right], \quad
    \overset{\circ}{b}{}_2^{j,\beta} \coloneq -\frac{c_0^je^{1,j}}{Y_N^j}, \quad
    I_N \coloneq \sum_{i=2}^{N-2}\frac{1}{i+1} = \psi(N) + \gamma - \frac{3}{2}, \quad
    \psi(x) = \frac{d}{dx}\ln\Gamma(x), \label{eq: bo}
\end{align}
where $\psi$ denotes the digamma function and $\gamma$ here is the Euler constant.

The interior values of the top lines of the segments $\Delta{ij}$ in the $j$-th slice $\Sigma_j$ have the same structure as the ones for $\Sigma_1$ with generalized substitutes:
\begin{align}
    \overset{\circ}{b}{}_3^{1,j} =& -\frac{e^{1,j}}{Y_N^j}\left( \frac{(c_0^j)^3e^{1,j}}{X_N^j}\left[ \frac{K_{1,j}}{3} + K_{2,j} \right] - \overset{\circ}{a}{}_3^{2,j} \right), \qquad
    \overset{\circ}{b}{}_3^{ij} = -\frac{e^{1,j}}{(i+1)Y_N^j}\left( \frac{(c_0^j)^3e^{1,j}}{6X_N^j}i\overset{\circ}{\mathcal{K}}{}_b^{ij}
    - \overset{\circ}{a}{}_3^{i+1,j} \right), \\
    \overset{\circ}{b}{}_3^{N-1,j} =& -\frac{(c_0^j)^2 \left(e^{1,j}\right)^2}{6 Y_N^j}\left( (N-1)^2K_{N-1,j}\frac{X_1^j}{X_N^j}
    + (2N-1)K_{N,j} + \frac{6}{(c_0^j)^2e^{1,j}}\overset{\circ}{a}{}_3^{N-1,j} \right), \notag \\
    \overset{\circ}{\mathcal{K}}{}_b^{ij} =&\ \frac{\overset{\circ}{C}{}^{i-1,j}}{(X_N^j)^2}
    + iK_{ij}\mathrm{X}_b^{ij} + (i+1)K_{i+1,j}\mathrm{Z}_b^{ij}, \qquad
    \mathrm{X}_b^{ij} \coloneq i + 5 + \frac{i-1}{i} + (i+3)\frac{(e^{1,j})^2}{(X_N^j)^2}, \notag \\
    \mathrm{Z}_b^{ij} \coloneq& \left( i + 2 + \frac{e^{1,j}e^{i+1,j}}{(X_N^j)^2} \right)
   \left( 1 + \frac{(e^{i+1,j})^2}{(k_N^{i+1,j})^2} \right) + \frac{(i+1)^2(X_N^j)^2 + (2i+1)e^{1,j}e^{i+1,j}}{i(k_N^{i+1,j})^2}, \qquad i\in\{2,\ldots,N-2\}. \notag
\end{align}

In the argument coefficients of the opening angle of $\Sigma_j$ we see the substitute $Z_N^j$ appearing, which we never used in $\Sigma_1$. We drop the $0$-index in the argument coefficients here, since we do not need them for other index values.
\begin{align}\label{eq: alphao}
    \alpha_2^{0,j} =& \overset{\circ}{\alpha}{}_2^{0,j}
    + \overset{\circ}{\alpha}{}_2^{j,c} c_2^j + \overset{\circ}{\alpha}{}_2^{j,\beta} \beta_2^j, \\
    \overset{\circ}{\alpha}{}_2^{0,j} =& \frac{(X_N^j)^2}{2(Y_N^j)^2}\left( \frac{(c_0^j)^2e^{1,j}}{3X_N^j}\left[ K_{1,j}\left( 3 + 4\frac{(e^{1,j})^2}{(X_N^j)^2} \right)
    + K_{2,j}\frac{(k_N^{2,j})^2}{(X_N^j)^2} \right] + \frac{\overset{\circ}{a}{}_3^{2,j}}{c_0^j} \right), \quad
    \overset{\circ}{\alpha}{}_2^{j,c} = \frac{a\cos\beta_0^j - X_N^j}{(Y_N^j)^2c_0^j}e^{1,j}, \quad
    \overset{\circ}{\alpha}{}_2^{j,\beta} = \frac{aZ_N^j}{(Y_N^j)^2} \notag
\end{align}

We note that none of the argument coefficients contain curvature values. The dependence them is given by the interior values alone and thus each slice depends on it's curvature values in the same way.
\begin{align}
  \bar{\gamma}_2^j =& \overset{\circ}{\gamma}{}_2^j + \overset{\circ}{\gamma}{}_2^{j,c}c_2^j
    + \overset{\circ}{\gamma}{}_2^{j,\beta}\beta_2^j, \qquad
    \overset{\circ}{\gamma}{}_2^j = \overset{\circ}{\hat{\alpha}}{}_2^{N-1,j} + \overset{\circ}{\gamma}{}_2^{N-1,j}, \quad
    \overset{\circ}{\gamma}{}_2^{j,c} = \frac{e^{N j}}{(Y_N^j)^2}, \quad
    \overset{\circ}{\gamma}{}_2^{j,\beta} = Nc_0^j\frac{X_N^j}{(Y_N^j)^2}, \label{eq: gammao} \\
    \overset{\circ}{\hat{\alpha}}{}_2^{N-1,j} =& \frac{1}{X_1^j(Y_N^j)^2}\left( \frac{c_0^j\,e^{N-1,j}}{6}\left[
    K_{N-1,j}\frac{X_1^j}{X_N^j}(X_1^j(Y_N^j)^2 - c_0^j\,e^{N-1,j}e^{1,j}) + K_{Nj}c_0^j(e^{1,j})^2\left( \frac{N(c_0^j)^2 - a^2}{(Y_1^j)^2} - 1 \right) \right] \right. \notag \\
    &\left.+ (X_N^j)^2\overset{\circ}{a}{}_3^{N-1,j} \vphantom{\frac{X_1^j}{X_N^j}}\right), \qquad
    \overset{\circ}{\gamma}{}_2^{N-1,j} = \frac{1}{(Y_1^j)^2}\left( \frac{K_{Nj}}{6}c_0^je^{1,j}\left[ (Y_1^j)^2
    + c_0^jX_1^j \right] \right).
\end{align} 

The contribution of the argument coefficients of the recursion parameters were already included into in the ones of the rib-lines and those into the argument coefficients presented above. Since they are not required on their own, we leave them away here. We list the interior values though, since we will calculate the limits in Sec.~\ref{sec: Limit calculations} starting from the inner most recursion level, working ourselves from the bottom of this list to the top.
\begin{align}
    \overset{\circ}{a}{}_3^{ij} =& \frac{(c_0^j)^3e^{1,j}}{6X_N^j}\sum_{n=i}^{N-2}\frac{i}{n+1}\overset{\circ}{\mathcal{K}}{}_a^{nj}
    + \frac{i}{N-1}\overset{\circ}{a}{}_3^{N-1,j}, \qquad
    \overset{\circ}{\mathcal{K}}{}_a^{nj} = \frac{\overset{\circ}{C}\,^{n-1,j}}{(X_N^j)^2}
    + n\,K_{nj}\mathrm{X}^{nj} + (n+1)K_{n+1,j}\mathrm{Z}^{nj}, \\
    \mathrm{X}^{ij} \coloneq& \frac{1}{(k_N^{ij})^2}\left[ \left(5-\frac{1}{i}\right)(X_N^j)^2
    + i^2(i+3)\frac{(e^{1,j})^4}{(X_N^j)^2} + (5i^2 + 3)(e^{1,j})^2 \right], \qquad
    \mathrm{Z}^{ij} \coloneq 1 + 2(i+1)\frac{(e^{1,j})^2}{(X_N^j)^2}, \\
    i\in&\{2,\ldots,N-2\}, \qquad j\in\{1,\ldots,N\} \notag
\end{align}
\begin{align}
    \overset{\circ}{a}{}_3^{N-1,j} =& \frac{(N-1)(c_0^j)^2X_1^j}{6X_N^j\left[ (N-1)c_0^j + X_1^j \right]}\left( \frac{c_0^j\,e^{1,j}}{(X_N^j)^2}\overset{\circ}{C}{}^{N-2,j}
    + K_{N-1,j}\chi^j + K_{N,j}\zeta^j \right), \\
    \chi^j =& (N-1)\frac{e^{1,j}}{X_N^j}\left( a^2 + N(c_0^j)^2 + (N+1)a\,c_0^j\cos\beta_0^j + \frac{X_1^j}{X_N^j}(Y_N^j)^2 \right)
    \notag \\
    &+ \left( (Y_N^j)^2\frac{(k_N^{N-1,j})^2 + (X_N^j)^2}{(k_N^{N-1,j})^2(X_N^j)^2} + \frac{N-2}{(k_N^{N-1,j})^2}\left[
    \frac{(X_N^j)^2}{N-1} + (2N-1)(e^{1,j})^2 \right] \right)c_0^j\,e^{N-1,j} \\
    \zeta^j =& \frac{e^{1,j}}{X_1^j}\left( 2a^2 + N(c_0^j)^2 + (2N+1)a\,c_0^j\cos\beta_0^j \right), \qquad j\in\{1,\ldots,N\}
\end{align}
\begin{align}
    \overset{\circ}{C}{}^{1,j} =&\, K_{1,j}\left[ (X_N^j)^2 + 2(Y_N^j)^2 \right], \qquad
    \overset{\circ}{C}{}^{ij} = \frac{2}{i+1}\overset{\circ}{C}{}^{1,j} + \sum_{k=2}^i \frac{k}{i+1}\mathcal{K}^{kj},
    &\qquad i\in&\{2,\ldots,N-2\}, \notag \\
    \mathcal{K}^{ij} =&\, K_{ij}\left[ 2i(X_N^j)^2 + i\left( 3 + \frac{(X_N^j)^2}{(k_N^{ij})^2} \right)(Y_N^j)^2
    + \left( 1 + i(i^2-1)\frac{(X_N^j)^2}{(k_N^{ij})^2} \right)(e^{1,j})^2 \right], &\qquad j\in&\{1,\ldots,N\}.
\end{align}
We repeat the same kind of test as we did for the first slice, to check whether we correctly accounted for the second order correction terms to the arguments $c_2^j$ and $\beta_2^j$. Fig.~\ref{fig: jth Slice Test-plot} shows one such plot.
\begin{figure}[h!]
    \centering
    \includegraphics[width=\linewidth]{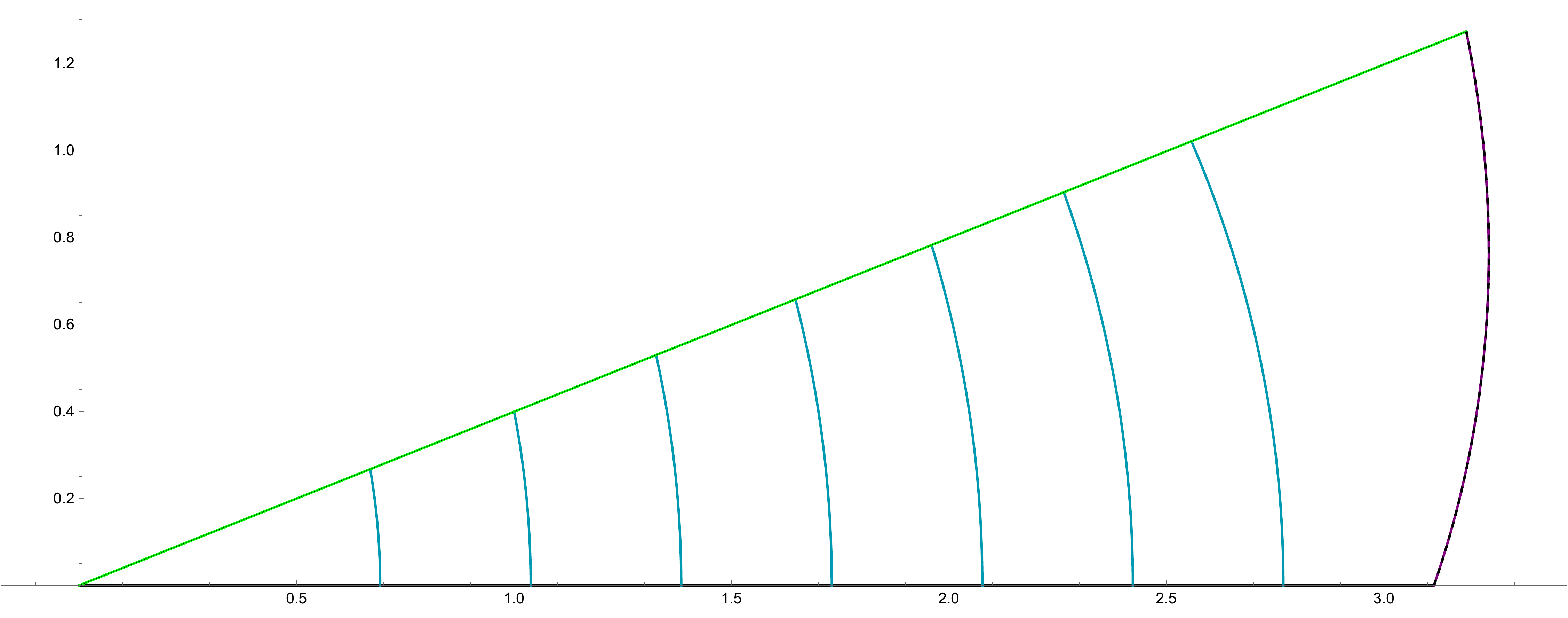}
    \caption{We check the consistency of our results by plotting the triangulation. The angle $\gamma^{N.1,j}$ is tested by plotting the continuing line backwards (black dashed line on top of the purple side-line). We used the parameters $N=9$, $a=5$, $c_0^j=3$, $c_2^j=9.3$, $\beta_0^j=\frac{\pi}{3}$ and $\beta_2^j=-20.8$.}
    \label{fig: jth Slice Test-plot}
\end{figure}

\section{Recursion across slices}\label{sec: Recursion across slices - full}
In the beginning of the previous Sec.~\ref{sec: The $j$-th slice} we calculated the zeroth order term of the base-line $c_0^j(a,c,\beta)$ and the angles $\beta_0^j(a,c,\beta)$ and $\bar{\alpha}_0^j$ for the $j$-th slice. Since $\bar{b}_0^j = c_0^{j+1}$, $b_0 = \bar{b}_0^N$ and $\alpha_0 = \bar{\alpha}_0^N$ we solved the zeroth order of the full triangle.\\
We also derived the function functions for $\bar{b}_2^j$, $\beta_2^{j+1}$ and $\alpha_2^{0,j}$ in terms of $a$, $c_2^j$ and $\beta_2^j$. So, we solved the second order problem of $j$-th slice in terms of the second slice and what remains to do, is to connect all slices together and establish a recursion across the slices.\\

Just as for the zeroth order the base-line of the $j$-th slice is the top-line of the previous slice. The expression for $\beta_2^{j+1}$ depends on $\beta_2^j$ and in this form it becomes obvious, that we have a coupled recursion:
\begin{align}
    c_2^j = \bar{b}_2^{j-1} = \overset{\circ}{b}{}_2^{j-1}
    + \overset{\circ}{b}{}_2^{j-1,c}c_2^{j-1} + \overset{\circ}{b}{}_2^{j-1,\beta}\beta_2^{j-1}, \quad
    \beta_2^{j} = \bar{\gamma}_2^{j-1} = \overset{\circ}{\gamma}{}_2^{j-1}
    + \overset{\circ}{\gamma}{}_2^{j-1,c}c_2^{j-1} + \overset{\circ}{\gamma}{}_2^{j-1,\beta}\beta_2^{j-1}
\end{align}
If we now replace the $c_2^{j-1}$ and $\beta_2^{j-1}$ with their respecting expressions according to the equations above, we get constant term for each of the two and the terms depending on the iterates. We collect the terms with $c_2^{j-2}$ and $\beta_2^{j-2}$ and pull the constant terms, which are not dependent on the recursion variables, out front.
\begin{align}\label{eq: b_2^j expansion}
    \bar{b}_2^j =& c_2^{j+1} = \overset{\circ}{b}\,_2^j
    + \overset{\circ}{b}\,_2^{j,c}\left( \overset{\circ}{b}{}_2^{j-1}
    + \overset{\circ}{b}{}_2^{j-1,c}c_2^{j-1} + \overset{\circ}{b}{}_2^{j-1,\beta}\beta_2^{j-1} \right)
    + \overset{\circ}{b}{}_2^{j,\beta}\left( \overset{\circ}{\gamma}{}_2^{j-1}
    + \overset{\circ}{\gamma}{}_2^{j-1,c}c_2^{j-1} + \overset{\circ}{\gamma}{}_2^{j-1,\beta}\beta_2^{j-1} \right) \notag \\
    =& \overset{\circ}{b}\,_2^j
    + \underbrace{\overset{\circ}{b}\,_2^{j,c}}_{O_1^j} \overset{\circ}{b}\,_2^{j-1}
    + \underbrace{\overset{\circ}{b}\,_2^{j,\beta}}_{Q_1^j} \overset{\circ}{\gamma}\,_2^{j-1}
    + \underbrace{\left( \overset{\circ}{b}\,_2^{j,c}\overset{\circ}{b}\,_2^{j-1,c}
    + \overset{\circ}{b}\,_2^{j,\beta}\overset{\circ}{\gamma}\,_2^{j-1,c} \right)}_{O_2^j} c_2^{j-1}
    + \underbrace{\left( \overset{\circ}{b}\,_2^{j,c}\overset{\circ}{b}\,_2^{j-1,\beta}
    + \overset{\circ}{b}\,_2^{j,\beta}\overset{\circ}{\gamma}\,_2^{j-1,\beta} \right)}_{Q_2^j} \beta_2^{j-1}
    \notag \\
    =& \overset{\circ}{b}\,_2^j
    + O_1^j\overset{\circ}{b}\,_2^{j-1} + Q_1^j\overset{\circ}{\gamma}\,_2^{j-1}
    +  O_2^j\overset{\circ}{b}\,_2^{j-2} + Q_2^j\overset{\circ}{\gamma}{}_2^{j-2}
    \ldots + O_{j-2}^j\overset{\circ}{b}\,_2^2 + Q_{j-2}^j\overset{\circ}{\gamma}\,_2^2
    + O_{j-1}^jc_2^2 + Q_{j-1}^j\beta_2^2
\end{align}
We observe, that the constant terms always end with a second order coefficient. More importantly, the prefactors of these second order coefficients are the prefactors of the recursion variables in the previous step. We substitute the prefactors of $c_2^{j-n}$ with $O_n^j$ and the prefactors of $\beta_2^{j-n}$ with $Q_n^{j-n}$.\\

We can easily verify, that the last two terms seamlessly fit into the sequence we sum over:
\begin{equation}
    c_2^2 = \bar{b}_2^1(K_{i1};a,c^1,\beta^1) = \bar{b}_2(K_{i1};a,c_0^1,\beta_0^1)
    = \overset{\circ}{b}\,_2^1, \qquad
    \beta_2^2 = \bar{\gamma}_2^1(K_{i1};a,c^1,\beta^1) = \bar{\gamma}_2(K_{i1};a,c_0^1,\beta_0^1)
    = \overset{\circ}{\gamma}\,_2^1,
\end{equation}
because the first $c$ and $\beta$ are precise and have no higher order terms:
\begin{equation}
    c^1 = c_0^1 = c, \qquad \beta^1 = \beta_0^1 = \beta
\end{equation}

\begin{proposition}
Thus, $c_2^j$ can be written as a sum over $O_n^j$ and $Q_n^j$:
\begin{equation}\label{eq: b_2^j}
    \bar{b}_2^j = c_2^{j+1} = \overset{\circ}{b}\,_2^{j}
    + \sum_{n=1}^{j-1}\left( O_n^j\overset{\circ}{b}\,_2^{j-n} + Q_n^j\overset{\circ}{\gamma}\,_2^{j-n} \right)
\end{equation}

From the calculation above Eq.~\eqref{eq: b_2^j expansion} we can see, that the $O_n^j$ and $Q_n^j$ follow the recursion rules:
\begin{equation}
    O_{n+1}^j = O_n^j\overset{\circ}{b}\,_2^{j-n,c} + Q_n^j\overset{\circ}{\gamma}\,_2^{j-n,c}, \quad
    Q_{n+1}^j = O_n^j\overset{\circ}{b}\,_2^{j-n,\beta} + Q_n^j\overset{\circ}{\gamma}\,_2^{j-n,\beta} \qquad
    O_1^j = \overset{\circ}{b}\,_2^{j,c}, \quad Q_1^j = \overset{\circ}{b}\,_2^{j,\beta}
\end{equation}
\end{proposition}
We can prove this together with~\eqref{eq: b_2^j} via induction.

\begin{proof}$ $\\
Induction start:
\begin{align}
    n &= 1: &\quad &c_2^{j+1} =\, \overset{\circ}{b}{}_2^j + \overset{\circ}{b}{}_2^{j,c}c_2^j
    - \overset{\circ}{b}{}_2^{j,\beta}\beta_2^j
        = \overset{\circ}{b}\,_2^j
        + O_1^j \overset{\circ}{b}\,_2^{j-1} + Q_1^j \overset{\circ}{\gamma}\,_2^{j-1} \notag \\
    n &= 2: &\quad &+ \underbrace{\left( \overset{\circ}{b}\,_2^{j,c}\overset{\circ}{b}\,_2^{j-1,c}
        + \overset{\circ}{b}\,_2^{j,\beta}\overset{\circ}{\gamma}\,_2^{j-1,c} \right)}_{O_2^j} c_2^{j-1}
        + \underbrace{\left( \overset{\circ}{b}\,_2^{j,c}\overset{\circ}{b}\,_2^{j-1,\beta}
        + \overset{\circ}{b}\,_2^{j,\beta}\overset{\circ}{\gamma}\,_2^{j-1,\beta} \right)}_{Q_2^j} \beta_2^{j-1}
\end{align}
Induction step: Under the assumption, that~\eqref{eq: b_2^j} holds up to $n-1$ we get:
\begin{align}
    c_2^{j+1} =& \overset{\circ}{b}\,_2^j
    + \sum_{m=1}^{n-1} \left( O_m^j\overset{\circ}{b}\,_2^{j-m} + Q_m^j\overset{\circ}{\gamma}\,_2^{j-m} \right) + O_n^jc_2^{j-n+1} + Q_n^j\beta_2^{j-n+1} \notag \\
    =& \overset{\circ}{b}\,_2^j
    + \sum_{m=1}^{n-1} \left( O_m^j\overset{\circ}{b}\,_2^{j-m} + Q_m^j\overset{\circ}{\gamma}\,_2^{j-m} \right)
    + O_n^j\overset{\circ}{b}\,_2^{j-n} + Q_n^j\overset{\circ}{\gamma}\,_2^{j-n} \notag \\
    &+ \underbrace{\left( O_n^j\overset{\circ}{b}\,_0^{j-n,c}
    + Q_n^j\overset{\circ}{\gamma}\,_0^{j-n,c} \right)}_{O_{n+1}^j}c_2^{j-n}
    + \underbrace{\left( O_n^j\overset{\circ}{b}\,_0^{j-n,\beta} + Q_n^j\overset{\circ}{\gamma}\,_0^{j-n,\beta} \right)}_{Q_{n+1}^j}\beta_2^{j-n}
\end{align}
\end{proof}

We can do an analogue calculation for the angles $\beta^j$,
\begin{equation}
    \bar{\gamma}_2^j = \beta_2^{j+1} = \overset{\circ}{\gamma}{}_2^j + \underbrace{\overset{\circ}{\gamma}{}_2^{j,c}}_{V_1^j}\overset{\circ}{b}{}_2^{j-1}
    + \underbrace{\overset{\circ}{\gamma}{}_2^{j,\beta}}_{W_1^j}\overset{\circ}{\gamma}{}_2^{j-1}
    + \underbrace{ \overset{\circ}{\gamma}{}_2^{j,c}\overset{\circ}{b}{}_2^{j-1,c}
    + \overset{\circ}{\gamma}{}_2^{j,\beta}\overset{\circ}{\gamma}{}_2^{j-1,c} }_{V_2^j} c_2^{j-1}
    + \underbrace{ \overset{\circ}{\gamma}{}_2^{j,c}\overset{\circ}{b}{}_2^{j-1,\beta}
    + \overset{\circ}{\gamma}{}_2^{j,\beta}\overset{\circ}{\gamma}{}_2^{j-1,\beta} }_{W_2^j} \beta_2^{j-1},
\end{equation}
which allows us to write $\bar{\gamma}_2^j$ explicitly given the recursion substitutes $V_n^j$ and $W_n^j$:
\begin{equation}\label{eq: gamma_2^j}
    \bar{\gamma}_2^j = \beta_2^{j+1} = \overset{\circ}{\gamma}{}_2^j
    + \sum_{n=1}^{j-1}\left( V_n^j\overset{\circ}{b}{}_2^{j-n} + W_n^j\overset{\circ}{\gamma}{}_2^{j-n} \right),
\end{equation}
and we find the same recursion rules as we did for $O_n^j$ and $Q_n^j$, with different starting points:
\begin{equation}
    V_{n+1}^j = V_n^j\overset{\circ}{b}{}_2^{j-n,c} + W_n^j\overset{\circ}{\gamma}{}_2^{j-n,c}, \quad
    W_{n+1}^j = V_n^j\overset{\circ}{b}{}_2^{j-n,\beta} + W_n^j\overset{\circ}{\gamma}{}_2^{j-n,\beta}, \qquad
    V_1^j = \overset{\circ}{\gamma}{}_2^{j,c}, \quad W_1^j = \overset{\circ}{\gamma}{}_2^{j,\beta}
\end{equation}

With these quantities given, we can write the explicit formula for the opening angle of the $j$-th slice:
\begin{align}\label{eq: alpha_2^j}
    \alpha_2^{0,j} =& \overset{\circ}{\alpha}{}_2^{0,j} + \overset{\circ}{\alpha}{}_2^{j,c}c_2^j
    + \overset{\circ}{\alpha}{}_2^{j,\beta}\beta_2^j \\
    =& \overset{\circ}{\alpha}{}_2^{0,j} + \overset{\circ}{\alpha}{}_2^{j,c}\left\{ \overset{\circ}{b}\,_2^{j-1}
    + \sum_{n=1}^{j-2}\left( O_n^{j-1}\overset{\circ}{b}\,_2^{j-1-n} + Q_n^{j-1}\overset{\circ}{\gamma}\,_2^{j-1-n} \right) \right\}
    + \overset{\circ}{\alpha}{}_2^{j,\beta}\left\{ \overset{\circ}{\gamma}{}_2^{j-1}
    + \sum_{n=1}^{j-2}\left( V_n^{j-1}\overset{\circ}{b}{}_2^{j-1-n} + W_n^{j-1}\overset{\circ}{\gamma}{}_2^{j-1-n} \right) \right\}  \notag
\end{align}

Now, what is missing are the explicit forms of the recursion substitutes $O_n^j$, $Q_n^j$, $W_n^j$ and $V_n^j$.

We insert the generalized substitutes from Sec.~\ref{sec: generalized substitutes} into the argument coefficients~\eqref{eq: bo},~\eqref{eq: gammao} and~\eqref{eq: alphao}:
\begin{align}\label{eq: CoeffFuncs}
    \overset{\circ}{b}{}_2^{j,c} =& \frac{Y_N^j}{Y_N^{j-1}}\left( 1 - \frac{aZ_N^j}{(Y_N^j)^2} + \frac{\Lambda_N}{(Y_N^j)^2L_N^j} \right),
    \quad \overset{\circ}{b}{}_2^{j,\beta} = -\frac{ce^1}{Y_N^j}, \qquad
    \overset{\circ}{\gamma}{}_2^{j,c} = \frac{N^2ce^1}{(Y_N^j)^2Y_N^{j-1}}, \quad \overset{\circ}{\gamma}{}_2^{j,\beta} = 1 - \frac{aZ_N^j}{(Y_N^j)^2} \\
    \overset{\circ}{\alpha}{}_2^{j,c} =& -\frac{N^2ce^1}{(Y_N^j)^2Y_N^{j-1}}, \quad
    \overset{\circ}{\alpha}{}_2^{j,\beta} = \frac{aZ_N^j}{(Y_N^j)^2}
\end{align}
where we introduced additional substitutes:
\begin{equation}
    \Lambda_N \coloneq NI_Nc^2(e^1)^2, \qquad L_N^j \coloneq (Y_N^j)^2 - aZ_N^j
\end{equation}

We also define the following quantities to simplify the following calculations and help reveal the structure of the recursions:
\begin{align}
    f_N \coloneq& Nce^1, \quad \tilde{Q}_n^j \coloneq -\frac{1}{ce^1}Q_n^j \quad
    \Rightarrow \quad Q_n^j = -ce^1\tilde{Q}_n^j \\
    P_m^n \coloneq& \prod_{l=m}^n\overset{\circ}{\gamma}{}_2^{j-l,\beta} = \prod_{l=m}^n\left( 1 - \frac{aZ_N^{j-l}}{(Y_N^{j-l})^2} \right),
    \qquad \hat{P}_m^n \coloneq \prod_{l=m}^n\left( 1 - \frac{aZ_N^{j-l}}{(Y_N^{j-l})^2} + \frac{\Lambda_N}{(Y_N^{j-l})^2L_N^{j-l}} \right)
\end{align}

We attempt to solve the two coupled recursions by writing down the first few terms of the sequence in terms of the coupled sequence.
\begin{proposition}
This way we get a semi-explicit form for the four sequences:\\
\begin{minipage}{.5\linewidth}
    \begin{align}
    O_n^j =& \frac{Y_N^j}{Y_N^{j-n}}\hat{P}_0^{n-1} - \frac{f_N^2}{Y_N^{j-n}}\sum_{m=1}^{n-1}\frac{\hat{P}_{m+1}^{n-1}}{(Y_N^{j-m})^2}\tilde{Q}_m^j, \label{eq: Onj(Qmj)} \\
    V_n^j =& \frac{Nf_N}{Y_N^{j-n}}\sum_{m=0}^{n-1}\frac{\hat{P}_{m+1}^{n-1}}{(Y_N^{j-m})^2}W_m^j,
    \quad W_0^j \coloneq 1 \label{eq: Vnj(Wmj)}
\end{align}
\end{minipage}
\begin{minipage}{.5\linewidth}
    \begin{align}
    Q_n^j =& -c\,e^1\sum_{m=0}^{n-1}\frac{P_{m+1}^{n-1}}{Y_N^{j-m}}O_m^j, \quad O_0^j \coloneq 1 \label{eq: Qnj(Omj)} \\
    W_n^j =& P_0^{n-1} - c\,e^1\sum_{m=1}^{n-1}\frac{P_{m+1}^{n-1}}{Y_N^{j-m}}V_m^j \label{eq: Wnj(Vmj)}
\end{align}
\end{minipage}
\end{proposition}
We prove the validity of these semi-explicit expressions via induction:\\

\begin{proof}
Induction start:
\begin{align}
    O_1^j =& \frac{Y_N^j}{Y_N^{j-1}}\hat{P}_0^0 - \frac{f_N^2}{Y_N^{j-1}}\sum_{m=1}^0\frac{\hat{P}_{m+1}^1}{(Y_N^{j-m})^2}\tilde{Q}_m^j \quad \checkmark &\qquad
    Q_1^j =& -c\,e^1\sum_{m=0}^0\frac{\cancel{P_1^0}}{Y_N^j}O_0^j = -\frac{c\,e^1}{Y_N^j} \quad \checkmark \notag \\
    V_1^j =& \frac{Nf_N}{Y_N^{j-1}}\sum_{m=0}^0\frac{\cancel{\hat{P}_1^0}}{(Y_N^j)^2}W_0^j = \frac{Nf_N}{Y_N^{j-1}(Y_N^j)^2} \quad \checkmark &\qquad
    W_1^j =& P_0^0 - c\,e^1\cancelto{0}{\sum_{m=1}^0\frac{P_{m+1}^0}{Y_N^{j-m}}V_m^j} = \overset{\circ}{\gamma}{}_2^{j,\beta} \quad \checkmark
\end{align}

Induction step:
\begin{align}
    O_{n+1}^j =& O_n^j\overset{\circ}{b}\,_2^{j-n,c} + Q_n^j\overset{\circ}{\gamma}\,_2^{j-n,c} \notag \\
    =& \left( \frac{Y_N^j}{\cancel{Y_N^{j-n}}}\hat{P}_0^{n-1} - \frac{f_N^2}{\cancel{Y_N^{j-n}}}\sum_{m=1}^{n-1}\frac{\hat{P}_{m+1}^{n-1}}{(Y_N^{j-m})^2}\tilde{Q}_m^j \right)
    \frac{\cancel{Y_N^{j-n}}}{Y_N^{j-n-1}}\hat{P}_n^n - ce^1\tilde{Q}_n^j \frac{N^2ce^1}{(Y_N^{j-n})^2Y_N^{j-n-1}} \notag \\
    =& \frac{Y_N^j}{Y_N^{j-(n+1)}}\hat{P}_0^n - \frac{f_N^2}{Y_N^{j-(n+1)}}\left( \sum_{m=1}^{n-1}\frac{\hat{P}_{m+1}^n}{(Y_N^{j-m})^2}\tilde{Q}_m^j + \frac{\tilde{Q}_n^j}{(Y_N^{j-n})^2} \right) \notag \\
    =& \frac{Y_N^j}{Y_N^{j-(n+1)}}\hat{P}_0^n
    - \frac{f_N^2}{Y_N^{j-(n+1)}} \sum_{m=1}^n\frac{\hat{P}_{m+1}^n}{(Y_N^{j-m})^2}\tilde{Q}_m^j \qquad \checkmark
\end{align}
\begin{align}
    Q_{n+1}^j =& O_n^j\overset{\circ}{b}\,_2^{j-n,\beta} + Q_n^j\overset{\circ}{\gamma}\,_2^{j-n,\beta}
    = -\frac{ce^1}{Y_N^{j-n}} O_n^j - c\,e^1\sum_{m=0}^{n-1}\frac{P_{m+1}^{n-1}}{Y_N^{j-m}}O_m^j P_n^n
    = - c\,e^1\sum_{m=0}^n\frac{P_{m+1}^n}{Y_N^{j-m}}O_m^j \quad \checkmark
\end{align}
\begin{align}
    V_{n+1}^j =& V_n^j\overset{\circ}{b}{}_2^{j-n,c} + W_n^j\overset{\circ}{\gamma}{}_2^{j-n,c}
    = \frac{Nf_N}{\cancel{Y_N^{j-n}}}\sum_{m=0}^{n-1}\frac{\hat{P}_{m+1}^{n-1}}{(Y_N^{j-m})^2}W_m^j
    \frac{\cancel{Y_N^{j-n}}}{Y_N^{j-n-1}}\hat{P}_n^n + W_n^j \frac{N^2ce^1}{(Y_N^{j-n})^2Y_N^{j-n-1}} \notag \\
    =& \frac{Nf_N}{Y_N^{j-(n+1)}}\sum_{m=0}^n\frac{\hat{P}_{m+1}^n}{(Y_N^{j-m})^2}W_m^j \qquad \checkmark
\end{align}
\begin{align}
    W_{n+1}^j =& V_n^j\overset{\circ}{b}{}_2^{j-n,\beta} + W_n^j\overset{\circ}{\gamma}{}_2^{j-n,\beta}
    = -\frac{ce^1}{Y_N^{j-n}} V_n^j + \left( P_0^{n-1} - c\,e^1\sum_{m=1}^{n-1}\frac{P_{m+1}^{n-1}}{Y_N^{j-m}}V_m^j \right)P_n^n \notag \\
    =& P_0^n - c\,e^1\sum_{m=1}^n\frac{P_{m+1}^n}{Y_N^{j-m}}V_m^j \quad \checkmark
\end{align}
\end{proof}

Now we can use \eqref{eq: Qnj(Omj)} to replace $\tilde{Q}_m^j$ in \eqref{eq: Onj(Qmj)}. Then $O_n^j$ is given in terms of $O_m^j$'s and we can replace them using \eqref{eq: Onj(Qmj)} again and so on:
\begin{align}
    O_n^j =& \frac{Y_N^j}{Y_N^{j-n}}\hat{P}_0^{n-1} - \frac{f_N^2}{Y_N^{j-n}}\sum_{m_1=1}^{n-1}\frac{\hat{P}_{m_1+1}^{n-1}}{(Y_N^{j-m_1})^2}  \sum_{m_2=0}^{m_1-1}\frac{P_{m_2+1}^{m_1-1}}{Y_N^{j-m_2}}O_{m_2}^j \notag \\
    =& \frac{Y_N^j}{Y_N^{j-n}}\hat{P}_0^{n-1} - \frac{f_N^2}{Y_N^{j-n}}\sum_{m_1=1}^{n-1}\frac{\hat{P}_{m_1+1}^{n-1}}{(Y_N^{j-m_1})^2}  \sum_{m_2=0}^{m_1-1}\frac{P_{m_2+1}^{m_1-1}}{Y_N^{j-m_2}} \left( \frac{Y_N^j}{Y_N^{j-m_2}}\hat{P}_0^{m_2-1} - \frac{f_N^2}{Y_N^{j-m_2}}\sum_{m_3=1}^{m_2-1}\frac{\hat{P}_{m_3+1}^{m_2-1}}{(Y_N^{j-m_3})^2}\tilde{Q}_{m_3}^j \right)
\end{align}
Every time we do this we get another layer of nested sums. We observe, that the upper bounds of the sums decrease by one in each step.
\begin{equation}
    m_1 \leq n-1, \quad m_2 \leq n-2, \quad \ldots, \quad m_{\imath} \leq n-\imath, \quad \ldots, \quad m_n \leq n-n = 0
\end{equation}
Thus the nesting ends at the $n$-th sum and we define
\begin{equation}\label{eq: S_l}
    S_l \coloneq -\frac{f_N^2}{Y_N^{j-m_{2(l-1)}}}
    \sum_{m_{2l-1}=1}^{m_{2(l-1)}-1} \frac{\hat{P}_{m_{2l-1}+1}^{m_{2(l-1)}-1}}{(Y_N^{j-m_{2l-1}})^2}
    \sum_{m_{2l}=0}^{m_{2l-1}-1} \frac{P_{m_{2l}+1}^{m_{2l-1}-1}}{Y_N^{j-m_{2l}}}
\end{equation}
to simplify the expression and write out the entire expression, until we arrive at $m_n$:
\begin{align}
    O_n^j =& \frac{Y_N^j}{Y_N^{j-n}}\hat{P}_0^{n-1} + S_1\left( \frac{Y_N^j}{Y_N^{j-m_2}}\hat{P}_0^{m_2-1} +S_2\left( \ldots 
    \left( \frac{Y_N^j}{Y_N^{j-m_{n-2}}}\hat{P}_0^{m_{n-2}-1} + S_\frac{n}{2}O_{m_n}^j \right) \ldots \right) \right)
    \label{eq: Onj_explicit_long}
\end{align}
This expression is only valid for even $n$. For odd $n$ we would have to stop at $S_\frac{n-1}{2}O_{n-1}^j$.\\

We can go through the same process, starting with $Q_n^j$ to calculate it explicitly as well and write the two in a more elegant and general form:
\begin{equation}\label{eq: Onj&Qnj_explicit}
    O_n^j = \sum_{\underset{\textit{\small even}}{\imath=0,}}^n \prod_{l=1}^\frac{\imath}{2} S_l
    \cdot \frac{Y_N^j}{Y_N^{j-m_\imath}} \hat{P}_0^{m_\imath-1}, \quad m_0 \coloneq n, \qquad
    Q_n^j = -c\,e^1 \sum_{m_1=0}^{n-1} \frac{P_{m_1+1}^{n-1}}{Y_N^{j-m_1}} \sum_{\underset{\textit{\small odd}}{\imath=0,}}^n
    \prod_{l=1}^\frac{\imath-1}{2} S_{l+\frac{1}{2}} \cdot \frac{Y_N^j}{Y_N^{j-m_\imath}} \hat{P}_0^{m_\imath-1}
\end{equation}
The dot $\cdot$ stands for the action of an operator group. So, we take the product over the operators and then act the resulting operator on the argument on the right of the dot. Thus the argument is not taken to the power of $\frac{\imath}{2}$, as it would be, if the product would multiply over it as well.\\
We check, in the example of $O_n^j$ that this expression \eqref{eq: Onj&Qnj_explicit} is equally valid for even and odd $n$, by writing out the last sum and using, that $m_\imath\leq n-\imath$, the first sum in $S_l$ starts from $1$ and the second one from $0$:\\
\textbf{$n$ even:}
\begin{align}
    S_\frac{n}{2}\frac{Y_N^j}{Y_N^{j-m_n}}\hat{P}_0^{m_n-1} \quad \overset{\text{last sum}}{\longrightarrow} \quad
    \sum_{m_n = 0}^{m_{n-1}}\frac{P_{m_n+1}^{m_{n-1}-1}}{Y_N^{j-m_n}}\frac{Y_N^j}{Y_N^{j-m_n}}\hat{P}_0^{m_n-1}
    \overset{m_n=0}{=} \frac{P_1^{m_{n-1}-1}}{Y_N^j}\cancel{\frac{Y_N^j}{Y_N^j}}\cancelto{1}{\hat{P}_0^{-1}}
    \overset{m_1=1}{=} \frac{\cancelto{1}{P_1^0}}{Y_N^j} = \frac{1}{Y_N^j}
\end{align}
\textbf{$n$ odd:}
\begin{align}
    &S_\frac{n-1}{2}\frac{Y_N^j}{Y_N^{j-m_{n-1}}}\hat{P}_0^{m_{n-1}-1} \quad \overset{\text{last sum}}{\longrightarrow} \quad
    \sum_{m_{n-1}=0}^{m_{n-2}-1}\frac{P_{m_{n-1}+1}^{m_{n-2}-1}}{Y_N^{j-m_{n-1}}}\frac{Y_N^j}{Y_N^{j-m_{n-1}}}\hat{P}_0^{m_{n-1}-1}
    = \dots \notag \\
    &\dots \overset{m_{n-2}\in\{1,2\}}{=} \sum_{m_{n-1}=0}^0\frac{\cancelto{1}{P_1^0}}{Y_N^j}\cancel{\frac{Y_N^j}{Y_N^j}}\cancelto{1}{\hat{P}_0^{-1}}
    + \sum_{m_{n-1}=0}^1\frac{P_{m_{n-1}+1}^1Y_N^j}{(Y_N^{j-m_{n-1}})^2}\hat{P}_0^{m_{n-1}-1}
    = \frac{1}{Y_N^j} + \frac{P_1^1Y_N^j}{(Y_N^j)^2}\cancelto{1}{\hat{P}_0^{-1}} + \frac{\cancelto{1}{P_2^1}Y_N^j}{(Y_N^{j-1})^2}\hat{P}_0^0
    \notag \\
    &= \frac{1}{Y_N^j} + \frac{P_1^1}{Y_N^j} + \frac{O_1^j}{Y_N^{j-1}} = \sum_{m_{n-1}=0}^{m_{n-2}-1}\frac{P_{m_{n-1}+1}^{m_{n-2}-1}}{Y_N^{j-m_{n-1}}} O_{m_{n-1}}^j, \quad O_0^j = 1 \quad
    \longrightarrow \quad S_\frac{n-1}{2}O_{m_{n-1}}
\end{align}
We find, that in both cases we get the correct last terms, comparing with .\\

The same procedure applies to the coupled recursion of $W_n^j$ and $V_n^j$, especially since it is the same recursion with different starting points. But because they have different starting expressions the repetitive term is a bit different and we therefore define $T_l$ in the place of $S_l$.
\begin{equation}\label{eq: T_l}
    T_l \coloneq -c\,e^1 \sum_{m_{2l-1}=1}^{m_{2(l-1)}-1} \frac{P_{m_{2l-1}+1}^{m_{2(l-1)}-1}}{Y_N^{j-m_{2l-1}}}
    \frac{Nf_N}{Y_N^{j-m_{2l-1}}} \sum_{m_{2l}=0}^{m_{2l-1}-1} \frac{\hat{P}_{m_{2l+1}}^{m_{2l-1}-1}}{(Y_N^{j-m_{2l}})^2}
\end{equation}
This definition allows us to write $V_n^j$ and $W_n^j$ explicitly in the same manor:
\begin{equation}\label{eq: Vnj&Wnj_explicit}
    V_n^j = \frac{Nf_N}{Y_N^{j-n}} \sum_{m_1=0}^{n-1} \frac{\hat{P}_{m_1+1}^{n-1}}{(Y_N^{j-m_1})^2}
    \sum_{\underset{\textit{\small odd}}{\imath=0}}^n \prod_{l=1}^\frac{\imath-1}{2} T_{l+\frac{1}{2}} \cdot P_0^{m_\imath-1}, \qquad
    W_n^j = \sum_{\underset{\textit{\small even}}{\imath=0,}}^n \prod_{l=1}^\frac{\imath}{2} T_l \cdot P_0^{m_\imath-1}, \quad m_0=n
\end{equation}

\section{Mathematical methods needed for the limit calculations}\label{sec: Mathematical methods}
From equations \eqref{eq: b_2^j},\eqref{eq: gamma_2^j},\eqref{eq: alpha_2^j},\eqref{eq: Onj&Qnj_explicit} and \eqref{eq: Vnj&Wnj_explicit} we see, that the limiting second order terms of the top line $b_2$ and the opening angle $\alpha_2$ will infinitely many nested series, over infinite products.\\
In this section we develop the methods to convert the series \ref{subsec: SumtoInt} and infinite products \ref{subsec: InfProd} into integrals and product integrals. This leads us to an infinitely nested integral. We find the expression for the $n$-layered nested integral in \ref{subsec: nNestInt}, which will then allow us to relate it to the power series of sine and cosine in the next section.

\subsection{Sum to integral conversion}\label{subsec: SumtoInt}
The upper bound of the sums appearing in $b_2$ and $\alpha_2$ can be up to $n$ which can become as large as $N$ which we want to take to infinity. We cannot treat them as series and take the limit of the sequence of partial sums, because the summand is dependent on $N$ instead of summing over a fixed infinite sequence. Thus we study here, how to write them into a form such that they converge to integrals in the case of $m_\imath\to n\to N$.

We can for example approximate the integral over $x^2$ via step functions and get an upper and lower bound for it. If we do that, our upper and lower bounds are the sums we are interested in~\eqref{eq: IntBound}. We can now turn these bounds inside out, to get upper and lower bounds for the sum~\eqref{eq: SumBound}.\\
\begin{figure}[h!]
\begin{minipage}{0.6\linewidth}
    \centering
    \includegraphics[width=\linewidth]{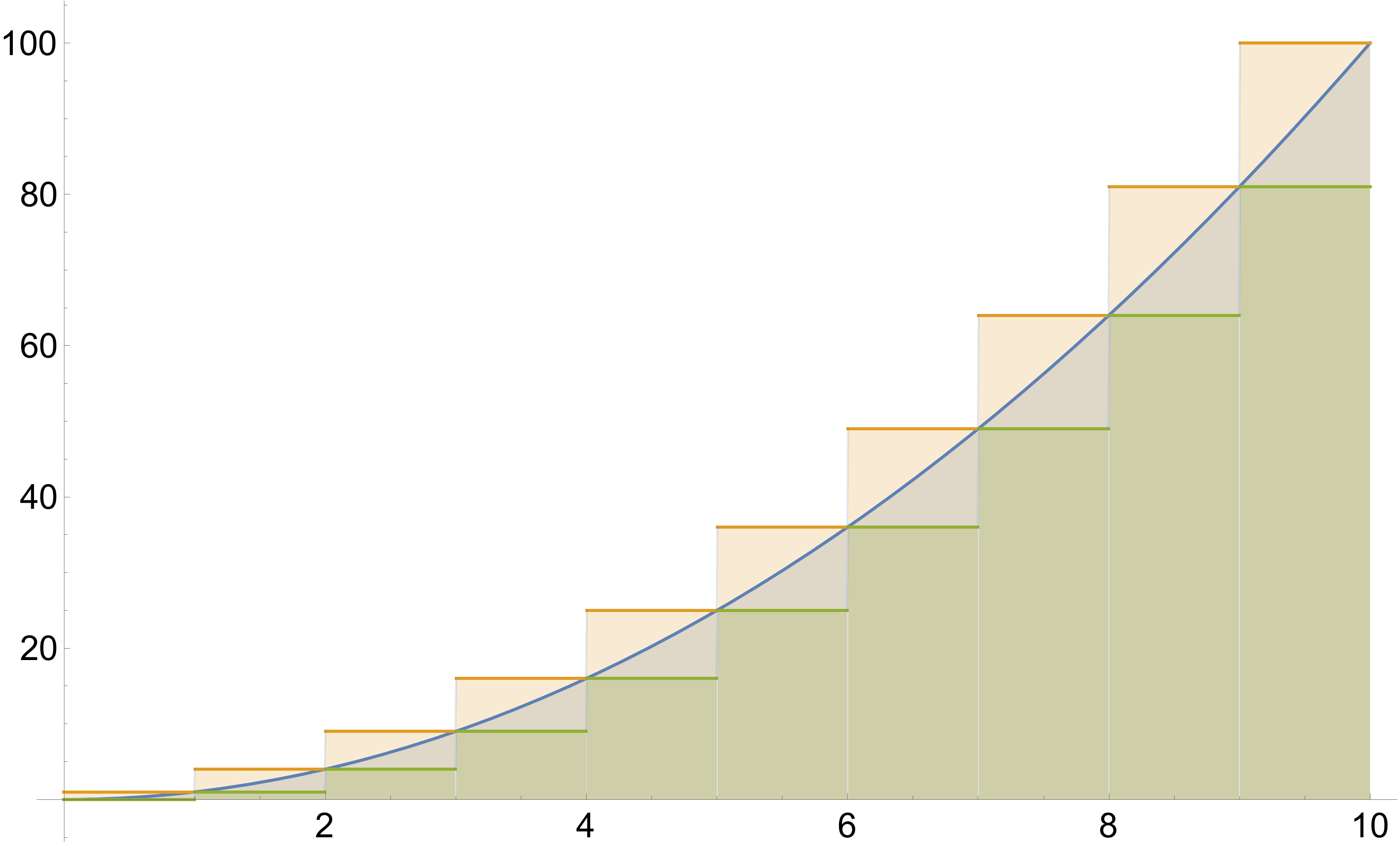}
\end{minipage}
\begin{minipage}{0.4\linewidth}
    \begin{align}
        &\sum_{i=n}^N i^2 < \int_n^{N+1} x^2 dx < \sum_{i=n+1}^{N+1} i^2 \label{eq: IntBound} \\
        \Rightarrow \quad &\int_{n-1}^N x^2 dx < \sum_{i=n}^N i^2 < \int_n^{N+1} x^2 dx \label{eq: SumBound}
    \end{align}
\end{minipage}
\caption{The step-functions provide upper (orange) and lower bounds (green) for the area beneath the function $x^2$ (blue).}
\label{fig:my_label}
\end{figure}

These bounds will however not converge to the sum since the missing peaces to the lower bound are growing and getting more, as the function grows. Both the upper and lower bounds can be anywhere between $0$ and $N$, which we let go to $\infty$. If we consider the summand as a function $f$ of $i$ and divide the argument by $N$. Then this becomes an increasingly fine sampling of the function $f$ between $0$ and $1$. Since we are now adding $N$ numbers of order $\mathcal{O}(1)$ we need to divide the sum by $N$. If we now turn the upper and lower bounds of step-functions inside out, we get upper and lower integral bounds for the sum which converge to the same value as $N$ tends to infinity. Thus the sum converges to the limiting integral.\\
In a more generic case the smallest and largest argument values we sum over determine the lower and upper limits of the integral.
\begin{figure}[h!]
\begin{minipage}{.3\linewidth}
    \centering
    \includegraphics[width=\linewidth]{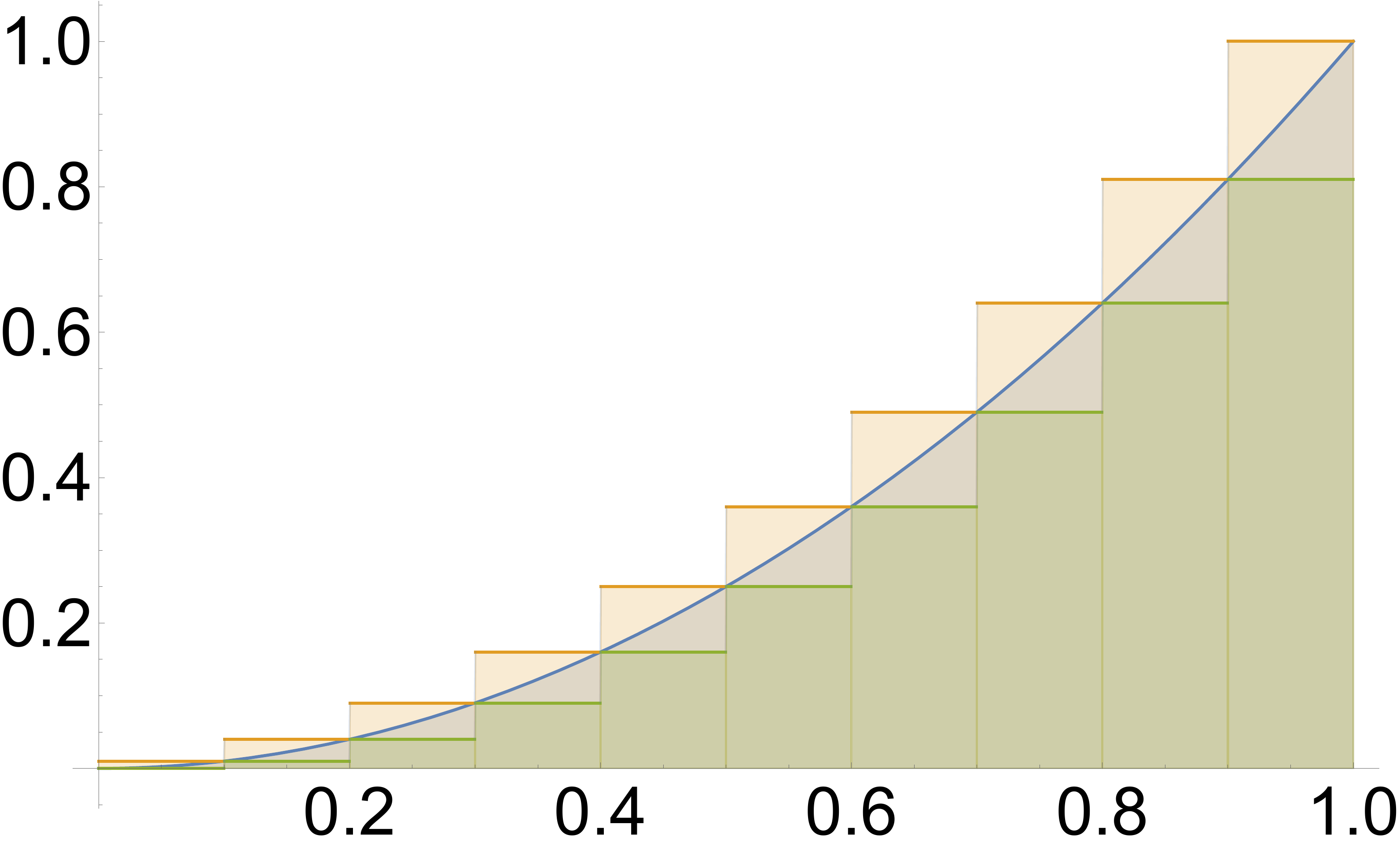}
\end{minipage}
\begin{minipage}{.08\linewidth}
    $\overset{N=10\mapsto20}{\longrightarrow}$
\end{minipage}
\begin{minipage}{.3\linewidth}
    \centering
    \includegraphics[width=\linewidth]{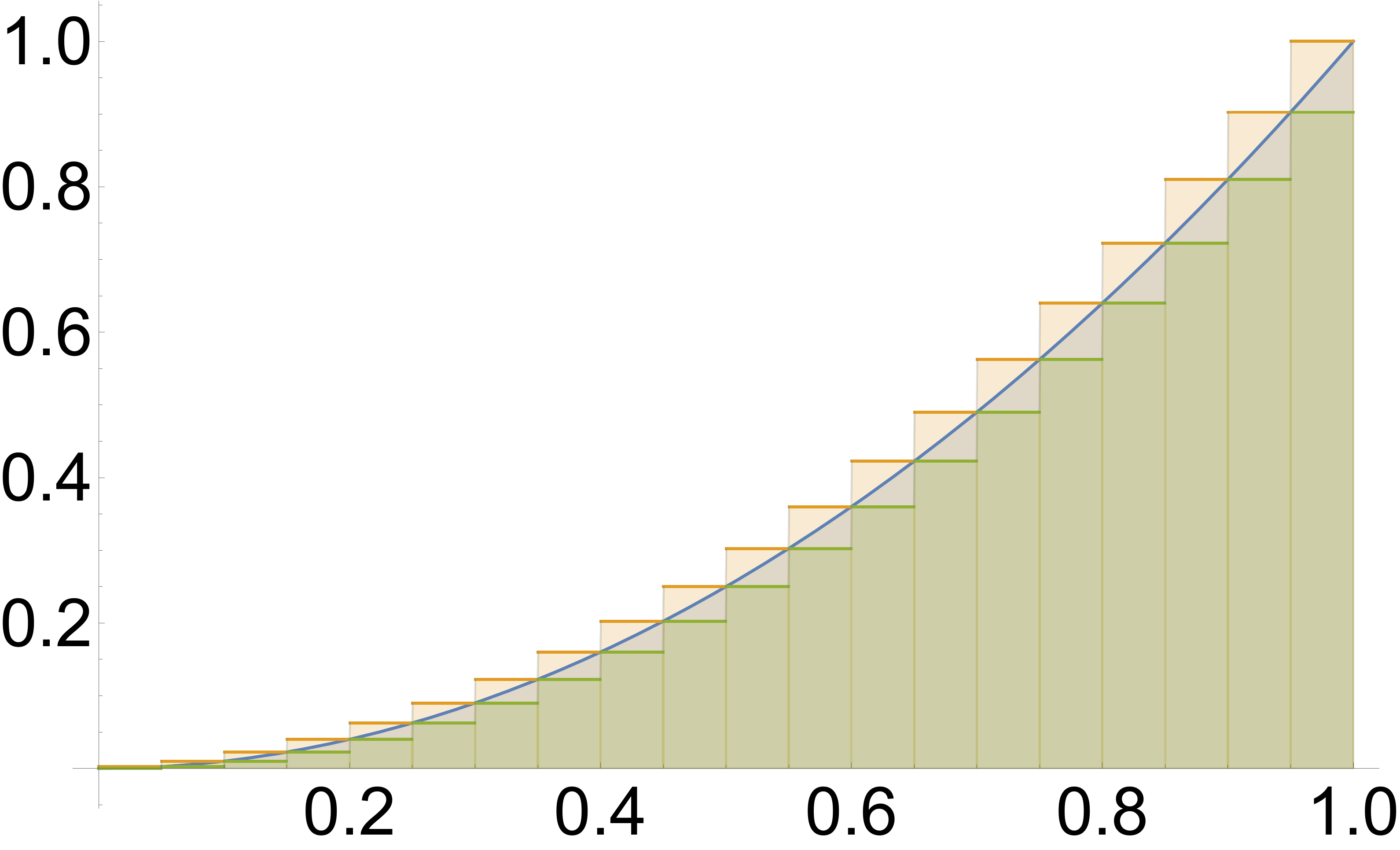}
\end{minipage}
\begin{minipage}{.32\linewidth}
    \begin{align}
        &\frac{1}{N}\sum_{j=n}^M\frac{j}{N} \sim \int_\frac{n}{N}^\frac{M}{N} j^2\, dj\,, \\
        &N \to \infty \notag
    \end{align}
\end{minipage}
    \caption{As $N$ increases, the area under the stepfunctions converge to the area under $x^2$ (blue) between $0$ and $1$.}
    \label{fig:SamplingSums}
\end{figure}\\
The $\frac{1}{N}=\Delta j$ turns our sum over sampling points into the area under a stepfunction which converges to an integral as $N\to\infty$ and $\frac{1}{N}\to dj$. We can always multiply with an apropriate one $\frac{N}{N}$ to get the $\Delta j$ factor we need.\\
For a general function we get the following asymptotic behaviour for the sum when $N$ tends to infinity.
\begin{equation}
    \frac{1}{N}\sum_{j=n(N)}^{M(N)} f\left( \frac{j}{N} \right) \sim \int_\frac{n}{N}^\frac{M}{N} f(j) dj, \quad N\to\infty
\end{equation}
Essentially, we rewrite our sums until we arrive at a special case of the definition of the Riemann integral, where $x_j=\frac{j}{N}$ and $\Delta x=\frac{1}{N}\to dx$.

It is now straight forward to extend this to nested sums, which lead to multidimensional integrals. We are here especially interested, in a case of the following type:
\begin{equation}
    \underbrace{\frac{1}{N}}_{\to dx}\sum_{i=1}^{N-1}\underbrace{\frac{1}{N}}_{\to dy}\sum_{j=1}^{i-1}
    f\left(\frac{i}{N},\frac{j}{N}\right)\sim \frac{1}{N}\sum_{i=1}^{N-1} \underbrace{\int_\frac{1}{N}^\frac{i}{N}f\left(\frac{i}{N},j\right)dj}_{F\left(\frac{i}{N}\right)}
    \sim \int_0^1\int_0^i f(i,j) dj\,di
\end{equation}
We see from \eqref{eq: SumBound} that get a $+1$ to the upper limit, when we take the limit to the integral. The integral limits are the smallest and largest values of the argument $j$ in the sum and thus get a $\frac{1}{N}$, which converts the upper limit $i-1\to\frac{i}{N}$ into a sampling point as required for the next sum to integral conversion. The $\frac{1}{N}$ vanishes for $N\to\infty$ which does not apply to $\frac{i}{N}$, since $i$ can be as big as $N$.\\
On this note it is useful to look at some examples involving finite shifts in the arguments of sums and products, to be aware, when we can drop them and when not. Let us first consider the following test-case:
\begin{align}\label{eq: AsympSumExample}
    \frac{1}{N}\sum_{j=2}^{N-2}\frac{j(j-1)}{N^2} &\overset{\frac{1}{N}\to 0}{\sim}
    \frac{1}{N}\sum_{j=2}^{N-2}\frac{j^2}{N^2} \sim \int_\frac{2}{N}^\frac{N-2}{N} j^2 dj
    \overset{N\to\infty}{\longrightarrow} \int_0^1 j^2 dj \\
    &\not\sim \int_\frac{2}{N}^\frac{N-2}{N} j(j-1) dj
    \sim \frac{1}{N}\sum_{j=2}^{N-2}\frac{j}{N}\left( \frac{j}{N} - 1 \right)
\end{align}
We can see, that we can neglect the shift by $-1$ in the argument by considering the converse. In the case of a nested sum, the constant shifts in the limits become shifts in the arguments, when we approximate the inner sum with an integral.\\
Such shifts are less harmless, when they appear in infinite products:
\begin{align}\label{eq: AsympProdExample}
    \frac{1}{N}\prod_{j=3}^{N-2}\frac{j}{j-1} =& \frac{1}{N}\frac{\cancel{3}}{2}\frac{\cancel{4}}{\cancel{3}}\cdot
    \ldots\cdot\frac{N-2}{\cancel{N-3}} = \frac{N-2}{2N} \to \frac{1}{2}, \quad N\to\infty \\
    \not\sim& \frac{1}{N}\prod_{j=3}^{N-2}\frac{j}{j} = \frac{N-2-3+1}{N} \to 1, \quad N\to\infty
\end{align}
We again compare with the converse assumption, that we can just neglect the shift by $-1$ and compare the results. The lower limit of the product leads to the factor of $\frac{1}{2}$, whereas the the $-2$ in the upper limit vanishes for $N\to\infty$. We present a more thorough approach to infinite products in the following section.\\

It may seem a bit unnatural choice, that we keep the name of the index and write the factor $\frac{1}{N}$, which represents $dx$ in front of the sum. We do this, because it is simpler this way to count the powers in $N$ later in the limit calculation and we keep the name of the index so we can track better, which sum transformed into what integral.

\subsection{A main theorem of calculus for products}\label{subsec: InfProd}
In the main text we motivated and introduced the notion of product integrals as limits of infinite products. In this section we discuss their convergence and proof an analogue to fundamental theorem of calculus for product integrals.\\

We create a lower bound by taking the smallest of all correction terms. When we then expand the lower bound of the product, using the binomial theorem, we see that the contribution of the correction terms in general do not vanish.
\begin{align}
    &a \coloneq \text{min}\{a_n(N)|n\in\{0,\ldots,N\},N\in\mathbb{N}\} = \underset{x\in[0,1)}{\text{min}}\{f(x)\} \\
    &\lim_{N\to\infty}\prod_{n=0}^{N-1}\left( 1 + \frac{a_n(N)}{N} \right)
    \geqslant \lim_{N\to\infty}\prod_{n=0}^{N-1}\left( 1 + \frac{a}{N} \right)
    = \lim_{N\to\infty}\left( 1 + \frac{a}{N} \right)^N
    = \lim_{N\to\infty}\sum_{k=0}^N\frac{N!}{(N-k)!k!}\frac{a^k}{N^k} \notag \\
    &= \lim_{N\to\infty} 1 + \cancel{N}\frac{a}{\cancel{N}} + \frac{\cancel{N(N-1)}}{2}\frac{a^2}{\cancel{N^2}}
    + \frac{\cancel{N(N-1)(N-2)}}{3!}\frac{a^3}{\cancel{N^3}} + \ldots = 1 + a + \frac{a^2}{2} + \frac{a^3}{3!} + \ldots < e^a
\end{align}
We observe that for the first couple of terms, powers of $N$ in the denominator cancel with the binomial coefficients in the limit $N\to\infty$. But the binomial coefficients cannot keep up with the growth of $N^k$ and terms with to large $k$ don't contribute significantly anymore. Thus, the sequence of terms we sum over decays faster then the ones in the exponential series, which guaranties convergence. It is not immediately obvious from this expression however, to what number this product converges to.\\

We conclude, that while a Riemann integral is the limit of adding an increasing number of values increasingly close to $0$, the neutral element to addition, we multiply an increasing number of values increasingly close to $1$, the neural element to multiplication, in the case of these products. It thus makes sense to consider them as product integrals.
\begin{equation}
    \int f dx = \lim_{N\to\infty}\sum_{n=0}^N\left( 0 + \frac{f(x_n)}{N} \right) \qquad \longrightarrow \qquad
    \mathcal{P} f dx \coloneq \lim_{N\to\infty}\prod_{n=0}^N\left( 1 + \frac{f(x_n)}{N} \right)
\end{equation}

To derive an analogue to the fundamental theorem of calculus for product integrals we first have a closer look at how the fundamental theorem of calculus is related to surfaces under step functions.\\

If we hold the starting point $a$ fixed and move the end point $b$ away from $a$, starting at $a$, then the surface under the graph of a smooth function $f(x)$ from $a$ to $b$ is given by some function of $b$. If we now move $a$ up towards $b$ we are decreasing the enclosed surface and this decrease clearly must be given by the same function since we just moved $b$ through this interval. Thus, the integral of $f$ from $a$ to $b$ must be given by the difference of some a priori unknown function $F$ evaluated at $a$ and $b$. We construct the Riemann integral backwards until we arrive at the integrand expressed in terms of the primitive function $F$. This is a simplistic derivation of the fundamental theorem of calculus, which will show us how to do derive an equivalent theorem for product integrals.
\begin{align}
    &x_n \coloneq a + n\Delta x, \quad \Delta x = \frac{b-a}{N} \quad \Rightarrow \quad x_0 = a, \quad x_N = b \\
    \vphantom{a} \notag \\
    &F(b) - F(a) = F(x_N) - F(x_{N-1}) + F(x_{N-1}) - \ldots - F(x_1) + F(x_1) - F(x_0) \notag \\
    &= F(x_{N-1}+\Delta x) - F(x_{N-1}) + \ldots + F(x_0+\Delta x) - F(x_0)
    = \sum_{n=0}^{N-1}F(x_n+\Delta x) - F(x_n) \quad \forall N\in\mathbb{N} \\
    &= \lim_{N\to\infty}\sum_{n=0}^{N-1} \underbrace{\frac{F(x_n+\Delta x) - F(x_n)}{\Delta x}\Delta x}_{a_n\to0}
    = \lim_{N\to\infty} \frac{b-a}{N} \sum_{n=0}^{N-1} \lim_{\Delta x\to0} \frac{F(x_n+\Delta) - F(x_n)}{\Delta x}
    = \lim_{N\to\infty} \frac{b-a}{N} \sum_{n=0}^{N-1} F'(x_n) \notag \\
    &= \int_a^b F'(x) dx
\end{align}

$ $\\
We are now prepared to proof an analogue version to this theorem for product integrals.

\begin{theorem}[Fundamental Theorem of Product Integration] The product integral of a smooth, real function $f$ is given by:
\begin{equation}
    \underset{a}{\overset{b}{\mathcal{P}}} f dx = \frac{F(b)}{F(a)},
\end{equation}
where $F$ is the product primitive function of $f$.
\end{theorem}

\begin{proof}
We assume the existence of a primitive function for product integrals and do the analogue calculation on multiplication. Subtracting is adding the additive inverse thus for products we multiply with the multiplicative inverse i.e., we divide. Then we expand by multiplying with ones instead of adding zeros to include the sample points. Finally, we need to bring it into a form, where we multiply over small deviations from $1$ instead of summing over small deviations from $0$:
\begin{align}\label{eq: MainThProd}
    \frac{F(b)}{F(a)} =& F(x_N)\frac{F(x_{N-1})}{F(x_{N-1})}\cdot\ldots\cdot\frac{F(x_1)}{F(x_1)}\frac{1}{F(x_0)}
    = \frac{F(x_N)}{F(x_{N-1})}\cdot\ldots\cdot\frac{F(x_1)}{F(x_0)} = \prod_{n=0}^{N-1}\frac{F(x_{n+1})}{F(x_n)}
    \quad \forall N\in \mathbb{N} \\
    =& \prod_{n=0}^{N-1}\frac{F(x_{n+1}) + F(x_n) - F(x_n)}{F(x_n)}
    = \lim_{N\to\infty}\prod_{n=0}^{N-1}\left( 1 + \frac{F(x_{n+1}) - F(x_n)}{F(x_n)} \right) \notag \\
    =& \lim_{N\to\infty}\prod_{n=0}^{N-1}
    \left( 1 + \lim_{\Delta x\to0}\frac{F(x_{n+1}) - F(x_n)}{F(x_n)\Delta x}\Delta x \right)
    = \lim_{N\to\infty} \prod_{n=0}^{N-1}\left( 1 + \frac{F'(x_n)}{F(x_n)}\Delta x \right)
    = \underset{a}{\overset{b}{\mathcal{P}}} \frac{F'(x)}{F(x)} dx
\end{align}
With this result we see, that $F$ has to be $C^1([a,b])$ and nonzero on the interval $[a,b]$, to be a product primitive function. Such functions exist, which justifies the assumption at the beginning of the proof.
\end{proof}

When we add a constant to a primitive function $F$ we get another primitive function $F+c$ to the same function $f$. The same is true for multiplying factors to product primitive functions.

\begin{proposition}
Multiplying a product primitive function with a constant factor leads to another product primitive function to the same product integrand $f$:
\begin{equation}
    F = \mathcal{P}f \quad \Rightarrow \quad c\,F = \mathcal{P}f, \quad \forall\,c\in\mathbb{R}
\end{equation}
\end{proposition}

\begin{proof}
\begin{equation}
    F = \mathcal{P}f \quad \Rightarrow \quad \frac{(c\,F(x))'}{cF(x)} = \frac{cF'(x)}{cF(x)} = \frac{F'(x)}{F(x)}
    = f(x)
\end{equation}
\end{proof}

\begin{corollary}
Thus, $\mathcal{P}f$ analogously to $\int f$, represents a set of functions rather than a single one:
\begin{equation}
    \mathcal{P}f = \left\{ F\in C^\infty(\mathbb{R}) \left.\vphantom{\frac{F'(x)}{F(x)}}\right| \frac{F'(x)}{F(x)} = f(x) \right\}, \qquad \forall f\in C^\infty(\mathbb{R})
\end{equation}
\end{corollary}

We saw, that the small terms of order $\frac{1}{N}$ in $P_m^n$ stack up and we know now that we need to find the product primitive function associated to that term to calculate the limit of this infinite product. But the product $\hat{P}_m^n$ has an additional term of order $\frac{\ln{N}}{N^2}$. We construct an upper bound and expand the product as above, to show that terms suppressed with higher powers then $\frac{1}{N}$ vanish in the limit $N\to\infty$:

\begin{proposition}
Terms of order $\mathcal{O}(\frac{1}{N^2})$ in a product integral vanish in the limit $N\to\infty$.
\end{proposition}

\begin{proof}
Let
\begin{equation}
    \alpha > 1, \quad A \coloneq \underset{n\in \mathbb{N}}{max}\{a_n\}, \quad B \coloneq \underset{n\in \mathbb{N}}{max}\{b_n\}
\end{equation}
then
\begin{align}\label{eq: 2ndOrdProd}
    &\lim_{N\to\infty}\prod_{n=0}^{N-1}\left( 1 + \frac{a_n}{N} + \frac{b_n}{N^\alpha} \right) \leqslant \lim_{N\to\infty}\prod_{n=0}^{N-1}\left( 1 + \frac{A}{N} + \frac{B}{N^\alpha} \right)
    = \lim_{N\to\infty}\sum_{k=0}^N\binom{N}{k}\left( 1 + \frac{A}{N} \right)^{N-k}\frac{B^k}{N^{k\alpha}} \notag \\
    &= \lim_{N\to\infty} \left( 1 + \frac{A}{N} \right)^N
    + \frac{N!}{(N-1)!}\left( 1 + \frac{A}{N} \right)^{N-1}\frac{B}{N^\alpha}
    + \frac{N!}{(N-2)!2}\left( 1 + \frac{A}{N} \right)^{N-2}\frac{B^2}{N^{2\alpha}} + \ldots \notag \\
    &= \lim_{N\to\infty} \left( 1 + \frac{A}{N} \right)^N
    + \underbrace{\left( 1 + \frac{A}{N} \right)^{N-1}\frac{B}{N^{\alpha-1}}}_{\to0}
    + \underbrace{\frac{1}{2}\left( 1 + \frac{A}{N} \right)^{N-2}\frac{B^2}{N^{2(\alpha-1)}}}_{\to0} + \ldots
    = \lim_{N\to\infty} \left( 1 + \frac{A}{N} \right)^N.
\end{align}
\end{proof}

\subsection{The $n$-layered nested integrals}\label{subsec: nNestInt}
The last mathematical problem we need to deal with, so that we can calculate limits, is the increasing number of nested sums. As we know now, this leads to $n$ nested integrals, where $n$ can be as large as $N$, which we take to infinity. In the next section it will turn out that we repeatedly integrate over the same function. More precisely the $n$-dimensional integral we are interested in is of the form:
\begin{equation}
    I_n \coloneq \int^{x_n} f(x_{n-1})\int^{x_{n-1}} f(x_{n-2}) \ldots \int^{x_1} f(x_0)\, dx_0 \ldots dx_{n-2}\,dx_{n-1}
\end{equation}
\begin{equation}
    I_n \coloneq \int^{x_n} f(x_{n-1})I_{n-1} dx_{n-1}, \qquad I_1 \coloneq \int^{x_1} f(x_0) dx_0
\end{equation}
We left away the lower integration limits to signify that we are neglecting its contribution for now i.e., we assume that $F(a) = 0$. This is not true for the case we want to calculate in the next section, but we can use the result of this simplified problem, to derive the solution of the sufficiently generic one.\\
When we write out the first few integrals we see, that we can use integration by parts (P.I.) and the structure becomes apparent very quickly. We again, as in the previous subsection use $F$ as a primitive function to $f$:
\begin{align}
    I_1 =& \int^{x_1} f(x_0) dx_0 = F(x_1) \label{eq: I1} \\
    I_2 =& \int^{x_2} f(x_1)\underbrace{\int^{x_1} f(x_0) dx_0}_{I_1} dx_1 = \int^{x_2} f(x_1)F(x_1) dx_1
    \overset{P.I.}{=} F^2(x_2) - \int^{x_2} F(x_1)f(x_1) dx_1 \notag \\
    \Rightarrow& \quad I_2 = \frac{F^2(x_2)}{2} \label{eq: I2} \\
    I_3 =& \int^{x_3} f(x_2)I_2 dx_2 = \frac{1}{2}\int^{x_3} f(x_2)F^2(x_2) dx_2
    = \frac{1}{2}\left( F^3(x_3) - 2\int^{x_3} F(x_2)F(x_2)f(x_2) dx_2 \right) \notag \\
    =& \frac{F^3(x_3)}{2} - \int^{x_3} F^2(x_2)f(x_2) dx_2 \quad \Rightarrow \quad I_3 = \frac{F^3(x_3)}{2\cdot3} \label{eq: I3}
\end{align}

\begin{proposition}
From here it is reasonable to assume, that the general formula is given by:
\begin{equation}
    I_n = \frac{F^n(x_n)}{n!},
\end{equation}
\end{proposition}
which we proof by induction.

\begin{proof}
Induction start: See \eqref{eq: I1}, \eqref{eq: I2}, \eqref{eq: I3}.\\
Induction step:
\begin{align}
    &I_{n+1} = \int^{x_{n+1}} f(x_n)I_n dx_n \overset{I.A.}{=} \frac{1}{n!}\int^{x_{n+1}} f(x_n)F^n(x_n) dx_n \notag \\
    &\quad\ \ \overset{P.I.}{=} \frac{1}{n!}\left( F^{n+1}(x_{n+1}) - \int^{x_{n+1}} F(x_n)nF^{n-1}(x_n)f(x_n) dx_n \right)
    = \frac{F^{n+1}(x_{n+1})}{n!} - \frac{1}{(n-1)!}\int^{x_{n+1}} F^n(x_n)f(x_n) dx_n \notag \\
    &\Rightarrow \quad \frac{F^{n+1}(x_{n+1})}{n!} = \frac{1}{(n-1)!}\left( \frac{1}{n} + 1 \right)\int^{x_{n+1}} f(x_n)F^n(x_n) dx_n
    = \frac{n+1}{n!}\int^{x_{n+1}} f(x_n)F^n(x_n) dx_n = (n+1)I_{n+1} \notag \\
    &\Rightarrow \quad I_{n+1} = \frac{F^{n+1}(x_{n+1})}{(n+1)!}
\end{align}
\end{proof}

Now we turn to the more complicated problem and define the type of integrals to which the nested sums will converge:
\begin{equation}
    J_n = \int_0^{x_n} f(x_{n-1})\int_0^{x_{n-1}} f(x_{n-2}) \ldots \int_0^{x_1} f(x_0)\,dx_0 \ldots dx_{n-2}\,dx_{n-1}
\end{equation}
It is the same type of integral as the $I_n$ including the lower limits and thus in every step one gets a $-F(0)$-term. In our case the contribution of the lower integral limit does not vanish $F(0)\neq0$.\\

\begin{proposition}
We find $J_n$ by applying the same strategy of writing down the first terms and using $I_n$ and arrive at:
\begin{equation}\label{eq: J_imath}
    J_\imath = \sum_{s=1}^\imath\frac{(-1)^{\imath-s}}{(\imath-s)!}F^{\imath-s}(0)[I_s]_0^{x_\imath}
\end{equation}
\end{proposition}

\begin{proof}
\begin{flalign}
\text{Induction start:} \qquad J_1 = \sum_{s=1}^1\frac{(-1)^0}{0!}F^0(0)[I_s]_0^{x_1} = F(x_1) - F(0) = \int_0^{x_1} f(x_0) dx_0 \qquad \checkmark&&
\end{flalign}
Induction step:
\begin{align}
    J_{\imath+1} =& \int_0^{x_{\imath+1}} f(x_\imath)J_\imath dx_\imath
    \overset{I.A.}{=} \int_0^{x_{\imath+1}} f(x_\imath)\sum_{s=1}^\imath\frac{(-1)^{\imath-s}}{(\imath-s)!}
    F^{\imath-s}(0)\frac{F^s(x_\imath)-F^s(0)}{s!} dx_\imath \notag \\
    =& \sum_{s=1}^\imath\frac{(-1)^{\imath-s}}{(\imath-s)!}F^{\imath-s}(0)\int_0^{x_{\imath+1}}
    f(x_\imath)\frac{F^s(x_\imath)}{s!} dx_\imath - \sum_{s=1}^\imath\frac{(-1)^{\imath-s}}{(\imath-s)!s!}
    F^\imath(0)\int_0^{x_{\imath+1}} f(x_\imath) dx_\imath \notag \\
    \overset{\eqref{eq: Identity}}{=}& \sum_{\imath=2}^{n+1}\frac{(-1)^{\imath-s+1}}{(\imath-s+1)!}F^{\imath-s+1}(0)\int_0^{x_{\imath+1}}
    \underbrace{f(x_\imath)\frac{F^{s-1}(x_\imath)}{(s-1)!} dx_\imath}_{I_s}
    - \frac{(-1)^{\imath+1}}{\imath!}F^\imath(0)\int_0^{x_{\imath+1}} \underbrace{f(x_\imath) dx_\imath}_{I_1} \notag \\
    =& \sum_{s=1}^{\imath+1}\frac{(-1)^{\imath+1-s}}{(\imath+1-s)!}F^{\imath+1-s}(0)[I_s]_0^{x_{\imath+1}},
\end{align}
where we used
\begin{equation}\label{eq: Identity}
    \sum_{s=1}^\imath\frac{(-1)^{\imath-s}}{(\imath-s)!s!} = \frac{(-1)^{\imath+1}}{\imath(\imath-1)!}.
\end{equation}
\end{proof}

\section{Limit calculations}\label{sec: Limit calculations}
Now we have all results we need in explicit form and we developed all the tools we need to calculate the limits. We first calculate the limits of the quantities in each slice, which we need. Then we calculate the limits of the recursion substitutes and in the last subsection we put everything together to derive our final result.

\subsection{Limits of the interior values}
We start by calculating the asymptotic limit of the quantities of the $j$-th slice presented in Sec.~\ref{subsubsec: 2-nd order of Sigma_j}, starting from the inner most level of substitution. We then insert the asymptotic expressions until we arrive at $\overset{\circ}{b}{}_2^j$, $\overset{\circ}{\gamma}_2^j$ and $\overset{\circ}{\alpha}_2^{0,j}$, which we need for \eqref{eq: b_2^j}, \eqref{eq: gamma_2^j} and \eqref{eq: alpha_2^j}. We note that these are being summed over and thus we have to calculate the asymptotic behaviour and drop terms according to the example \eqref{eq: AsympSumExample}.\\

We drop no terms in the following substitutes for now, but it will be very useful to track their algebraic order in $N$:
\begin{align}
    Y_N^j = \sqrt{j^2a^2 + N^2c^2 + 2jNac\cos\beta} \sim \mathcal{O}(N), \qquad Z_N^j \coloneq ja + Nc\cos\beta \sim \mathcal{O}(N) \\
    Y_1^j = \frac{Y_N^{N+j-1}}{N} \sim \mathcal{O}(1), \qquad X_1^j = \frac{(Y_N^{j-1})^2+NaZ_N^{j-1}}{NY_N^{j-1}} \sim \mathcal{O}(1) \\
    c_0^j = \frac{Y_N^{j-1}}{N} \sim \mathcal{O}(1), \qquad \beta_0^j = \sin^{-1}\frac{Nc\sin\beta}{Y_N^{j-1}} \sim \mathcal{O}(1),
    \qquad \cos\beta_0^j = \frac{Z_N^{j-1}}{Y_N^{j-1}} \sim \mathcal{O}(1)
\end{align}

The asymptotic limit of the remaining substitutes are:
\begin{align}
    X_N^j = \frac{(Y_N^j)^2 - aZ_N^j}{Y_N^{j-1}} \sim Y_N^j \sim \mathcal{O}(N), \qquad e^{ij} = \frac{Nce^i}{Y_N^{j-1}} \sim \frac{Nce^i}{Y_N^j} \sim \mathcal{O}(0) \\
    k_N^{ij} = \frac{1}{Y_N^{j-1}}\sqrt{[(Y_N^j)^2-aZ_N^j]^2 + N^2c^2(e^i)^2} \sim \frac{1}{Y_N^{j-1}}\sqrt{(Y_N^j)^4 + N^2c^2(e^i)^2}
    \sim \mathcal{O}(N)
\end{align}

We now use these asymptotic forms to calculate the limits of the recursion parameters:
\begin{equation}
    \mathcal{K}^{ij} \sim 6iK_{ij}(Y_N^j)^2 \sim \mathcal{O}(N^3), \qquad
    \overset{\circ}{C}{}^{1,j} \sim 3K_{1,j}(Y_N^j)^2 \sim \mathcal{O}(N^2) \qquad
    \overset{\circ}{C}^{ij} \sim 6(Y_N^j)^2\frac{1}{i}\sum_{k=1}^i k^2K_{kj}
\end{equation}

Next we calculate the limits of the rib-lines.
\begin{align}
    &\chi^j \sim \left( 5 + 2N\frac{aZ_N^{j-1}}{(Y_N^j)^2} \right)c\,e^N \sim \mathcal{O}(N), \qquad
    \zeta^j \sim Nc\,e^1\left( 1 + \frac{NaZ_N^j}{(Y_N^j)^2 + NaZ_N^j} \right) \sim \mathcal{O}(N), \\
    &\overset{\circ}{a}_3^{N-1,j} \sim \frac{ce^1}{Y_N^j}\left[ (Y_N^{j-1})^2 + NaZ_N^{j-1} \right]
    \frac{1}{N}\sum_{k=2}^{N-2} \frac{k^2}{N^2}K_{kj} \sim \mathcal{O}(N), \\
    &\mathrm{X}^{ij} \sim 5 \sim \mathcal{O}(1), \qquad \mathrm{Z}^{ij} \sim 1 \sim \mathcal{O}(1), \qquad
    \overset{\circ}{\mathcal{K}}{}_a^{nj} \sim \frac{6}{n}\sum_{k=1}^n k^2K_{kj} \sim \mathcal{O}(N^2) \\
    &\overset{\circ}{a}{}_3^{ij} \sim iY_N^jc\,e^1\frac{1}{N}\sum_{n=i}^N \frac{1}{n^2} \frac{1}{N}\sum_{k=1}^nk^2K_{kj}
    \sim \mathcal{O}(N^2).
\end{align}

Finally we can calculate the asymptotic interior values of the quantities we are interested in, starting with the top-line:
\begin{align}
    &\overset{\circ}{b}{}_3^{N-1,j} \sim -\frac{c^2(e^1)^2}{6Y_N^j}\left( N\frac{(Y_N^{j-1})^2+NaZ_N^{j-1}}{(Y_N^j)^2}K_{N-1,j}
    + 2NK_{Nj} + \frac{(Y_N^{j-1})^2+NaZ_N^{j-1}}{(Y_N^j)^2}\sum_{k=2}^{N-2}\frac{k^2}{N^2}K_{kj} \right) \sim \mathcal{O}(1) \\
    &\mathrm{X}_b^{ij} \sim i \sim \mathcal{O}(N), \qquad \mathrm{Z}_b^{ij} \sim 2i \sim \mathcal{O}(N), \qquad
    \overset{\circ}{\mathcal{K}}{}_b^{ij} \sim \frac{6}{i}\sum_{k=1}^i k^2K_{kj} + 3i^2K_{ij} \sim \mathcal{O}(N^2) \\
    &\overset{\circ}{b}{}_3^{ij} \sim \frac{(c\,e^1)^2}{Y_N^j}\left( \sum_{n=i}^N\frac{1}{n^2} \frac{1}{N}\sum_{k=1}^n k^2K_{kj}
    - \frac{1}{iN}\sum_{k=1}^i k^2K_{kj} - \frac{i^2}{2N}K_{ij} \right) \sim \mathcal{O}(1) \\
    &\overset{\circ}{b}{}_3^{1,j} \sim \frac{2c^2(e^1)^2}{Y_N^j}\sum_{n=2}^N\frac{1}{n^2} \frac{1}{N}\sum_{k=1}^nk^2K_{kj} \sim \mathcal{O}(1)
\end{align}

When we put everything together, to calculate $\overset{\circ}{b}{}_2^j$ we see, that the boundary terms $\overset{\circ}{b}{}_3^{1,j}$ and $\overset{\circ}{b}{}_3^{N-1,j}$ are suppressed by $\varepsilon$ which is balanced for the $\overset{\circ}{b}{}_3^{ij}$ by adding of order $N$ of them.
\begin{equation}
    \overset{\circ}{b}_2^j = \overset{\circ}{b}{}_3^{1,j}\varepsilon + \sum_{i=2}^{N-2}\overset{\circ}{b}{}_3^{ij}\varepsilon
    + \overset{\circ}{b}{}_3^{N-1,j}\varepsilon \sim \sum_{i=2}^{N-2}\overset{\circ}{b}{}_3^{ij}\varepsilon
\end{equation}
We insert the explicit expressions of the $K_{ij}$ from \eqref{eq: K_ij} as a function of position, to prepare for the conversion to integrals. To simplify the calculations we introduce the semi-discrete version of the Gaussian curvature field $K_j[i]$.
\begin{align}\label{eq: Kij_to_K(b,alpha)}
    &\bar{\alpha}_0^j = \sin^{-1}\frac{j\,a\sin\beta}{Y_N^j}, \qquad K_{ij} = K\left((i-1)c_0^j\varepsilon,\bar{\alpha}_0^{j-1}\right)
    \sim K\left(\frac{i\,Y_N^j}{N^2},\,\sin^{-1}\frac{j\,a\sin\beta}{Y_N^j}\right)
    \eqcolon K_j\left[ \frac{i}{N} \right], \\
    &K_j[i]\in C^\infty([0,1];\mathbb{R}), \quad \forall j\in\{1,\ldots,N\}
\end{align}
We multiply with appropriate ones to complete the terms with the appropriate powers of $N$ and introduce $y(x)$, so we can convert the sums into integrals:
\begin{align}
    \overset{\circ}{b}{}_2^j =& \sum_{i=2}^{N-2} \frac{(c\,e^1)^2}{Y_N^j}\left\{ \frac{1}{N}\sum_{n=i}^N\frac{N^2}{n^2} \frac{1}{N}\sum_{k=1}^n \frac{k^2}{N^2}\,K_j\left[\frac{k}{N}\right]
    - \frac{N}{i}\frac{1}{N}\sum_{k=1}^i\frac{k^2}{N^2}\,K_j\left[\frac{k}{N}\right] 
    - \frac{i^2}{2N^2}\,K_j\left[\frac{i}{N}\right] \right\} \\
    =& \sum_{i=2}^{N-2} \frac{(c\,e^1)^2}{Y_N^j}\left\{ \frac{1}{N}\sum_{n=i}^N\frac{N^2}{n^2} \int_0^\frac{n}{N} k^2\,K_j[k]\, dk
    - \frac{N}{i} \int_0^\frac{i}{N} k^2\,K_j[k]\, dk - \frac{i^2}{2N^2}\,K_j\left[\frac{i}{N}\right] \right\} \\
    =& \frac{N(c\,e^1)^2}{Y_N^j}\left\{ \frac{1}{N}\sum_{i=2}^{N-2}\int_\frac{i}{N}^1 \frac{1}{n^2} \int_0^n k^2\,K_j[k]\, dk\,dn
    - \frac{1}{N}\sum_{i=2}^{N-2}\frac{N}{i} \int_0^\frac{i}{N} k^2\,K_j[k] dk
    - \frac{1}{N}\sum_{i=2}^{N-2}\frac{i^2}{2N^2}\,K_j\left[\frac{i}{N}\right] \right\} \\
    =& \frac{N(c\,e^1)^2}{Y_N^j}\left\{ \int_0^1 \int_i^1 \frac{1}{n^2} \int_0^n k^2\,K_j[k]\, dk\,dn\,di
    - \int_0^1 \frac{1}{i} \int_0^i k^2\,K_j[k] dk\,di - \frac{1}{2}\int_0^1 i^2\,K_j[i]\,di \right\}
\end{align}

By expanding the curvature field $K_j[k]$ along the $j$-th slice in a Laurent series, we can prove, that the first two terms in the bracket always cancel. We need to impose a requirement on the curvature field though: a power of $\frac{1}{k^3}$ leads to $\ln0$ i.e. a non-existing integral.
\begin{equation}
    K_j[k] = \sum_{l=-\infty}^\infty\kappa_l^j k^l, \qquad \kappa_{-3}^j = 0
\end{equation}
\begin{align}
    &\int_0^1 \int_i^1 \frac{1}{n^2} \int_0^n k^2\,K_j[k]\, dk\,dn\,di - \int_0^1 \frac{1}{i} \int_0^i k^2\,K_j[k] dk\,di
    = \sum_{l=-\infty}^\infty\kappa_l^j\left( \int_0^1 \int_i^1 \frac{1}{n^2} \int_0^n k^2\,k^l\, dk\,dn\,di
    - \int_0^1 \frac{1}{i} \int_0^i k^2\,k^l dk\,di \right) \notag \\
    &= \sum_{\underset{\text{\scriptsize$l\neq-2$}}{l=-\infty}}^\infty\kappa_l^j
    \left( \int_0^1 \int_i^1 \frac{n^{l+1}}{l+3}\, dn\,di - \int_0^1 \frac{i^{l+2}}{l+3}\,di \right) + \kappa_{-2}^jI_{-2}
    = \sum_{\underset{\text{\scriptsize$l\neq-2$}}{l=-\infty}}^\infty\kappa_l^j
    \left( \int_0^1 \frac{1-i^{l+2}}{(l+2)(l+3)}\, di - \frac{1}{(l+3)^2} \right) + \kappa_{-2}^jI_{-2} \notag \\
    &= \sum_{\underset{\text{\scriptsize$l\neq-2$}}{l=-\infty}}^\infty\kappa_l^j
    \underbrace{\left( \frac{1}{(l+2)(l+3)} - \frac{1}{(l+2)(l+3)^2} - \frac{1}{(l+3)^2} \right)}_{=0} + \kappa_{-2}^jI_{-2}
    = \kappa_{-2}^jI_{-2}
\end{align}
The special case of the power $\frac{1}{k^2}$ needs to shown to vanish separately, because it's integrand cannot be written in this form:
\begin{align}
    I_{-2} =& \int_0^1 \int_i^1 \frac{1}{n^2} \int_0^n dk\,dn\,di - \int_0^1 \frac{1}{i} \int_0^i dk\,di
    = \int_0^1 \int_i^1 \frac{1}{n}\, dn\,di - \int_0^1 di  = \int_0^1 \ln1-\ln i \,di - 1 \notag \\
    =& -\underbrace{\int_0^1\ln i\,di}_{=-1} - 1 = 0
\end{align}
And thus we find that $\overset{\circ}{b}{}_2^j$ asymptotically approaches:
\begin{equation}
    \overset{\circ}{b}{}_2^j \sim -\frac{N(ce^1)}{2Y_N^j}\int_0^1 i^2K_j[i]\, di.
\end{equation}

When we calculate the interior values of the top-angles $\overset{\circ}{\gamma}{}_2^j$ we find that the contribution of the last segment $\overset{\circ}{\gamma}{}_2^{N-1,j}$ vanishes asymptotically.
\begin{align}
    &\overset{\circ}{\gamma}{}_2^{N-1,j} = \frac{K_{Nj}}{6}c\,e^1\left[ 1
    + \frac{(Y_N^{j-1})^2+NaZ_N^{j-1}}{(Y_N^{N+j-1})^2} \right] \sim \mathcal{O}(1), \qquad \overset{\circ}{\hat{\alpha}}{}_2^{N-1,j} \sim \frac{c\,e^1}{N^2}\sum_{k=2}^{N-2}k^2K_{kj} \sim \mathcal{O}(N), \\
    &\overset{\circ}{\gamma}{}_2^j = \overset{\circ}{\hat{\alpha}}{}_2^{N-1,j} + \overset{\circ}{\gamma}{}_2^{N-1,j}
    \sim f_N\int_0^1 k^2\,K_j[k]\,dk \sim \mathcal{O}(N)
\end{align}
There is however a contribution from that segment $\Delta_N$ in the recursion substitutes $V_n$ and $W_n$.\\

Finally we get to the asymptotic interior value of the opening angle:
\begin{equation}
    \overset{\circ}{\alpha}{}_2^{0,j} \sim f_N\int_0^1 \frac{1}{n^2} \int_0^n k^2K_j[k]\, dk\,dn \sim \mathcal{O}(N)
\end{equation}

\subsection{Limits of the recursion substitutes}
Now we get to the tricky part of the limit calculation, where we need all the mathematical methods we derived in the previous section. We want to calculate the asymptotic limit of the recursion substitutes $O_n$, $Q_n$, $V_n$ and $W_n$, which contain infinite products and infinitely nested sums.\\
We start with the products $\hat{P}_m^n$ and $P_m^n$. If we look at \eqref{eq: S_l}, \eqref{eq: Onj&Qnj_explicit},
\eqref{eq: T_l} and \eqref{eq: Vnj&Wnj_explicit} wee see, that the lower bound $m$ can range from $0$ to $n+1$ and the upper bound $n$ can range from $0$ to $N$, which we take to infinity. Thus both products multiply over at most order $N$ factors. We saw in \eqref{eq: 2ndOrdProd}, a suppression larger than $\frac{1}{N}$ cannot be balanced by multiplying over $N$ factors and vanishes. Thus $\hat{P}_n^m$ converges to $P_m^n$ and it converges to the product integral \eqref{eq: MainThProd} as we take $N$ to infinity:
\begin{equation}
    \lim_{N\to\infty}\hat{P}_m^n = \lim_{N\to\infty}\prod_{l=m}^n\left(\vphantom{\frac{Z_N^j}{(Y_N^j)^2}}\right.
    1 - \frac{aZ_N^{j-l}}{(Y_N^{j-l})^2} + \underbrace{\frac{\Lambda_N}{(Y_N^{j-l})^2L_N^{j-l}}}_{\to0} \left.\vphantom{\frac{Z_N^j}{(Y_N^j)^2}}\right)
    = \lim_{N\to\infty}\prod_{l=m}^n\left( 1 - \frac{aZ_N^{j-l}}{(Y_N^{j-l})^2} \right) = \lim_{N\to\infty}P_m^n
\end{equation}

We introduce the functions $y(x)$ and $z(x)$, to bring the product into the right form, so that we can use the main theorem for product integrals \eqref{eq: MainThProd}:
\begin{align}
    y(x) \coloneq&\ \sqrt{x^2a^2 + c^2 + 2xac\cos\beta}, \quad \text{with} \quad y(1) = y, \quad y(0) = c \quad \text{and} \quad
    Y_N^j = N\,y\left(\frac{j}{N}\right) \\
    z(x) \coloneq&\ x\,a + c\cos\beta, \quad \text{with} \quad Z_N^j = N\,z\left( \frac{j}{N} \right) \\
    \lim_{N\to\infty}P_m^n =& \lim_{N\to\infty}\prod_{l=m}^n\left( 1 - \frac{a\,z\left(\frac{j-l}{N}\right)}{N\,y^2\left(\frac{j-l}{N}\right)} \right)
    = \lim_{N\to\infty}\prod_{l=m}^n\left( 1 + \frac{f\left(\frac{l}{N}\right)}{N} \right)
    = \lim_{N\to\infty}\underset{\frac{m}{N}}{\overset{\frac{n}{N}}{\mathcal{P}}}f\,dx
\end{align}

It is straight forward to check, that $F(x) = y\left(\frac{j}{N}-x\right)$ is a product primitive function to $f(x)$.
\begin{align}
    \frac{F'(x)}{F(x)} = \frac{-2\left(\frac{j}{N}-x\right)a^2-2ac\cos\beta}{2y^2\left(\frac{j}{N}-x\right)}
    = -\frac{a\,z\left(\frac{j}{N}-x\right)}{y^2\left(\frac{j}{N}-x\right)} = f(x)
\end{align}

And thus the asymptotic limit of the infinite products is given by:
\begin{equation}
    P_m^n \sim \underset{\frac{m}{N}}{\overset{\frac{n}{N}}{\mathcal{P}}}f\,dx
    = \frac{F\left(\frac{n}{N}\right)}{F\left(\frac{m}{N}\right)}
    = \frac{y\left(\frac{j-n}{N}\right)}{y\left(\frac{j-m}{N}\right)} = \frac{Y_N^{j-n}}{Y_N^{j-m}}
\end{equation}

We can use this result to calculate the asymptotic limits of the two summation operators $S_l$ and $T_l$. As we have seen in \eqref{eq: AsympSumExample}, constant shifts in the sums go to zero.
\begin{align}
    S_l =& -\frac{f_N^2}{Y_N^{j-m_{2(l-1)}}}
    \sum_{m_{2l-1}=1}^{m_{2(l-1)}-1} \frac{\hat{P}_{m_{2l-1}+1}^{m_{2(l-1)}-1}}{(Y_N^{j-m_{2l-1}})^2}
    \sum_{m_{2l}=0}^{m_{2l-1}-1} \frac{P_{m_{2l}+1}^{m_{2l-1}-1}}{Y_N^{j-m_{2l}}} \notag \\
    \sim& -\frac{f_N^2}{ \cancel{Y_N^{j-m_{2(l-1)}}} }\sum_{m_{2l-1}=1}^{m_{2(l-1)}-1}
    \frac{1}{(Y_N^{j-m_{2l-1}})^2} \frac{ \cancel{Y_N^{j-m_{2(l-1)}-1}} }{ \cancel{Y_N^{j-m_{2l-1}+1}} }
    \sum_{m_{2l}=0}^{m_{2l-1}-1} \frac{1}{Y_N^{j-m_{2l}}} \frac{ \cancel{Y_N^{j-m_{2l-1}-1}} }{Y_N^{j-m_{2l}+1}} \notag \\
    \sim& -(c\,e^1)^2\frac{1}{N}\sum_{m_{2l-1}=1}^{m_{2(l-1)}-1} \frac{1}{ y^2\left(\frac{j-m_{2l-1}}{N}\right) }
    \frac{1}{N}\sum_{m_{2l}=0}^{m_{2l-1}-1} \frac{1}{ y^2\left(\frac{j-m_{2l}}{N}\right) } \\
    \vphantom{a} \notag \\
    T_l =& -c\,e^1 \sum_{m_{2l-1}=1}^{m_{2(l-1)}-1} \frac{P_{m_{2l-1}+1}^{m_{2(l-1)}-1}}{Y_N^{j-m_{2l-1}}}
    \frac{Nf_N}{Y_N^{j-m_{2l-1}}} \sum_{m_{2l}=0}^{m_{2l-1}-1} \frac{\hat{P}_{m_{2l}+1}^{m_{2l-1}-1}}{(Y_N^{j-m_{2l}})^2}
    \notag \\
    \sim& -c\,e^1 \sum_{m_{2l-1}=1}^{m_{2(l-1)}-1} \frac{1}{Y_N^{j-m_{2l-1}}}
    \frac{ Y_N^{j-m_{2(l-1)}-1} }{ Y_N^{j-m_{2l-1}+1} } \frac{Nf_N}{ \cancel{Y_N^{j-m_{2l-1}}} }
    \sum_{m_{2l}=0}^{m_{2l-1}-1} \frac{1}{(Y_N^{j-m_{2l}})^2} \frac{ \cancel{Y_N^{j-m_{2l-1}-1}} }{ Y_N^{j-m_{2l}+1} } \notag \\
    \sim& -(c\,e^1)^2 \frac{1}{N}\sum_{m_{2l-1}=1}^{m_{2(l-1)}-1}
    \frac{ y\left(\frac{j-m_{2(l-1)}}{N}\right) }{ y^2\left(\frac{j-m_{2l-1}}{N}\right) }
    \frac{1}{N}\sum_{m_{2l}=0}^{m_{2l-1}-1} \frac{1}{ y^3\left(\frac{j-m_{2l}}{N}\right) } \label{eq: T_l}
\end{align}
The $y$ in the first sum is going to cancel with a $\frac{1}{y}$ from the second sum of the next $T_{l+1}$ we multiply to $T_l$:
\begin{equation}
    2l = 2(l+1-1)
\end{equation}

We convert nested sums created by the product over the summation operators $S_l$ into integrals according to \ref{subsec: SumtoInt} which leads to arbitrarily many nested integrals of the type we calculated in \ref{subsec: nNestInt}. With these results we can calculate the asymptotic limit of $O_n^j$:
\begin{align}
    O_n^j =& \sum_{\underset{\textit{\small even}}{\imath=0,}}^n \prod_{l=1}^\frac{\imath}{2} S_l
    \cdot \frac{Y_N^j}{Y_N^{j-m_\imath}} \hat{P}_0^{m_\imath-1}
    \sim \sum_{\underset{\textit{\small even}}{\imath=0,}}^n \prod_{l=1}^\frac{\imath}{2} S_l
    \cdot \frac{ \cancel{Y_N^j} }{ \cancel{Y_N^{j-m_\imath}} } \frac{ \cancel{Y_N^{j-m_\imath-1}} }{ \cancel{Y_N^j} } \notag \\
    \sim& \sum_{\underset{\textit{\small even}}{\imath=2,}}^n \left( -(c\,e^1)^2 \right)^\frac{\imath}{2}
    \int_0^\frac{n}{N} \frac{1}{y^2\left(\frac{j}{N}-m_1\right)} \int_0^{m_1} \frac{1}{y^2\left(\frac{j}{N}-m_2\right)} \ldots
    \int_0^{m_{\imath-1}} \frac{1}{y^2\left(\frac{j}{N}-m_\imath\right)}  dm_\imath \ldots  dm_2\,dm_1
    + \cancelto{\mathbbb{1}}{\prod_{l=1}^0S_l}\cdot1 \notag \\
    \overset{\eqref{eq: J_imath}}{=}& \sum_{\underset{\textit{\small even}}{\imath=0,}}^n (-1)^\frac{\imath}{2}(c\,e^1)^\imath
    J_\imath\left(\frac{n}{N}\right) + 1 \sim \mathcal{O}(1), \quad
    J_\imath(x) = \sum_{s=1}^\imath\frac{(-1)^{\imath-s}}{(\imath-s)!}F^{\imath-s}(0)[I_s]_0^x,
\end{align}
where now $F$ is the primitive function of:
\begin{align}
    f(x) = \frac{1}{y^2\left(\frac{j}{N}-x\right)}; \qquad F' = f \quad \Rightarrow \quad
    F(x) = -\frac{1}{c\,e^1}\tan^{-1}\frac{a\,z\left(\frac{j}{N}-x\right)}{c\,e^1}
\end{align}
Note, that in the special case $\imath = 0$, where the product over the summation is empty and thus becomes the unit operator $\mathbbb{1}$ we need to start the sum from $\imath = 2$ and add the special term separately. We can start the sum from $\imath=0$ again in the next step, which we do for convenience, since $J_0 = 0$.\\

We follow the same method to calculate the limit of $Q_n^j$.
\begin{align}
    Q_n^j =& -c\,e^1 \sum_{m_1=0}^{n-1} \frac{P_{m_1+1}^{n-1}}{Y_N^{j-m_1}} \sum_{\underset{\textit{\small odd}}{\imath=0,}}^n
    \prod_{l=1}^\frac{\imath-1}{2} S_{l+\frac{1}{2}} \cdot \frac{Y_N^j}{Y_N^{j-m_\imath}} \hat{P}_0^{m_\imath-1}
    \sim -c\,e^1 \sum_{m_1=0}^{n-1} \frac{1}{Y_N^{j-m_1}} \frac{Y_N^{j-n+1}}{Y_N^{j-m_1-1}}
    \sum_{\underset{\textit{\small odd}}{\imath=0,}}^n \prod_{l=1}^\frac{\imath-1}{2} S_{l+\frac{1}{2}} \cdot 1 \notag \\
    \sim& -c\,e^1\,y\left(\frac{j-n}{N}\right) \int_0^\frac{n}{N} \frac{1}{y\left(\frac{j}{N}-m_1\right)}
    \sum_{\underset{\textit{\small odd}}{\imath=3,}}^n (-1)^\frac{\imath-1}{2}(c\,e^1)^{\imath-1}
    \int_0^{m_1} \ldots \int_0^{m_{\imath-1}} \frac{1}{y\left(\frac{j}{N}-m_\imath\right)} dm_\imath \ldots dm_2
    + \cancelto{\mathbbb{1}}{\prod_{l=1}^0S_l}\cdot1 \,dm_1 \notag \\
    \sim&\ y\left(\frac{j-n}{N}\right) \sum_{\underset{\textit{\small odd}}{\imath=0,}}^n (-1)^\frac{\imath+1}{2}(c\,e^1)^\imath
    \int_0^\frac{n}{N} \frac{1}{y\left(\frac{j}{N}-m_1\right)} \ldots \int_0^{m_{\imath-1}} \frac{1}{y\left(\frac{j}{N}-m_\imath\right)}
    dm_\imath \ldots dm_1 \notag \\
    \sim&\ y\left(\frac{j-n}{N}\right) \sum_{\underset{\textit{\small odd}}{\imath=0,}}^n (-1)^\frac{\imath+1}{2}(c\,e^1)^\imath
    J_\imath\left(\frac{n}{N}\right) \sim \mathcal{O}(1)
\end{align}
For $O_n^j$ we could stop at this stage where we proved that the recursion substitutes are of order $\mathcal{O}(1)$, since it will turn out, that it's contribution will vanish. The only contribution to the top-line comes from $Q_n^j$ and so we insert the expressions for $J_\imath$ and $F$ to finish it's limit calculation.
\begin{align}
    Q_n^j \sim& y\left(\frac{j-n}{N}\right) \sum_{\underset{\textit{\small odd}}{\imath=0,}}^n (-1)^\frac{\imath+1}{2}(c\,e^1)^\imath
    \sum_{s=1}^\imath\frac{(-1)^{\imath-s}}{(\imath-s)!}F^{\imath-s}(0) \frac{F^s\left(\frac{n}{N}\right)-F^s(0)}{s!} \notag \\
    \sim&\ y\left(\frac{j-n}{N}\right) \sum_{\underset{\textit{\small odd}}{\imath=0,}}^n (-1)^\frac{\imath+1}{2}\cancel{(c\,e^1)^\imath}
    \sum_{s=1}^\imath\frac{(-1)^{\imath-s}}{(\imath-s)!s!} \frac{(-1)^\imath}{\cancel{(c\,e^1)^\imath}}
    \eta^{\imath-s} \left[ \xi^s - \eta^s \right] \notag \\
    \sim&\ y\left(\frac{j-n}{N}\right) \sum_{\underset{\textit{\small odd}}{\imath=0,}}^n
    \sum_{s=1}^\imath \frac{(-1)^{\frac{\imath+1}{2}+s}}{(\imath-s)!s!} \eta^{\imath-s}\left[ \xi^s - \eta^s \right],
\end{align}
where we substitute
\begin{equation}
    \eta = \tan^{-1}\frac{a\,z\left(\frac{j}{N}\right)}{c\,e^1}, \qquad
    \xi = \tan^{-1}\frac{a\,z\left(\frac{j-n}{N}\right)}{c\,e^1}
\end{equation}
to make the terms more handy, when we deal with the double sum.\\

As compared to the previous case we do not have to worry about an extra term for $\imath = 1$, since we integrate over it. Then it is consistent with the $J_1$ term and we can include it into the sum again in the third step.\\

We observe, that $Q_n^j$ now looks similar to a Cauchy product of two exponential series. If we want to calculate the limit we need to disentangle these sums into a product of two sums which separately converge to a power series. To do this, we have to deal with two problems: the first sum only sums over odd values of $\imath$ and the $\imath$ in the power of $-1$ does not come with a $-s$.\\
Because the first sum is over odd indices the total power $t$ in $x$ and $y$ is always odd $t = \imath - s + s = \imath$. We can create a power series with odd total power, by taking the Cauchy product of an even and an odd series. To get all the terms appearing in $O_n^j$ we need to add both possible combinations and we can solve the issue with the $(-1)^\imath$ by adjusting the alternation of the two series appropriately.
\begin{align}
    &\sum_{\underset{\textit{\small odd}}{\imath=0}}^\infty \sum_{s=1}^\imath
    \frac{(-1)^{\frac{\imath+1}{2}+s}}{(\imath-s)!s!} \eta^{\imath-s} \xi^s
    = \sum_{\underset{\textit{\small odd}}{\imath=0}}^\infty \sum_{s=1}^\imath
    \frac{(-1)^\frac{\imath+1}{2}}{(\imath-s)!s!} \eta^{\imath-s} (-\xi)^s \notag \\
    &= \sum_{\underset{\textit{\small odd}}{\imath=0}}^\infty \frac{(-1)^\frac{\imath+1}{2}}{\imath!} \eta^\imath
    \cdot \sum_{\underset{\textit{\small even}}{\imath=2}}^\infty \frac{(-1)^\frac{\imath}{2}}{\imath!} (-\xi)^\imath
    + \sum_{\underset{\textit{\small even}}{\imath=0}}^\infty \frac{(-1)^\frac{\imath}{2}}{\imath!} \eta^\imath
    \cdot \sum_{\underset{\textit{\small odd}}{\imath=1}} \frac{(-1)^\frac{\imath+1}{2}}{\imath!} (-\xi)^\imath \notag \\
    &= \sum_{\imath=0}^\infty \frac{(-1)^{\imath+1}}{(2\imath+1)!} \eta^{2\imath+1}
    \cdot \left( \sum_{\imath=0}^\infty \frac{(-1)^\imath}{(2\imath)!} \xi^{2\imath} - 1 \right)
    + \sum_{\imath=0}^\infty \frac{(-1)^\imath}{(2\imath)!} \eta^{2\imath}
    \cdot \sum_{\imath=0}^\infty \frac{(-1)^{\imath+1}}{(2\imath+1)!} (-\xi)^{2\imath+1} \notag \\
    &= -\sin{\eta}(\cos{\xi}-1) + \cos{\eta}\sin{\xi} = \sin{\eta} - \sin(\eta-\xi)
\end{align}

Using this result we find, that for large $n$ the recursion substitute $Q_n^j$ converges to:
\begin{align}
    Q_n^j \sim&\ y\left(\frac{j-n}{N}\right) \left[ \cancel{\sin{\eta}} - \sin(\eta-\xi) - \cancel{\sin{\eta}} + \cancelto{0}{\sin(\eta-\eta)} \right]
    = - y\left(\frac{j-n}{N}\right) \sin\left( \tan^{-1}\tilde{\eta} - \tan^{-1}\tilde{\xi} \right) \notag \\
    =& - y\left(\frac{j-n}{N}\right) \left[ \sin(\tan^{-1}\tilde{\eta})\cos(\tan^{-1}\tilde{\xi})
    - \cos(\tan^{-1}\tilde{\eta})\sin(\tan^{-1}\tilde{\xi}) \right] \notag \\
    =& - y\left(\frac{j-n}{N}\right) \left[ \frac{\tilde{\eta}}{\sqrt{1+\tilde{\eta}^2}} \frac{1}{\sqrt{1+\tilde{\xi}^2}}
    - \frac{1}{\sqrt{1+\tilde{\eta}^2}} \frac{\tilde{\xi}}{\sqrt{1+\tilde{\xi}^2}} \right],
\end{align}
where we substitute:
\begin{equation}
    \tilde{\eta} = \frac{a\,z\left(\frac{j}{N}\right)}{c\,e^1}, \qquad
    \tilde{\xi} = \frac{a\,z\left(\frac{j-n}{N}\right)}{c\,e^1}.
\end{equation}

The square roots in the denominators are of the form:
\begin{equation}
    \sqrt{ 1 + \left(\frac{z\left(\frac{j}{N}-x\right)}{c\sin\beta}\right)^2 }
    = \frac{1}{c\sin\beta}\sqrt{ c^2\sin\beta + \left(\frac{j}{N}-x\right)^2a^2 + 2\left(\frac{j}{N}-x\right)ac\cos\beta + c^2\cos^2\beta }
    = \frac{y\left(\frac{j}{N}-x\right)}{c\sin\beta}
\end{equation}

After re-substitution many terms cancel and we get:
\begin{equation}\label{eq: limQnj}
    Q_n^j \sim -\frac{c\,e^1}{y\left(\frac{j}{N}\right)}\frac{n}{N}.
\end{equation}

The last puzzle peaces we need to take limits of are $V_n^j$ and $W_n^j$.
\begin{align}
    V_n^j =& \frac{Nf_N}{Y_N^{j-n}} \sum_{m_1=0}^{n-1} \frac{\hat{P}_{m_1+1}^{n-1}}{(Y_N^{j-m_1})^2}
    \sum_{\underset{\textit{\small odd}}{\imath=0}}^n \prod_{l=1}^\frac{\imath-1}{2} T_{l+\frac{1}{2}}
    \cdot P_0^{m_\imath-1}
    \sim \frac{Nf_N}{ \cancel{Y_N^{j-n}} } \sum_{m_1=0}^{n-1} \frac{1}{(Y_N^{j-m_1})^2}
    \frac{ \cancel{Y_N^{j-n+1}} }{Y_N^{j-m_1-1}}
    \sum_{\underset{\textit{\small odd}}{\imath=0}}^n \prod_{l=1}^\frac{\imath-1}{2} T_{l+\frac{1}{2}}
    \cdot \frac{Y_N^{j-m_\imath-1}}{Y_N^j} \notag \\
    \sim&\ Nf_N \sum_{m_1=0}^{n-1} \frac{1}{(Y_N^{j-m_1})^2}
    \frac{1}{ \cancel{N}\,y\left(\frac{j-m_1}{N}\right) }
    \sum_{\underset{\textit{\small odd}}{\imath=0}}^n \prod_{l=1}^\frac{\imath-1}{2} T_{l+\frac{1}{2}}
    \cdot \frac{ \cancel{N}\,y\left(\frac{j-m_{\imath}}{N}\right) }{Y_N^j} \label{eq: Vnj step 2} \\
    \sim&\ \cancel{N} \frac{c\,e^1}{N} \sum_{m_1=0}^{n-1} \frac{1}{ y^2\left(\frac{j-m_1}{N}\right) }
    \sum_{\underset{\textit{\small odd}}{\imath=0}}^n \prod_{l=1}^\frac{\imath-1}{2} T_{l+\frac{1}{2}}
    \cdot \frac{1}{ \cancel{N}\,y\left(\frac{j}{N}\right) } \label{eq: Vnj step 3}
\end{align}

In the last step, \eqref{eq: Vnj step 2} to \eqref{eq: Vnj step 3}, a $\frac{1}{y}$-term cancels with the $y$-term in the numerator of the first summation operator \eqref{eq: Cancellation l=1} and the same happens with the denominator in the second sum of the last summation operator \eqref{eq: Cancellation l=last} and the numerator of \eqref{eq: Vnj step 2} on which the operators are acting.
\begin{align}
    l =& 1 \ \to \ T_{1+\frac{1}{2}}: & y\left(\frac{j-m_{2(1+\frac{1}{2}-1)}}{N}\right) =&\ y\left(\frac{j-m_1}{N}\right) \label{eq: Cancellation l=1} \\
    l =& \frac{\imath-1}{2} \ \to \ T_{\frac{\imath}{2}}: & y\left(\frac{j-m_{2\frac{\imath}{2}}}{N}\right) =&\ y\left(\frac{j-m_{\imath}}{N}\right) \label{eq: Cancellation l=last}
\end{align}
Next we convert all sums to integrals. We can pull the last term in \eqref{eq: Cancellation l=last} out of the integrals since it does not depend on the integration variables:
\begin{align}
    V_n^j \sim&\ \frac{c\,e^1}{y\left(\frac{j}{N}\right)} \int_0^\frac{n}{N}
    \frac{1}{y^2\left(\frac{j}{N}-m_1\right)} \sum_{\underset{\textit{\small odd}}{\imath=0}}^n
    (-1)^\frac{\imath-1}{2}(c\,e^1)^{\imath-1} \int_0^{m_1} \frac{1}{y^2\left(\frac{j}{N}-m_2\right)} \ldots
    \int_0^{m_{\imath-1}} \frac{1}{y^2\left(\frac{j}{N}-m_\imath\right)} dm_\imath \ldots dm_2\, dm_1 \notag \\
    =& \frac{1}{y\left(\frac{j}{N}\right)} \sum_{\underset{\textit{\small odd}}{\imath=0}}^n
    (-1)^\frac{\imath-1}{2} (c\,e^1)^\imath J_\imath\left(\frac{n}{N}\right) \sim \mathcal{O}(1)
\end{align}
As in the case of $Q_n^j$, the additional integral over the sum allows us to write $V_n^j$ in one sum over $\imath$.\\

In $W_n^j$ the cancellation only happens for $T_\frac{\imath}{2}$ and thus we get an additional $y$-factor from $T_1$.
\begin{align}
    W_n^j =& \sum_{\underset{\textit{\small even}}{\imath=0,}}^n \prod_{l=1}^\frac{\imath}{2} T_l
    \cdot P_0^{m_\imath-1}
    \sim \sum_{\underset{\textit{\small even}}{\imath=0,}}^n \prod_{l=1}^\frac{\imath}{2} T_l
    \cdot \frac{Y_N^{j-m_\imath+1}}{Y_N^j}
    \sim \sum_{\underset{\textit{\small even}}{\imath=0,}}^n \prod_{l=1}^\frac{\imath}{2} T_l
    \cdot \frac{ y\left(\frac{j-m_{\imath}}{N}\right) }{ y\left(\frac{j}{N}\right) } \\
    \sim& \frac{1}{ y\left(\frac{j}{N}\right) } \sum_{\underset{\textit{\small even}}{\imath=2,}}^n (-1)^\frac{\imath}{2}(c\,e^1)^\imath
    \int_0^\frac{n}{N} \frac{ y\left(\frac{j-n}{N}\right) }{ y^2\left(\frac{j}{N}-m_1\right) } \ldots
    \int_0^{m_{\imath-1}} \frac{1}{y^2\left(\frac{j}{N}-m_\imath\right)} dm_\imath \ldots dm_1
    + \cancelto{\mathbbb{1}}{\prod_{l=1}^0 T_l}\cdot \frac{ y\left(\frac{j-m_0}{N}\right) }{ y\left(\frac{j}{N}\right) } \\
    \sim& \frac{ y\left(\frac{j-n}{N}\right) }{ y\left(\frac{j}{N}\right) }\left[ \sum_{\underset{\textit{\small even}}{\imath=0,}}^n (-1)^\frac{\imath}{2}(c\,e^1)^\imath J_\imath\left(\frac{n}{N}\right) + 1 \right]
    \sim \mathcal{O}(1)
\end{align}

In the case of $V_n^j$ and $W_n^j$ the latter turns out to be the only contributing recursion substitute and we proceed with it as we did for $Q_n^j$:
\begin{align}
    W_n^j \sim& \frac{ y\left(\frac{j-n}{N}\right) }{ y\left(\frac{j}{N}\right) }
    \sum_{\underset{\textit{\small even}}{\imath=0,}}^n (-1)^\frac{\imath}{2}\cancel{(c\,e^1)^\imath} 
    \sum_{s=1}^\imath \frac{(-1)^{\imath-s}}{(\imath-s)!}\frac{(-1)^\imath}{\cancel{(c\,e^1)^\imath} s!}\eta^{\imath-s}[\xi^s - \eta^s]
    \sim \frac{ y\left(\frac{j-n}{N}\right) }{ y\left(\frac{j}{N}\right) }
    \sum_{\underset{\textit{\small even}}{\imath=0,}}^n 
    \sum_{s=1}^\imath \frac{(-1)^{\frac{\imath}{2}-s}}{(\imath-s)!s!}\eta^{\imath-s}[\xi^s - \eta^s]
\end{align}

Analogue to the case before, we multiply two even and two odd series with adjusted powers on the signs to get all terms with even total power appearing in the series:
\begin{align}
    &\sum_{\underset{\textit{\small even}}{\imath=0}}^\infty \sum_{s=1}^\imath
    \frac{(-1)^\frac{\imath}{2}}{(\imath-s)!s!} \eta^{\imath-s} (-\xi)^s
    = \sum_{\underset{\textit{\small even}}{\imath=0}}^\infty \frac{(-1)^\frac{\imath}{2}}{\imath!} \eta^\imath
    \cdot \sum_{\underset{\textit{\small even}}{\imath=2}}^\infty \frac{(-1)^\frac{\imath}{2}}{\imath!} (-\xi)^\imath
    + \sum_{\underset{\textit{\small odd}}{\imath=0}}^\infty \frac{(-1)^\frac{\imath-1}{2}}{\imath!} \eta^\imath
    \cdot \sum_{\underset{\textit{\small odd}}{\imath=1}}^\infty \frac{(-1)^\frac{\imath+1}{2}}{\imath!} (-\xi)^\imath \notag \\
    &= \sum_{\imath=0}^\infty \frac{(-1)^\imath}{(2\imath)!} \eta^{2\imath}
    \cdot \left( \sum_{\imath=0}^\infty \frac{(-1)^\imath}{(2\imath)!} \xi^{2\imath} - 1 \right)
    + \sum_{\imath=0}^\infty \frac{(-1)^\imath}{(2\imath+1)!} \eta^{2\imath+1}
    \cdot \left( -\sum_{\imath=0}^\infty \frac{(-1)^\imath}{(2\imath+1)!} (-\xi)^{2\imath+1} \right) \notag \\
    &= \cos{\eta}\,(\cos{\xi}-1) + \sin{\eta}\,(-\sin(-\xi)) = \cos(\eta-\xi) - \cos{\eta}
\end{align}
And thus $W_n^j$ converges to:
\begin{align}
    W_n^j \sim& \frac{ y\left(\frac{j-n}{N}\right) }{ y\left(\frac{j}{N}\right) }
    \left[ \cos(\eta-\xi) - \cancel{\cos{\eta}} - \cos(\eta-\eta) + \cancel{\cos{\eta}} + 1 \right]
    = \frac{ y\left(\frac{j-n}{N}\right) }{ y\left(\frac{j}{N}\right) }
    \cos\left(\tan^{-1}\tilde{\eta}-\tan^{-1}\tilde{\xi}\right) \notag \\
    =& \frac{ y\left(\frac{j-n}{N}\right) }{ y\left(\frac{j}{N}\right) }
    \left[ \cos\left(\tan^{-1}\tilde{\eta}\right)\cos\left(\tan^{-1}\tilde{\xi}\right)
    + \sin\left(\tan^{-1}\tilde{\eta}\right)\sin\left(\tan^{-1}\tilde{\xi}\right) \right] \notag \\
    =& \frac{ y\left(\frac{j-n}{N}\right) }{ y\left(\frac{j}{N}\right) }
    \left[ \frac{1}{\sqrt{1+\tilde{\eta}^2}}\frac{1}{\sqrt{1+\tilde{\xi}^2}}
    + \frac{\tilde{\eta}}{\sqrt{1+\tilde{\eta}^2}}\frac{\tilde{\xi}}{\sqrt{1+\tilde{\xi}^2}} \right] \notag \\
    =& \frac{ \cancel{y\left(\frac{j-n}{N}\right)} }{ y\left(\frac{j}{N}\right) }
    \frac{ c^2\sin^2\beta }{ \cancel{y\left(\frac{j-n}{N}\right)} y\left(\frac{j}{N}\right) }
    \left[ 1 + \frac{\frac{j}{N}a^2 + ac\cos\beta}{c\,e^1}\frac{\frac{j-n}{N}a^2 + ac\cos\beta}{c\,e^1} \right] \notag \\
    =& \frac{1}{ y^2\left(\frac{j}{N}\right) } \left[ c^2\sin^2\beta
    + \left( \frac{j}{N}a + c\cos\beta \right)\left( \frac{j-n}{N}a + c\cos\beta \right) \right] \notag \\
    =& \frac{1}{ y^2\left(\frac{j}{N}\right) } \left[ \frac{j}{N}\frac{j-n}{N}a^2
    + \left(\frac{j}{N}+\frac{j+n}{N}\right)ac\cos\beta + c^2 \right]
    = \frac{ y^2\left(\frac{j}{N}\right) - \frac{j}{N}\frac{n}{N}a^2 - \frac{n}{N}ac\cos\beta }
    { y^2\left(\frac{j}{N}\right) }
    = 1 - \frac{n}{N}\frac{ a\,z\left(\frac{j}{N}\right) }{ y^2\left(\frac{j}{N}\right) }
\end{align}

\subsection{The second order of the full triangulation}
Now we are ready to calculate the second order term of the top-line. We insert the limits we calculated into \eqref{eq: b_2^j} for the case of $j=N$:
\begin{align}
    \bar{b}_2^N\varepsilon^2 =&\ \overset{\circ}{b}\,_2^N\varepsilon^2
    + \sum_{n=1}^{N-1}\left( O_n^N\overset{\circ}{b}\,_2^{N-n} + Q_n^N\overset{\circ}{\gamma}\,_2^{N-n} \right)\varepsilon^2 \\
    \sim& -\underbrace{\frac{N(ce^1)}{2Y_N^N}\int_0^1 i^2K_N[i]\, di\,\varepsilon^2}
    _{\sim\mathcal{O}(N^{-2})\overset{N\to\infty}{\longrightarrow}0}
    + \frac{1}{N}\sum_{n=1}^{N-1}\left(\vphantom{\frac{N}{N}}\right.
    -\underbrace{\frac{N(ce^1)}{2NY_N^{N-n}}\int_0^1 i^2K_{N-n}[i]\, di\, O_n^N}_{\sim\mathcal{O}(N^{-1})\overset{N\to\infty}{\longrightarrow}0}
    + \underbrace{\frac{f_N}{N}\int_0^1 k^2\,K_{N-n}[k]\,dk\, Q_n^N}_{\sim\mathcal{O}(1)} \left.\vphantom{\frac{N}{N}}\right)
\end{align}
The sum over $n$ grows with $N$. Since the exponential series converges very fast, almost all terms become arbitrarily close to the power series, as $N$ tends to infinity. Thus we can use \eqref{eq: limQnj} and we re-substitute the curvature field according to \eqref{eq: Kij_to_K(b,alpha)}, when we convert the sum over $n$ to an integral. To simplify the expressions we introduce the index function $K[i,j]$, which is defined over the Gaussian curvature field by:
\begin{equation}
    K[i,j] \coloneq K\left( i\,y(j), \,\sin^{-1}\frac{j\,a\sin\beta}{y(j)} \right)
    \in C^\infty([0,1]^2;\mathbb{R}).
\end{equation}
Then the semi-discrete version is related to it via:
\begin{equation}
    K_j(k) = K\left(k\frac{Y_N^j}{N},\,\sin^{-1}\frac{j\,a\sin\beta}{Y_N^j}\right)
    = K\left( k\,y\left(\frac{j}{N}\right), \,\sin^{-1}\frac{j\,a\sin\beta}{N\,y\left(\frac{j}{N}\right)} \right) = K\left[ k, \frac{j}{N} \right].
\end{equation}
Using this relation we calculate the second order term to the top-line:
\begin{align}\label{eq: b_2}
    b_2\varepsilon^2 =&\ \bar{b}_2^N\varepsilon^2 \sim \frac{c\,e^1}{N}\sum_{n=1}^{N-1} \int_0^1 k^2\,K_{N-n}[k]\,dk\, Q_n^N
    \sim -\frac{c\,e^1}{N}\sum_{n=1}^{N-1} \int_0^1 k^2\, K\left[ k,\frac{N-n}{N} \right]\, dk\,
    \frac{c\,e^1}{y\left(\frac{N}{N}\right)}\frac{n}{N} \\
    \overset{N\to\infty}{\longrightarrow}& -\frac{(c\,e^1)^2}{y(1)}
    \int_0^\frac{N}{N} \int_0^1 k^2\, K[k,1-n] \,dk\, n\, dn
    = -\frac{a^2c^2\sin^2\beta}{\sqrt{a^2 + c^2 + 2ac\cos\beta}}
    \int_0^1 n \int_0^1 k^2\, K[k,1-n] \,dk\,dn
\end{align}

In the same manner we can calculate the limit of the second order term to the top angle of the geodesic triangle:
\begin{align}\label{eq: gamma_2}
    \gamma_2\varepsilon^2 =&\ \bar{\gamma}_2^N\varepsilon^2 = \overset{\circ}{\gamma}{}_2^N\varepsilon^2
    + \sum_{n=1}^{N-1}\left( V_n^N\overset{\circ}{b}{}_2^{N-n} + W_n^N\overset{\circ}{\gamma}{}_2^{N-n} \right)\varepsilon^2 \\
    \sim& \underbrace{f_N\int_0^1 k^2\,K_N[k]\,dk\,\varepsilon^2}_{\sim\mathcal{O}(N^{-1})\overset{N\to\infty}{\longrightarrow}0}
    + \frac{1}{N}\sum_{n=1}^{N-1}\left(\vphantom{\frac{N}{N}}\right. -\underbrace{\frac{N(ce^1)}{2NY_N^{N-n}}
    \int_0^1 i^2K_{N-n}[i]\, di\, V_n^N}_{\sim\mathcal{O}(N^{-1})\overset{N\to\infty}{\longrightarrow}0}
    + \underbrace{\frac{f_N}{N}\int_0^1 k^2\,K_{N-n}[k]\,dk\,W_n^N}_{\sim\mathcal{O}(1)}
    \left.\vphantom{\frac{N}{N}}\right) \notag \\
    =& \frac{c\,e^1}{N}\sum_{n=1}^{N-1} \int_0^1 k^2\, K\left[ k, \frac{N-n}{N} \right] \,dk\,
    \left( 1 - \frac{n}{N}\frac{ a\,z\left(\frac{N}{N}\right) }{ y^2\left(\frac{N}{N}\right) } \right)
    \overset{N\to\infty}{\longrightarrow} c\,e^1 \int_0^1 \left( 1 - n\frac{ a\,z(1) }{ y^2(1) } \right) \int_0^1 k^2\, K[k,1-n] \,dk\,dn \notag
\end{align}

In the case of the sphere we have:
\begin{equation}
    \gamma_2\varepsilon^2 = SC^{(2)}(K;\pi-\beta,c,a) \qquad \alpha_2\varepsilon^2 = SC^{(2)}(K;\pi-\beta,a,c)
\end{equation}
The arguments $a$ and $c$ are just swapped for the two angles but a non-uniform curvature field in general breaks this symmetry. The arguments in the curvature field cannot be swapped, the $K_{ij}$ asymptotically converge to $K[i,j]$ independent of which quantity we calculate. Thus we cannot just swap the arguments in the case of a varying curvatrue field $K(l,\phi)$. Since we find no shortcut to this we calculate $\alpha_2\varepsilon^2$ via Eq.~\eqref{eq: alpha_2^j}:
\begin{align}
    \alpha_2^{0,j} =& \overset{\circ}{\alpha}{}_2^{0,j} + \overset{\circ}{\alpha}{}_2^{j,c}\left[ \overset{\circ}{b}\,_2^{j-1}
    + \sum_{n=1}^{j-2}\left( O_n^{j-1}\overset{\circ}{b}\,_2^{j-1-n} + Q_n^{j-1}\overset{\circ}{\gamma}\,_2^{j-1-n} \right) \right]
    + \overset{\circ}{\alpha}{}_2^{j,\beta}\left[ \overset{\circ}{\gamma}{}_2^{j-1}
    + \sum_{n=1}^{j-2}\left( V_n^{j-1}\overset{\circ}{b}{}_2^{j-1-n} + W_n^{j-1}\overset{\circ}{\gamma}{}_2^{j-1-n} \right) \right]  \notag
\end{align}

First we need the limits of the argument coefficients $\overset{\circ}{\alpha}{}_2^{j,c}$ and $\overset{\circ}{\alpha}{}_2^{j,\beta}$ from \eqref{eq: CoeffFuncs}. In this case they appear in a sum. The other argument coefficients only appear in the recursion substitutes, in the context of a product.
\begin{equation}
    \overset{\circ}{\alpha}{}_2^{j,c} = -\frac{N^2ce^1}{(Y_N^j)^2Y_N^{j-1}} \sim -\frac{N^2ce^1}{(Y_N^j)^3}
    \sim \mathcal{O}(N^{-1}),
    \qquad \overset{\circ}{\alpha}{}_2^{j,\beta} = \frac{aZ_N^j}{(Y_N^j)^2} \sim \mathcal{O}(N^{-1})
\end{equation}

Now, that we know the asymptotic behaviour of all involved quantities we see, that $4$ of the $7$ terms tend to $0$ and can thus be discarded:
\begin{align}
    \alpha_2\varepsilon^2 =&\ \bar{\alpha}_2^N\varepsilon^2 = \sum_{j=1}^N\alpha_2^{0,j}\varepsilon^2
    = \sum_{j=1}^N\left\{ \overset{\circ}{\alpha}{}_2^{0,j} + \overset{\circ}{\alpha}_2^{j,c}c_2^j
    + \overset{\circ}{\alpha}{}_2^{j,\beta}\beta_2^j \right\}\varepsilon^2 \\
    =& \sum_{j=1}^N\left\{ \overset{\circ}{\alpha}{}_2^{0,j} + \overset{\circ}{\alpha}{}_2^{j,c}\left[ \overset{\circ}{b}\,_2^{j-1}
    + \sum_{n=1}^{j-2}\left( O_n^{j-1}\overset{\circ}{b}\,_2^{j-1-n} + Q_n^{j-1}\overset{\circ}{\gamma}\,_2^{j-1-n} \right) \right] \right. \notag \\
    &\left. + \overset{\circ}{\alpha}{}_2^{j,\beta}\left[ \overset{\circ}{\gamma}{}_2^{j-1}
    + \sum_{n=1}^{j-2}\left( V_n^{j-1}\overset{\circ}{b}{}_2^{j-1-n} + W_n^{j-1}\overset{\circ}{\gamma}{}_2^{j-1-n} \right) \right] \right\}\frac{1}{N^2} \\
    \sim& \underbrace{ \frac{1}{N} \sum_{j=1}^N \frac{f_N}{N}\int_0^1 \frac{1}{n^2} \int_0^n k^2K_j[k] dk\,dn }
    _{\sim\mathcal{O}(1)} - \sum_{j=1}^N \frac{N^2(ce^1)^2}{(Y_N^j)^3}\left[\vphantom{\frac{(e^1)^2}{Y_N^j}}\right.
    \underbrace{ -\frac{N(ce^1)^2}{2N^2Y_N^j}\int_0^1 i^2K_j[i] di }_{\sim\mathcal{O}(N^{-2})\overset{N\to\infty}{\longrightarrow}0} \notag \\
    &+ \frac{1}{N}\sum_{n=1}^{j-2}\left(\vphantom{\frac{(e^1)^2}{Y_N^j}}\right.
    \underbrace{ -O_n^{j-1}\frac{N(ce^1)^2}{2NY_N^{j-n}}\int_0^1 i^2K_{j-n}[i] di }
    _{\sim\mathcal{O}(N^{-1})\overset{N\to\infty}{\longrightarrow}0}
    + \underbrace{ Q_n^{j-1}\frac{f_N}{N}\int_0^1 k^2K_{j-n}[k] dk }_{\sim\mathcal{O}(1)}
    \left.\vphantom{\frac{(e^1)^2}{Y_N^j}}\right) \left.\vphantom{\frac{(e^1)^2}{Y_N^j}}\right] \\
    &+ \sum_{j=1}^N \frac{aZ_N^j}{(Y_N^j)^2}\left[\vphantom{\frac{(e^1)^2}{Y_N^j}}\right.
    \underbrace{ \frac{f_N}{N^2}\int_0^1 k^2K_j[k] dk }_{\sim\mathcal{O}(N^{-1})\overset{N\to\infty}{\longrightarrow}0}
    + \frac{1}{N}\sum_{n=1}^{j-2}\left(\vphantom{\frac{(e^1)^2}{Y_N^j}}\right. 
    \underbrace{ -V_n^{j-1}\frac{N(ce^1)^2}{2NY_N^{j-n}}\int_0^1 i^2K_{j-n}[i] di }
    _{\sim\mathcal{O}(N^{-1})\overset{N\to\infty}{\longrightarrow}0}
    + \underbrace{ W_n^j\frac{f_N}{N}\int_0^1 k^2K_{j-n}[k] dk }_{\sim\mathcal{O}(1)}
    \left.\vphantom{\frac{(e^1)^2}{Y_N^j}}\right) \left.\vphantom{\frac{(e^1)^2}{Y_N^j}}\right] \notag
\end{align}

Next we convert the sums in the remaining terms into integrals:
\begin{align}
    \alpha_2\varepsilon^2 \sim&\ c\,e^1\frac{1}{N}\sum_{j=1}^N
    \int_0^1 \frac{1}{n^2} \int_0^n k^2\,K\left[k,\frac{j}{N}\right] dk\,dn
    + \frac{1}{N}\sum_{j=1}^N \frac{c\,e^1}{y\left(\frac{j}{N}\right)^3} \frac{1}{N}\sum_{n=1}^{j-2} \frac{c\,e^1}{y\left(\frac{j}{N}\right)}\frac{n}{N} c\,e^1\int_0^1 k^2\,K\left[k,\frac{j-n}{N}\right] dk \notag \\
    &+ \frac{1}{N}\sum_{j=1}^N \frac{a\,z\left(\frac{j}{N}\right)}{y^2\left(\frac{j}{N}\right)}
    \frac{1}{N}\sum_{n=1}^{j-2}\left( 1 - \frac{n}{N}\frac{ a\,z\left(\frac{j}{N}\right) }{ y^2\left(\frac{j}{N}\right) } \right) c\,e^1\int_0^1 k^2\,K\left[k,\frac{j-n}{N}\right] dk \\
    \sim&\ c\,e^1\int_0^1 \int_0^1 \frac{1}{n^2} \int_0^n k^2\,K[k,j] dk\,dn\,dj
    + \frac{(c\,e^1)^3}{N}\sum_{j=1}^N \frac{1}{y\left(\frac{j}{N}\right)^4}
    \int_0^\frac{j}{N} n \int_0^1 k^2\,K\left[k,\frac{j}{N}-n\right] dk\,dn \notag \\
    &+ \frac{c\,e^1}{N}\sum_{j=1}^N \frac{a\,z\left(\frac{j}{N}\right)}{y^2\left(\frac{j}{N}\right)}
    \int_0^\frac{j}{N} \left( 1 - n\frac{ a\,z\left(\frac{j}{N}\right) }{ y^2\left(\frac{j}{N}\right) } \right) \int_0^1 k^2\,K\left[k,\frac{j}{N}-n\right] dk\,dn \\
    \overset{N\to\infty}{\longrightarrow}&\ c\,e^1\int_0^1 \int_0^1 \frac{1}{n^2} \int_0^n k^2\,K[k,j]\, dk\,dn\,dj
    + (c\,e^1)^3 \int_0^1 \frac{1}{y(j)^4}
    \int_0^j n \int_0^1 k^2\,K[k,j-n]\, dk\,dn\,dj \notag \\
    &+ c\,e^1 \int_0^1 \frac{a\,z(j)}{y^2(j)}
    \int_0^j \left( 1 - n\frac{ a\,z(j) }{ y^2(j) } \right) \int_0^1 k^2\,K[k,j-n] dk\,dn\,dj
\end{align}

Finally, we write the term in the square brackets in a more compact form:
\begin{equation}\label{eq: alpha_2}
    \alpha_2\varepsilon^2 = c\,e^1\int_0^1 \int_0^1 \frac{1}{n^2}\int_0^n k^2\,K[k,j] dk\,dn
    + \int_0^j \left[ \frac{a\,z(j)}{y^2(j)} + n\frac{ (c\,e^1)^2 - a^2z^2(j) }{y^4(j)} \right]
    \int_0^1 k^2\,K[k,j-n] dk\,dn\, dj
\end{equation}

\section{Testing the results}\label{sec: Tests}
After such a lengthy calculation it makes sense to test the results on some examples. The flat space is a trivial test: All three second order terms are proportional to an integral over $K$ and thus vanish if $K=0$.\\
The $2$-sphere allows us to compare with a known exact result. We know no other ways to calculate comparable quantities for less symmetric examples, so we have to content ourselves with consistency checks.

\subsection{On the 2-sphere}\label{subsec: S^2 test}
We only expanded up to second order in all calculations on the segments and then took the limit. What we got is the second order term to the cosine and sine laws for an arbitrary curvature field. Instead of expanding in each variable $a$, $c$, $\beta$ separately, we introduce a scale-factor $\delta>0$. We consider $\delta$ to be small and thus use it as the variable, in which the expansion happens, while $a$, $c$ and $\beta$ are left undetermined and thus cover all possibilities and retain the generality of the expansion.\\
This agrees with the setup in Sec.~\ref{sec: Expansion for small triangles}, Eq.~\eqref{eq: variable expansion}, with $a_1 = a$, $b_1 = c$, $\gamma_0 = \pi-\beta$ and $a_n = b_n = 0$ for $n>1$, $\gamma_n = 0$ for $n > 0$. The scale parameter $\delta$ assumes the role of $\varepsilon$. But now we work with the limit i.e. we do not cut the triangle into pieces which is equivalent to having infinite pieces and thus we have no sensible notion of $N$-pieces.\\
We calculate the substitutes we used for the expansion for this special case:
\begin{align}
    C_0 =& 1, \qquad
    C_1 = 0 \\
    C_2 =& -\frac{K}{2}\left[ a^2 + c^2 - ac\cos(\pi-\beta) \right] \\
    C_3 =& 0 \\
    C_4 =& \frac{K^2}{24}\left( a^4 + 6a^2c^2 + c^4 - 4[a^3c + ac^3]\cos(\pi-\beta) \vphantom{\sqrt{2}}\right)
\end{align}
\begin{equation}
    y \coloneq \sqrt{-\frac{2C_2}{K}} = \sqrt{a^2 + c^2 + 2ac\cos\beta}, \qquad x \coloneq \frac{C_3}{K} = 0, \qquad z \coloneq \frac{C_4}{K}
\end{equation}
The cosine law, expanded up to third order in the scale parameter is given by \eqref{eq: C(K,gamma,a,b)_expanded} for this special case:
\begin{align}\label{eq: CO3}
    C(K;\pi-\beta,a,c) = \sum_{n=1}^\infty c_n\delta^n \approx&\ y\left( \delta - \frac{x}{y^2}\delta^2 + \left[ \frac{Ky^2}{24}
    - \frac{z}{y^2} - \frac{x^2}{2y^2} \right]\delta^3 + \mathcal{O}(\delta^4) \right)
\end{align}
And thus the first order recovers the expression for the plane $K=0$, the second order vanishes and thus the fird order is the first non-trivial term:
\begin{equation}
     c_1 = y = \sqrt{a^2 + c^2 + 2ac\cos\beta} \qquad c_2 = -\frac{x}{y} = 0 \qquad
     c_3 = y\left[ \frac{Ky^2}{24} - \frac{z}{y^2} - \frac{x^2}{2y^2} \right]
     = -\frac{K a^2 c^2 \sin^2\beta}{6\sqrt{a^2 + c^2 + 2ac\cos\beta}}
\end{equation}
We compare $c_3$ with $b_2\varepsilon^2$ for constant curvature $K(l,\varphi) = const.$:
\begin{align}
    b_2\varepsilon^2 =& -\frac{a^2c^2\sin^2\beta}{\sqrt{a^2 + c^2 + 2ac\cos\beta}}
    \int_0^1 n \int_0^1 k^2\,K \,dk\,dn
    = -\frac{K\,a^2c^2\sin^2\beta}{6\sqrt{a^2 + c^2 + 2ac\cos\beta}} = C_2(K;\pi-\beta,a,c)
\end{align}
The in this case disentangled double integral contributes the missing factor of $\frac{1}{6}$.

\begin{align}
    S_0 = \frac{a}{y}\sin(\pi-\beta), \qquad S_1 = 0, \qquad
    S_2 = \frac{1}{y}\left\{ \frac{K}{6}\left[ \frac{3a\,y^2}{4} - a^3 \right]\sin(\pi-\beta)
    + a\frac{z}{y^2}\sin(\pi-\beta) \right\}
\end{align}
\begin{align}
    SC(K;\pi-\beta,a,c) = \sum_{n=0}^\infty \alpha_n\delta^n \approx& \sin^{-1}(S_0) + \frac{S_1}{\sqrt{1-S_0^2}}\delta
    + \left\{ \frac{S_2}{\sqrt{1-S_0^2}} + \frac{yS_1^2}{\sqrt{1-S_0^2}^3} \right\}\delta^2 + \mathcal{O}(\delta^3)
\end{align}
\begin{align}
    \alpha_0 =& \sin^{-1}(S_0) = \sin^{-1}\left(\frac{a}{y}\sin\beta\right), \qquad \alpha_1 = \frac{S_1}{\sqrt{1-S_0^2}} = 0, \\
    \alpha_2 =& \frac{S_2}{\sqrt{1-S_0^2}} + \frac{yS_1^2}{\sqrt{1-S_0^2}^3}
    = \frac{Kac\sin\beta\left(2a^2 + c^2 + 3ac\cos\beta\right)}{6\left(a^2 + c^2 + 2ac\cos\beta\right)}
    \textcolor{gray}{ = Kac\sin\beta\frac{y^2(1)+a\,z(1)}{6\,y^2(1)} }
\end{align}

We calculate $\gamma_2\varepsilon^2$ for $K=const.$, which should give us the second order term of $SC$, with $a$ and $c$ swapped, because for the angle $\gamma$ the two side-lengths change the roles as compared to $\alpha$.
\begin{align}
    \gamma_2\varepsilon^2 =& c\,e^1 \int_0^1 \left( 1 - n\frac{ a\,z(1) }{ y^2(1) } \right) \int_0^1 k^2\,K\, dk\,dn
    = \frac{Kac\sin\beta}{3} \int_0^1  1 - n\frac{ a\,z(1) }{ y^2(1) }\, dn
    = \frac{Kac\sin\beta}{3}\left( 1 - \frac{ a\,z(1) }{ 2y^2(1) } \right) \notag \\
    =& Kac\sin\beta \frac{ 2y^2(1) - a\,z(1) }{ 6y^2(1) }
    = Kac\sin\beta \frac{ 2( a^2 + c^2 + 2ac\cos\beta ) - a( a + c\cos\beta ) }{ 6y^2(1) } \notag \\
    =& \frac{ Kac\sin\beta(a^2 + 2c^2 + 3ac\cos\beta) }{ 6(a^2 + c^2 + 2ac\cos\beta) } = SC^{(2)}(K,\pi-\beta,c,a)
\end{align}

Finally we check, whether $\alpha_2\varepsilon^2$ indeed reduces to $SC_2(K;\pi-\beta,a,c)$ and thus becomes equal to $\gamma_2\varepsilon^2$ with $a$ and $c$ swapped in the case of constant curvature. For a generic curvature field this is clearly not the case.
\begin{align}
    \alpha_2\varepsilon^2 =&\ c\,e^1\int_0^1 \int_0^1 \frac{1}{n^2}\int_0^n k^2\,K dk\,dn
    + \int_0^j \left[ \frac{a\,z(j)}{y^2(j)} + n\frac{ (c\,e^1)^2 - a^2z^2(j) }{y^4(j)} \right]
    \int_0^1 k^2\,K dk\,dn\, dj \\
    =& K\frac{c\,e^1}{3}\int_0^j \int_0^1 n\, dn
    + \int_0^j \frac{a\,z(j)}{y^2(j)} + n\frac{ (c\,e^1)^2 - a^2z^2(j) }{y^4(j)}\, dn\, dj
    = K\frac{c\,e^1}{3}\int_0^j \frac{1}{2} + j\frac{a\,z(j)}{y^2(j)} + j^2\frac{ (c\,e^1)^2 - a^2z^2(j) }{2y^4(j)}\, dj \notag
\end{align}
We use Mathematica~\cite{Mathematica} to calculate the integral over $j$ and simplify the expression:
\begin{align}
    \alpha_2\varepsilon^2 =&\ K\frac{c\,e^1}{12}\left( 3 + \frac{(a-c)(a+c)}{a^2 + c^2 + 2ac\cos\beta} \right)
    = \frac{Kac\sin\beta\left(2a^2 + c^2 + 3ac\cos\beta\right)}{6\left(a^2 + c^2 + 2ac\cos\beta\right)}
    = SC_2(K;\pi-\beta,a,c)
\end{align}

\subsection{The hilly landscape approach}
We use the analogy of a two dimensional metric to a chart of a hilly landscape. The distances we would have to walk does not necessarily correspond to the distance seen on the chart. The curvature of the surface can be visualized, in this case, by plotting the height as a function of position $h(p)$.\\
In the case of rotation symmetric geometry / landscape for example a rotation symmetric hill it is enough to know the profile of the hill. We will thus call the height function profile function in this case, which will only depend on the radial variable $h(r)$.\\
\begin{figure}[h!]
    \centering\includegraphics[width=0.6\linewidth]{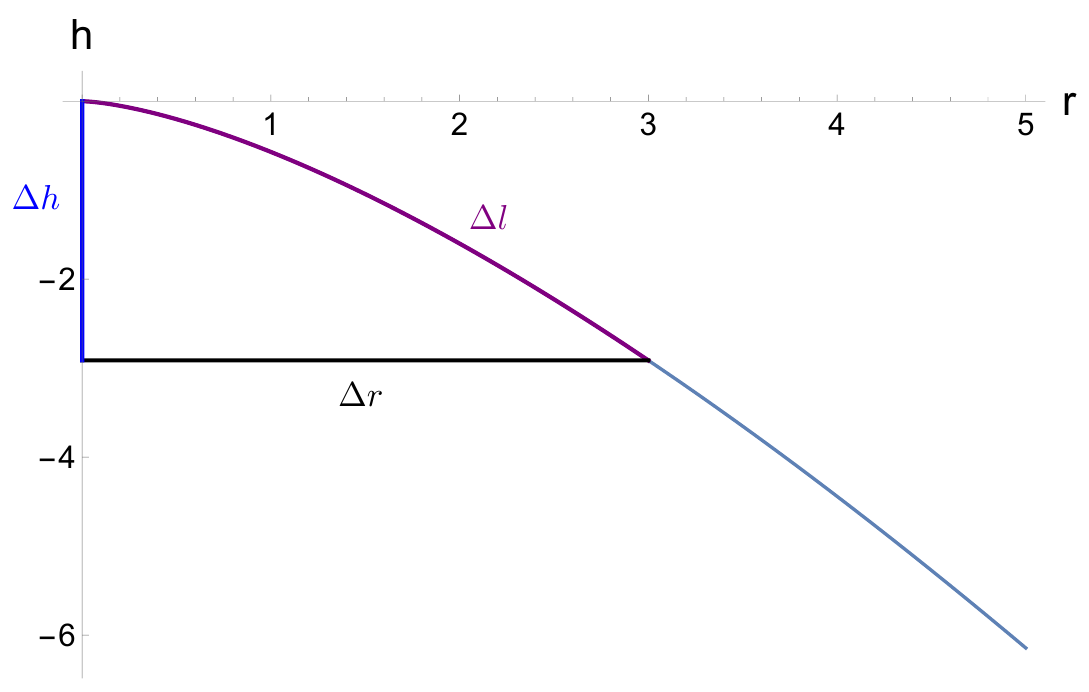}
    \caption{We sketch the idea of the hilly landscape approach on a plot of the example $h(r) = -\int \sqrt{\ln^2(e+r) - 1}\,dr$ we use in this section.}
    \label{fig:my_label}
\end{figure}

The metric of a general $2$-dimensional geometry can be written in the form:
\begin{equation}
    ds^2 = f(r)dr^2 + r^2d\varphi^2,
\end{equation}
where the function $f$ quantifies in a rather non-transparent way how much the observed distance on the chart differs from the actual one. The chart is representative of the true distances, wherever the value of $f(r)$ is close to one.\\
We can transform the metric such that the radial component becomes one and thus the new radial component $l$ represents the true distance from the center. Since the geometry is rotational symmetric the coordinate line along $r$ is a geodesic and the new coordinate $l$ is it's arc-length. In this case the metric could be mistaken for the flat one, if we overlook that $r$ must now be written in terms of $l$.
\begin{equation}
    r\mapsto l, \quad dl = \sqrt{f(r)}dr \quad \Rightarrow \quad ds^2 = dl^2 + r^2(l)d\varphi^2 \quad
    \wedge \quad l(r) = \int_0^r\sqrt{f(\rho)}d\rho
\end{equation}
Since the $\varphi$-coordinate was the angle at the center by definition the metric is now written in faithful normal coordinates and the curvature is pushed into the angular component. We also note, that we could calculate $r^2(l)$ from the geodesic flow bundle as discussed in the main text via:
\begin{equation}
    r^2(l) = \frac{1-\dot{l}(p;0)}{\dot{\varphi}(p;0)}
    \approx \frac{1-\partial_a C^{(2)}\left(K;\frac{\pi}{2},0,c\right)}
    {\partial_a SC^{(2)}\left(K;\frac{\pi}{2},0,c\right)}
\end{equation}

The integral must start from $0$, since the distance of the center to the central point is trivially $0$. By taking the derivative of $l(r)$ we can calculate function $f(r)$ from the relation between $l$ and $r$, which must be an invertible function. The arc-length must be a strictly monotonically growing function of $r$, which also puts a constrain on $f(r)$. 
\begin{equation}
    f(r) = l'^2(r)>0, \quad l(r) \rightleftharpoons r(l)
\end{equation}

We now use the idea of the hilly landscape approach and add a direction (a height $h$) to make the intrinsic curvature appear as extrinsic in the flat ambient space $\mathbb{R}^3$. By construction the metric in the $r,h$-plane must assume the flat Cartesian form, from which we can determine the profile function in terms of $f(r)$:
\begin{equation}
    dl^2 = dr^2 + dh^2 \quad \Rightarrow \quad dh = \sqrt{dl^2 + dr^2} = \sqrt{f(r)dr^2 + dr^2}
    = \sqrt{f(r) + 1}\,dr \quad \Rightarrow \quad h(r) = \int\sqrt{f(r)-1}\,dr
\end{equation}
The integration limits here correspond to a choice of embedding and have no further consequence.\\

Due to the rotation symmetry it is straight forward to extend this to an embedding of the two dimensional manifold into $\mathbb{R}^3$ and $h(r)$ allows us to construct an immersion.
\begin{equation}\label{eq: immersion}
    dS^2 = dr^2 + r^2d\varphi^2 + dh^2, \qquad \imath:\ [0,\infty)\times[0,2\pi) \to \mathbb{R}^3, \quad
    (r,\varphi) \mapsto \begin{pmatrix} r\cos\varphi \\ r\sin\varphi \\ h(r) \end{pmatrix}
\end{equation}

When trying to construct a solvable example it is inconvenient to work from $f(r)$ since we need to integrate over it's square root and hope that we can invert the $l(r)$ which comes out, if we can even calculate the integral. We can instead search for a suitable relation between $l$ and $r$ and calculate the profile function from this:
\begin{equation}
    dl = \sqrt{dr^2 + dh^2} = \sqrt{dr^2 + \left(\frac{dh}{dr}\right)^2dr^2} = \sqrt{1 + h'^2(r)}\,dr \quad
    \Rightarrow \quad l(r) = \int_0^r\sqrt{1 + h'^2(\rho)}\,d\rho
\end{equation}
We invert this relation, using the main theorem of calculus, to obtain the profile function from $l(r)$:
\begin{equation}
    l'(r) = \frac{d}{dr}\int_0^r\sqrt{1 + h'^2(\rho)}\,d\rho = \sqrt{1 + h'^2(r)} \quad \Rightarrow \quad l'^2(r) = 1 + h'^2(r) \quad \Rightarrow \quad l'(r)\geqslant 1
\end{equation}
From this result we see, that the derivative of $l(r)$ must be bigger or equal to $1$. This makes sense geometrically, since a constant profile function $h(r) = const.$ represents a flat space with $l=r$ and as soon as $h(r)$ bends we would expect the length of it's graph from $0$ to $r$ to be longer then $r$:
\begin{equation}
    h(r) = const. \quad \Rightarrow \quad h'(r) = 0 \quad \Rightarrow \quad l(r) = \int_0^r\sqrt{1 + h'^2(\rho)}\,d\rho = \int_0^r\,d\rho = r
\end{equation}
Since the relation between the derivatives of the arc-length and the profile function with respect to $r$ is quadratic we get two possible solutions which again merely correspond to a choice of the embedding and can thus be chosen by preference.
\begin{equation}
    h'(r) = \pm\sqrt{l'^2(r)-1} \quad \Rightarrow \quad h(r) = \pm\int\sqrt{l'^2(r)-1}dr
\end{equation}

\subsubsection{On a non-trivial rotation symmetric immersion into $\mathbb{R}^3$}
It is convenient to search for a function $l(r)$, satisfying the following conditions, to construct an example of such an immersion.
\begin{equation}
    i)\ \exists\ l^{-1}(l) = r(l), \qquad ii)\ l(0) = 0, \qquad iii)\ l'(r) \geqslant 1
\end{equation}
With the condition $iii)$ the strict monotony is satisfied automatically. For this example we would like a geometry which is initially curved and becomes asymptotically flat as $r$ tends to infinity. It is not immediately obvious, that the function
\begin{equation}
    l(r) = (e+r)\ln(e+r) - e - r, \quad \text{which satisfies} \quad l(0) = e\,\cancelto{1}{\ln{e}} - e = 0
\end{equation}
satisfies this requirement, since it's first derivative
\begin{equation}
    l'(r) = \ln(e+r) > 1, \quad r > 0
\end{equation}
is not asymptotically constant. The constant $e$ denotes Euler's number. It follows, that also the derivative of the profile function
\begin{equation}
    h'(r) = -\sqrt{\ln^2(e+r) - 1},
\end{equation}
where we chose a negative sign, is not asymptotically constant. So, the profile function does not asymptotically behave like a linear function, which would correspond to a cone. We need to keep in mind though that both principle curvatures are multiplied and we will see later, that the Gauss curvature field indeed tends to zero in the limit $r\to\infty$.\\
The inverse of $l$ is given by:
\begin{equation}
    r(l) = \frac{l}{W_0\left(\frac{l}{e}\right)} - e, \qquad x\,e^x = y, \quad \text{if} \quad x = W_0(y), \quad x\geqslant0
\end{equation}
where $W_0$ denotes the principal branch of the Lambert $W$ function, also called the product logarithm.\\

We calculate the integral for the profile function numerically to plot the surface embedded into $\mathbb{R}^3$.
\begin{equation}\label{eq: h(r)}
    h(r) = -\int \sqrt{\ln^2(e+r) - 1}\,dr
\end{equation}
The calculation of the function $f(r)$ is more straight forward and we calculate the Gaussian curvature in $r$,$\varphi$-coordinates for comparison with the faithful normal ones.
\begin{equation}
    f(r) = \ln^2(e+r), \qquad ds^2 = \ln^2(e+r)\,dr^2 + r^2d\varphi^2
\end{equation}
The non-vanishing Christoffel symbols are:
\begin{equation}
    \Gamma^r_{rr} = \frac{1}{(e+r)\ln(e+r)}, \qquad \Gamma^r_{\varphi\varphi} = -\frac{r}{\ln^2(e+r)}, \qquad
    \Gamma^\varphi_{r\varphi} = \frac{1}{r}
\end{equation}
The non-vanishing components of the Riemann tensor are given by:
\begin{equation}
    R^r_{\varphi r\varphi} = -R^r{\varphi\varphi r} = \frac{r}{(e+r)\ln^3(e+r)}, \qquad
    R^\varphi_{r\varphi r} = -R^\varphi_{rr\varphi} = \frac{1}{r(e+r)\ln(e+r)}
\end{equation}
We contract to calculate the Ricci-tensor:
\begin{equation}
    R_{ij} = \begin{pmatrix} R_{rr} & 0 \\ 0 & R_{\varphi\varphi} \end{pmatrix}, \qquad
    R_{rr} = \frac{1}{r(e+r)\ln(e+r)}, \quad R_{\varphi\varphi} = \frac{r}{(e+r)\ln^3(e+r)}
\end{equation}
Finally the Gaussian curvature field in two dimensions is half of the Ricci-scalar:
\begin{equation}
    K(r,\varphi) = \frac{R}{\det{g}} = \frac{R}{2} = \frac{1}{r(e+r)\ln^3(e+r)}
\end{equation}
And we see, that this curvature field goes to $0$ rather quickly as $r$ tends to infinity.

To calculate the geodesic flow bundle, we using the formulas derived in the previous section, we need the curvature field in faithful normal coordinates. We transform the metric using $r(l)$ and repeat the calculation above:
\begin{equation}
    ds^2 = dl^2 + r^2(l)d\varphi^2 = dl^2 + \left( \frac{l}{W_0\left(\frac{l}{e}\right)} - e \right)d\varphi^2
\end{equation}
We list the non-zero Christoffel symbols which are not the analogue components, since the curvature moved to the $\varphi$-component.
\begin{equation}
    \Gamma^l_{\varphi\varphi} = \frac{e\,W_0\left(\frac{l}{e}\right)-l}{W_0\left(\frac{l}{e}\right) \left(1+W_0\left(\frac{l}{e}\right)\right)}, \qquad
    \Gamma^\varphi_{l\varphi} = \frac{W_0\left(\frac{l}{e}\right)}{\left(1 + W_0\left(\frac{l}{e}\right)\right)
    \left(l - e\,W_0\left(\frac{l}{e}\right)\right)}
\end{equation}
The curvature is then given by:
\begin{align}
    R^l_{\varphi l\varphi} =& -R^l_{\varphi\varphi l}
    = \frac{l-e\,W_0\left(\frac{l}{e}\right)}{l \left(1 + W_0\left(\frac{l}{e}\right)\right)^3}, \qquad
    R^\varphi_{l\varphi l} = -R^\varphi_{ll\varphi}
    = \frac{W_0\left(\frac{l}{e}\right)^2}{l \left(1+W_0\left(\frac{l}{e}\right)\right)^3
    \left(l-e\,W_0\left(\frac{l}{e}\right)\right)} \\
    R_{ij} =& \begin{pmatrix} R_{ll} & 0 \\ 0 & R_{\varphi\varphi} \end{pmatrix}, \qquad
    R_{ll} = \frac{W_0\left(\frac{l}{e}\right)^2}{l \left(1+W_0\left(\frac{l}{e}\right)\right)^3
    \left(l-e\,W_0\left(\frac{l}{e}\right)\right)}, \quad
    R_{\varphi\varphi} = \frac{l-e\,W_0\left(\frac{l}{e}\right)}{l \left(1 + W_0\left(\frac{l}{e}\right)\right)^3}
\end{align}
And the Gaussian curvature field in faithful normal coordinates is given by:
\begin{equation}\label{eq: K(l,varphi)}
    K(l,\varphi) = \frac{W_0\left(\frac{l}{e}\right)^2}{l \left(1+W_0\left(\frac{l}{e}\right)\right)^3
    \left(l-e\,W_0\left(\frac{l}{e}\right)\right)}
\end{equation}
We observe, that the curvature field in the $r$,$\varphi$-coordinates is distorted in radial direction in a non-trivial way with respect to the distances $l$ from the center. We see from Fig.~\ref{fig: K(l)} that since the curvature drops very quickly the numerical values of $K(r)$ are very close to those of $K(l)$.
\begin{figure}[h!]
    \centering
    \includegraphics[width=.6\linewidth]{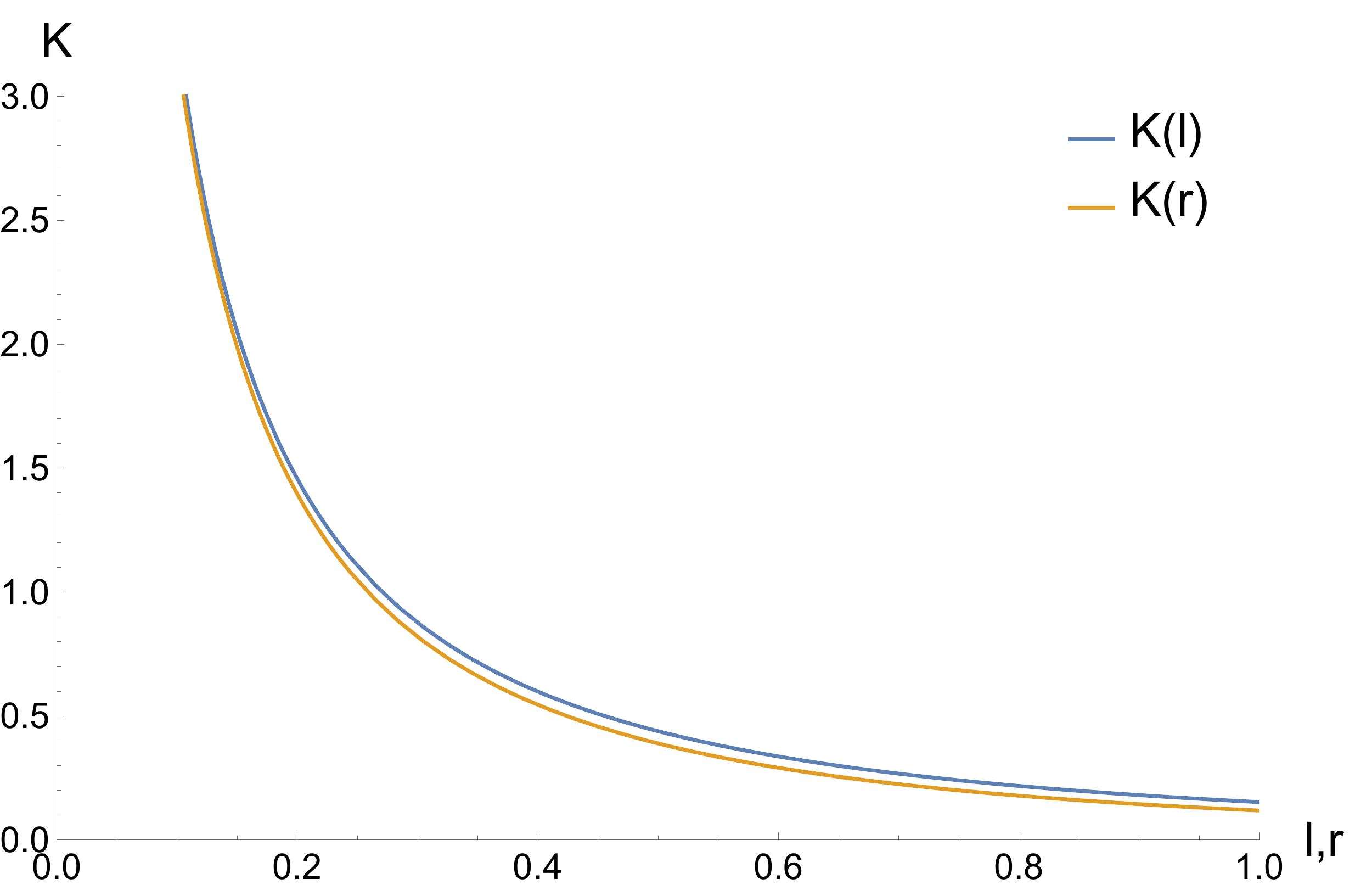}
    \caption{We plot the Gaussian curvature field $K$ in radial direction versus arclength $l$ and radial coordinate $r$, both running from $0$ to $1$.}
    \label{fig: K(l)}
\end{figure}

We use the results of our triangulation \eqref{eq: b_2} and \eqref{eq: alpha_2} together with the immersion $\imath$~\eqref{eq: immersion} constructed via the profile function $h(r)$~\eqref{eq: h(r)} to plot a sample of the geodesic flow bundle onto the embedded curved surface in $\mathbb{R}^3$ on the left in Fig.~\ref{fig: GFB on embedding}. On the right hand side we compare the same sample with the analogue sample of the unit $2$-sphere $\mathcal{S}^2$ and the Euclidean plane $E$ in a faithful normal chart.\\
\begin{figure}[h!]
    \begin{minipage}{.33\linewidth}
        \centering\includegraphics[width=\linewidth]{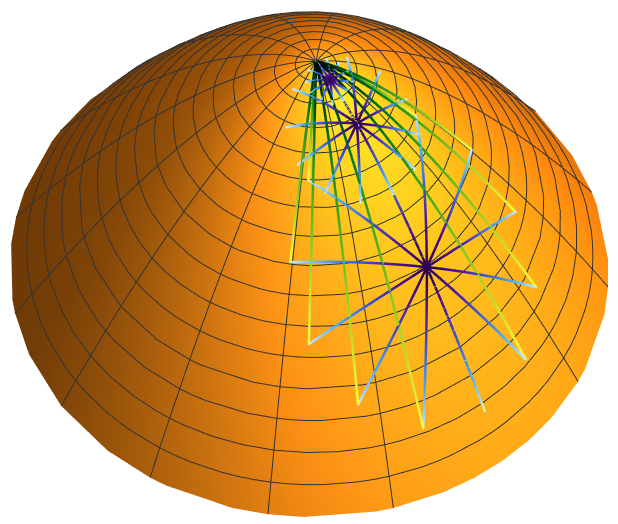}
    \end{minipage}
    \begin{minipage}{.66\linewidth}
        \centering\includegraphics[width=\linewidth]{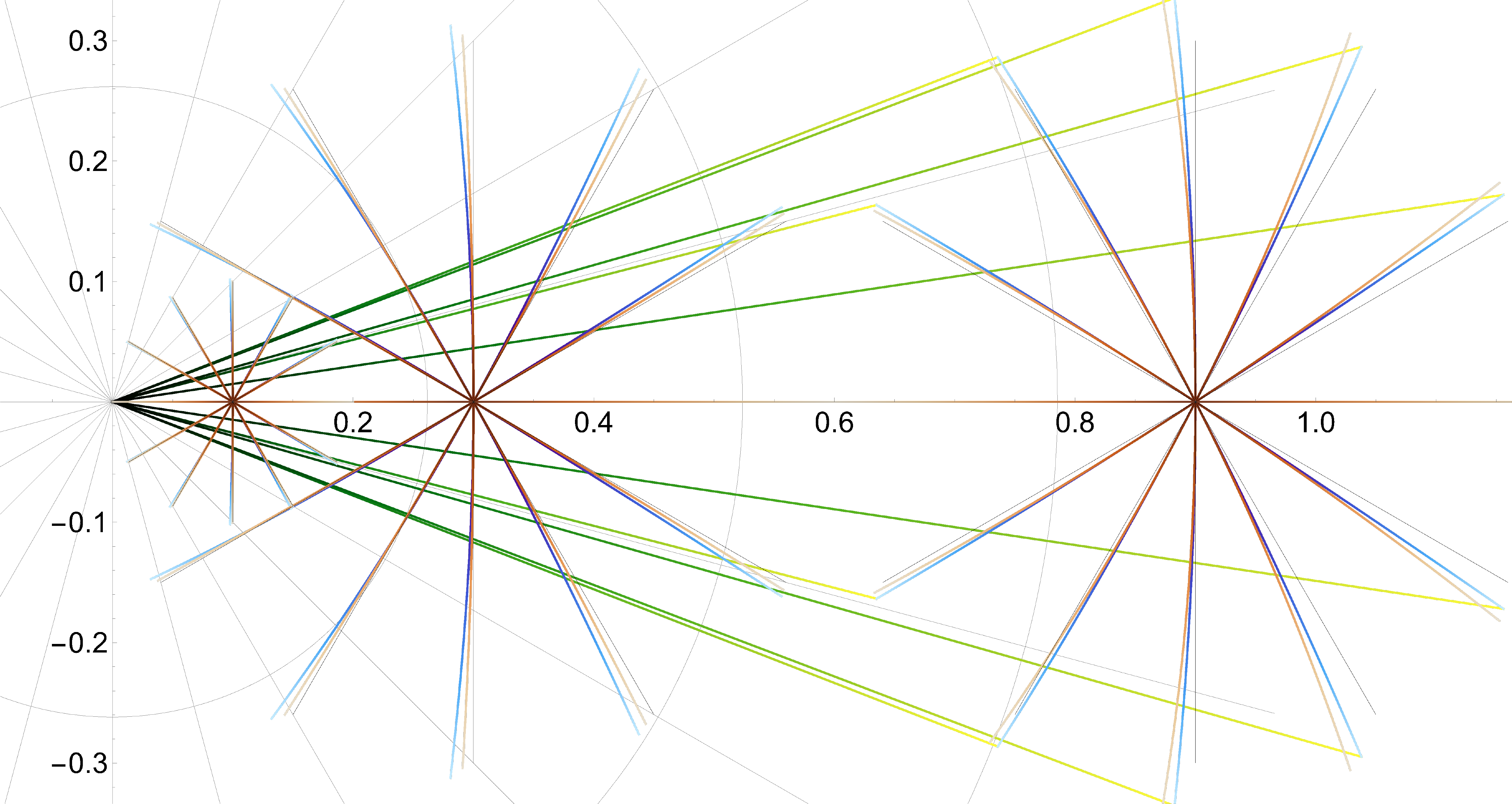}
    \end{minipage}
    \caption{We sample the geodesic flow bundle to the Gaussian curvature field $K(l,\varphi)$ from equation~\eqref{eq: K(l,varphi)} in blue tones at $l=0.1,0.3,0.9$ in the directions $\beta=n\frac{\pi}{6}$, $n\in\mathbb{Z}$ for the arclengt parameter range $\lambda\in[0,0.1),\,[0,0.3),\,[0,0.9)$ respectively. We plot the flow from $o$ to the end points of the flow from $l=0.9$ in green tones. On the left we plot it onto the embedded surface in $\mathbb{R}^3$ and on the right in faithful normal coordinates with the flows of the $2$-sphere in orange and the Euclidean plane in thin black lines for comparison.}
    \label{fig: GFB on embedding}
\end{figure}

We observe, that the geodesics on this surface are closer to the ones of the Euclidean plane then the $2$-sphere as $l$ increases as we would expect from a decreasing curvature field. The geodesics will however never become straight lines, since this chart is from the perspective of an observer in the origin and thus the image is distorted by the curvature surrounding the origin $o$. In more mathematical language: the curvature around $o$ always contributes to the integrals over $K(l,\varphi)$.\\
We also plot the geodesics from $o$ which are supposed to match up with the flow from $q(0.9)$. The fact that the geodesic triangles match up in faithful normal coordinates as well as in the constructed immersion shows the consistency of the results on a non-trivial curvature field.\\
It may seem strange that the curvature field has a pole at the center, especially since the derivative of the profile function vanishes there and thus seems to match up. The second derivative diverges however and it is thus not possible to fit a sphere with a finite radius to this point.
\begin{equation}
    h'(0) = -\sqrt{\ln^2(e) - 1} = 0, \qquad
    h''(r) = \frac{\ln(e+r)}{ (e+r)\sqrt{\ln^2(e+r) - 1} } \overset{r\to0}{\longrightarrow} \infty
\end{equation}

\subsection{On a central $\frac{1}{l}$-curvature field}
In this section we test our results on a central $\frac{1}{l}$ curvature field.
\begin{equation}
    K(l,\varphi) = \frac{1}{l} \quad \Rightarrow \quad
    K(i,j) = K\left( i\,y(j), \sin^{-1}\frac{j\,a\sin\beta}{y(j)} \right) = \frac{1}{i\,y(j)}
\end{equation}
A point on a geodesic from $q(c) = \gamma_{o,0}(c)$ heading into direction $\beta$ in faithful normal coordinates is given by the two components $l$ and $\varphi$, which we expand to second order:
\begin{align}
    l(c,\beta;\lambda) \approx&\ C^{(2)}(K(l,\varphi);\pi-\beta,\lambda,c) = b_0 + b_2\varepsilon^2
    = y -\frac{f^2}{y}\int_0^1 n \int_0^1 \frac{k}{y(1-n)} \,dk\,dn, \\
    \varphi(c,\beta;\lambda) \approx& SC^{(2)}(K(l,\varphi);\pi-\beta,\lambda,c) = \alpha_0 + \alpha_2\varepsilon^2 \notag \\
    =& \sin^{-1}\frac{e^1}{y} + f\int_0^1 \int_0^1 \frac{1}{n^2}\int_0^n \frac{k}{y(j)} dk\,dn
    + \int_0^j \left[ \frac{a\,z(j)}{y^2(j)} + n\frac{ f^2 - a^2z^2(j) }{y^4(j)} \right]
    \int_0^1 \frac{k}{y(j-n)} dk\,dn\, dj \\
    \vphantom{a}\notag\\
    y(j) =&\ \sqrt{j^2a^2 + c^2 + 2jac\cos\beta}, \quad y = y(1), \qquad z(j) = j\,a + c\cos\beta \qquad f = f_1 = c\,e^1 = a^2c^2\sin^2\beta
\end{align}
The integrals are not analytically calculatable for most curvature fields. Also in this case the integral over $j$ to calculate $\varphi$ is not solvable. For the numerical integration, which we use instead, it is more convenient to split the integral for $\alpha_2\varepsilon^2$ into two, due to the different integration boundaries of the $n$- and $k$-integrals:
\begin{equation}
    \alpha_2\varepsilon^2 = f(I_1 + I_2), \qquad b_2\varepsilon^2 = -\frac{f^2}{y}I_3
\end{equation}
\begin{align}
    I_1 \coloneq& \int_0^1 \int_0^1 \frac{1}{n^2} \int_0^n k^2\,K[k,j]\, dk\,dn\,dj, & I_3 \coloneq \int_0^1 n \int_0^1 k^2\,K[k,1-n]\, dk\,dn \\
    I_2 \coloneq& \int_0^1 \int_0^j \left[ \frac{a\,z(j)}{y^2(j)} + n\frac{f^2-a^2z^2(j)}{y^4(j)} \right]
    \int_0^1 k^2\,K[k,j-n]\, dk\,dn\,dj
\end{align}
We plot a sample of the geodesic flow bundle of the $\frac{1}{l}$-curvature field in Fig.~\ref{fig: 1/l-field}. The geodesics get closer to the ones of the unit sphere as we get closer to $l=1$, where the curvature field drops to $K=1$, consistent with our expectation. We avoid plotting outside the range of $1$, since the approximation becomes increasingly less accurate towards $1$.
\begin{figure}[h!]
    \centering
    \includegraphics[width=.8\linewidth]{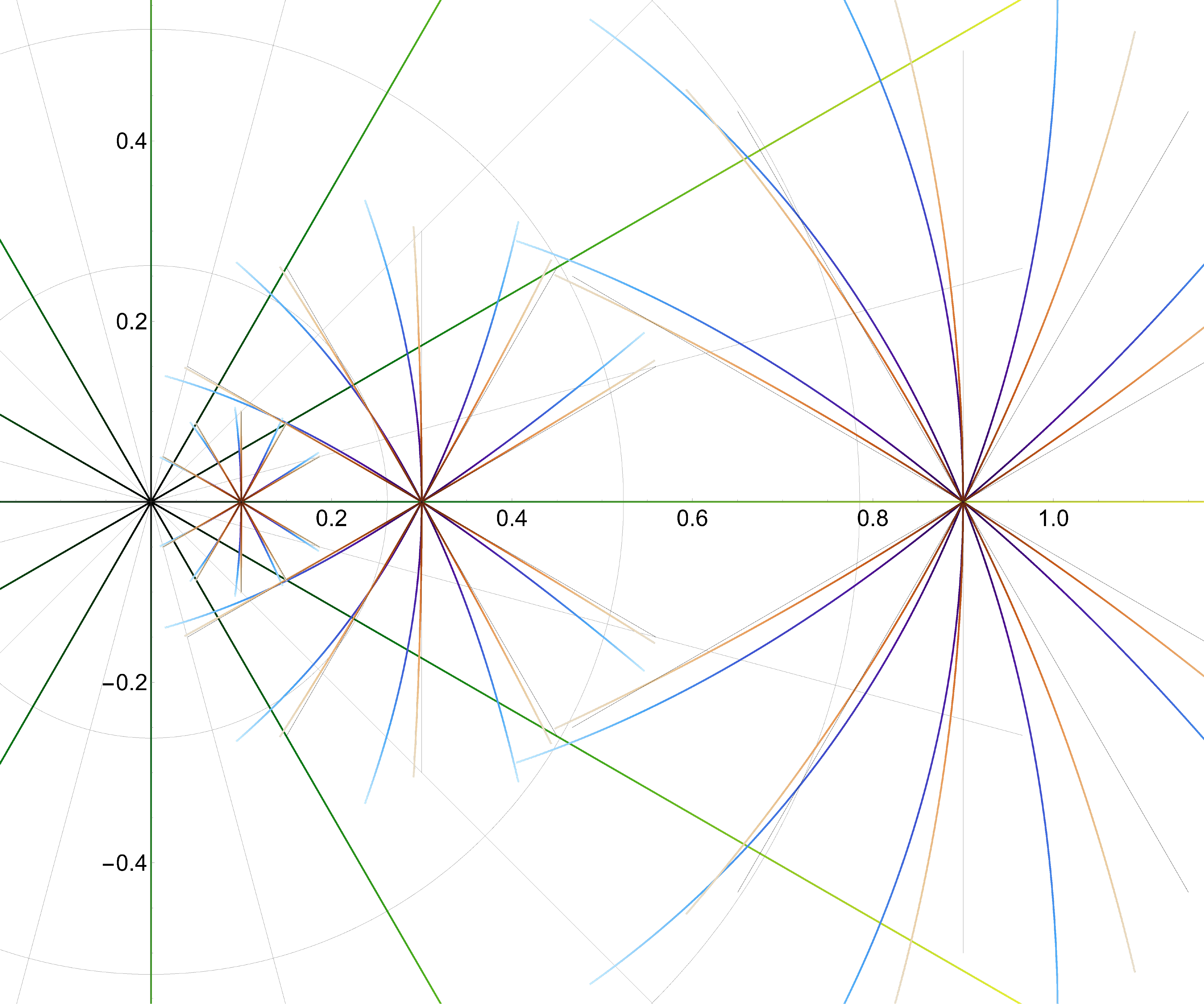}
    \caption{We plot the $K(l,\varphi) = \frac{1}{l}$-curvature field in blue and the constant curvature $K=1$ and $K=0$ in orange and thin gray lines respectively for comparison. The geodesic flow bundle is sampled at $l=0.1,0.3,0.9$ and w.l.o.g. in direction $\varphi=0$ up to second order.}
    \label{fig: 1/l-field}
\end{figure}

\subsection{On a circular curvature wave}
We model a curvature wave, traveling outwards from the center with a finite amplitude at the center, which decays with $\frac{1}{l}$. We treat the time $t$ as an extrinsic parameter and assure, that the curvature values are always positive.
\begin{equation}
    K(l,\varphi) = \frac{\kappa}{l+\frac{1}{2}}\left( 1 + \sin(2\pi l - \omega t) \right)
\end{equation}
We can see from the Figs.~\ref{fig: curvature wave} and \ref{fig: curvature wave GFB} that all geodesics behind the peak are bent most, since the rays from the center collect more curvature until they reach them. The flows closer to the center become flatter as the peak travels away, consistent with the low curvature around center at $\omega t = \frac{2\pi}{3}$.
\begin{figure}[h!]
    \begin{minipage}{.33\linewidth}
        \centering\includegraphics[width=\linewidth]{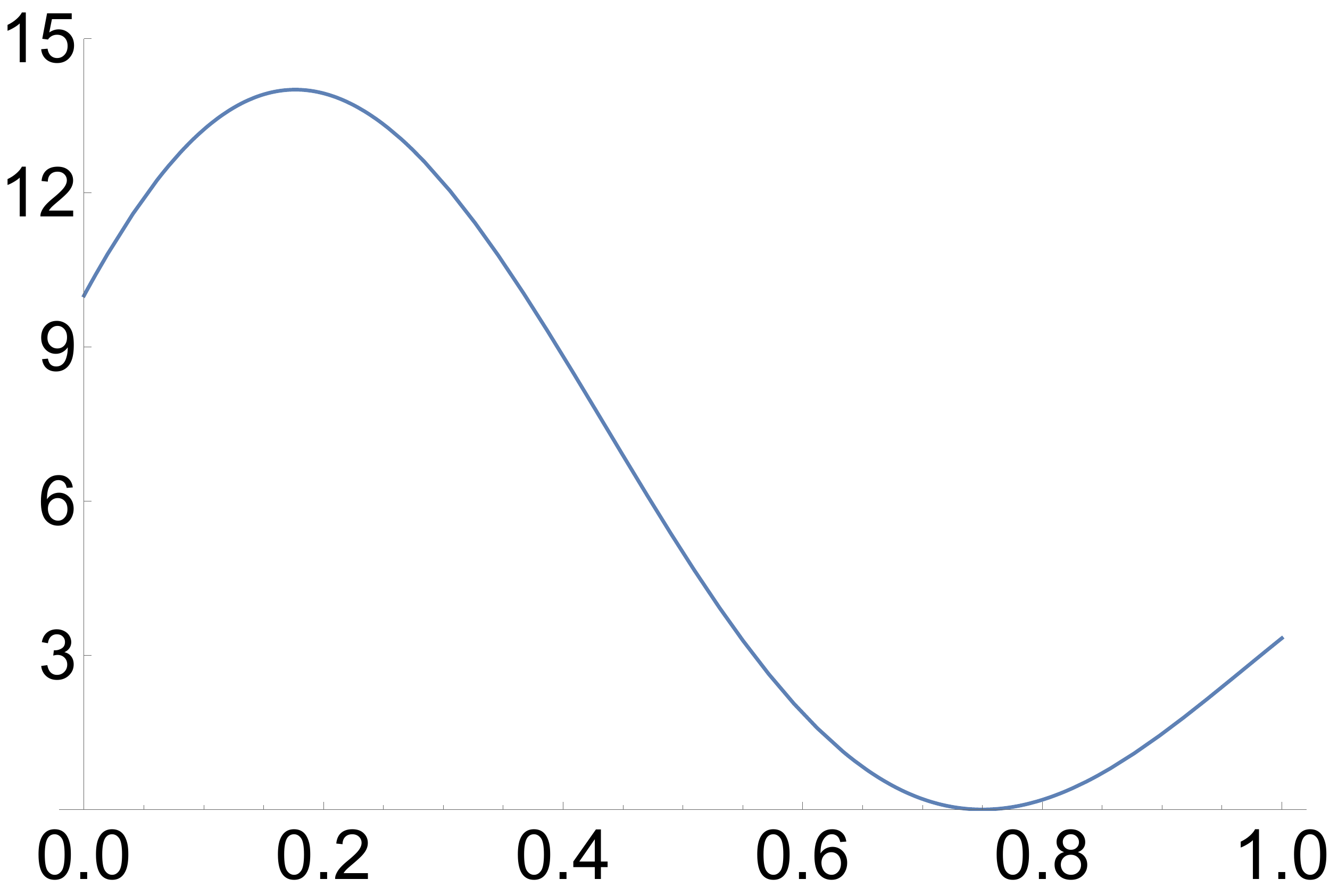}
    \end{minipage}
    \begin{minipage}{.33\linewidth}
        \centering\includegraphics[width=\linewidth]{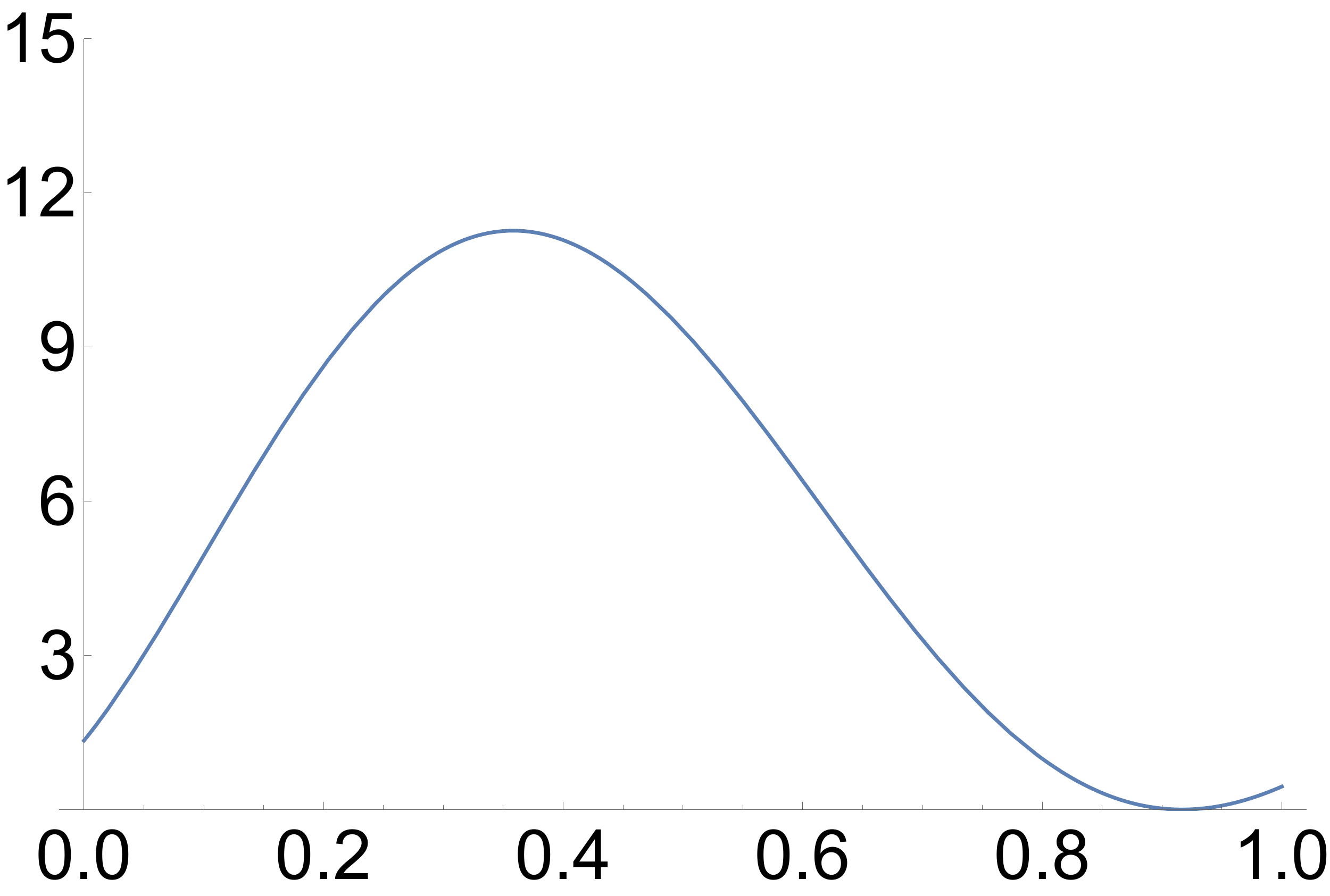}
    \end{minipage}
    \begin{minipage}{.33\linewidth}
        \centering\includegraphics[width=\linewidth]{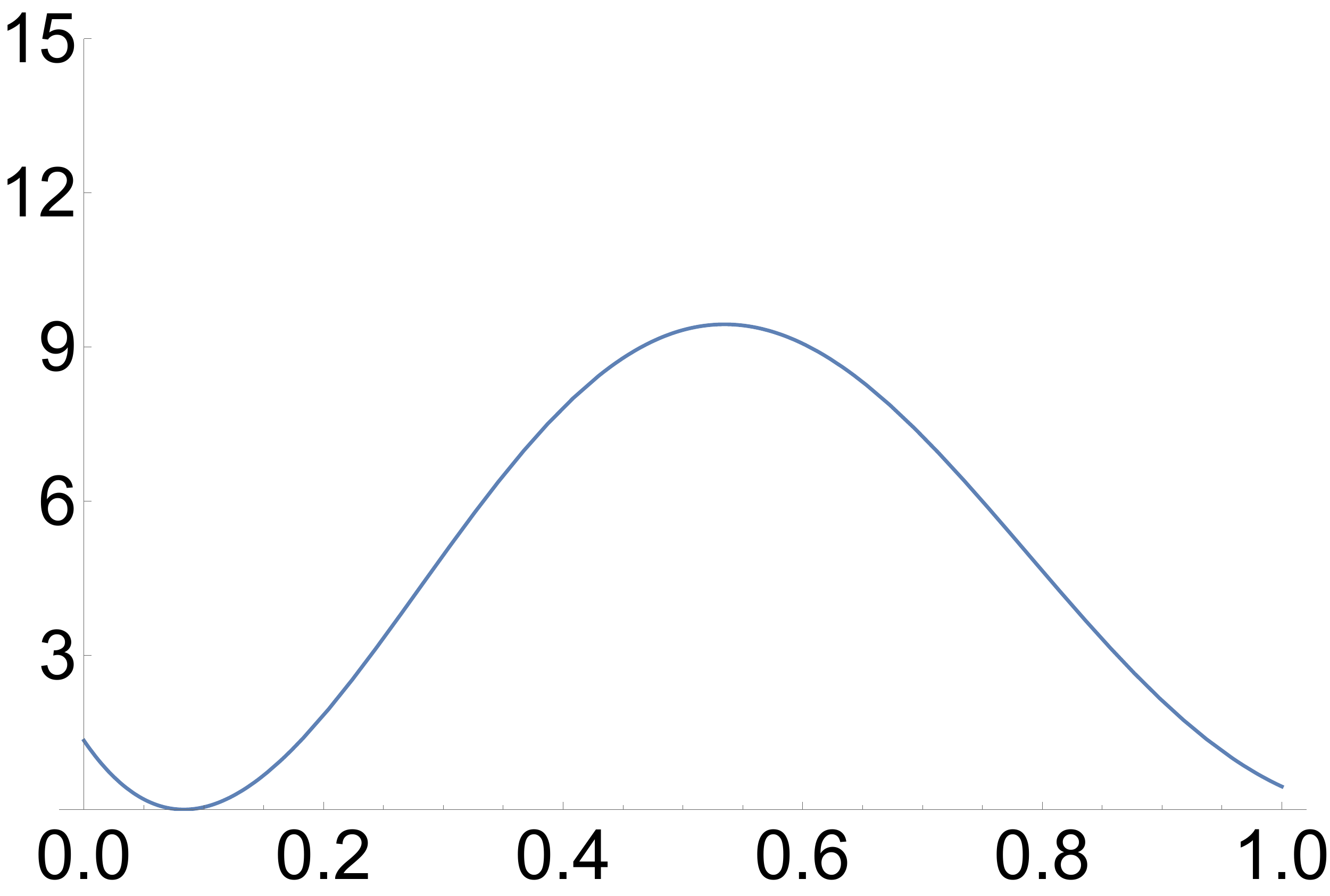}
    \end{minipage}
    \caption{We plot the curvature wave profile at $\omega t = 0,\,\frac{\pi}{3},\,\frac{2\pi}{3}$ from left to right.}
    \label{fig: curvature wave}
    \begin{minipage}{\linewidth}
        \centering\includegraphics[width=.9\linewidth]{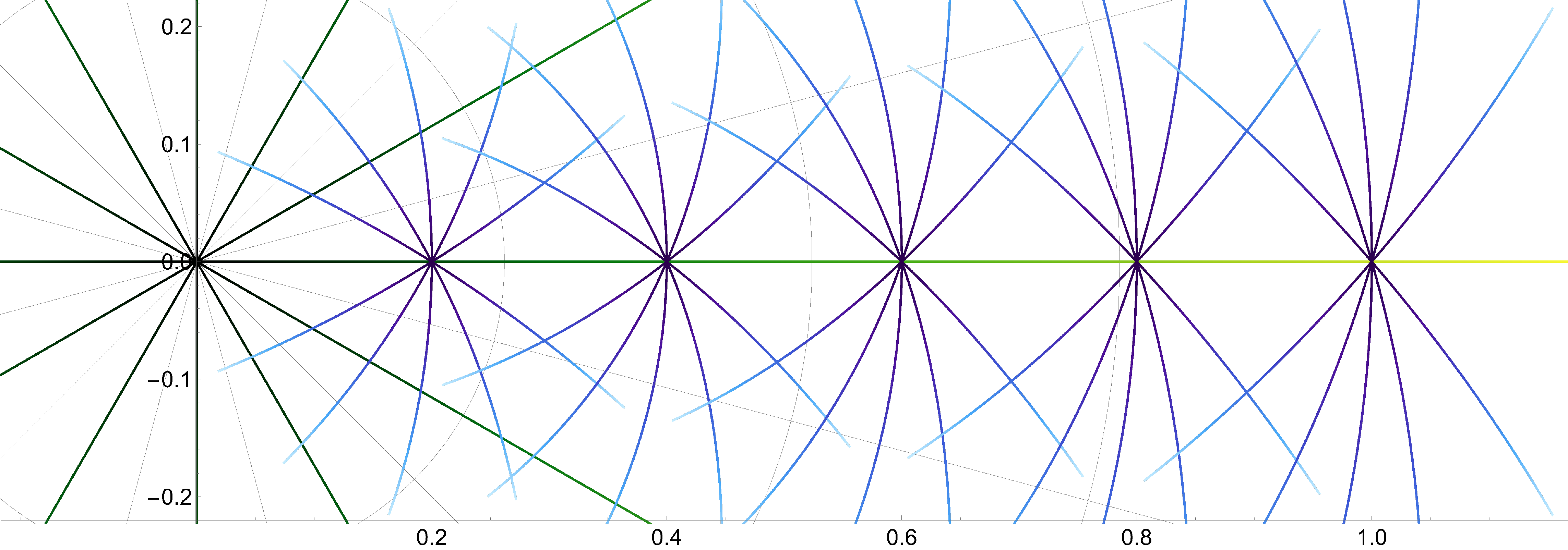}
    \end{minipage}
    \begin{minipage}{\linewidth}
        \centering\includegraphics[width=.9\linewidth]{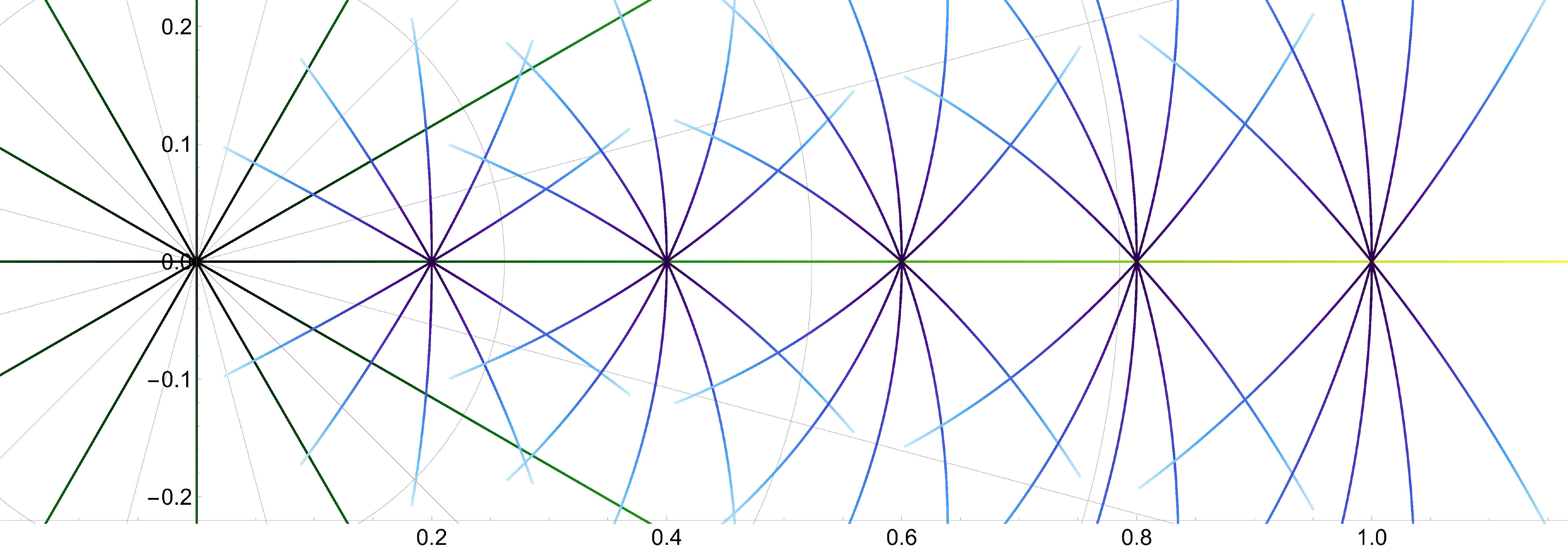}
    \end{minipage}
    \begin{minipage}{\linewidth}
        \centering\includegraphics[width=.9\linewidth]{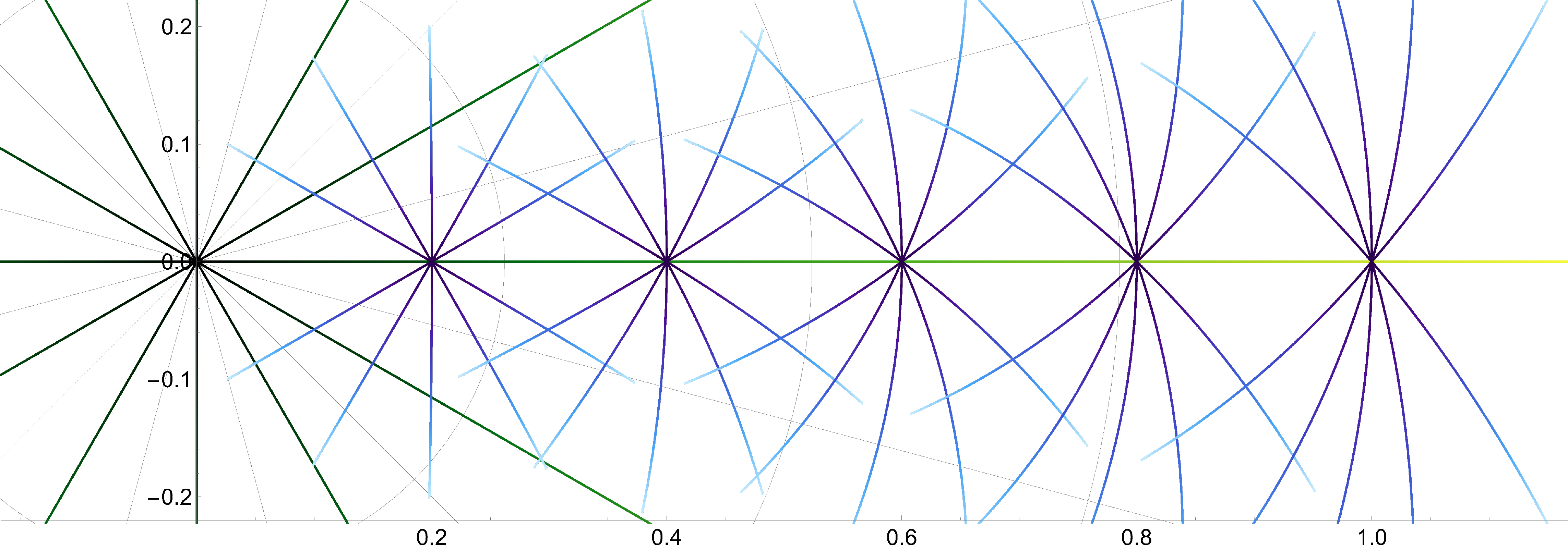}
    \end{minipage}
    \caption{We plot a sample of the geodesic flow bundle for the curvature wave at $\omega t = 0,\,\frac{\pi}{3},\,\frac{2\pi}{3}$ from top to bottom, at the points $c = 0,\,0.2,\,0.4,\,0.6,\,0.8,\,1$, $\varphi = 0$ in the directions $\beta = n\frac{\pi}{6}$, $n\in\mathbb{Z}$.}
    \label{fig: curvature wave GFB}
\end{figure}

\appendix

\section{Arcus functions}\label{sec: Arcfunc}
If we restrict sine to $\left[-\frac{\pi}{2},\frac{\pi}{2}\right]$ it becomes a bijective funcion and the inverse is well defined. We can derive the inverse to the sine operation, using the inverse function theorem:
\begin{equation}\label{eq: IFT}
    (f^{-1})'(y) = (f'(x))^{-1}, \quad y = f(x)
\end{equation}
\begin{align}
    \Rightarrow \quad (\sin^{-1})'(\sin x) = \frac{1}{\cos x} = \frac{1}{\sqrt{1-\sin^2x}}, \quad 
    x\in\left[-\frac{\pi}{2},\frac{\pi}{2}\right] \quad
    \text{and} \quad \sin^{-1}y = \int_0^y (\sin^{-1})'(z) dz
\end{align}
Sine on this domain assumes values between $-1$ and $1$, so arcsine is defined by the map:
\begin{equation}
    \begin{matrix}
    \sin^{-1}: & [-1,1] & \to & \left[-\frac{\pi}{2},\frac{\pi}{2}\right] \\ \\
    & y & \mapsto & \int_0^y \frac{1}{\sqrt{1-z^2}} dz 
    \end{matrix}
\end{equation}

The binomial series is absolutely convergent $\forall\alpha\in\mathbb{C}$ on the open unit disc $z\in B_1(0)\subset\mathbb{C}$.
\begin{equation}
    (1 + z)^\alpha = \sum_{k=0}^\infty \binom{\alpha}{k} z^k, \quad \binom{\alpha}{k} \coloneq \frac{(\alpha)_k}{k!},
\end{equation}
where we used the falling factorial $(\alpha)_k \coloneq \prod_{j=0}^{k-1}(\alpha-j)$ in the definition of the generalized binomial coefficients.\\
Since zero measures are irrelevant to integrals, we can use this power series on the entire range of arcsine.
\begin{equation}
    \sin^{-1}y = \int_0^y (1-z^2)^{-\frac{1}{2}} dz
    = \int_0^y \sum_{k=0}^\infty\binom{-\frac{1}{2}}{k}(-z^2)^k dz
\end{equation}

Writing out and manipulating the falling factorial allows us, to rewrite the binomial coefficient
\begin{align}
    \left(-\frac{1}{2}\right)_k =& \prod_{j=0}^{k-1}\left( -\frac{1}{2} - j \right)
    = (-1)^k\prod_{j=0}^{k-1}\left( \frac{1}{2} + j \right) = (-1)^k\prod_{j=0}^{k-1}\frac{2j + 1}{2}
    = \left( -\frac{1}{2} \right)^k 1\cdot3\cdot\ldots\cdot(2k-1) \notag \\
    =& \left( -\frac{1}{2} \right)^k 1\cdot\frac{2}{2\cdot1}\cdot3\cdot\frac{4}{2\cdot2}\cdot\ldots
    \cdot(2k-1)\cdot\frac{2k}{2\cdot k} = \left( -\frac{1}{2} \right)^k \frac{(2k)!}{2^kk!} \\
    \Rightarrow \quad \binom{-\frac{1}{2}}{k} =& \frac{(-1)^k(2k)!}{(2^kk!)^2}
\end{align}
and thereby see, that the alternation cancels. The absolute convergence allows us to swap the order of summation and integration.
\begin{align}
    \sin^{-1}y = \sum_{k=0}^\infty \int_0^y \frac{(2k)!}{(2^kk!)^2} z^{2k} dz = \sum_{k=0}^\infty \frac{(2k)!}{(2^kk!)^2} \frac{y^{2k+1}}{2k+1}
\end{align}

The cosine is just sine, shifted by $\frac{\pi}{2}$ and therefore we can use the same power series: $\cos^{-1}y = \frac{\pi}{2} - \sin^{-1}y$.\\
We restrict the hyperbolic cosine to $\mathbb{R}_+$ apply the inverse function theorem \eqref{eq: IFT} again, to get the integral expressions of the hyperbolic arc-sine and arc-cosine.
\begin{align}
    (\sinh^{-1})'(\sinh{x}) &= \frac{1}{\cosh{x}} = \frac{1}{\sqrt{1+\sinh^2x}}, \quad x\in\mathbb{R} \quad
    \text{and} \\
    (\cosh^{-1})'(\cosh{x}) &= \frac{1}{\sinh{x}}
    = \frac{1}{\sqrt{\cosh^2x-1}}, \quad x\in\mathbb{R}_+
\end{align}

\begin{align}
    \begin{matrix}
    \sinh^{-1}: & \mathbb{R} & \to & \mathbb{R} \\
    & y & \mapsto & \int_0^y \frac{1}{\sqrt{1+z^2}} dz
    \end{matrix} \qquad
    \begin{matrix}
    \cosh^{-1}: & [1,\infty) & \to & \mathbb{R}_+ \\
    & y & \mapsto & \int_1^y \frac{1}{\sqrt{z^2-1}} dz
    \end{matrix}
\end{align}

\begin{align}
    \sinh^{-1}y &= \int_0^y (1+z^2)^{-\frac{1}{2}} dz
    = \sum_{k=0}^\infty \binom{-\frac{1}{2}}{k} \int_0^y z^{2k} dz
    = \sum_{k=0}^\infty \frac{(-1)^k(2k)!}{(2^kk!)^2} \frac{y^{2k+1}}{2k+1}, \quad y\in(-1,1)
\end{align}

Using the substitution $z = \sqrt{\eta^2-1}$, $dz = \frac{\eta d\eta}{\sqrt{\eta^2-1}}$ we can assure ourselves, that $\cosh^{-1}y = \sin^{-1}\sqrt{y^2-1}$:
\begin{align}
    \sinh^{-1}\sqrt{y^2-1} = \int_0^{\sqrt{y^2-1}} \frac{1}{\sqrt{1+z^2}} dz = \int_1^y \frac{d\eta}{\sqrt{\eta^2-1}} d\eta = \cosh^{-1}y
\end{align}
The square root is always positive, so the hyperbolic arc-sine always maps it to a positive number and the identity holds for $y\in[1,\infty)$. This allows us to reuse the power series of the hyperbolic arc-sine for the hyperbolic arc-cosine, analogue to the non-hyperbolic case:
\begin{equation}
    \cosh^{-1}y = \sum_{k=0}^\infty \frac{(-1)^k(2k)!}{(2^kk!)^2} \frac{\sqrt{y^2-1}^{2k+1}}{2k+1} = \sqrt{y^2-1}\sum_{k=0}^\infty \frac{(-1)^k(2k)!}{(2^kk!)^2} \frac{(y^2-1)^k}{2k+1}, \quad y\in(1,\sqrt{2})
\end{equation}
We see, that the power series becomes an expansion around $1$ instead of $0$ and we have the condition $\sqrt{y^2-1}<1$, to assure absolute convergence.

\bibliographystyle{apsrev4-1}
\bibliography{GeodFBRef}